\documentclass{article}
\usepackage[margin=1in]{geometry}

\usepackage{amsmath}
\usepackage{amssymb}
\usepackage{amsthm}
\usepackage{mathtools}
\usepackage{bm}
\usepackage{booktabs}
\usepackage{mathrsfs}
\usepackage{enumerate}
\usepackage{color}
\usepackage{hyperref}
\usepackage{float}
\usepackage{xcolor}
\hypersetup{
  colorlinks   = true,
  urlcolor     = blue,
  linkcolor    = blue,
  citecolor   = red
}
\usepackage{dsfont}
\usepackage[square,numbers]{natbib}
\usepackage{graphicx}
\usepackage{subcaption}
\usepackage{enumitem}

\def\R{\mathbb{R}}

\def\P{\mathbb{P}}
\def\E{\mathbb{E}}

\def\N{\mathcal{N}}
\def\cA{\mathcal{A}}
\def\cF{\mathcal{F}}
\def\cH{\mathcal{H}}

\def\cT{\mathcal{T}}

\def\cR{\mathcal{R}}
\def\cE{\mathcal{E}}

\def\A{\bm{A}}

\def\M{\bm{M}}
\def\e{\bm{e}}
\def\m{\bm{m}}
\def\x{\bm{x}}
\def\y{\bm{y}}
\def\a{\bm{a}}
\def\b{\bm{b}}
\def\g{\bm{g}}

\def\G{\bm{G}}
\def\H{\bm{H}}
\def\I{\bm{I}}

\def\n{\bm{n}}
\def\X{\bm{X}}
\def\Q{\bm{Q}}

\def\W{\bm{W}}
\def\u{\bm{u}}
\def\r{\bm{r}}
\def\s{\bm{s}}
\def\q{\bm{q}}
\def\v{\bm{v}}
\def\w{\bm{w}}
\def\z{\bm{z}}
\def\bS{\bm{S}}

\def\bR{\bm{R}}
\def\bZ{\bm{Z}}
\def\bO{\mathbf{O}}

\let\epsilon\varepsilon

\def\btheta{\boldsymbol{\theta}}
\def\bbeta{\boldsymbol{\eta}}

\def\bphi{\boldsymbol{\phi}}

\def\bomega{\boldsymbol{\omega}}
\def\eps{\varepsilon}
\def\beps{\boldsymbol{\varepsilon}}

\def\sP{\mathsf{P}}
\def\sS{\mathbb{S}}

\DeclareMathOperator{\Var}{Var}
\DeclareMathOperator{\Cov}{Cov}
\DeclareMathOperator{\corr}{corr}

\DeclareMathOperator{\err}{err}

\DeclareMathOperator{\diag}{diag}
\DeclareMathOperator{\Tr}{Tr}

\DeclareMathOperator{\sech}{sech}
\DeclareMathOperator{\spa}{span}
\DeclareMathOperator{\poly}{poly}
\newcommand{\op}{\mathrm{op}}
\DeclareMathOperator{\low}{low}
\DeclareMathOperator{\AMP}{{}}
\newcommand{\weak}{\mathsf{wk}}
\newcommand{\str}{\mathsf{str}}
\newcommand{\SE}{\mathsf{SE}}
\DeclareMathOperator{\fix}{\mathsf{fix}}

\DeclareMathOperator{\Ons}{Ons}

\newcommand{\pnorm}[2]{\lVert #1\rVert_{#2}}

\newcommand{\biggpnorm}[2]{\bigg\lVert#1\bigg\rVert_{#2}}
\newcommand{\Bigpnorm}[2]{\Big\lVert#1\Big\rVert_{#2}}
\newcommand{\iprod}[2]{\left\langle#1,#2\right\rangle}

\newcommand{\av}[1]{\left\langle#1\right\rangle}

\newcommand{\ampskel}{%
  \vphantom{\displaystyle\sum_{\ell=0}^{t}\gamma_{t,\ell}\g^\ell}%
}

\newcommand{\ampblock}[2]{%
  \begin{array}{@{}c@{}}
  \ampskel #1 \\[0.5ex]
  \ampskel #2
  \end{array}%
}

\newcommand{\amprel}[1]{%
  \;\begin{array}{@{}c@{}}
  \ampskel \mathrel{#1} \\[0.5ex]
  \ampskel \mathrel{#1}
  \end{array}\;%
}

\newcommand{\ampbin}[1]{%
  \;\begin{array}{@{}c@{}}
  \ampskel \mathbin{#1} \\[0.5ex]
  \ampskel \mathbin{#1}
  \end{array}\;%
}

\newcommand{\amplabel}[2]{%
  \underbrace{#1}_{\mathclap{\textsf{#2}}}%
}

\newtheorem{theorem}{Theorem}[section]
\newtheorem{assumption}[theorem]{Assumption}
\newtheorem{definition}[theorem]{Definition}
\newtheorem{lemma}[theorem]{Lemma}

\newtheorem{proposition}[theorem]{Proposition}
\newtheorem{corollary}[theorem]{Corollary}

\title{Approximate Message Passing with Random Initialization for Phase Retrieval}
\author{Yuchen Chen, Yandi Shen, Xingyu Xu}

\begin{document}
\maketitle

\begin{abstract}
Randomly initialized first-order iterative algorithms are simple, widely used, and empirically effective across many high-dimensional inference problems, yet their theoretical behavior remains far less understood than that of algorithms started from more elaborate, problem-specific initializations such as spectral initialization. This paper takes a step toward closing this gap by studying randomly initialized approximate message passing (AMP) for noiseless phase retrieval, a canonical nonconvex inference problem. We give a non-asymptotic, growing-horizon characterization of the AMP trajectory from random initialization, and show that it achieves both the information-theoretic weak recovery threshold $\delta_{\weak}=1/2$ and the AMP strong recovery threshold $\delta_{\str}^{\AMP}\approx1.13$, matching the recovery thresholds established for spectrally initialized AMP. The key technical ingredients are twofold: (i) a non-asymptotic coupling that decomposes the AMP iterates as a signal component, a linear combination of independent Gaussian vectors, and an approximation error, and (ii) a second-order analysis of the Onsager correction error built around a non-asymptotic version of Bolthausen's condition. These allow us to analyze the $O_\delta(\log n)$ steps required for weak recovery, and to faithfully track the asymptotic state evolution up to $n^{1/3}/\operatorname{polylog}(n)$ steps in the intermediate recovery regime $\delta\in(\delta_{\weak},\delta_{\str}^{\AMP})$, both beyond the $O(1)$-iteration horizon of classical AMP theory. Above the strong recovery threshold, AMP reaches any prescribed fixed recovery accuracy within $O_{\delta,\varepsilon}(\log n)$ iterations. Both the coupling and the error analysis are developed for a broad class of generalized AMP dynamics and may be of independent interest for general single-index models.
\end{abstract}

\setcounter{tocdepth}{2}
\tableofcontents

\section{Introduction}

Consider the noiseless phase retrieval problem
\begin{align}\label{eq:model}
y_i = (\x_i^\top \btheta^\star)^2, \quad i=1,\ldots,n,
\end{align}
where the sensing vectors $\x_i$ are drawn independently from $\N(0,\I_d/d)$ and $\btheta^\star$ is an unknown signal. The objective is to reconstruct $\btheta^\star$, up to its global sign, from $\{(\x_i,y_i)\}_{i=1}^n$. Phase retrieval is a canonical nonlinear single-index model for studying the information-theoretic and algorithmic limits of high-dimensional inference.

Earlier reconstruction methods were based primarily on semidefinite relaxations \cite{candes2015phase,candes2013phaselift} and related convex formulations \cite{goldstein2018phasemax,walkspurger2015phase}. More recent methods include gradient descent \cite{cai2016optimal,chen2017solving,zhang2017nonconvex,wang2017solving,luo2019solving,chen2019gradient,chi2019nonconvex,Soltanolkotabi2019structured,gao2020perturbed,li2020toward,hand2020optimal,ma2020implicit,cai2022solving}, online stochastic gradient descent \cite{tan2023online,arous2021online}, and message-passing algorithms \cite{mondelli_approximate_2022,maillard2020phase,ma2019optimization}. Because of the sign symmetry, these iterative methods are commonly analyzed from an informative initialization such as spectral initialization which uses the leading eigenvector of
\begin{align}\label{eq:spec_intro}
\M_n = \frac{1}{n}\sum_{i=1}^n \cT_s(y_i)\x_i\x_i^\top,
\end{align}
for a suitable preprocessing function $\cT_s$. In the proportional asymptotic regime $n/d\rightarrow\delta\in(0,\infty)$, the optimal spectral method attains the information-theoretic weak-recovery threshold $\delta_{\weak}=1/2$ \cite{lu2020phase,mondelli2019fundamental}. Subsequent
refinement by approximate message passing (AMP) and gradient descent has also been analyzed \cite{mondelli_approximate_2022,chen2025learning}.

Random initialization avoids the construction of a model-specific initial estimator, but it is less well understood than spectral initialization \cite{chen2017solving,wang2017solving,zhang2017nonconvex,ma2020implicit,lu2020phase,mondelli2019fundamental,mondelli_approximate_2022}. Recent work establishes weak recovery with $n=\widetilde O(d)$ samples for several single-index models \cite{chen2019gradient,tan2019phase,berthier2025learning,arous2021online,arnaboldi2024repetita,lee2024neural,montanari2025dynamical} and multi-index models \cite{abbe2022merged,abbe2023sgd,arnaboldi2023high,dandi2024benefits,zhang2025neural}. These results generally do not identify the sharp proportional-regime sample complexity, and therefore do not permit a precise comparison between algorithms. This motivates the following question:
\begin{quote}
\emph{Can a randomly initialized iterative algorithm attain the sharp recovery (both weak and strong) thresholds in the proportional regime?}
\end{quote}

In this work, we answer this question in the affirmative for randomly initialized Bayes-optimal AMP for phase retrieval. Under the normalization $\pnorm{\btheta^\star}{}=\sqrt d$, we initialize the generalized AMP algorithm \cite{rangan2012iterative,javanmard2013state,barbier2019optimal} by
\begin{align}\label{eq:init_intro}
\v^0 \sim \N(0, \I_d),
\end{align}
set $\u^0=\X f_0(\v^0)$, and iterate
\begin{equation}\label{eq:main_amp}
\begin{aligned}
\v^{t+1} &= \X^\top h_t(\u^t,\y) - \hat c_t f_t(\v^t),\\
\u^{t+1} &= \X f_{t+1}(\v^{t+1}) - \hat b_{t+1}h_t(\u^t,\y).
\end{aligned}
\end{equation}
Here $\hat c_t,\hat b_{t+1}$ are Onsager correction coefficients defined in Section~\ref{sec:main_result}, and $f_t:\R\rightarrow\R$ and $h_t:\R^2 \rightarrow \R$ are the normalized Bayes-optimal nonlinearities (see Appendix~\ref{sec:bayes_amp} for a review).


The following is an informal statement of the main result; see Theorem~\ref{thm:main} for the precise statement.
\begin{theorem}[Informal]\label{thm:informal_main}
Let $\delta_{\weak}=1/2$ and $\delta_{\str}\approx1.13$. Fix $\delta>\delta_{\weak}$ with $\delta\neq\delta_{\str}$, and let $n/d=\delta$. The distribution of the AMP iterate $\v^{t}$ satisfies
\begin{align}\label{eq:distribution_intro}
\sP(\v^{t}) \approx \sP\Big(\rho_t\btheta^\star + \bZ_t\Big),
\end{align}
where $\bZ_t\sim\N(0,\I_d)$, and $\{\rho_t\}_{t\geq0}$ are random scalar coefficients that may depend on $\bZ_t$. The following conclusions hold with high probability:
\begin{itemize}
\item (Weak recovery) There exist $c_{\weak}>0$ and $\tau_{\weak}=O(\log n)$ such that $|\rho_{\tau_{\weak}}|\geq c_{\weak}$.
\item (Intermediate recovery) If $\delta\in(\delta_{\weak},\delta_{\str})$, then there is a finite positive state-evolution fixed point $\rho_\infty^2(\delta)$ such that, uniformly for $\tau_{\weak}\leq t\leq\widetilde O(n^{1/3})$,
\begin{align}\label{eq:intermedite_recovery_intro}
\left|\rho_t^2-\rho_\infty^2(\delta)\right| \leq C_\delta\exp\Big(-\frac{t-\tau_{\weak}}{C_\delta}\Big) + \widetilde O\Big(\frac{t+1}{\sqrt n}\Big).
\end{align}
\item (Strong recovery) If $\delta>\delta_{\str}$, then, for every fixed $\eps>0$, there exists $\tau_\eps=O_{\delta,\eps}(\log n)$ at which the correlation between $\v^{\tau_\eps}$ and $\btheta^\star$ is at least $1-\eps$.
\end{itemize}
\end{theorem}

The distributional characterization (\ref{eq:distribution_intro}) states that the AMP iterate $\v^t$ is correlated with the signal $\btheta^\star$ with strength $\rho_t$ while maintaining a constant noise level (the latter follows from the normalization of $f_t,h_t$ in (\ref{eq:main_amp})). Under the initialization (\ref{eq:init_intro}), the initial signal strength
\begin{align}\label{eq:rho_0}
\rho_0 := \frac{1}{d}\iprod{\v^0}{\btheta^\star}
\end{align}
is of the order $\rho_0 = O\big(1/\sqrt{d}\big)$, and the evolution of $\{\rho_t\}_{t\geq 0}$ naturally decomposes into three phases:
\begin{itemize}
\item {\em Phase I: Symmetry breaking.} Initially, $\rho_t$ follows the recursion
\begin{align*}
\rho_{t+1} \approx \sqrt{2\delta}\, \rho_t + \varphi_t,
\end{align*}
where, conditionally on the past $\{\rho_s\}_{s\leq t}$, the perturbation $\varphi_t$ is approximately $\N(0,5/d)$. In this phase $\varphi_t$ is of the same order as $\rho_t$, so it cannot be discarded. When $\delta>\delta_{\weak}$, the factor $\sqrt{2\delta}$ exceeds one and amplifies these fluctuations. Following \cite{li2023approximate}, an anti-concentration argument then shows that within $O(\log n)$ iterations $|\rho_t|$ grows from order $n^{-1/2}$ to order $n^{-1/4}$, up to logarithmic factors.
\item {\em Phase II: Exponential growth.} Once $|\rho_t|$ reaches order $n^{-1/4}$, it dominates the Gaussian perturbation and grows as  $|\rho_{t+1}| \geq (1+c) |\rho_t|$ for some $c>0$ depending on $\delta$, until $|\rho_t|$ reaches a fixed positive constant. Thus weak recovery occurs within $O(\log n)$ iterations.
        
\item {\em Phase III: State evolution.} Once weak recovery is achieved, the distribution of $\v^t,\u^t$ tracks the asymptotic state evolution $\rho_{t+1}^2=F_\delta(\rho_t^2)$, with $F_\delta$ defined in \eqref{eq:def_F_delta}; see Section~\ref{sec:main_result}. This recursion behaves in two distinct ways. When $\delta\in(\delta_{\weak},\delta_{\str})$, it converges to a finite fixed point $\rho_\infty^2(\delta)$, as in \eqref{eq:intermedite_recovery_intro}, and $\rho_\infty^2(\delta)$ approaches approximately $5.52$ as $\delta\uparrow\delta_{\str}$. By contrast, when $\delta>\delta_{\str}$, there is no positive fixed point: $|\rho_t|$ grows exponentially as
\begin{align*}
\rho_{t+1}^2 \geq \Big(\frac{\delta}{\delta_{\str}} + o(1)\Big)\rho_t^2.
\end{align*}
This shows that the final correlation achieved by AMP undergoes a discontinuous phase transition at $\delta = \delta_{\str}$.
\end{itemize}

Theorem~\ref{thm:informal_main} states that randomly initialized Bayes-optimal AMP achieves weak recovery above $\delta_{\weak}=1/2$ and arbitrarily accurate recovery above $\delta_{\str}\approx1.13$. The weak-recovery threshold coincides with the information-theoretic threshold under a Gaussian prior derived in \cite{mondelli2019fundamental}, while the strong-recovery threshold matches the best currently known threshold under a Gaussian prior for AMP \cite{barbier2019optimal,celentano2020estimation}.

Both thresholds can be read off heuristically from the Bayes-optimal state evolution. For the weak-recovery threshold, however, there is also a complementary interpretation from the perspective of spectral methods: Phases I and II described above approximately compute the spectral initialization \eqref{eq:spec_intro} with optimal processing function $\cT_s(y) = \frac{y-1}{\sqrt{2\delta} + y - 1}$, whose optimality is known to hold for general single-index models under a symmetry condition \cite{mondelli2019fundamental}; see Section \ref{sec:main_result} for a more detailed discussion of the relation to spectral methods.

Beyond the thresholds, Theorem~\ref{thm:informal_main} tracks the AMP trajectory over a growing horizon: the estimate \eqref{eq:intermedite_recovery_intro} holds for a number of iterations of order $n^{1/3}$, up to polylogarithmic factors, whereas the classical AMP analysis covers an $O(1)$ horizon \cite{bolthausen2014iterative,bayati_message_passing_2011,javanmard2013state}. Reaching this horizon is the main technical contribution of the paper. The proof has two main ingredients.

The first is an exact finite-sample coupling between the AMP trajectory $(\u^t,\v^t)$ in \eqref{eq:main_amp} and the random data $(\X,\y)$, which decomposes each iterate into a signal component, a normalized Gaussian component, and a residual. Under the coupling this decomposition holds exactly, and the subsequent analysis shows that the residual is small with high probability, jointly over the growing time horizons considered in the theorem. The decomposition is related to the Householder representation of \cite{lu2021householder} and to the non-asymptotic AMP decompositions of \cite{li2022non,li2024non}.

The second is a non-asymptotic error-recursion analysis that controls this residual, which is created by the adaptive Onsager correction. In the intermediate regime, a second-order surrogate removes the leading adaptive terms and produces a contracting error recursion, which yields the $n^{1/3}/\operatorname{polylog}(n)$ time horizon. In the strong-recovery regime, the proof instead uses direct first-order recursions until the signal coefficient reaches an arbitrary fixed target.

Both ingredients are developed for the AMP dynamics \eqref{eq:main_amp} with general nonlinearities in the single-index model
\begin{align}\label{def:single_index}
\y=\varphi(\X\btheta^\star),
\end{align}
where the link is applied entrywise and may include external randomness. Section~\ref{subsec:intro_amp_nonasymptotic} gives a technical summary, and Section~\ref{sec:main_proof} the precise statements.

\medskip

\noindent\textbf{Further related work.}
\begin{itemize}
\item {\em Bayes-optimal AMP.} Bayes-optimal generalized AMP has been analyzed for a broad class of high-dimensional inference problems \cite{mondelli2019fundamental,barbier2019optimal,maillard2020phase,celentano2020estimation,mondelli_approximate_2022,montanari2024statistically}. For several estimation problems, it is optimal among broad classes of first-order methods at any fixed iteration \cite{celentano2020estimation,montanari2024statistically}. The present work addresses the growing number of iterations required by a random initialization.

\item {\em Learning Gaussian single- and multi-index models.} For Gaussian single-index models, the sample complexity of gradient-based learning is often of order $\widetilde O(d^\ell)$ \cite{dudeja2018learning,chen2019gradient,tan2019phase,arous2021online,damian2023smoothing,arnaboldi2024repetita,lee2024neural,berthier2025learning,montanari2025dynamical,kovavcevic2026full}, where $\ell$ is determined by the information or generative exponent of the link function \cite{arous2021online,damian2024computational}. Sharp constant-level thresholds in the proportional regime are known principally for spectral methods \cite{lu2020phase,mondelli2019fundamental,luo2019optimal} and message-passing algorithms \cite{barbier2019optimal,mondelli2019fundamental,mondelli_approximate_2022,maillard2020phase}. Computational lower bounds for more general link functions are studied in \cite{damian2022neural,damian2024computational}. Related results for multi-index models appear in \cite{damian2022neural,abbe2023sgd,arnaboldi2023high,dandi2024benefits,collins2024hitting,zhang2025neural,montanari2026phase,troiani2024fundamental,zhang2026precise,kovavcevic2025spectral,defilippis2025optimal,diakonikolas2025robust}, with lower bounds in \cite{troiani2024fundamental,damian2025generative,diakonikolas2026algorithms,latourelle2026statistical}.

\item {\em Additional phase retrieval literature.} Convex formulations and their proportional-regime behavior are studied in \cite{thrampoulidis2014gaussian,abbasi2017performance,dhifallah2017phase}. The nonconvex optimization landscape and its relation to gradient methods are studied in \cite{sun2018geometric,li2020toward,sarao2019afraid,sarao2020complex,sarao2020marvels,antenucci2019glassy,cai2022solving,maillard2020landscape}. Spectral methods are analyzed in \cite{netrapalli2013phase,lu2020phase,luo2019optimal,maillard2022construction}. AMP was applied to phase retrieval in \cite{schniter2015compressive}; an optimization-based AMP formulation was studied in \cite{ma2019optimization}.
\end{itemize}

\noindent\textbf{Paper organization.} The rest of the paper is organized as follows. Section \ref{sec:main_result} introduces the Bayes-optimal AMP algorithm and states the main result. Section \ref{sec:main_proof} presents the non-asymptotic coupling and error recursion analysis. Section \ref{sec:proof_PR} outlines the analysis of the three phases and completes the proof of the main theorem. Omitted proofs and technical details are collected in Appendices~\ref{Sec:Error_Proofs}--\ref{sec:proof_strong_recovery}. Appendix~\ref{sec:first_order_analysis} provides the first-order analysis of the AMP approximation errors used in the strong-recovery regime, and Appendix~\ref{sec:background} gives an overview of Bayes-optimal AMP for single-index models and of Bolthausen's condition in the asymmetric AMP setting.

\medskip

\noindent\textbf{Notation.} Bold quantities such as $\btheta^\star, \X$ will denote objects with growing dimension (with respect to $n,d$). For a vector $\v \in \R^d$, we use $\pnorm{\v}{}$ to denote the $2$-norm, $\diag(\v)$ to denote the diagonal matrix with $\v$ as its diagonal, and $\av{\v} = \frac{1}{d}\sum_{i=1}^d v_i$. For a matrix $\M$, we use $\pnorm{\M}{F},\pnorm{\M}{\op},\Tr(\M)$ to denote the usual Frobenius norm, operator norm, and trace. For two matrices $\M,\M'$, we use $\M\circ\M'$ to denote the Hadamard (entrywise) product. For a matrix $\M$ with orthonormal columns, we denote the projection to its span by $P_{\M} = \M\M^\top$ and the projection to its orthonormal complement by $P_{\M}^\perp = I-P_{\M}$. For a random variable $X$, we use $\sP(X)$ to denote its law. We use $\cF(\cdot)$ to denote the generated sigma-field. Unless otherwise indicated, $C,C',c,c'>0$ denote universal constants, independent of $n,d,\delta$, which may change from line to line; $C_\delta,c_\delta$ (and $C_{\delta,\varepsilon}$, $C_{\delta,M}$) denote constants depending only on the indicated parameters. For two nonnegative sequences $a_n,b_n$, we write $a_n=O(b_n)$ if $a_n\leq Cb_n$, and $a_n=O_\delta(b_n)$ if $a_n\leq C_\delta b_n$; $O_{\delta,\varepsilon}$ and $O_{\delta,M}$ are defined in the same way.  We use ``for $n$ large enough'' to mean for $n$ larger than a constant depending only on $\delta$ (and on $\varepsilon$ or $M$ where indicated). We use $\widetilde O,\widetilde\Omega,\widetilde\Theta$ to suppress polylogarithmic factors in $n,d$.

\section{Main result}\label{sec:main_result}

\subsection{AMP and state evolution}
Given functions $f_t:\R\rightarrow \R$ and $h_t:\R^2 \rightarrow \R$, and initialization $\v^0$ independent of the data $(\X,\y)$, the (generalized) AMP algorithm \cite{rangan2012iterative, javanmard2013state, barbier2019optimal} takes the form \eqref{eq:main_amp}, which we repeat here for convenience:
\begin{equation*}
\begin{aligned}
\v^{t+1} &= \X^\top h_t(\u^t,\y) - \hat c_t f_t(\v^t),\\
\u^{t+1} &= \X f_{t+1}(\v^{t+1}) - \hat b_{t+1}h_t(\u^t,\y),
\end{aligned}
\end{equation*}
where the first iterate is $\u^0 = \X f_0(\v^0)$. Here $f_t$ and $h_t$ are applied entrywise, $h_t'$ denotes the derivative of $h_t$ in its first argument, and the Onsager coefficients are
\begin{align}\label{eq:general_onsager}
\hat c_t = \frac{1}{d}\sum_{i=1}^n h'_t(u^t_i, y_i), \quad \hat b_{t+1} = \frac{1}{d}\sum_{j=1}^d f'_{t+1}(v^{t+1}_j),
\end{align}
assuming these derivatives exist. Under appropriate asymptotic assumptions, the empirical distributions of $(\v^{t+1},\btheta^\star)$ and $(\u^t, \bbeta^\star)$ with $\bbeta^\star := \X\btheta^\star$ admit an almost sure limit:
\begin{align*}
\frac{1}{d}\sum_{j=1}^d \delta_{(v^{t+1}_j, \theta^\star_j)} \overset{W_2}{\to} \sP(V_{t+1},\theta^\star), \quad \frac{1}{n}\sum_{i=1}^n \delta_{(u^{t}_i, \eta^\star_i)} \overset{W_2}{\to} \sP(U_t, \eta^\star),
\end{align*}
where the asymptotic variables are given by
\begin{align*}
V_{t+1} = \mu_{V,t+1}\theta^\star + \sigma_{V,t+1}W_{V,t+1},\quad U_{t} = \mu_{U,t}\eta^\star + \sigma_{U,t}W_{U,t}, \quad Y = (\eta^\star)^2,
\end{align*}
and the four scalar sequences $\{\mu_{U,t},\sigma_{U,t}^2,\mu_{V,t+1}, \sigma_{V,t+1}^2\}_{t \geq 0}$ evolve by the state evolution equation
\begin{equation}\label{eq:gamp_se}
\begin{aligned}
\mu_{U,t} &= \E[\theta^\star f_t(V_t)],\\
\sigma_{U,t}^2 &= \E[f_t(V_t)^2] - \mu_{U,t}^2,\\
\mu_{V,t+1} &= \delta\Big(\E[\eta^\star h_t(U_t,Y)] - \E[h'_t(U_t,Y)]\E[\theta^\star f_t(V_t)]\Big),\\
\sigma_{V,t+1}^2 &= \delta\E[h_t(U_t,Y)^2],
\end{aligned}
\end{equation}
initialized from $(\mu_{U,0}, \sigma_{U,0}^2) = \lim_{n,d\rightarrow\infty} (\frac{1}{d}\iprod{\btheta^\star}{f_0(\v^0)}, \frac{1}{d}\pnorm{P_{\btheta^\star}^\perp f_0(\v^0)}{}^2)$, assuming these limits exist almost surely. Here the joint law of $(V_{t+1},U_t,Y)$ is specified with mutually independent $\theta^\star, \eta^\star, W_{V,t+1},W_{U,t}$, where $\theta^\star$ (resp. $\eta^\star$) is drawn from the limiting empirical distribution of $\btheta^\star$ (resp. $\bbeta^\star$), and $W_{V,t+1},W_{U,t}$ are standard Gaussian variables. We review more background on AMP in Appendix \ref{sec:bayes_amp} and refer to \cite{rangan2012iterative, javanmard2013state, barbier2019optimal} for detailed treatments.

In this work, we focus on the Bayes-optimal choice of nonlinearities in (\ref{eq:main_amp}) (see Appendix  \ref{sec:bayes_amp} for a review of Bayes-optimal AMP for general single-index models), by taking for all $t\geq 0$
\begin{align}\label{eq:rescaled_nonlinearity}
f_t(x) = x, \quad h_t(u,y) = \frac{h_t^\star(u,y)}{\pnorm{h_t^\star(\u^t,\y)}{}/\sqrt{d}}. 
\end{align}
Here 
\begin{align}\label{eq:ht_ast}
h^\star_t(u,y) = (\hat\mu_t^2 + 1)\sqrt y \tanh\Big(\hat\mu_t u\sqrt{y}\Big) - \hat\mu_t u,
\end{align}
and we use the following data-driven $\hat\mu_t$  as a proxy for $\mu_{U,t}$ in (\ref{eq:gamp_se}):
\begin{align}\label{eq:hat_mu_t}
\hat\mu_t = \sqrt{\Big(\frac{1}{n}\pnorm{\u^t}{}^2 - 1\Big) \vee n^{-1/2}}.
\end{align}
With this choice, the general AMP dynamic (\ref{eq:main_amp}) is simplified as
\begin{align}\label{eq:simplified_amp}
\v^{t+1} = \X^\top h_t(\u^t,\y) - \hat{c}_t\v^t, \qquad \u^{t+1} = \X \v^{t+1} - h_t(\u^t,\y),
\end{align}
where the Onsager coefficient is given by
\begin{align}\label{eq:onsager_c}
\hat c_t = \frac{1}{d}\sum_{i=1}^n \frac{(\hat\mu_t^2 +1)\hat\mu_t y_i \big(1-\tanh^2(\hat\mu_tu^t_i\sqrt{y_i})\big) - \hat\mu_t}{\pnorm{h_t^\star(\u^t,\y)}{}/\sqrt{d}}.
\end{align}
With the random initialization (\ref{eq:init_intro}), the simplified AMP (\ref{eq:simplified_amp}) is executed in the order of
\begin{align*}
\v^0 \rightarrow \u^0 (\rightarrow h_0) \rightarrow \v^1 \rightarrow \u^1 (\rightarrow h_1) \rightarrow \v^2 \rightarrow \ldots.
\end{align*}

We need a few more definitions to state our main result. Let 
\begin{align}\label{eq:m_function}
m(\mu) = \E_{G,W\sim \N(0,1)} \Big[G^2\tanh^2(\mu G|G| + \sqrt{\mu}W|G|)\Big], \quad \mu\in\R_+
\end{align}
and
\begin{align}\label{eq:def_F_delta}
F_\delta(\mu)\coloneqq\delta(1+\mu)\big((1+\mu)m(\mu)-\mu\big).
\end{align}
Further define
\begin{align}\label{eq:strong_recovery_threshold}
\delta_{\str}^{\AMP} := \sup_{\mu > 0} \frac{\mu\delta}{F_\delta(\mu)} =  \sup_{\mu > 0} \frac{\mu}{(\mu + 1) \big((\mu + 1) m(\mu) - \mu\big)},
\end{align}
which is numerically evaluated as $\delta_{\str}^{\AMP} \approx 1.13$. For $\delta\in(\delta_{\weak},\delta_{\str}^{\AMP})$, let $\mu_\infty(\delta)$ denote the
smallest positive fixed point of $F_\delta$, and set
$\rho_\infty(\delta)\coloneqq\sqrt{\mu_\infty(\delta)}$ (see Proposition~\ref{prop:phase_3_asymp} for their existence and properties).

We are now ready to state the main result. We give the main steps of its proof in Sections \ref{subsec:couple_amp}--\ref{subsec:analysis_phase3}, and complete the proof in Section \ref{subsec:complete_main}.

\begin{theorem}\label{thm:main}
Fix $\delta>\delta_{\weak}=1/2$ with $\delta\neq\delta_{\str}$, and let $n/d=\delta$. The AMP trajectory can be coupled to scalar coefficients $\{\rho_t\}$ and independent vectors $\s^\ell\sim\N(0,\I_d)$ so that, for $t\geq1$,
\begin{align}\label{eq:main_decomp}
\v^t=\rho_t\btheta^\star+\sum_{\ell=0}^{t-1}\lambda_{t-1,\ell}\s^\ell+\Delta_{V,t},
\end{align}
where $\pnorm{\lambda_{t-1,[0:t-1]}}{}=1$, and the coefficients $\lambda_{t-1,\ell}$ may be arbitrarily correlated with $\{\s^\ell\}_{\ell=0}^{t-1}$. There exist a universal constant $C>0$ and constants $C_\delta,c_\delta>0$ depending only on $\delta$ such that the following hold.
\begin{itemize}
\item \textnormal{\textbf{Weak recovery.}} With probability $1-O_\delta(n^{-10})$, there
is a time $\tau_{\weak}\leq C_\delta\log n$ such that
\begin{align*}
|\rho_{\tau_{\weak}}|\geq c_\delta,\qquad
\pnorm{\Delta_{V,\tau_{\weak}}}{}\leq C_\delta\log^{12}n.
\end{align*}

\item \textnormal{\textbf{Intermediate recovery.}} If $\delta_{\weak}<\delta<\delta_{\str}$, initialize $\rho_{\SE,\tau_{\weak}}^2=\rho_{\tau_{\weak}}^2$ and iterate $\rho_{\SE,t+1}^2=F_\delta(\rho_{\SE,t}^2)$. On the same event, uniformly for
\begin{align*}
\tau_{\weak}\leq t\leq\frac{n^{1/3}}{C_\delta\log^{20}n},
\end{align*}
we have
\begin{align*}
|\rho_t^2-\mu_\infty(\delta)|
&\leq|\rho_{\SE,t}^2-\mu_\infty(\delta)|+C_\delta\frac{(t+1)\log^{12}n}{\sqrt n},\\
\pnorm{\Delta_{V,t}}{}&\leq C_\delta(t+1)\log^{12}n.
\end{align*}
Moreover,
\begin{align*}
|\rho_{\SE,t}^2-\mu_\infty(\delta)|
\leq C_\delta\exp\Big\{-\frac{t-\tau_{\weak}}{C_\delta}\Big\}.
\end{align*}

\item \textnormal{\textbf{Strong recovery.}} If $\delta>\delta_{\str}$, fix $\eps\in(0,1)$. With probability $1-O_{\delta,\eps}(n^{-10})$, there is a
time $\tau_\eps\leq C_{\delta,\eps}\log n$ satisfying
\begin{align*}
|\rho_{\tau_\eps}|\geq\frac{2}{\sqrt\eps},\qquad
\pnorm{\Delta_{V,\tau_\eps}}{}\leq C_{\delta,\eps}\log^C n.
\end{align*}
Consequently, defining the correlation $\corr(\v,\btheta) := |\langle\v,\btheta\rangle|/(\pnorm{\v}{}\,\pnorm{\btheta}{})$ for nonzero $\v,\btheta\in\R^d$, we have
\begin{align}\label{eq:strong_correlation}
\corr(\v^{\tau_\eps},\btheta^\star)\geq1-\eps.
\end{align}
\end{itemize}
\end{theorem}
The decomposition \eqref{eq:main_decomp} writes the AMP iterate $\v^t$ as the sum of three components: (i) the signal component $\rho_t\btheta^\star$; (ii) the Gaussian component $\sum_{\ell=0}^{t-1}\lambda_{t-1,\ell}\s^\ell$; and (iii) an error term $\Delta_{V,t}$, whose norm is at most $\widetilde O(t)$ in all regimes. The coefficients $\rho_t$ and $\{\lambda_{t-1,\ell}\}_{\ell=0}^{t-1}$ depend on the Gaussian vectors $\{\s^\ell\}_{\ell=0}^{t-1}$ in a complicated way, but there are only $t+1$ of them, compared with the dimension $d$ of each $\s^\ell$. Hence, for moderate $t$, uniform concentration arguments can be used to understand summary statistics of $\v^t$ (see for example \cite[Lemma 9]{li2022non}).

Let us highlight one such consequence of Theorem~\ref{thm:main} on the phase transition of the correlation between $\v^t$ and the signal $\btheta^\star$. Its proof is given in Section \ref{subsec:complete_main}.

\begin{corollary}\label{cor:corr}
Fix $\delta\in(\delta_{\weak},\delta_{\str})$. With probability $1-O_\delta(n^{-10})$, for all $\tau_{\weak}\leq t\leq n^{1/3}/(C_\delta\log^{20}n)$,
\begin{align}\label{eq:corr_intermediate}
\Big|\corr(\v^t,\btheta^\star)-\frac{\rho_\infty(\delta)}{\sqrt{1+\rho_\infty^2(\delta)}}\Big|
\leq C_\delta\exp\Big\{-\frac{t-\tau_{\weak}}{C_\delta}\Big\}+C_\delta n^{-1/6}.
\end{align}
In particular, on the same event, the left-hand side is at most $C_\delta n^{-1/6}$ for
all $\tau_{\weak}+C_\delta\log n\leq t\leq n^{1/3}/(C_\delta\log^{20}n)$.

On the other hand, for every fixed $\delta>\delta_{\str}$ and $\eps\in(0,1)$, there exists
$C_{\delta,\eps}>0$ such that
\begin{align}\label{eq:strong-correlation-iterated-limit}
\lim_{\substack{n,d\rightarrow\infty\\n/d=\delta}} \P\Big(\exists\,t\leq C_{\delta,\eps}\log n:\corr(\v^t,\btheta^\star)\geq1-\eps\Big)=1.
\end{align}
\end{corollary}
Corollary \ref{cor:corr} completes the following picture of the phase diagram of the final correlation. Below $\delta_{\weak}$, weak recovery is information-theoretically impossible under a Gaussian prior \cite{mondelli2019fundamental}. Between the two thresholds, the correlation converges to the deterministic value in the corollary, with left limit approximately $0.92$ at $\delta_{\str}$. Above $\delta_{\str}$, AMP reaches any fixed accuracy as given in \eqref{eq:strong-correlation-iterated-limit}. Figure~\ref{fig:univ-phase} illustrates this phase diagram and provides numerical evidence.

Next we discuss a few other aspects of Theorem \ref{thm:main}.
\begin{itemize}
\item {\em Rotational equivariance of AMP.} Theorem \ref{thm:main} differs from the classical asymptotic AMP results \cite{bayati_message_passing_2011,javanmard2013state,barbier2019optimal} in two ways: it does not require the pair $(\v^0,\btheta^\star)$ to have a deterministic limiting empirical distribution, and the characterization (\ref{eq:main_decomp}) concerns the entire $d$-dimensional law of $\v^t$ rather than its empirical distribution. Both features come from the rotational invariance of the Gaussian matrix $\X$. Let $\bO$ be any rotation matrix in $\R^{d\times d}$, and define the rotated data
\begin{align*}
\btheta_O^\star = \bO\btheta^\star, \qquad \v^0_O = \bO\v^0, \qquad \y_O = (\X \btheta^\star_O)^2.
\end{align*}
The AMP dynamic (\ref{eq:simplified_amp}) run on the rotated data is
\begin{align*}
\v^{t+1}_O = \X^\top h_t(\u_O^t,\y_O) - \hat c_{O,t} \v_O^t,\qquad \u_O^{t+1} &= \X \v_O^{t+1} - h_t(\u_O^t,\y_O),
\end{align*}
initialized at $\v_O^0$, with $\hat c_{O,t} = \delta \av{h_t'(\u_O^t,\y_O)}$. Since $\X\bO\overset{d}{=}\X$, one checks by induction that, for every $t\geq 0$,
\begin{align*}
\u_O^t \overset{d}{=} \u^t, \quad \v_O^t\overset{d}{=}\bO\v^t.
\end{align*}
Thus rotating the signal and the initialization rotates the iterates in the same way, and the law of any rotation-invariant statistic, such as $\rho_t$ or $\corr(\v^t,\btheta^\star)$, depends on $\btheta^\star$ only through its norm. This is why no assumption on the empirical distribution of $\btheta^\star$ is needed.
\item {\em Weak and strong recovery threshold.} Theorem \ref{thm:main} identifies $\delta_{\weak} = 1/2$ and $\delta_{\str}^{\AMP} \approx 1.13$ as the thresholds of weak and strong recovery for the randomly initialized AMP dynamic (\ref{eq:simplified_amp}). Heuristically, both can be read off from the population state evolution (\ref{eq:gamp_se}). Under the normalized Bayes-optimal choice (\ref{eq:rescaled_nonlinearity}), the state evolution (\ref{eq:gamp_se}) is simplified as (see Appendix \ref{sec:bayes_amp} for details)
\begin{align*}
\sigma_{U,t}^2 = \sigma_{V,t}^2 = 1, \quad \mu_{U,t} = \mu_{V,t} =: \mu_t
\end{align*}
for all $t\geq 1$, and $\{\mu_t\}_{t\geq 0}$ evolves as
\begin{align}\label{eq:amp_se_intro}
\mu_{t+1}^2 = F_\delta(\mu_t^2),
\end{align}
where $F_\delta$ is defined in \eqref{eq:def_F_delta} (note that its argument is the squared mean $\mu_t^2$). Both weak and strong recovery boil down to studying the properties of this one-dimensional map. $0$ is always a fixed point of $F_\delta$, and $\delta_{\weak} = 1/2$ is the smallest value where $0$ becomes unstable (i.e., $F_\delta'(0) > 1$). For random initialization (\ref{eq:init_intro}), $\mu_0^2$ starts from a vanishingly small neighborhood of $0$, and thus once $\delta > \delta_{\weak}$, the recursion (\ref{eq:amp_se_intro}) allows $\{\mu_t^2\}_{t\geq 0}$ to escape from $0$ and eventually achieve weak recovery. On the other hand, since $\sigma_{U,t}^2 = \sigma_{V,t}^2 = 1$, strong recovery is only achieved if $\mu_t^2 \rightarrow \infty$, and $\delta_{\str}^{\AMP}$ is precisely the threshold above which $F_\delta$ has no positive fixed point. Below it, $\mu_t^2$ converges to the fixed point $\mu_\infty(\delta)$ defined above. We refer to Proposition \ref{prop:phase_3_asymp} for the details of these calculations.

\item {\em Connection to spectral method.} Another way to see why the AMP (\ref{eq:simplified_amp}) achieves weak recovery at $\delta_{\weak} = 1/2$ is through its connection to the spectral method for phase retrieval \cite{mondelli2019fundamental,mondelli_approximate_2022}. The spectral method constructs the matrix (\ref{eq:spec_intro}) with some processing function $\mathcal T_s$, and outputs the leading eigenvector of $\M_n$ as an estimate of $\btheta^\star$. Before weak recovery is achieved, the AMP dynamic (\ref{eq:simplified_amp}) can be linearized as (see Lemma
\ref{lem:linearization} for details)
\begin{align}\label{eq:linear_amp_intro}
\v^{t+1} \approx \X^\top \frac{(\y-1)\circ \u^t}{\sqrt{2\delta}}, \qquad \u^t \approx \X\v^t - \frac{(\y-1)\circ \u^{t-1}}{\sqrt{2\delta}}.
\end{align}
A fixed point $\v^\infty$ of this iteration satisfies
\begin{align*}
\v^\infty = \X^\top \diag\Big(\frac{\y-1}{\sqrt{2\delta}}\Big)\bigg(\I_n + \diag\Big(\frac{\y-1}{\sqrt{2\delta}}\Big)\bigg)^{-1}\X \v^\infty,
\end{align*}
that is, $\v^\infty$ is an eigenvector of the matrix (\ref{eq:spec_intro}) with processing function $\mathcal T_s(y) = \frac{y-1}{\sqrt{2\delta} + y-1}$. Hence the AMP (\ref{eq:simplified_amp}) can be understood as iteratively approximating the spectral method with this $\mathcal T_s$, which is the optimal processing function identified in \cite[Eq. (48)]{mondelli2019fundamental}. This correspondence between the Bayes-optimal AMP and the optimal spectral method is known to hold for general single-index models under a symmetry condition \cite{mondelli2019fundamental}.

\item {\em Polynomial in $n$ horizon.} Another notable feature of Theorem \ref{thm:main} is the length of the horizon. In both regimes $\delta\in (\delta_{\weak}, \delta_{\str}^{\AMP})$ and $\delta > \delta_{\str}^{\AMP}$, the distributional characterization is valid up to a horizon that grows with $n$, and in the intermediate regime this horizon is polynomial in $n$. This significantly extends the classical asymptotic AMP analysis \cite{bayati_message_passing_2011,javanmard2013state,barbier2019optimal}, and some more recent developments \cite{rush2018finite,bao2025leave,han2024entrywise} that are largely confined to $t = O(\frac{\log n}{\log\log n})$ iterations (albeit some with non-Gaussian $\X$). This long-time analysis constitutes the technical core of Theorem \ref{thm:main}, and is based on a second-order error recursion analysis (as opposed to the first-order analysis in \cite{rush2018finite,bao2025leave,han2024entrywise}, loosely speaking) that is inspired by the recent line of work \cite{li2023approximate, li2024non}. More discussion on the proof technique, including the connection to \cite{li2023approximate, li2024non}, is given in the next subsection.
\end{itemize}

\subsection{Non-asymptotic analysis of AMP}\label{subsec:intro_amp_nonasymptotic}
The proof of Theorem~\ref{thm:main} consists of the following two main ingredients:
\begin{enumerate}
\item An exact finite-sample coupling of the AMP iterates $\u^t,\v^t$ with the random data $(\X,\y)$.
\item A non-asymptotic error recursion for approximating $\u^t,\v^t$ by their structurally simpler surrogates.
\end{enumerate}
Both ingredients are developed for the general AMP recursion \eqref{eq:main_amp} in the single-index model \eqref{def:single_index}. We give a technical summary here, deferring the construction to Section~\ref{sec:main_proof} and the proof details to Appendix~\ref{Sec:Error_Proofs}. 

\medskip
\noindent\textbf{Non-asymptotic coupling.} The starting point of the analysis is the following constructive coupling. Its underlying idea stems from \cite{lu2021householder}, where it was originally developed primarily for computational purposes. For an integer $t\geq0$ define recursively
\begin{align}\label{eq:intro__uv}
\u^t&=\sum_{j=-1}^t\gamma_{t,j}\g^j+\sum_{j=0}^{t-1}\frac{\r^j}{\sqrt d}\iprod{\s^j}{\z^t}-\hat b_t\w^{t-1}+\Delta_{U,t}^{\low},\notag\\
\v^{t+1}&=\sum_{j=0}^t\lambda_{t,j}\s^j+\sum_{j=-1}^t\frac{\q^j}{\sqrt d}\iprod{\g^j}{\w^t}-\hat c_t\z^t+\Delta_{V,t+1}^{\low},
\end{align}
with $\w^{-1}=0$, coupled with the random data
\begin{align}\label{eq:intro_X}
\X^{(t)}
&=\sum_{\ell=-1}^tP_{\bR^{[0:\ell-1]}}^\perp\frac{\g^\ell\q^{\ell\top}}{\sqrt d}
+\sum_{\ell=0}^t\frac{\r^\ell\s^{\ell\top}}{\sqrt d}P_{\Q^{[-1:\ell]}}^\perp+P_{\bR^{[0:t]}}^\perp\W P_{\Q^{[-1:t]}}^\perp,\qquad \y=\varphi(\g^{-1}).
\end{align}
We first explain the notation:
\begin{itemize}
\item Here $\{\g^\ell\}_{\ell=-1}^t$ are i.i.d. $\N(0, \I_n)$ vectors, $\{\s^\ell\}_{\ell=0}^t$ are i.i.d. $\N(0, \I_d)$ vectors, and they are mutually independent.
\item For each $\u^t$, $\v^t$, we define 
\begin{align*}
\z^{-1} = \btheta^\star, \quad \z^t = f_t(\v^t),\quad \w^t = h_t(\u^t,\y), 
\end{align*}
and $\{\q^\ell\}_{\ell=-1}^t$ (resp. $\{\r^\ell\}_{\ell=0}^t$) are the
Gram--Schmidt orthonormal directions generated by
$\{\z^\ell\}_{\ell=-1}^t$ (resp. $\{\w^\ell\}_{\ell=0}^t$):
\begin{equation}\label{def:gamp_q_and_r}
\begin{aligned}
\q^{-1} &:= \frac{\btheta^\star}{\pnorm{\btheta^\star}{}}, \quad \q^\ell = \frac{P_{\Q^{[-1:\ell-1]}}^\perp \z^\ell}{\pnorm{P_{\Q^{[-1:\ell-1]}}^\perp \z^\ell}{}}, \quad \ell\geq 0, \\
\r^0 &:= \frac{\w^0}{\pnorm{\w^0}{}}, \quad \r^\ell = \frac{P^\perp_{\bR^{[0:\ell-1]}}\w^\ell}{\pnorm{P^\perp_{\bR^{[0:\ell-1]}}\w^\ell}{}}, \quad \ell \geq 1.
\end{aligned}
\end{equation}
We define $\bR^{[0:t]} = [\r^0,\ldots,\r^t]\in\R^{n\times (t+1)}$ and $\Q^{[-1:t]} = [\q^{-1},\ldots,\q^t] \in \R^{d\times (t+2)}$, and understand $P_{\bR^{[0:-2]}}^\perp, P_{\bR^{[0:-1]}}^\perp$ as $\I_n$.
\item The random coefficients $\{\lambda_{t,\ell}\}_{\ell=0}^t$, $\{\gamma_{t,\ell}\}_{\ell=-1}^t$ are defined by
\begin{align}\label{eq:inner_product}
\gamma_{t,\ell} \coloneqq \frac{1}{\sqrt d}\iprod{\q^\ell}{\z^t}, \quad \lambda_{t,\ell} \coloneqq \frac{1}{\sqrt{d}} \iprod{\r^\ell}{\w^t},
\end{align}
which are random: they depend on $\v^0$, $\y$ and the Gaussian vectors $\{\g^\ell\}_{\ell=-1}^t,\{\s^\ell\}_{\ell=0}^t$; see \eqref{eq:uv_dependence} for the precise measurability.

\item $\hat b_t, \hat c_t$ are Onsager coefficients defined in (\ref{eq:general_onsager}).
\item $\W \in \R^{n\times d}$ has i.i.d. $\N(0,1/d)$ entries and is independent of everything else. 
\item $\Delta_{U,t}^{\low}$ and $\Delta_{V,t+1}^{\low}$, whose exact definitions are given by (\ref{eq:amp_low_d}) below, are error terms resulting from low-dimensional projections of $\{\g^\ell\}_{\ell=-1}^t, \{\s^\ell\}_{\ell=0}^t$, and satisfy $\pnorm{\Delta_{U,t}^{\low}}{}\vee \pnorm{\Delta_{V,t+1}^{\low}}{} = \widetilde O(\sqrt t)$; see Lemma \ref{lem:bound_lowd_projection} for details. 
\end{itemize}
As shown by Lemma \ref{lem:X_coupling} below, the coupling of (\ref{eq:intro__uv}) and (\ref{eq:intro_X}) posits that: (i) For any $t\geq 0$, $\X^{(t)}$ has i.i.d. $\N(0,1/d)$ entries and 
\begin{align}\label{eq:conjugate_signal}
\bbeta^\star \coloneqq \g^{-1} = \X^{(t)}\btheta^\star,
\end{align}
which is the conjugate signal, and (ii) if we implement the AMP dynamic (\ref{eq:main_amp}) with data $(\X^{(t)},\y)$, we obtain exactly (\ref{eq:intro__uv}). Consequently, (\ref{eq:intro__uv}) can be viewed as a specific realization of the distribution of the AMP dynamic (\ref{eq:main_amp}).

This construction is close in spirit to the non-asymptotic AMP decompositions of \cite{li2022non,li2024non}, which extract independent Gaussian vectors from a given matrix along the AMP trajectory. The difference is that we first generate the Gaussian vectors and use them to construct both the iterates in \eqref{eq:intro__uv} and the matrix in \eqref{eq:intro_X}. Both approaches are constructive counterparts of Gaussian conditioning \cite{bolthausen2014iterative,bayati_message_passing_2011}: the iterates up to time $t$ determine the matrix in only $O(t)$ directions, and the last term of \eqref{eq:intro_X} carries the remaining Gaussian part, which they have not yet seen.

\medskip
\noindent\textbf{A state-evolution decomposition of the iterates.} It remains to understand the terms in \eqref{eq:intro__uv}. Motivated by the state-evolution representation \eqref{eq:gamp_se}, we decompose them as
\begin{align}\label{eq:uv_intro_repeat}
\ampblock{\u^t}{\v^{t+1}}
\amprel{=}
\amplabel{\ampblock{\gamma_{t,-1}\bbeta^\star}{\rho_{t+1}\btheta^\star}}{signal}
\quad\ampbin{+}\quad
\amplabel{\ampblock{\displaystyle\sum_{\ell=0}^t\gamma_{t,\ell}\g^\ell}{\displaystyle\sum_{\ell=0}^t\lambda_{t,\ell}\s^\ell}}{Gaussian noise}
\quad\ampbin{+}\quad
\amplabel{\ampblock{\Delta_{U,t}^{\Ons}}{\Delta_{V,t+1}^{\Ons}}}{Onsager}
\quad\ampbin{+}\quad
\amplabel{\ampblock{\Delta_{U,t}^{\low}}{\Delta_{V,t+1}^{\low}}}{low-dim projection},
\end{align}
where
\begin{align}\label{eq:onsager_u_intro}
\Delta_{U,t}^{\Ons} := \sum_{j=0}^{t-1}\frac{\r^j}{\sqrt d}\iprod{\s^j}{\z^t}-\hat b_t\w^{t-1},\qquad \Delta_{V,t+1}^{\Ons} := \sum_{\ell=0}^t\frac{\q^\ell}{\sqrt d}\iprod{\g^\ell}{\w^t}-\hat c_tP_{\q^{-1}}^\perp\z^t,
\end{align}
and the signal strength is given by
\begin{align}\label{eq:key_signal}
\rho_{t+1}\coloneqq\frac{1}{d}\iprod{\g^{-1}}{\w^t}-\hat c_t\frac{1}{\sqrt d}\iprod{\q^{-1}}{\z^t}
=\frac{1}{d}\iprod{\g^{-1}}{\w^t}-\hat c_t\gamma_{t,-1}.
\end{align}
In \eqref{eq:uv_intro_repeat}, the coefficients $\gamma_{t,-1}$ and $\rho_{t+1}$ of the signal terms are finite-sample analogues of the population means $\mu_{U,t}$ and $\mu_{V,t+1}$. In phase retrieval, $\gamma_{t,-1}=\rho_t$ because $f_t$ is the identity (Lemma~\ref{lem:conjugate_signal}), and $\pnorm{\lambda_{t,[0:t]}}{}=1$ by the normalization of $h_t$ in \eqref{eq:rescaled_nonlinearity}.

The Gaussian noise terms are linear combinations of the independent standard Gaussian vectors $\{\g^\ell\}$ and $\{\s^\ell\}$. Their coefficients depend on these same vectors, but there are only $t+1$ of them, so uniform concentration still controls the norms and inner products of these terms, and the empirical averages that enter the recursion for $\rho_t$. The Onsager terms $\Delta_{U,t}^{\Ons},\Delta_{V,t+1}^{\Ons}$ are what remains after the Onsager correction cancels the leading part of the memory terms. Controlling their magnitude is the key step in extending the AMP characterization beyond a fixed time horizon.

\medskip
\noindent\textbf{Non-asymptotic approximation.} Ignoring the $\Delta^{\Ons}$ and $\Delta^{\low}$ error terms in \eqref{eq:uv_intro_repeat}, define the first-order approximations
\begin{align}\label{eq:tilde_u_intro}
\tilde\u^t&\coloneqq\gamma_{t,-1}\bbeta^\star+\sum_{j=0}^t\gamma_{t,j}\g^j,\qquad
\tilde\v^{t+1}\coloneqq\rho_{t+1}\btheta^\star+\sum_{\ell=0}^t\lambda_{t,\ell}\s^\ell,
\end{align}
with $\tilde\v^0=\v^0$ and first-order errors
\begin{align}\label{eq:first_order_error}
\Delta_{U,t}=\u^t-\tilde\u^t,\qquad
\Delta_{V,t+1}=\v^{t+1}-\tilde\v^{t+1}.
\end{align}
The low-dimensional projection errors $\Delta^{\low}$ can be bounded by uniform concentration, as in Lemma~\ref{lem:bound_lowd_projection}. The main difficulty lies in the Onsager terms $\Delta^{\Ons}$. A simple idea is to formulate a recursion directly for $\pnorm{\Delta_{U,t}}{}$ and $\pnorm{\Delta_{V,t+1}}{}$. Unfortunately, this analysis is too coarse and only controls $\pnorm{\Delta_{U,t}}{}, \pnorm{\Delta_{V,t+1}}{}$ up to $t = c\log n$ iterates for some small $c > 0$. For a finer analysis that remains valid over a poly-$n$ horizon, we introduce the auxiliary second-order AMP approximation as
\begin{equation}
\begin{aligned}
\hat\u^t := \tilde\u^t + \sum_{k=0}^{t-1} \alpha^k_{t-1}h_k(\tilde \u^k,\y), \qquad \hat\v^{t+1} &:= \tilde\v^{t+1} + \sum_{k=0}^{t} \beta^k_t \big(f_k(\tilde\v^k) - \gamma_{k,-1}\btheta^\star\big).
\end{aligned}
\end{equation}
Here $\{\alpha_{t-1}^k\}_{k=0}^{t-1},\{\beta_t^k\}_{k=0}^t$ are correction coefficients recursively defined in (\ref{eq:beta_correct_def}) and (\ref{eq:alpha_correct_def}) below, and the form of $\hat\u^t,\hat\v^{t+1}$ is motivated by the fact that 
\begin{align*}
\Delta^{\Ons}_{U,t} &\in \spa(\w^0,\ldots,\w^{t-1}) \approx \spa\Big(h_0(\tilde\u^0,\y),\ldots,h_{t-1}(\tilde\u^{t-1},\y)\Big), \\
\Delta^{\Ons}_{V,t+1} &\in \spa(\q^0,\ldots,\q^t) \approx \spa\Big(f_0(\tilde\v^0) - \gamma_{0,-1}\btheta^\star,\ldots,f_t(\tilde\v^t) - \gamma_{t,-1}\btheta^\star\Big).
\end{align*} We define the second-order approximation error by
\begin{align}\label{eq:second_order_error}
\Delta_{U,t}'&\coloneqq\u^t-\Big(\tilde\u^t+\sum_{k=0}^{t-1}\alpha_{t-1}^k\w^k\Big),\qquad
\Delta_{V,t+1}'\coloneqq\v^{t+1}-\Big(\tilde\v^{t+1}+\sum_{k=0}^t\beta_t^kP_{\btheta^\star}^\perp\z^k\Big).
\end{align}
Note that $\Delta_{U,t}'$ is not $\u^t-\hat\u^t$: its correction uses the AMP iterates $\w^k=h_k(\u^k,\y)$ themselves, whereas $\hat\u^t$ replaces them by $h_k(\tilde\u^k,\y)$. 

Section~\ref{sec:second_order_analysis} proves the explicit inequalities
\eqref{eq:induction_Delta_U} and \eqref{eq:induction_Delta_V}. The core of these inequalities is the recursive bound
\begin{equation}\label{eq:so_intro}
\begin{aligned}
\pnorm{\Delta_{U,t}'}{} &\leq \Big[\sqrt{\mathcal{B}_{f_t}(\rho_t)} + \cE_t^f\Big]\pnorm{\Delta_{V,t}'}{} + \err_{U,t}',\\
\pnorm{\Delta_{V,t+1}'}{} &\leq\Big[\sqrt{\mathcal{B}_{h_t}(\rho_t)}+\cE_t^h\Big] \pnorm{\Delta_{U,t}'}{} + \err_{V,t+1}'. 
\end{aligned}
\end{equation}
Here $\cE_t^f,\cE_t^h$ are defined in \eqref{eq:E_multiplier},
$\err'_{U,t},\err'_{V,t+1}$ are additive error terms which we will control inductively,
and, with $G,W\stackrel{\mathrm{i.i.d.}}{\sim}\N(0,1)$,
\begin{align}\label{eq:B_intro}
\mathcal B_{f_t}(\rho) := \frac1d\sum_{j=1}^d\E\big[f_t'(\rho\theta^\star_j+W)^2\big], \qquad \mathcal B_{h_t}(\rho)\coloneqq\delta\,\E\big[h_t'(\rho G+W,G^2)^2\big]
\end{align}
are the key multipliers \eqref{eq:F_def} at unit noise level, where for $\mathcal B_{h_t}$ the average over the entries of the Gaussian vector $\bbeta^\star$ is
replaced by an expectation over $G$. Notably, (\ref{eq:B_intro}) can be viewed as a dynamical generalization of Bolthausen's constant \cite{bolthausen2014iterative} (see Appendix \ref{sec:bolthausen} for a review in the asymmetric AMP setting), originally defined at the fixed point of AMP state evolution to study convergence over an $O(1)$ horizon. Here the non-asymptotic Bolthausen constant $\mathcal{B}_{f_t}(\rho_t)\cdot \mathcal{B}_{h_t}(\rho_t)$ plays an equally essential role in the error recursion (\ref{eq:so_intro}), determining the horizon over which Onsager correction remains valid.

The second-order analysis above builds on \cite{li2024non} and extends it in two directions:
\begin{itemize}
\item The analysis in \cite{li2024non} is for the linear model; ours covers the general AMP recursion \eqref{eq:main_amp} for single-index models.
\item The analysis in \cite{li2024non} requires the multiplier in the error recursion \eqref{eq:so_intro} to be a strict contraction. This fails in our symmetry-breaking phase, where the multiplier is $1+o(1)$, so the recursion has to be controlled without a contraction.
\end{itemize}

\subsection{Further discussion}
We end this section by discussing some open problems for future investigation.
\begin{itemize}
\item (\textit{General first-order methods under random initialization}) This
paper identifies sharp weak- and strong-recovery thresholds for randomly
initialized Bayes-optimal AMP. It remains important to determine the
corresponding proportional-regime thresholds for other first-order methods,
whose trajectories are described by more complicated dynamical mean-field
equations \cite{celentano_high-dimensional_2021,gerbelot2024rigorous}.

\item (\textit{Multi-index models}) Recent work gives nearly linear sample
complexity for feature learning in several multi-index models
\cite{dandi2024benefits,zhang2025neural,montanari2026phase}. Sharp
proportional-regime thresholds are known in fewer settings, notably for
Bayes-optimal AMP and spectral estimators
\cite{troiani2024fundamental,defilippis2025optimal,kovavcevic2025spectral}.
Understanding random AMP initialization in this broader class is open.

\item (\textit{Universality}) The coupling used here relies on Gaussian
rotational invariance. Figure~\ref{fig:univ-phase} suggests that the recovery
thresholds persist more broadly. Proving this would require combining the
growing-horizon argument with AMP universality techniques
\cite{bayati2015universality,chen2021universality,fan2022approximate,dudeja2023universality,wang2024universality,bao2025leave}.
\end{itemize}

\begin{figure}[H]
\centering
\includegraphics[width=0.5\linewidth]{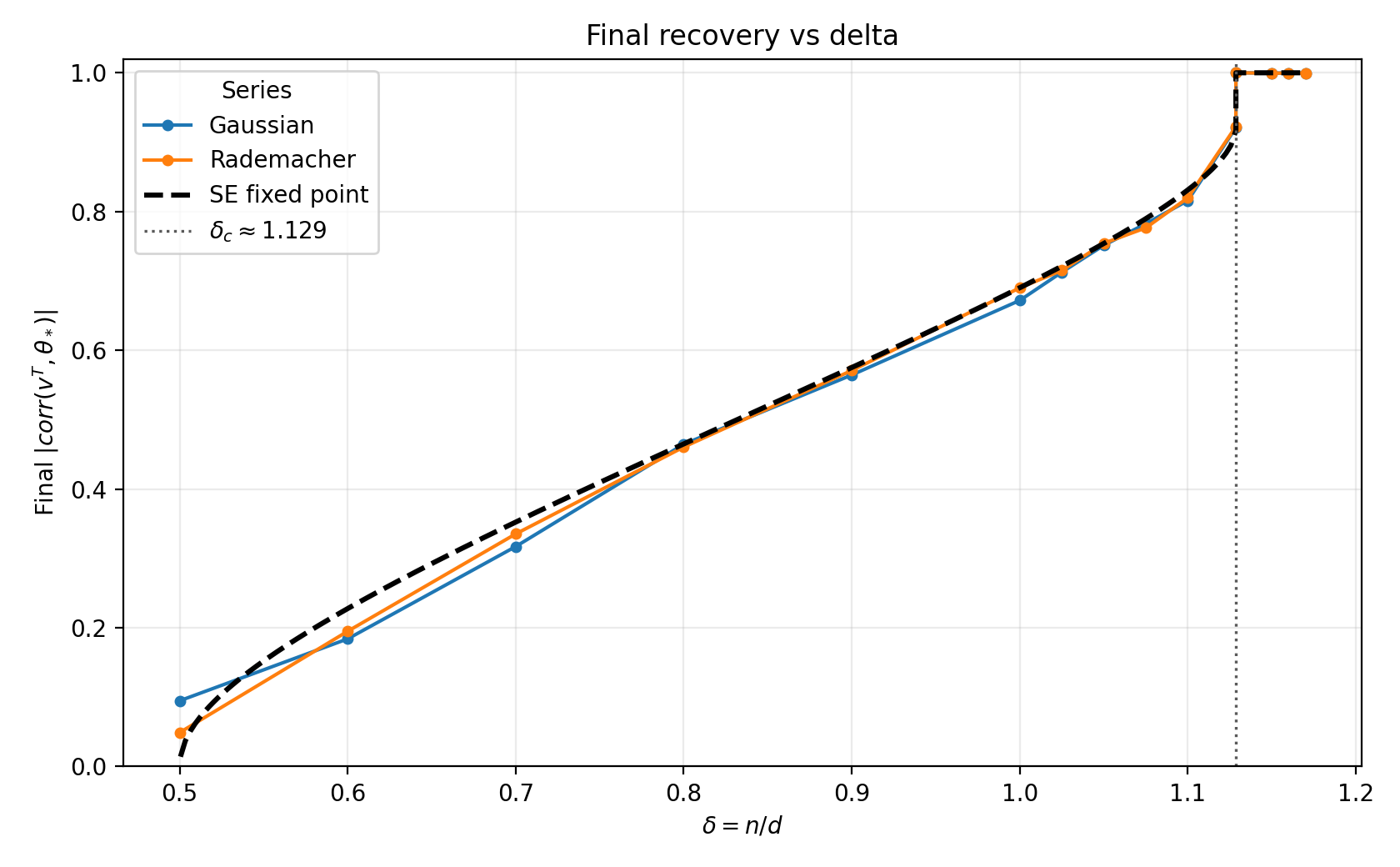}
\caption{The final correlation of randomly initialized Bayes-optimal AMP for sensing matrices with i.i.d. Gaussian or Rademacher entries is plotted at various aspect ratios. The Gaussian simulations agree with the phase diagram of Theorem \ref{thm:main}. The agreement for Rademacher design provides numerical evidence that the recovery thresholds extend beyond Gaussian design.}
\label{fig:univ-phase}
\end{figure}

\section{Non-asymptotic analysis of AMP for single-index models}\label{sec:main_proof}
Throughout this section, we study the generalized AMP recursion \eqref{eq:main_amp} in the single-index model \eqref{def:single_index}. The proofs of the results in this section are given in Appendix \ref{Sec:Error_Proofs}. Our analysis allows the nonlinearities to contain random parameters, on which we impose the following assumptions.
\begin{assumption}
\label{assump:nonlinearity}
We assume the AMP nonlinearities $f_k,h_k$ can be written as  $f_k(v)=\mathsf f_k(v;\xi_k)$ and $h_k(u,y)=\mathsf h_k(u,y;\vartheta_k)$, where $\mathsf f_k$ and $\mathsf h_k$ are deterministic maps and $\xi_k\in\Xi_k \subset\mathbb R^{p_k^f}$ and $\vartheta_k\in \Theta_k\subset\mathbb R^{p_k^h}$ contain their random parameters. There is a universal integer $C_{\rm par}\geq1$ such that the following conditions hold.
\begin{enumerate}
\item For every $t\geq0$,
\begin{align}\label{eq:parameter_cover_assumption}
\sum_{k=0}^{t+1}\big(p_k^f+p_k^h\big) \leq C_{\rm par}(t+2).
\end{align}

\item Let $\mathcal Y\subseteq\R$ be a set containing every possible value of the responses $y_i$. For every $k\geq0$, $\xi\in\Xi_k$, $\vartheta\in\Theta_k$ and $y\in\mathcal Y$, the maps $v\mapsto\mathsf f_k(v;\xi)$ and $u\mapsto\mathsf h_k(u,y;\vartheta)$ are three times differentiable. Moreover, for $0\leq j\leq3$, all $\xi,\tilde\xi\in\Xi_k$ with $\|\xi\|\vee\|\tilde\xi\|\leq n^6$, and all $\vartheta,\tilde\vartheta\in\Theta_k$ with
$\|\vartheta\|\vee\|\tilde\vartheta\|\leq n^6$,
\begin{equation}\label{eq:parameter_lipschitz_assumption}
\begin{aligned}
\big|\partial_v^j\mathsf f_k(v;\xi)-\partial_v^j\mathsf f_k(v;\tilde\xi)\big| &\leq n^{C_{\rm par}}\big(1+|v|^{C_{\rm par}}\big)\pnorm{\xi-\tilde\xi}{},\\
\big|\partial_u^j\mathsf h_k(u,y;\vartheta)-\partial_u^j\mathsf h_k(u,y;\tilde\vartheta)\big| &\leq n^{C_{\rm par}}\big(1+|u|^{C_{\rm par}}+|y|^{C_{\rm par}}\big)\pnorm{\vartheta - \tilde\vartheta}{}
\end{aligned}
\end{equation}
for all $v,u\in\R$ and $y\in\mathcal Y$.

\item The response satisfies
\begin{align}
\label{eq:response_polynomial_envelope}
\P\Big(\max_{1\leq i\leq n}|y_i|>n^{C_{\rm par}}\Big)
&\leq n^{-20}.
\end{align}

\item The random parameters may depend only on the information available when the corresponding nonlinearity is chosen. Specifically, for every $k\geq0$,
\begin{equation}\label{eq:parameter_nonanticipation}
\begin{aligned}
\vartheta_k &\text{ is }\cF\Big(\cR,\v^0,\{\g^j\}_{j=-1}^{k},\{\s^j\}_{j=0}^{k-1}\Big)\text{-measurable},\\
\xi_{k+1} &\text{ is }\cF\Big(\cR,\v^0,\{\g^j\}_{j=-1}^{k},\{\s^j\}_{j=0}^{k}\Big)\text{-measurable}.
\end{aligned}
\end{equation}
Here $\cR$ denotes any randomness in the link function $\varphi$, independent of $(\X,\v^0)$.
\end{enumerate}
\end{assumption}

We now check that the phase retrieval nonlinearities \eqref{eq:rescaled_nonlinearity} satisfy Assumption \ref{assump:nonlinearity}. Define
\begin{align}\label{eq:pi_parameter}
N_t\coloneqq \frac{\pnorm{h_t^\star(\u^t,\y)}{}}{\sqrt d},
\qquad \pi_t\coloneqq \frac{\sqrt{2\delta}\,\hat\mu_t}{N_t}-1.
\end{align}
On the event $N_t>0$, the nonlinearity $h_t$ \eqref{eq:rescaled_nonlinearity} satisfies
\begin{align}\label{eq:h_parameterization}
h_t(u,y) =\mathsf h_t(u,y;\hat\mu_t,\pi_t),\qquad \mathsf h_t(u,y;\mu,\pi) := \frac{1+\pi}{\sqrt{2\delta}}\Big\{(1+\mu^2)\sqrt y \frac{\tanh(\mu u\sqrt y)}{\mu}-u\Big\}.
\end{align}
Since $f_t(v)=v$ with no random parameters, the parameter vectors are
\begin{align}\label{eq:phase_retrieval_parameter_tuple}
\begin{alignedat}{3}
p_k^f&=0,\qquad
&\Xi_k&=\mathbb R^0,\qquad
&\xi_k&=(),\\
p_k^h&=2,\qquad
&\Theta_k&=\mathbb R_+\times\mathbb R,\qquad
&\vartheta_k&=(\hat\mu_k,\pi_k).
\end{alignedat}
\end{align}
From this, it is clear that conditions \eqref{eq:parameter_cover_assumption} and \eqref{eq:parameter_nonanticipation} are satisfied. Condition \eqref{eq:response_polynomial_envelope} follows from $\bm y=(\bm g^{-1})^{\circ2}$ and the Gaussian maximum bound. The next lemma verifies the remaining condition \eqref{eq:parameter_lipschitz_assumption}. 

\begin{lemma}\label{lem:phase_retrieval_parameter_lipschitz}
Suppose $\delta>1/2$ and let $D\geq1$. Write $\mathsf h\coloneqq\mathsf h_t$ for the map in \eqref{eq:h_parameterization}, which does not depend on $t$. For $0\leq j\leq3$, $0\leq\mu,\tilde\mu\leq n^D$, and $|\pi|\vee|\tilde\pi|\leq n^D$,
\begin{align}\label{eq:phase_retrieval_parameter_lipschitz}
&\Big|\partial_u^j\mathsf h(u,y;\mu,\pi)
-\partial_u^j\mathsf h(u,y;\tilde\mu,\tilde\pi)\Big|\leq
C_Dn^{C_D}\big(1+|u|^5+y^4\big)
\Big(|\mu-\tilde\mu|+|\pi-\tilde\pi|\Big),
\end{align}
for all $u\in\mathbb R$ and $y\geq0$, where $C_D$ depends only on $D$. In particular, taking $D=6$ shows that phase retrieval satisfies condition \eqref{eq:parameter_lipschitz_assumption} of Assumption~\ref{assump:nonlinearity} with $\mathcal Y=[0,\infty)$, for all $n\geq2$, once $C_{\rm par}$ is a sufficiently large universal constant.
\end{lemma}

\subsection{Coupling of AMP and random data}\label{subsec:couple_amp}
We first construct a coupling of the AMP iterates and the Gaussian design that makes the independent Gaussian vectors underlying the approximation explicit. Define the low-dimensional projection errors
\begin{align}\label{eq:amp_low_d}
\Delta_{U,t}^{\low}
&\coloneqq-\sum_{\ell=-1}^tP_{\bR^{[0:\ell-1]}}\g^\ell\gamma_{t,\ell}
-\sum_{\ell=0}^{t-1}\frac{\r^\ell}{\sqrt d}\big(P_{\Q^{[-1:\ell]}}\s^\ell\big)^\top\z^t,\notag\\
\Delta_{V,t+1}^{\low}
&\coloneqq-\sum_{\ell=0}^tP_{\Q^{[-1:\ell]}}\s^\ell\lambda_{t,\ell}
-\sum_{\ell=0}^t\frac{\q^\ell}{\sqrt d}\big(P_{\bR^{[0:\ell-1]}}\g^\ell\big)^\top\w^t.
\end{align}
Then the AMP iterates in \eqref{eq:intro__uv} take the following form, with $\w^{-1}=\bm0_n$:
\begin{align}\label{eq:construct_uv}
\u^\ell
&=\sum_{j=-1}^\ell P_{\bR^{[0:j-1]}}^\perp\g^j\frac{\q^{j\top}\z^\ell}{\sqrt d}
+\sum_{j=0}^{\ell-1}\r^j\frac{\big(P_{\Q^{[-1:j]}}^\perp\s^j\big)^\top\z^\ell}{\sqrt d}
-\hat b_\ell\w^{\ell-1},\notag\\
\v^{\ell+1}
&=\sum_{j=-1}^\ell\q^j\frac{\big(P_{\bR^{[0:j-1]}}^\perp\g^j\big)^\top\w^\ell}{\sqrt d}
+\sum_{j=0}^\ell P_{\Q^{[-1:j]}}^\perp\s^j\frac{\r^{j\top}\w^\ell}{\sqrt d}
-\hat c_\ell\z^\ell.
\end{align}
Starting from $(\btheta^\star,\q^{-1},\v^0)$, the above definition is executed in the order of
\begin{align*}
(\btheta^\star,\q^{-1},\v^0) &\overset{(\star)}{\rightarrow} \u^0 \rightarrow \w^0 \rightarrow \{\r^\ell\}_{\ell=0}^0 \overset{(\star)}{\rightarrow} \v^1 \rightarrow \z^1 \rightarrow \{\q^\ell\}_{\ell=-1}^1  \overset{(\star)}{\rightarrow} \u^1 \rightarrow \ldots \\
&\overset{(\star)}{\rightarrow} \u^t \rightarrow \w^t \rightarrow \{\r^\ell\}_{\ell=0}^t \overset{(\star)}{\rightarrow} \v^{t+1} \rightarrow \z^{t+1} \rightarrow \{\q^\ell\}_{\ell=-1}^{t+1} \overset{(\star)}{\rightarrow} \u^{t+1} \rightarrow \ldots,
\end{align*}
where arrows marked with $(\star)$ point to the main AMP iterates $\{\u^t\}, \{\v^t\}$. It can be checked that,
\begin{align}\label{eq:uv_dependence}
\u^\ell,\y,\w^\ell,\r^\ell &\in\cF\Big(\cR,\v^0,\{\g^j\}_{j=-1}^\ell,\{\s^j\}_{j=0}^{\ell-1}\Big),\notag\\
\v^{\ell+1},\z^{\ell+1},\q^{\ell+1} &\in\cF\Big(\cR,\v^0,\{\g^j\}_{j=-1}^\ell,\{\s^j\}_{j=0}^{\ell}\Big).
\end{align}

Recall the matrix $\X^{(t)}$ defined in (\ref{eq:intro_X}). The following lemma states the coupling property of $\X^{(t)}$ and (\ref{eq:construct_uv}). We emphasize once and for all that all subsequent probabilistic claims are made with respect to the probability space underlying $(\v^0,\{\g^\ell\}_{\ell=-1}^\infty, \{\s^\ell\}_{\ell=0}^\infty,\cR)$.

\begin{lemma}\label{lem:X_coupling}
For every deterministic integer $t\geq0$ satisfying $t+2\leq d$ and $t+1\leq n$:
\begin{enumerate}
\item Under the single-index model \eqref{def:single_index}, we have $(\X^{(t)},\y,\v^0,\cR)\overset{(\mathrm d)}=(\X,\y,\v^0,\cR)$.
\item The AMP recursion \eqref{eq:main_amp} driven by $(\X^{(t)},\y)$ agrees pathwise with \eqref{eq:construct_uv} through $\v^{t+1}$.
\end{enumerate}
\end{lemma}

\subsection{Second-order analysis} \label{sec:second_order_analysis}

Recall we want to control the first-order approximations in \eqref{eq:tilde_u_intro}. To get finer control of these approximation errors, we instead consider a second-order approximation introduced in Section \ref{subsec:intro_amp_nonasymptotic}, namely
\begin{equation}\label{eq:def_u_v_hat}
\begin{aligned}
\hat{\u}^t := \tilde{\u}^t + \sum_{\ell=0}^{t-1} \alpha_{t-1}^\ell h_\ell(\tilde{\u}^\ell,\y), \qquad \hat{\v}^{t+1} := \tilde{\v}^{t+1} + \sum_{\ell=0}^t \beta_t^\ell \Big(f_\ell(\tilde{\v}^\ell)  - \gamma_{\ell,-1}\btheta^\star\Big),
\end{aligned}
\end{equation}
with initialization $\hat\u^0 := \tilde{\u}^0$. The high level idea of this section is to derive a recursion which simultaneously bounds the second-order approximation errors $\Delta_{U,t}'$ and $\Delta_{V,t+1}'$ defined in \eqref{eq:second_order_error}, along with the correction coefficients $\alpha_t^\ell,\beta_t^\ell$. With $\bar\lambda_{t-1}$ and $\bar\gamma_t$ defined in (\ref{def:induction_constant}) below, the first-order approximation errors can then be recovered via triangle inequality since
\begin{align*}
\pnorm{\Delta_{U,t}}{} \leq \pnorm{\Delta_{U,t}'}{} + \bar\lambda_{t-1}\sum_{k=0}^{t-1}|\alpha_{t-1}^k|,\qquad \pnorm{\Delta_{V,t+1}}{} \leq \pnorm{\Delta_{V,t+1}'}{} + \bar\gamma_t\sum_{k=0}^{t}|\beta_t^k|.
\end{align*}

\subsubsection{Definitions and exact decomposition}
We start by rigorously constructing the second-order approximation. We use the initialization $\tilde\v^0=\hat\v^0=\v^0$, $\Delta_{V,0}=\Delta_{V,0}'=0$, and $\hat\u^0=\tilde\u^0=\u^0$. For notational simplicity, set
\begin{align}\label{def:hat_lambda_-1}
\lambda_{t,-1}\coloneqq\rho_{t+1},\qquad \s^{-1}=\btheta^\star,
\end{align}
so that $\tilde\v^{t+1}=\sum_{\ell=-1}^t\lambda_{t,\ell}\s^\ell$. Also let $\bar\z^t=P_{\btheta^\star}^\perp\z^t=\z^t-\gamma_{t,-1}\btheta^\star$.

We now define the correction coefficients $\alpha$ and $\beta$ recursively. For $t\geq0$, define
\begin{align}\label{eq:beta_correct_def}
\beta_t^k
=
\begin{cases}
\alpha_{t-1}^{k}\delta\av{h_t'(\hat\u^t,\y)\circ h_k'(\tilde{\u}^k,\y)}, & 0\leq k\leq t-1,\\
\begin{aligned}
\delta\av{h_t'(\hat\u^t,\y) - h_t'(\u^t,\y)} + \frac{\iprod{ \sum_{\ell=0}^t\gamma_{t,\ell}\g^\ell}{h_t(\u^t,\y)-h_t(\hat\u^t,\y)}}{d\pnorm{\gamma_{t,[0:t]}}{}^2}
\end{aligned} & k=t,
\end{cases}
\end{align}
where we note that $\pnorm{\gamma_{t,[0:t]}}{}^2=\pnorm{\bar\z^t}{}^2/d$, and we adopt the convention $0/0 = 0$ when $\pnorm{\gamma_{t,[0:t]}}{} = 0$. After forming $\hat\v^{t+1}$ from $\beta_t$, define
\begin{align}\label{eq:alpha_correct_def}
\alpha_t^k
=
\begin{cases}
\beta_t^{k+1}\av{f_{t+1}'(\hat\v^{t+1})\circ f_{k+1}'(\tilde{\v}^{k+1})}, & 0\leq k\leq t-1,\\
\begin{aligned}
\av{f_{t+1}'(\hat\v^{t+1})-f_{t+1}'(\v^{t+1})} +  \frac{\iprod{\sum_{\ell=0}^t\lambda_{t,\ell}\s^\ell}{f_{t+1}(\v^{t+1})-f_{t+1}(\hat\v^{t+1})}}{d\pnorm{\lambda_{t,[0:t]}}{}^2}
\end{aligned} & k=t,
\end{cases}
\end{align}
where $\pnorm{\lambda_{t,[0:t]}}{}^2=\pnorm{\w^t}{}^2/d$, and we adopt the convention $0/0 = 0$ when $\pnorm{\lambda_{t,[0:t]}}{} = 0$. Given $\alpha_t$, form $\hat\u^{t+1}$ as in \eqref{eq:def_u_v_hat}. Since $\hat\u^0=\tilde\u^0=\u^0$, we have $\beta_0^0=0$, from which the correction coefficients can be computed in the order
\begin{align*}
\beta_0 \rightarrow \hat\v^1 \rightarrow \alpha_0 \rightarrow \hat\u^1 \rightarrow \beta_1 \rightarrow \hat\v^2 \rightarrow \alpha_1 \rightarrow\cdots.
\end{align*}

\subsubsection{The main induction argument}
\label{subsec:main_ind}
Now we are ready to bound $\Delta_{U,t}'$, $\Delta_{V,t+1}'$, and the correction coefficients by induction. The bounds involve three kinds of quantities: the scale parameters $\bar\lambda_t,\bar\gamma_t$ in \eqref{def:induction_constant}, the multipliers $F,H$ in \eqref{eq:F_def}, and the error terms $\cE_t^f,\cE_t^h$ in \eqref{eq:E_multiplier}. We define them in turn. First, let
\begin{equation}\label{def:induction_constant}
\begin{aligned}
\bar\lambda_t &:= \max_{0\leq k\leq t} \Big\{\pnorm{h_k(\bm{0}_n,\y)}{},\pnorm{\bbeta^\star}{}, \sqrt d\,\pnorm{\lambda_{k,[0:k]}}{}\Big\},\\
\bar\gamma_t &:= \max_{0\leq k\leq t} \Big\{\pnorm{f_k(\bm{0}_d)}{},\pnorm{\btheta^\star}{}, \sqrt d\,\pnorm{\gamma_{k,[-1:k]}}{}\Big\}.
\end{aligned}
\end{equation}
We also use the shorthand
\begin{align*}
\pnorm{h_{[0:t]}'}{\infty} := \max_{0\leq\ell\leq t}\pnorm{h_\ell'}{\infty}, \qquad \pnorm{f_{[0:t]}'}{\infty} := \max_{0\leq\ell\leq t}\pnorm{f_\ell'}{\infty},
\end{align*}
and analogous notation for higher derivatives, and write $\gamma_{[0:t],-1}=(\gamma_{0,-1},\ldots,\gamma_{t,-1})$. Next, for $\rho\in\R$ and $\sigma\geq0$, define
\begin{equation}\label{eq:F_def}
\begin{aligned}
F_{t+1}(\rho,\sigma;\btheta^\star) := \frac1d\E_{\s} \Bigpnorm{f_{t+1}'(\rho\btheta^\star+\sigma\s)}{}^2,\qquad H_t(\rho,\sigma;\bbeta^\star) := \frac1d\E_{\g} \Bigpnorm{h_t'(\rho\bbeta^\star+\sigma\g,\y)}{}^2,
\end{aligned}
\end{equation}
where $\s\sim\N(0,\I_d)$ and $\g\sim\N(0,\I_n)$ are fresh auxiliary Gaussians, so that the expectations average only over these Gaussian inputs. Deterministically,
\begin{align}\label{eq:F_trivial_bound}
F_{t+1}(\rho,\sigma;\btheta^\star) \leq\pnorm{f_{t+1}'}{\infty}^2, \qquad H_t(\rho,\sigma;\bbeta^\star) \leq\delta\pnorm{h_t'}{\infty}^2.
\end{align}
These multipliers enter the bounds through the products
\begin{align*}
\mathfrak F_{k,t} &:= \sqrt{F_k\big(\rho_k,\pnorm{\lambda_{k-1,[0:k-1]}}{}; \btheta^\star\big)} \sqrt{F_t\big(\rho_t,\pnorm{\lambda_{t-1,[0:t-1]}}{}; \btheta^\star\big)}
\end{align*}
for $k,t\geq1$, and
\begin{align*}
\mathfrak H_{k,t} &:= \sqrt{H_k\big(\gamma_{k,-1},\pnorm{\gamma_{k,[0:k]}}{}; \bbeta^\star\big)} \sqrt{H_t\big(\gamma_{t,-1},\pnorm{\gamma_{t,[0:t]}}{}; \bbeta^\star\big)}
\end{align*}
for $k,t\geq0$. Finally, define the error terms
\begin{equation}\label{eq:E_multiplier}
\begin{aligned}
\cE_t^f &:= C\Bigg[\big(\pnorm{f_{[0:t]}'}{\infty}\vee1\big)^2 \sqrt{t+1}\Big(\frac{\log n}{d}\Big)^{1/4}\\
&\hspace{4.2em} + \big(\pnorm{f_{[0:t]}'}{\infty}\vee1\big)\pnorm{f_t''}{\infty} \frac{\sqrt{t+1}+\sqrt{\log n}}{\sqrt d} \big(\pnorm{\Delta_{V,t}}{} + \pnorm{\hat\v^t-\tilde\v^t}{}\big)\Bigg],\\
\cE_t^h &:= C\Bigg[\big(\delta^{1/4}+\delta^{3/4}\big) \big(\pnorm{h_{[0:t]}'}{\infty}\vee1\big)^2 \sqrt{t+1}\Big(\frac{\log n}{d}\Big)^{1/4}\\
&\hspace{4.2em} + \big(\pnorm{h_{[0:t]}'}{\infty}\vee1\big)\pnorm{h_t''}{\infty} (1+\sqrt\delta)\frac{\sqrt{t+1}+\sqrt{\log n}}{\sqrt d} \big(\pnorm{\Delta_{U,t}}{} + \pnorm{\hat\u^t-\tilde\u^t}{}\big)\Bigg].
\end{aligned}
\end{equation}

For $t \geq 1$, let $\cH_t^U$ denote the following recursive error bounds
\begin{subequations}
\label{eqs:HtU}
\begin{align}
\notag\pnorm{\Delta_{U,t}'}{} &\leq \Big(\sqrt{F_{t}(\rho_t, \pnorm{\lambda_{t-1,[0:t-1]}}{};\btheta^\star)} + \cE^f_t\Big)\Big(\pnorm{\Delta_{V,t}'}{} + \sum_{k=0}^{t-1}|\beta_{t-1}^k|\pnorm{f'_k}{\infty}\pnorm{\Delta_{V,k}}{}\Big)\\
&\quad + C\big(\bar\gamma_{t} + \bar\lambda_{t-1}\big)\big(\pnorm{f_{[0:t]}'}{\infty}\vee 1\big)^2(\pnorm{\rho_{[0:t]}}{\infty}\vee 1)t\sqrt{\frac{\log n}{d}}, \label{eq:induction_Delta_U}\\
|\alpha_{t-1}^{t-1}| &\leq \Big(\frac{\pnorm{f_t''}{\infty}}{\sqrt d} + \frac{2\pnorm{f_t'}{\infty}}{\sqrt d\,\pnorm{\lambda_{t-1,[0:t-1]}}{}}\Big)\notag\\
&\quad\times\Big(\pnorm{\Delta_{V,t}'}{} + \pnorm{f_{[0:t-1]}'}{\infty} \max_{0\leq\ell\leq t-1}\pnorm{\Delta_{V,\ell}}{} \sum_{\ell=0}^{t-1}|\beta_{t-1}^\ell|\Big), \label{eq:induction_alpha_last}\\
|\alpha_{t-1}^{t-2}| &\leq \pnorm{f_t'}{\infty}\pnorm{f_{t-1}'}{\infty} |\beta_{t-1}^{t-1}|, \qquad t\geq2,\label{eq:induction_alpha_middle_last}\\
|\alpha_{t-1}^k| & \leq |\alpha_{t-2}^{k+1}|\big(\mathfrak F_{k+1,t}+\cE_t^f\big) \big(\mathfrak H_{k+1,t-1}+\cE^h_{t-1}\big), \quad 0\leq k\leq t-3,\quad t\geq3.
\label{eq:induction_alpha_middle}
\end{align}
\end{subequations}
Similarly, for $t\geq0$, let $\cH_t^V$ denote
\begin{subequations}
\label{eqs:HtV}
\begin{align}
\notag\pnorm{\Delta_{V,t+1}'}{} &\leq \Big( \sqrt{H_{t}(\gamma_{t,-1},\pnorm{\gamma_{t,[0:t]}}{};\bbeta^\star)} + \cE^h_t\Big)\Big(\pnorm{\Delta_{U,t}'}{} + \sum_{k=0}^{t-1}|\alpha_{t-1}^k|\pnorm{h_k'}{\infty}\pnorm{\Delta_{U,k}}{}\Big)\\
&\quad + C(1+\sqrt\delta)\big(\bar\lambda_{t} + \bar\gamma_{t}\big)\big(\pnorm{h_{[0:t]}'}{\infty}\vee 1\big)^2\big(\pnorm{\gamma_{[0:t],-1}}{\infty}\vee 1\big)(t+1)\sqrt{\frac{\log n}{d}},\label{eq:induction_Delta_V}\\
|\beta_t^t| &\leq \sqrt\delta\Big(\frac{\pnorm{h_t''}{\infty}}{\sqrt d} + \frac{C\pnorm{h_t'}{\infty}} {\sqrt d\,\pnorm{\gamma_{t,[0:t]}}{}}\Big)\notag\\
&\quad\times\Big(\pnorm{\Delta_{U,t}'}{} + \pnorm{h'_{[0:t-1]}}{\infty} \max_{0\leq\ell\leq t-1}\pnorm{\Delta_{U,\ell}}{} \sum_{\ell=0}^{t-1}|\alpha_{t-1}^\ell|\Big), \label{eq:induction_beta_last}\\
|\beta_t^{t-1}| &\leq \delta\pnorm{h'_t}{\infty}\pnorm{h'_{t-1}}{\infty} |\alpha_{t-1}^{t-1}|, \qquad t\geq1,\label{eq:induction_beta_middle_last}\\ |\beta_t^{k}| & \leq |\beta_{t-1}^{k+1}|\big(\mathfrak H_{t,k}+\cE^h_t\big) \big(\mathfrak F_{t,k+1}+\cE^f_t\big), \quad 0\leq k\leq t-2,\quad t \geq 2. \label{eq:induction_beta_middle}
\end{align}
\end{subequations}
At $t=0$, all sums and maxima indexed by $[0:t-1]$ are interpreted as zero. 

The following definition records the corresponding inductive assumptions.
\begin{definition}
\label{assump:inductive_assumption}
For $t\geq0$, let $\cA_t^U$ denote the event
\begin{subequations}
\label{eqs:AtU}
\begin{align}
\bar\gamma_t\vee\bar\lambda_t &\leq Cn^5,\label{eq:gamma_main_bound}\\
\pnorm{\rho_{[0:t]}}{\infty} &\leq C\log^5n,\\
\max_{1\leq s\leq t}\sum_{k=0}^{s-1}|\alpha_{s-1}^k| &\leq \frac{C}{\log^5n}, \label{eq:cond_alpha_sum}\\
\pnorm{h_{[0:t]}'}{\infty}\vee\pnorm{h_{[0:t]}''}{\infty} &\leq C\log^5n, \label{eq:cond_derivative_linfty_h}\\
\max_{0\leq k\leq t}\pnorm{\vartheta_k}{} &\leq Cn^5. \label{eq:cond_parameter_h}
\end{align}
\end{subequations}
The maximum over an empty set is understood as zero.  Similarly, let $\cA_t^V$ denote
\begin{subequations}
\label{eqs:AtV}
\begin{align}
\bar\lambda_t\vee\bar\gamma_{t+1} &\leq Cn^5,\label{eq:lambda_main_norm}\\
\pnorm{\rho_{[0:t+1]}}{\infty} & \leq C\log^5n,\\
\max_{0\leq s\leq t}\sum_{k=0}^{s}|\beta_s^k| &\leq \frac{C}{\log^5n}, \label{eq:cond_beta_sum}\\
\pnorm{f_{[0:t+1]}'}{\infty}\vee\pnorm{f_{[0:t+1]}''}{\infty} &\leq C\log^5n, \label{eq:cond_derivative_linfty_f}\\
\max_{0\leq k\leq t+1}\pnorm{\xi_k}{} &\leq Cn^5. \label{eq:cond_parameter_f}
\end{align}
\end{subequations}
\end{definition}

\begin{proposition}\label{prop:general_error_induction}
Set $\cH_0^U=\{\Delta_{U,0}'=0\}$, which holds surely, and let $\cA_{-1}^V$ be the sure event. There is a universal constant $c>0$ such that for sufficiently large $n$  and every $t\geq0$ satisfying
\begin{align*}
t+1\leq\frac{c\sqrt n}{\log^3n},
\end{align*}
the following implications hold with probability $1-O(n^{-12})$:
\begin{itemize}
\item If $\cA_t^U$ and $\cA_{t-1}^V$ hold, then the bounds \eqref{eqs:HtV} defining $\cH_t^V$ hold.
\item If $\cA_t^V$ and $\cA_t^U$ hold, then the bounds \eqref{eqs:HtU} defining $\cH_{t+1}^U$ hold.
\end{itemize}
\end{proposition}

In summary, we have the chain of arguments
\begin{align*}
\cH^U_0 \rightarrow \cA^U_0 \overset{\checkmark}{\rightarrow} \cH^V_0 \rightarrow \cA^V_0 \overset{\checkmark}{\rightarrow} \cH^U_1 \rightarrow \cA^U_1 \overset{\checkmark}{\rightarrow} \cH^V_1 \rightarrow \cA^V_1 \rightarrow \ldots,
\end{align*}
and Proposition \ref{prop:general_error_induction} establishes those steps marked with a $\checkmark$. Later we will complete this chain of arguments by verifying the assumptions $\cA^U_t, \cA^V_t$ in our phase retrieval application (see e.g., Proposition \ref{prop:phase_1_tilde_close}).

\section{Proof of Theorem \ref{thm:main}}\label{sec:proof_PR}
We now apply the AMP decomposition and error recursion to randomly initialized phase retrieval, proving Theorem \ref{thm:main}. Throughout this section and Appendices \ref{sec:proof_phase_1}–\ref{sec:first_order_analysis}, $\delta$ is fixed and the implied constant in every probability bound may depend on $\delta$, and on $M$ in the strong-recovery regime.
\subsection{Preliminary results}\label{subsec:approximation_amp}
We first derive a few lemmas that will be helpful in the analysis of all three phases. Recall that \eqref{eq:tilde_u_intro} defines
\begin{align*}
\tilde \u^t := \gamma_{t,-1}\bbeta^{\star} + \sum_{j=0}^{t}\gamma_{t,j}\g^j,\qquad \tilde\v^{t+1} := \rho_{t+1}\btheta^\star + \sum_{\ell=0}^t \lambda_{t,\ell}\s^\ell.
\end{align*}
The coefficients are defined in \eqref{eq:inner_product}, $\rho_{t+1}$ in \eqref{eq:key_signal}, and $\bbeta^\star=\g^{-1}=\X^{(t)}\btheta^\star$ in \eqref{eq:conjugate_signal}. The following simple identity shows that $\gamma_{t,-1}$, the coefficient of $\bbeta^\star$, equals $\rho_t$.

\begin{lemma}\label{lem:conjugate_signal}
For every $t\geq 0$, we have $\gamma_{t,-1} = \rho_t$.
\end{lemma}
\begin{proof}
For $t=0$, \eqref{eq:rho_0} gives
\begin{align*}
   \gamma_{0,-1} = \rho_0 = \frac{1}{d}\iprod{f_0(\v^0)}{\btheta^\star}. 
\end{align*}
For $t\geq0$, the decomposition \eqref{eq:construct_uv} separates the component of $\v^{t+1}$ along $\btheta^\star$ from terms in its orthogonal complement. Taking the inner product with $\btheta^\star/d$ and using $f_{t+1}=\operatorname{id}$ gives
\begin{align*}
\gamma_{t+1,-1}
=\frac{1}{d}\iprod{\btheta^\star}{\v^{t+1}}
=\frac{1}{d}\iprod{\g^{-1}}{\w^t}-\hat c_t\gamma_{t,-1}
=\rho_{t+1}.
\end{align*}
The last equality is the definition \eqref{eq:key_signal}.
\end{proof}

The following lemma records a few useful derivative computations.

\begin{lemma}\label{lem:h_deriv_comp}
For $h_t^\star$ defined in \eqref{eq:ht_ast},
\begin{align}
\partial_u h_t^\star(u,y) &= (\hat\mu_t^2+1)\hat\mu_t y\Big(1-\tanh^2(\hat\mu_tu\sqrt{y})\Big) - \hat\mu_t,\label{eq:ht_ast_first}\\
\partial_u^2 h_t^\star(u,y) &= -2(\hat\mu_t^2+1)\hat\mu_t^2 y^{3/2}\tanh(\hat\mu_tu\sqrt{y})\Big(1-\tanh^2(\hat\mu_tu\sqrt{y})\Big),\label{eq:ht_ast_second}\\
\partial_u^3 h_t^\star(u,y) &= -2(\hat\mu_t^2 +1)\hat\mu_t^3 y^2 \big(1-\tanh^2(\hat\mu_t u\sqrt{y})\big)\big(1-3\tanh^2(\hat\mu_tu\sqrt{y})\big).\label{eq:ht_ast_third}
\end{align}
\end{lemma}
\begin{proof}
The identities follow by direct differentiation.
\end{proof}

The next lemma is a standard near-isometry estimate for Gaussian vectors, used repeatedly below.
\begin{lemma}\label{lem:gaussian_near_isometry}
Let $\bphi_1,\ldots,\bphi_m$ be independent $\N(0,\I_d/d)$ vectors, let $\a\in\R^d$ be deterministic with $\pnorm{\a}{}=1$, and set $\Phi=(\a,\bphi_1,\ldots,\bphi_m)$. There are universal constants $C,c>0$ such that, for every $\eta\in(0,1)$ with $m+\log(1/\eta)\leq cd$, with probability at least $1-\eta$,
\begin{align}\label{eq:claim_near_isometry}
\Bigpnorm{\Phi^\top\Phi-\I_{m+1}}{\op} \leq C \Big\{\sqrt{\frac{m+\log(1/\eta)}{d}}+\frac{m+\log(1/\eta)}{d}\Big\}.
\end{align}
\end{lemma}
\begin{proof}
This follows from a standard covering argument and Gaussian concentration.
\end{proof}

The final lemma gives preliminary bounds on the norms of the AMP iterates and their first-order approximations.

\begin{lemma}\label{lem:uv_norm}
Suppose $\delta>1/2$ and $n$ is larger than a universal constant. With probability $1-O(n^{-10})$, the following bounds hold simultaneously for every $t\geq0$ satisfying $t+1\leq c\sqrt d/\log^3n$, where $c>0$ is a sufficiently small universal constant:
\begin{align}
\frac{1}{\sqrt{d}}\pnorm{\tilde\v^{t+1}}{}
&= \Bigg(1+O\Big(\sqrt{\frac{t+\log n}{d}}\Big)\Bigg)
\sqrt{\rho_{t+1}^2+1},\label{eq:tilde_v_norm}\\
\frac{1}{\sqrt{d}}\pnorm{\v^{t+1}}{}
&= \Bigg(1+O\Big(\sqrt{\frac{t+\log n}{d}}\Big)\Bigg)
\sqrt{\rho_{t+1}^2+1}
+O\Big(\frac{\pnorm{\Delta_{V,t+1}}{}}{\sqrt d}\Big),\label{eq:v_norm}\\
\frac{1}{\sqrt{n}}\pnorm{\tilde\u^t}{}
&= \Bigg(1+O\Big(\sqrt{\frac{t+\log n}{d}}
+\sqrt{\frac{t+\log n}{n}}\Big)\Bigg)\sqrt{\rho_t^2+1}
+O\Big(\frac{\pnorm{\Delta_{V,t}}{}}{\sqrt d}\Big),\label{eq:tilde_u_norm}\\
\frac{1}{\sqrt{n}}\pnorm{\u^t}{}
&= \Bigg(1+O\Big(\sqrt{\frac{t+\log n}{d}}
+\sqrt{\frac{t+\log n}{n}}\Big)\Bigg)\sqrt{\rho_t^2+1}
+O\Big(\frac{\pnorm{\Delta_{V,t}}{}}{\sqrt d}
+\frac{\pnorm{\Delta_{U,t}}{}}{\sqrt n}\Big). \label{eq:u_norm}
\end{align}
Moreover, we have
\begin{align}
 \pnorm{\lambda_{t,[0:t]}}{} &= \frac{1}{\sqrt{d}}\pnorm{\w^t}{} = 1, \label{eq:lambda_norm_main}\\
\pnorm{\gamma_{t,[0:t]}}{}
&=1+O\Big(\sqrt{\frac{t+\log n}{d}}
+\frac{\pnorm{\Delta_{V,t}}{}}{\sqrt d}\Big), \label{eq:gamma_norm_loo}\\
\pnorm{\gamma_{t,[-1:t]}}{}
&=\frac{1}{\sqrt d}\pnorm{\z^t}{}
=\Bigg(1+O\Big(\sqrt{\frac{t+\log n}{d}}\Big)\Bigg)
\sqrt{\rho_t^2+1}
+O\Big(\frac{\pnorm{\Delta_{V,t}}{}}{\sqrt d}\Big). \label{eq:gamma_norm}
\end{align}
\end{lemma}

\subsection{Analysis of Phase I}\label{subsec:analysis_phase1}
Fix $\delta>1/2$, and let $C_\delta$ denote constants depending only on $\delta$. We first consider the regime where the signal strength $\rho_t$ in \eqref{eq:key_signal} satisfies $|\rho_t|\leq n^{-1/4}\log^{10}n$. Define the crossing time
\begin{align}\label{eq:cross_time_1}
\tau_I
&\coloneqq
\inf\big\{t\geq0:|\rho_t|>n^{-1/4}\log^{10}n\big\}.
\end{align}
Our goal is to show that $\tau_I=O_\delta(\log n)$ with high probability by deriving a non-asymptotic recursion for the signal strength $\rho_{t+1}$.

We start by showing that $\Delta_{U,t}, \Delta_{V,t+1}$ defined in (\ref{eq:first_order_error}) are small, and thus $\tilde\v^{t+1}$ and $\tilde\u^t$ closely track $\v^{t+1}$ and $\u^t$ throughout this phase.

\begin{proposition}\label{prop:phase_1_tilde_close}
With probability $1-O(n^{-10})$, the following bounds hold simultaneously for every $0\leq t<\tau_I$ with $t\leq\log^2n$:
\begin{align}\label{eq:Delta_bound_phase1}
\pnorm{\Delta_{U,t}}{} \vee \pnorm{\Delta_{V,t+1}}{}
\leq C_\delta\log^{11} n.
\end{align}
Moreover, on the same event,
\begin{gather}
|\hat\mu_t| \leq C_\delta\frac{\log^{10} n}{n^{1/4}},\label{eq:phase_1_mu_est}\\
\Big|\frac{\pnorm{h_t^\star(\u^t,\y)}{}}{\sqrt{d}} - \sqrt{2\delta}\hat\mu_t\Big|
\leq C_\delta|\hat\mu_t| \frac{\log^{20}n}{\sqrt n},\label{eq:phase_1_norm}\\
h_t'(u,y) = \frac{(1+\pi_t)}{\sqrt{2\delta}}\Big((\hat\mu_t^2+1)y\big(1-\tanh^2(\hat\mu_tu\sqrt{y})\big) - 1\Big), \label{eq:ht_first}
\end{gather}
where $\pi_{t}$ is defined by (\ref{eq:pi_parameter}) and satisfies $|\pi_t|\leq C_\delta\log^{20}n/\sqrt n$.
\end{proposition}

Next we introduce the key recursion formula for the signal strength $\rho_{t+1}$ in the current phase. The proof is given in Appendix \ref{subsec:proof_recursion_phase1}, after the proofs of the subsequent supporting Lemmas \ref{lem:linearization}--\ref{lem:approx_orthogonal_1}.
\begin{proposition}\label{prop:phase_1_main}
Let
\begin{align*}
\varphi_t \coloneqq \frac{1}{\sqrt{2\delta}}\frac{1}{d}
\iprod{\g^{-1}}{(\y-1)\circ \g^t}.
\end{align*}
With probability $1-O(n^{-10})$, the bounds in Proposition~\ref{prop:phase_1_tilde_close} hold and, simultaneously for $0\leq t<\tau_I$ with $t\leq\log^2n$, we have
\begin{align}\label{eq:recursion_phase_1}
\rho_{t+1} = \sqrt{2\delta}\rho_t + \varphi_t
+ O_\delta\Big(\frac{\log^{34} n}{n^{3/4}}\Big).
\end{align}
On the same event, $1\leq\tau_I\leq C_\delta\log n$.
\end{proposition}

We conclude this subsection by describing the main steps of proving Proposition \ref{prop:phase_1_main}, deferring the detailed arguments to Appendix \ref{sec:proof_phase_1}. We start by noting that the nonlinearity $h_t$ in (\ref{eq:rescaled_nonlinearity}) is driven by the nonlinear quantity $\sqrt{y}\tanh(\hat\mu_t u\sqrt{y})$. As $\hat\mu_t$ is small, we expect its behavior to be approximated by its Taylor expansion around 0, as quantified by the following result proved in Appendix \ref{subsec:h_linearization_phase1}.

\begin{lemma}\label{lem:linearization}
With probability $1-O(n^{-10})$, simultaneously for $0\leq t<\tau_I$ with $t\leq\log^2n$, we have
\begin{align}
h_t(u,y) &= (1+\pi_t)\Big(\frac{(y-1)u}{\sqrt{2\delta}} + r_t^{(1)}(u,y)\Big) \label{eq:h_first_order}\\
&= (1+\pi_t)\Big(\frac{(y-1)u}{\sqrt{2\delta}} + \frac{1}{\sqrt{2\delta}}\hat\mu_t^2(yu - \frac{1}{3} u^3y^2) + r_t^{(2)}(u,y)\Big) \label{eq:h_second_order},
\end{align}
where $|\pi_t|\leq C_\delta\log^{20}n/\sqrt n$ and, for a universal constant $C$, the remainders $r^{(1)}_t,r^{(2)}_t$ satisfy
\begin{align}\label{eq:taylor_error_bound}
|r_t^{(1)}(u,y)| \leq C\Big(\hat\mu_t^2|yu|+\hat\mu_t^2|y^2u^3|
+\hat\mu_t^4|y^3u^5|\Big), \qquad |r_t^{(2)}(u,y)|
\leq C\hat\mu_t^4\big(|y^2u^3|+|y^3u^5|\big).
\end{align}
Moreover, with $\r_t^{(1)} \coloneqq (r^{(1)}_t(u^t_i,y_i))_{i=1}^n$ and $\r_t^{(2)} \coloneqq (r^{(2)}_t(u^t_i,y_i))_{i=1}^n$, it holds that
\begin{align}\label{eq:r_vec_bound}
\frac{1}{\sqrt{d}}\pnorm{\r^{(1)}_t}{}
\leq C_\delta\frac{\log^{20} n}{\sqrt n},
\qquad
\frac{1}{\sqrt{d}}\pnorm{\r^{(2)}_t}{}
\leq C_\delta\frac{\log^{40} n}{n}.
\end{align}
\end{lemma}

An essential consequence of Lemma \ref{lem:linearization} is that  $\{\w^t\}$ are approximately orthogonal throughout this phase, as shown by the subsequent Lemma \ref{lem:lambda_almost_orthogonal} and Lemma \ref{lem:approx_orthogonal_1}. Their proofs are given in Appendices \ref{subsec:proof_orthogonality_2}  and \ref{subsec:proof_orthogonality_1} respectively.

\begin{lemma}\label{lem:lambda_almost_orthogonal}
With probability $1-O(n^{-10})$, simultaneously for $0\leq t<\tau_I$ with $t\leq\log^2n$, we have
\begin{align*}
\pnorm{\lambda_{t,[0:t-1]}}{}
\vee \pnorm{\gamma_{t,[-1:t-1]}}{}
\leq C_\delta\frac{\log^{12} n}{n^{1/4}}.
\end{align*}
\end{lemma}

\begin{lemma}\label{lem:approx_orthogonal_1}
With probability $1-O(n^{-10})$, simultaneously for $0\leq t<\tau_I$ with $t\leq\log^2n$, we have
\begin{align}\label{eq:w_approx}
\frac{1}{\sqrt{d}}\Bigpnorm{\w^t - \frac{(\y-1)\circ\g^t}{\sqrt{2\delta}}}{}
\leq C_\delta\frac{\log^{14} n}{n^{1/4}}.
\end{align}
Moreover, on the same event,
\begin{align}\label{eq:w_isometry}
\frac{1}{\sqrt{d}} \Bigpnorm{\sum_{\ell=0}^t a_\ell \w^\ell}{}
= \Big(1+O_\delta\Big(\frac{\log^{14} n}{n^{1/4}}\Big)\Big)\pnorm{a}{}.
\end{align}
The last display holds uniformly over $a=(a_0,\ldots,a_t)\in\R^{t+1}$.
\end{lemma}

Equipped with the above results, we can readily obtain a heuristic derivation of the recursion formula (\ref{eq:recursion_phase_1}). We first recall the definition of $\rho_{t+1}$ in (\ref{eq:key_signal}) and apply Lemma \ref{lem:conjugate_signal} to obtain
\begin{align*}
\rho_{t+1} = \frac{1}{d}\iprod{\g^{-1}}{h_t(\u^t,\y)} - \hat c_t\rho_t.
\end{align*}
We may approximate the first term by
\begin{align*}
\frac{1}{d}\iprod{\g^{-1}}{h_t(\u^t,\y)} &\overset{(i)}{\approx} \frac{1}{d}\iprod{\g^{-1}}{h_t(\tilde\u^t,\y)} \overset{(ii)}{\approx} \frac{1}{d}\iprod{\g^{-1}}{\frac{(\y-1)\circ \tilde\u^t}{\sqrt{2\delta}}}\\
&\overset{(iii)}{\approx} \rho_t \cdot \frac{1}{d}\iprod{\g^{-1}}{\frac{(\y-1)\circ \g^{-1}}{\sqrt{2\delta}}} + \underbrace{\frac{1}{d}\iprod{\g^{-1}}{\frac{(\y-1)\circ \g^t}{\sqrt{2\delta}}}}_{\varphi_t}.
\end{align*}
Here $(i)$ uses Proposition~\ref{prop:phase_1_tilde_close} and the estimate \eqref{eq:phase1_delta_projection}; $(ii)$ uses the Taylor expansion \eqref{eq:h_first_order}; and $(iii)$ uses $\tilde\u^t\approx\rho_t\g^{-1}+\g^t$ from Lemma~\ref{lem:lambda_almost_orthogonal}. Since $\y=\g^{-1}\circ\g^{-1}$, Gaussian polynomial concentration gives
\begin{align*}
\frac{1}{d}\iprod{\g^{-1}}{(\y-1) \circ \g^{-1}} \approx 2\delta.
\end{align*}
Lastly, a similar Taylor expansion argument as in Lemma \ref{lem:linearization} yields that $\hat c_t = \widetilde O(n^{-1/2})$, so using $|\rho_t| = \widetilde O(n^{-1/4})$, we arrive at the desired recursion in (\ref{eq:recursion_phase_1}):
\begin{align*}
\rho_{t+1} = \sqrt{2\delta} \rho_t + \varphi_t + \widetilde O(n^{-3/4}).
\end{align*}
The main Phase I result $\tau_I = O_\delta(\log n)$ follows by applying to the recursion (\ref{eq:recursion_phase_1}) a similar anticoncentration argument as introduced in \cite{li2023approximate}.

\subsection{Analysis of Phase II}\label{subsec:analysis_phase2}
Fix $\delta>1/2$ and let $c_{\weak}$ be a sufficiently small constant depending only on $\delta$.
Define the crossing time
\begin{align}\label{eq:def_tau_II}
\tau_{I\!I}\coloneqq\inf\big\{t\geq\tau_I:|\rho_t|\geq c_{\weak}\big\}.
\end{align}
 The crossing time $\tau_{I\!I}$ can be understood as the time when the AMP iterate $\v^{\tau_{I\!I}}$ achieves a weak correlation with the signal $\btheta^\star$. Our goal is to show that $\tau_{I\!I}=O_\delta(\log n)$ with high probability by proving that $|\rho_t|$ grows exponentially throughout Phase II.

\begin{proposition}\label{prop:phase_2_exp_growth}
Fix $\delta>1/2$. With probability $1-O(n^{-10})$, the conclusions of Propositions \ref{prop:phase_1_tilde_close} and \ref{prop:phase_1_main} hold and, simultaneously for every $\tau_I\leq t<\tau_{I\!I}$,
\begin{gather}\label{eq:phase2_recursion}
|\rho_{t+1}|\geq\Big(\frac{1+\sqrt{2\delta}}{2}\Big)|\rho_t|,\\
\pnorm{\Delta_{U,t}}{}\vee\pnorm{\Delta_{V,t+1}}{}\leq C_\delta\log^{12}n.
\end{gather}
On the same event,
\begin{align}\label{eq:phase2_explicit_length}
\tau_{I\!I}-\tau_I
\leq\frac{C\log n}{\log\Big(\frac{1+\sqrt{2\delta}}{2}\Big)},
\end{align}
where $C>0$ is universal. In particular, $\tau_{I\!I}\leq C_\delta\log n$.
\end{proposition}

We present the proof of Proposition \ref{prop:phase_2_exp_growth} in Appendix \ref{sec:proof_phase_2}. Compared to the Phase-I recursion in \eqref{eq:recursion_phase_1}, the recursion \eqref{eq:phase2_recursion} is simpler because the Gaussian perturbation $\varphi_t$ is lower order compared to the signal $|\rho_t|$. After time $\tau_I$, the signal is at least $n^{-1/4}\log^{10}n$, while $\varphi_t=\widetilde O(n^{-1/2})$ is smaller.

The proof follows the same steps as the derivation of \eqref{eq:recursion_phase_1}, with the main difference that the errors are now bounded relative to $\rho_t$.
First, we show that for every $\tau_I\leq t<\tau_{I\!I}$,
\begin{align*}
\pnorm{\Delta_{U,t}}{}\vee\pnorm{\Delta_{V,t+1}}{}\leq C_\delta\log^{12}n.
\end{align*}
As in Proposition~\ref{prop:phase_1_tilde_close}, this follows from the induction developed in Section~\ref{subsec:main_ind}. Next, recall that $\rho_{t+1}$ is defined by
\begin{align}\label{eq:rhot_expression}
\rho_{t+1} =\frac{1}{d}\iprod{\g^{-1}}{h_t(\u^t,\y)}-\hat c_t\rho_t =\frac{\sqrt d}{\pnorm{h_t^\star(\u^t,\y)}{}}
\Big[\frac{1}{d}\iprod{\g^{-1}}{h_t^\star(\u^t,\y)}
-\delta\rho_t\av{\partial_u h_t^\star(\u^t,\y)}\Big],
\end{align}
where $\partial_u$ is taken with $\hat\mu_t$ fixed. Using the above bound on $\pnorm{\Delta_{U,t}}{}$ to replace $\u^t$ by its Gaussian approximation $\tilde\u^t$, and a Taylor expansion as in Lemma \ref{lem:linearization}, we show that
\begin{gather*}
\frac{\pnorm{h_t^\star(\u^t,\y)}{}}{\sqrt d}
=\sqrt{2\delta}\hat\mu_t\big(1+O_\delta(\rho_t^2)\big),\\
\Big|\frac{1}{d}\iprod{\g^{-1}}{h_t^\star(\u^t,\y)}
-\delta\rho_t\av{\partial_u h_t^\star(\u^t,\y)}-2\delta\hat\mu_t\rho_t\Big|
\leq C_\delta\hat\mu_t\Big(|\rho_t|^3+\frac{\log^{14}n}{\sqrt d}\Big).
\end{gather*}
Finally, substituting these two estimates into \eqref{eq:rhot_expression} gives
\begin{align*}
|\rho_{t+1}-\sqrt{2\delta}\rho_t| \leq C_\delta|\rho_t|^3+C_\delta\frac{\log^{14}n}{\sqrt d}.
\end{align*}
Since $|\rho_t|<c_{\weak}$ and $c_{\weak}$ is small enough depending on $\delta$, the first term is at most $\frac{\sqrt{2\delta}-1}{4}|\rho_t|$. Since $|\rho_t|\geq n^{-1/4}\log^{10}n$, taking $n$ large enough makes the second term at most $\frac{\sqrt{2\delta}-1}{4}|\rho_t|$. Therefore,
\begin{align*}
|\rho_{t+1}|\geq\Big(\frac{1+\sqrt{2\delta}}{2}\Big)|\rho_t|,
\end{align*}
which proves the desired growth. Starting from $|\rho_{\tau_I}|>n^{-1/4}\log^{10}n$, this gives the time bound in \eqref{eq:phase2_explicit_length}. 

\subsection{Analysis of Phase III}\label{subsec:analysis_phase3}
The main idea of Phase III is that the asymptotic AMP state evolution (\ref{eq:amp_se_intro}) gives an accurate prediction of the behavior of randomly initialized AMP. Starting from the Phase II endpoint $\rho_{\SE,\tau_{I\!I}}^2 := \rho_{\tau_{I\!I}}^2$, it predicts that the signal strength
in Phase III evolves by
\begin{align}\label{eq:rho_SE}
\rho_{\SE,t+1}^2 = F_\delta\big(\rho_{\SE,t}^2\big),
\end{align}
with $F_\delta$ defined in \eqref{eq:def_F_delta}. Given $\rho_{\tau_{I\!I}}$, the sequence $\{\rho_{\SE,t}\}_{t\geq\tau_{I\!I}}$ is deterministic.

We start by studying the behavior of the sequence $\{\rho_{\SE,t}^2\}_{t=\tau_{I\!I}}^\infty$, or equivalently, the properties of $F_\delta$. An important quantity in this regard is Bolthausen's constant (we review Bolthausen's constant and condition in more generality in Appendix \ref{sec:bolthausen}). Define the map
\begin{align}\label{eq:h_rho}
h_\rho^\star(u,y) = (1+\rho^2)\sqrt{y}\tanh(\rho u \sqrt{y}) - \rho u,
\end{align}
so that $h_t^\star(u,y)$  as defined by (\ref{eq:ht_ast}) satisfies $h_t^\star(u,y) = h^\star_{\hat\mu_t}(u,y)$. Define Bolthausen's constant as
\begin{align}\label{def:bolthausen}
\mathcal B(\rho) := \frac{\E\Big[\Big(\partial_u h_\rho^{\star}(\rho G + W,G^2)\Big)^2\Big]}{\E\Big[\Big(h_\rho^{\star}(\rho G + W,G^2)\Big)^2\Big]},
\end{align}
where $W,G$ are independent $\N(0,1)$ variables. The following proposition describes the behavior of the map $F_\delta$ and the role of $\mathcal B(\cdot)$ in the stability of its fixed points. Recall that $\delta_{\weak} = 1/2$ and $\delta_{\str}^{\AMP}\approx1.13$ from \eqref{eq:strong_recovery_threshold}. We say $\mu$ is a stable (resp. unstable) fixed point of $F_\delta(\cdot)$ if $F_\delta(\mu) = \mu$ and $F_\delta'(\mu)\leq 1$ (resp. $F_\delta'(\mu) > 1$).
 
\begin{proposition}\label{prop:phase_3_asymp}
The map $\mu \mapsto F_\delta(\mu)$ admits the following behavior. 
\begin{enumerate}
\item If $\delta \leq \delta_{\weak}$, $F_\delta$ has a unique stable fixed point at $0$.
\item If $\delta>\delta_{\str}^{\AMP}$, the only fixed point of $F_\delta$ is $0$, which is unstable, and $F_\delta(\mu)\geq(\delta/\delta_{\str}^{\AMP})\mu$ for all $\mu\in\R_+$. Consequently, the iterates of $F_\delta$ started from any $\mu>0$ grow at least geometrically.
\item If $\delta_{\weak} < \delta \leq \delta_{\str}^{\AMP}$, we have two cases. 
\begin{itemize}
\item[(1)] When $\delta_{\weak} < \delta \leq 1$, in addition to $\mu=0$, $F_\delta$ has a unique positive and stable fixed point $\mu_\infty(\delta)$. 
\item[(2)] When $1 < \delta < \delta_{\str}^{\AMP},$ $F_\delta$ has two positive fixed points $\mu_\infty(\delta)\leq \tilde \mu_\infty(\delta)$, where the smaller one is stable and the larger one is unstable. As $\delta \rightarrow \delta_{\str}^{\AMP}$, these fixed points merge at $\mu_\infty(\delta_{\str}^{\AMP}) = \tilde\mu_\infty(\delta_{\str}^{\AMP})$, with the common value being
\begin{align*}
\mu_\infty(\delta_{\str}^{\AMP}) \approx 5.52.
\end{align*}
Moreover, we have
\begin{align}\label{eq:B_claim}
\mathcal B\left(\sqrt{\mu_{\fix}}\right) = F'_\delta(\mu_{\fix}), \quad \mu_{\fix} \in \big\{\mu_\infty(\delta), \tilde\mu_\infty(\delta)\big\},
\end{align}
with $\mathcal B\left(\sqrt{\mu_\infty(\delta)}\right) < 1$ and $\mathcal B\left(\sqrt{\tilde\mu_\infty(\delta)}\right) > 1$ for all $\delta \in (1, \delta_{\str}^{\AMP})$, and $\rho_\infty(\delta_{\str}^{\AMP}) := \sqrt{\mu_\infty(\delta_{\str}^{\AMP})} \approx 2.35$ being the unique point such that $\mathcal B(\rho) = 1$. 
\end{itemize}
\end{enumerate}
\end{proposition}

\begin{figure}[h]
  \centering
  \includegraphics[width=0.49\textwidth]{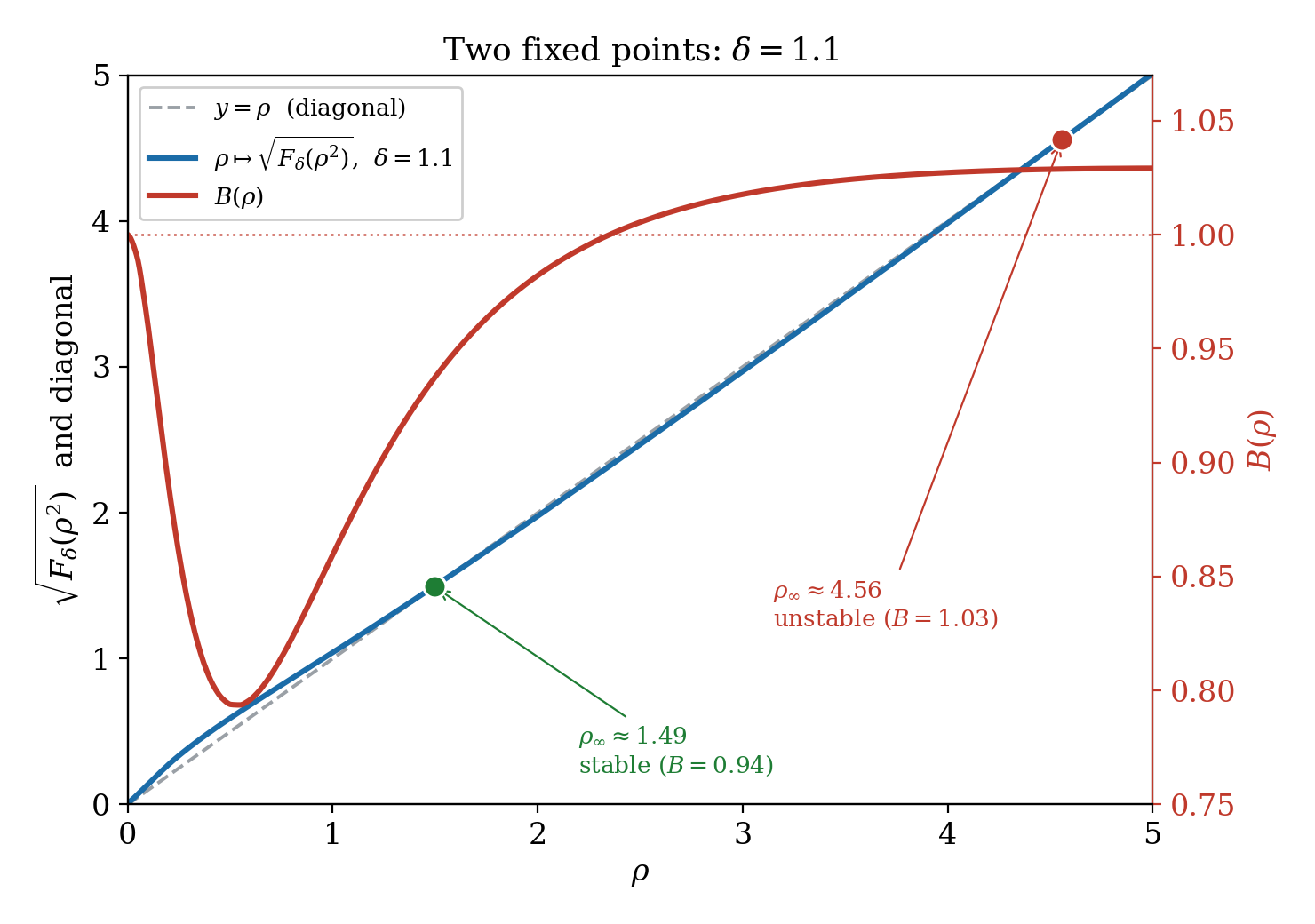}\hfill
  \includegraphics[width=0.49\textwidth]{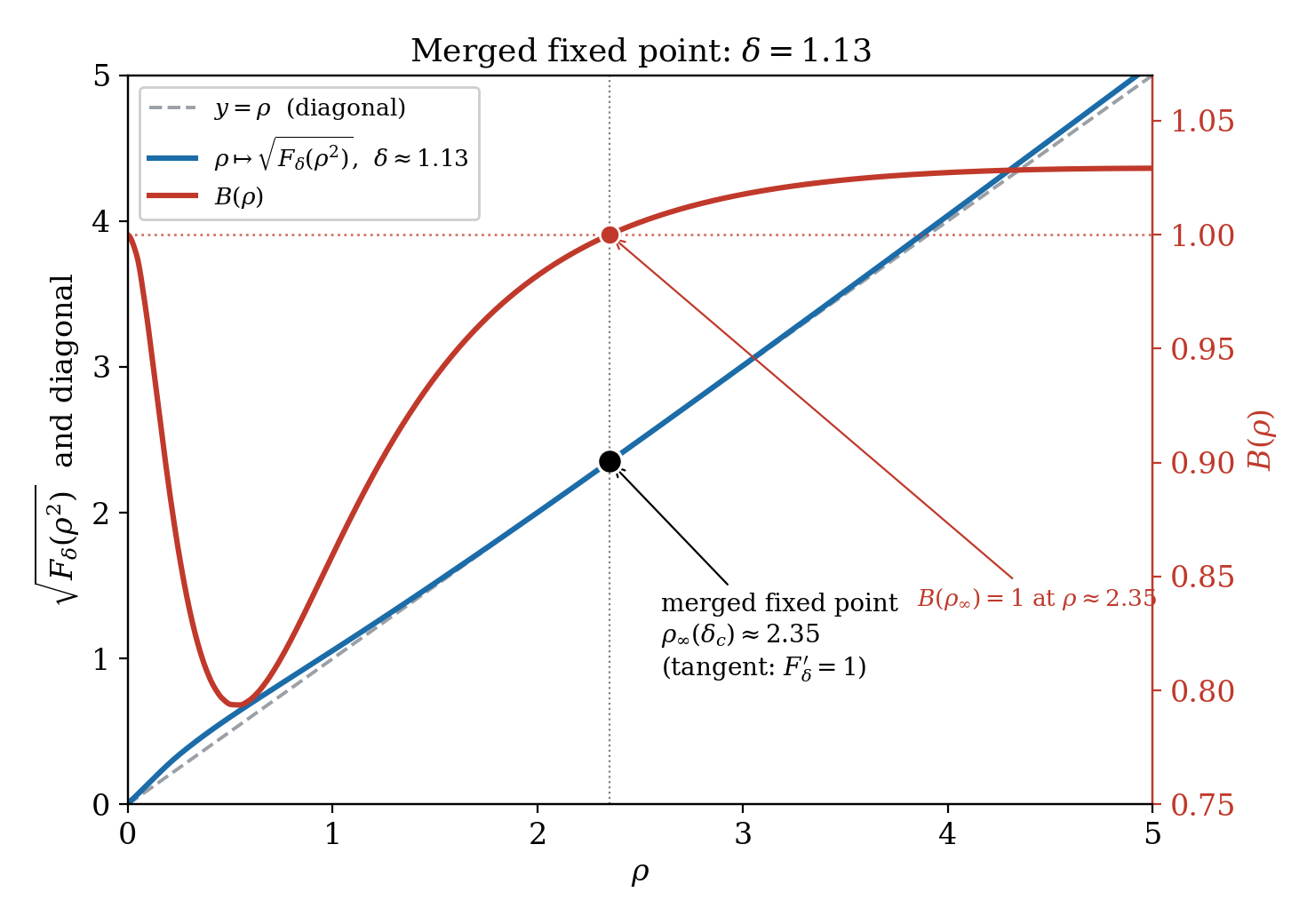}
  \caption{Left: $\delta=1.1$, two fixed points. Right:
           $\delta \approx1.13$, the merged fixed point at $\rho_\infty(\delta_{\str})\approx2.35$.}
  \label{fig:se-bolthausen}
\end{figure}

The proof of Proposition \ref{prop:phase_3_asymp} is given in Appendix \ref{subsec:proof_fixed_point}, and we visualize in Figure \ref{fig:se-bolthausen} the behavior of $\rho \mapsto \sqrt{F_\delta(\rho^2)}$ and $\mathcal B(\rho)$, where $\delta = 1.1$ with two fixed points in the left panel, and $\delta = \delta_{\str}^{\AMP} \approx 1.13$ with a merged fixed point in the right panel. Recall from Section~\ref{sec:main_result} that $\rho_\infty(\delta)=\sqrt{\mu_\infty(\delta)}$. Proposition \ref{prop:phase_3_asymp} shows the following behavior of the state evolution sequence.

\begin{table}[H]
\centering
\begin{tabular}{c|c|c}
\hline
Regime & $\delta$ & State evolution behavior \\
\hline
No recovery & $\delta \leq \delta_{\weak}$ &
$\rho_{\SE,t} \to 0$ \\
Intermediate recovery & $\delta_{\weak} < \delta \leq \delta_{\str}^{\AMP}$ &
$\rho_{\SE,t} \to \rho_\infty(\delta)<\infty$ \\
Strong recovery & $\delta>\delta_{\str}^{\AMP}$ &
$\rho_{\SE,t} \to \infty$ \\
\hline
\end{tabular}
\caption{State evolution phase diagram}
\label{tab:se_phase_diagram}
\end{table}

The main goal of the remainder of this section is to show that, in the current phase, the empirical signal strength $\{\rho_t\}$ of randomly initialized AMP exhibits the same behavior as the state evolution prediction in Table \ref{tab:se_phase_diagram}. In particular, we verify that the following holds with high probability:
\begin{itemize}
\item When $\delta_{\weak}<\delta<\delta_{\str}^{\AMP}$, $|\rho_t|$ converges to the SE fixed point $\rho_\infty(\delta)$.
\item When $\delta>\delta_{\str}^{\AMP}$, $|\rho_t|$ keeps growing geometrically until it exceeds any fixed level $M$.
\end{itemize}
\subsubsection{Intermediate recovery regime: $\delta_{\weak}<\delta<\delta_{\str}$}

In the intermediate regime we follow the trajectory over a polynomial time horizon and compare it with the state evolution (\ref{eq:rho_SE}).

\begin{proposition}\label{prop:phase_3_weak_main}
Fix $\delta\in(\delta_{\weak},\delta_{\str})$. There is a constant $c_\delta>0$ depending only on $\delta$ such that, for large enough $n$, the following hold with probability $1-O(n^{-10})$. For $\tau_{I\!I}\leq t<c_\delta n^{1/3}/\log^{20}n$,
\begin{align}
\big|\rho_{t+1}^2-F_\delta(\rho_t^2)\big|
&\leq C_\delta\frac{(t+1)\log^{12}n}{\sqrt n},
\label{eq:phase_3_recursion}\\
\pnorm{\Delta_{U,t}}{}\vee\pnorm{\Delta_{V,t+1}}{}
&\leq C_\delta(t+1)\log^{10}n.
\label{eq:phase3_direct_residual}
\end{align}
The state-evolution sequence $\rho_{\SE,t}^2$, initialized at $\rho_{\SE,\tau_{I\!I}}^2=\rho_{\tau_{I\!I}}^2$ and evolving by \eqref{eq:rho_SE}, satisfies
\begin{align}\label{eq:phase3_SE_geometric}
\big|\rho_{\SE,t}^2-\mu_\infty(\delta)\big| \leq C_\delta\exp\Big(-\frac{t-\tau_{I\!I}}{C_\delta}\Big), \qquad t\geq\tau_{I\!I}.
\end{align}
Uniformly for $\tau_{I\!I}\leq t\leq c_\delta n^{1/3}/\log^{20}n$,
\begin{align}\label{eq:convergence_decomposition}
\big|\rho_t^2-\rho_\infty^2(\delta)\big|
\leq\big|\rho_{\SE,t}^2-\mu_\infty(\delta)\big|
+C_\delta\frac{(t+1)\log^{12}n}{\sqrt n}.
\end{align}
\end{proposition}

The two errors in \eqref{eq:convergence_decomposition} have a natural interpretation:
\begin{align*}
\big|\rho_t^2-\rho_\infty^2(\delta)\big|
\leq\underbrace{\big|\rho_{\SE,t}^2-\mu_\infty(\delta)\big|}_{\text{SE convergence}}+\underbrace{\big|\rho_t^2-\rho_{\SE,t}^2\big|}_{\text{SE approximation}}.
\end{align*}
The second term is bounded by $C_\delta(t+1)\log^{12}n/\sqrt n$, while the first decreases exponentially.

We now outline how the state-evolution approximation error is controlled, and where the Bolthausen condition enters. In Appendix \ref{subsec:proof_phase_3_weak} we show that, on a suitable compact interval, there exists $0<\chi_\delta<1$ such that
\begin{align}\label{eq:phase3_exact_bolthausen_gap}
0\leq\frac{\mu F_\delta'(\mu)}{F_\delta(\mu)} = \mathcal B(\sqrt\mu)\leq1-\chi_\delta.
\end{align}
The second-order induction of Section \ref{subsec:main_ind} then gives the contraction
\begin{align*}
\pnorm{\Delta'_{V,t+1}}{} \leq \Big(1-\frac{\chi_\delta}{8}\Big)\pnorm{\Delta'_{V,t}}{} + C(t+1)\log^{5/2}n,
\end{align*}
where the multiplier is obtained by composing the two bounds in \eqref{eq:so_intro}, namely $\big[\sqrt{\mathcal B_{f_t}(\rho_t)}+\cE_t^f\big]\big[\sqrt{\mathcal B_{h_t}(\rho_t)}+\cE_t^h\big]$. This contraction yields \eqref{eq:phase3_direct_residual} over the whole polynomial time horizon. As in Phase II, we then combine \eqref{eq:phase3_direct_residual} with uniform concentration to obtain
\begin{align*}
\frac{1}{d}\pnorm{h_t^\star(\u^t,\y)}{}^2
&=F_\delta(\rho_t^2)+O_\delta\Big(\frac{(t+1)\log^{12}n}{\sqrt n}\Big),\\
\frac{1}{d}\big\langle\g^{-1},h_t^\star(\u^t,\y)\big\rangle
-\delta\rho_t\av{\partial_u h_t^\star(\u^t,\y)}
&=\operatorname{sign}(\rho_t)F_\delta(\rho_t^2)
+O_\delta\Big(\frac{(t+1)\log^{12}n}{\sqrt n}\Big).
\end{align*} 
Substituting these estimates into \eqref{eq:rhot_expression} proves \eqref{eq:phase_3_recursion}. Finally, \eqref{eq:phase3_exact_bolthausen_gap} gives
\begin{equation*}
\begin{aligned}
\Big|\log\Big(\frac{\rho_{t+1}^2}{\rho_{\mathsf{SE},t+1}^2}\Big)\Big| \leq (1-\chi_\delta) \Big|\log\Big(\frac{\rho_t^2}{\rho_{\mathsf{SE},t}^2}\Big)\Big| + C_\delta \big|\rho_{t+1}^2-F_\delta(\rho_t^2)\big|,
\end{aligned}
\end{equation*}
and iterating this contraction with \eqref{eq:phase_3_recursion} gives
\begin{align*}
\big|\rho_t^2-\rho_{\mathsf{SE},t}^2\big| \leq C_\delta \frac{(t+1)\log^{12}n}{\sqrt n}.
\end{align*}

\subsubsection{Strong recovery regime: $\delta>\delta_{\str}$}

Fix $\delta>\delta_{\str}$ and a target $M\geq1$, and let
\begin{align}\label{eq:strong_burnin_and_target}
\tau_{I\!I\!I}\coloneqq\inf\big\{t\geq\tau_{I\!I}:|\rho_t|\geq M\big\}.
\end{align}
As we show below, it takes a constant number of additional steps, depending on $\delta$ and $M$, for the signal strength to grow from the weak-recovery level $c_{\weak}$ to any fixed target $M$.

\begin{proposition}\label{prop:phase_3_strong_main}
Fix $\delta>\delta_{\str}$ and $M\geq1$. There exists some $C_{\delta,M} > 0$ such that the following holds for large enough $n$. With probability $1-O(n^{-10})$, for every $\tau_{I\!I}\leq t\leq\tau_{I\!I\!I}$,
\begin{align}\label{eq:strong_relative_residual}
\pnorm{\Delta_{U,t}}{} \vee \pnorm{\Delta_{V,t}}{} \leq C_{\delta,M}\log^Cn,
\end{align}
and
\begin{align}\label{eq:strong_recovery_recursion}
\big|\rho_{t+1}^2-F_\delta(\rho_t^2)\big| \leq C_{\delta,M}\frac{\log^Cn}{n^{1/4}},\qquad \rho_{t+1}^2 \geq\Big(\frac{1+\delta/\delta_{\str}}{2}\Big)\rho_t^2.
\end{align}
As a consequence, the stopping time $\tau_{I\!I\!I}$ is finite and
\begin{align}\label{eq:strong_total_iteration_bound}
\tau_{I\!I\!I}-\tau_{I\!I}\leq C_{\delta,M},
\qquad
|\rho_{\tau_{I\!I\!I}}|\geq M.
\end{align}
\end{proposition}

We briefly walk through the proof of Proposition~\ref{prop:phase_3_strong_main}, deferring the details to Appendix~\ref{sec:proof_strong_recovery}. As in the previous phases, an essential step is to control the AMP approximation errors. Since we only need to cover a constant number of additional steps, we use the first-order analysis in Appendix~\ref{sec:first_order_analysis}. Uniform concentration then gives \eqref{eq:strong_recovery_recursion}. Combining this with the population bound $F_\delta(\mu)\geq(\delta/\delta_{\str})\mu$ and taking $n$ sufficiently large gives the stated exponential growth.

\subsection{Completing the proof of Theorem \ref{thm:main}}\label{subsec:complete_main}

In this section we assemble the phase-specific analysis to prove Theorem \ref{thm:main} along with its corollary.

\begin{proof}[Proof of Theorem \ref{thm:main}]
The decomposition \eqref{eq:main_decomp} follows from the explicit coupling constructed in Section \ref{subsec:couple_amp}, specifically from \eqref{eq:tilde_u_intro}. The condition $\sum_{\ell=0}^{t-1}\lambda_{t-1,\ell}^2=1$ follows from \eqref{eq:lambda_norm_main}.

The weak recovery time follows from Propositions \ref{prop:phase_1_main} and \ref{prop:phase_2_exp_growth} with $\tau_{\weak}\coloneqq\tau_{I\!I}$ and $c_\delta=c_{\weak}$. The bound
\begin{align}\label{eq:main_weak_error_comparison}
\pnorm{\Delta_{V,\tau_{\weak}}}{}\leq C_\delta\log^{12}n
\end{align}
follows from Proposition~\ref{prop:phase_2_exp_growth}.

In the intermediate regime, the bound on $\pnorm{\Delta_{V,t}}{}$ follows from Propositions \ref{prop:phase_1_tilde_close}, \ref{prop:phase_2_exp_growth} and \ref{prop:phase_3_approx_weak}, and the comparison with state evolution and its convergence follow from Proposition \ref{prop:phase_3_weak_main}.

The strong recovery claims follow from Proposition \ref{prop:phase_3_strong_main} with $M=2/\sqrt\eps$ and $\tau_\eps\coloneqq\tau_{I\!I\!I}$. Combining the Phase-II crossing time with \eqref{eq:strong_total_iteration_bound} gives $\tau_\eps\leq C_{\delta,\eps}\log n$. It remains to prove (\ref{eq:strong_correlation}). The decomposition \eqref{eq:main_decomp}, together with \eqref{eq:claim_near_isometry} and \eqref{eq:v_norm}, gives
\begin{align}\label{eq:strong_theorem_correlation}
\corr(\v^{\tau_\eps},\btheta^\star)
&\geq\frac{|\rho_{\tau_\eps}|}{\sqrt{1+\rho_{\tau_\eps}^2}} - C\Bigg(\sqrt{\frac{\tau_\eps+\log n}{d}} + \frac{\pnorm{\Delta_{V,\tau_\eps}}{}}{\sqrt d}\Bigg)\notag\\
&\overset{(i)}{\geq} 1-\frac{1}{2\rho_{\tau_\eps}^2} - C_{\delta,\eps}\frac{\log^{C}n}{\sqrt n}
\overset{(ii)}{\geq} 1 - \frac{\eps}{8}-C_{\delta,\eps}\frac{\log^{C}n}{\sqrt n}
\overset{(iii)}{\geq} 1-\eps.
\end{align}
Here $(i)$ uses $x/\sqrt{1+x^2}\geq1-1/(2x^2)$ together with the bounds on $\tau_\eps$ and $\pnorm{\Delta_{V,\tau_\eps}}{}$, $(ii)$ uses $|\rho_{\tau_\eps}|\geq2/\sqrt\eps$, and $(iii)$ holds for $n$ large enough. The proof is complete.
\end{proof}

\begin{proof}[Proof of Corollary \ref{cor:corr}]
Throughout, we work on the event of Theorem~\ref{thm:main}, which has probability
$1-O_\delta(n^{-10})$, and let $\tau_{\weak}\leq t\leq n^{1/3}/(C_\delta\log^{20}n)$. Lemma \ref{lem:gaussian_near_isometry} gives, on an event of probability $1-O(n^{-10})$ and
simultaneously for all such $t$,
\begin{align}\label{eq:cor_conc_1}
\Big|\frac{1}{d}\iprod{\sum_{\ell=0}^{t-1}\lambda_{t-1,\ell}\s^\ell}{\btheta^\star}\Big|
\leq C\sqrt{\frac{t+\log n}{d}}.
\end{align}
The decomposition \eqref{eq:main_decomp} gives
\begin{align*}
\frac{1}{d}\iprod{\v^t}{\btheta^\star}
=\rho_t+\frac{1}{d}\iprod{\sum_{\ell=0}^{t-1}\lambda_{t-1,\ell}\s^\ell}{\btheta^\star}
+\frac{1}{d}\iprod{\Delta_{V,t}}{\btheta^\star}
=\rho_t+O\Bigg(\sqrt{\frac{t+\log n}{d}}+\frac{\pnorm{\Delta_{V,t}}{}}{\sqrt d}\Bigg).
\end{align*}
On the other hand, \eqref{eq:v_norm} gives
\begin{align*}
\frac{1}{\sqrt d}\pnorm{\v^t}{}
=\Bigg(1+O\Big(\sqrt{\frac{t+\log n}{d}}\Big)\Bigg)\sqrt{1+\rho_t^2}
+O\Big(\frac{\pnorm{\Delta_{V,t}}{}}{\sqrt d}\Big).
\end{align*}
By Theorem~\ref{thm:main}, $\pnorm{\Delta_{V,t}}{}\leq C_\delta(t+1)\log^{12}n$, and
$(t+1)\log^{12}n\leq C_\delta n^{1/3}$ for $t\leq n^{1/3}/(C_\delta\log^{20}n)$. Hence both error terms above are
at most $C_\delta n^{-1/6}$, which is at most $1/2$ for $n$ large enough. Moreover, by
\eqref{eq:phase3_rho_range},
\begin{align}\label{eq:cor_rho_range}
\frac{c_{\weak}^2}{2}<\rho_t^2\leq C.
\end{align}
Since $\pnorm{\btheta^\star}{}=\sqrt d$, combining the two expansions with the upper bound
in \eqref{eq:cor_rho_range} yields
\begin{align}\label{eq:cor_corr_expansion}
\Big|\corr(\v^t,\btheta^\star)-\frac{|\rho_t|}{\sqrt{1+\rho_t^2}}\Big|
\leq C_\delta n^{-1/6}.
\end{align}
Next, the map $x\mapsto\sqrt{x/(1+x)}$ is Lipschitz on $[c_{\weak}^2/2,\infty)$ with
constant at most $C/c_{\weak}$. Since $\rho_\infty^2(\delta)=\mu_\infty(\delta)\geq
c_{\weak}^2$, the lower bound in \eqref{eq:cor_rho_range} and Theorem~\ref{thm:main} give
\begin{align*}
\Big|\frac{|\rho_t|}{\sqrt{1+\rho_t^2}}
-\frac{\rho_\infty(\delta)}{\sqrt{1+\rho_\infty^2(\delta)}}\Big|
\leq C_\delta\big|\rho_t^2-\mu_\infty(\delta)\big|
\leq C_\delta\exp\Big\{-\frac{t-\tau_{\weak}}{C_\delta}\Big\}
+C_\delta\frac{(t+1)\log^{12}n}{\sqrt n},
\end{align*}
and the last term is at most $C_\delta n^{-1/6}$. Together with
\eqref{eq:cor_corr_expansion}, this proves \eqref{eq:corr_intermediate}. If moreover
$t-\tau_{\weak}\geq C_\delta\log n$, with $C_\delta$ at least the constant in the
exponent, then the exponential term is at most $n^{-1}\leq n^{-1/6}$, which proves the
last claim of the first part.

For strong recovery, fix $\eps\in(0,1)$. By Theorem~\ref{thm:main}, with probability
$1-O_{\delta,\eps}(n^{-10})$ the time $t=\tau_\eps\leq C_{\delta,\eps}\log n$ satisfies
$\corr(\v^t,\btheta^\star)\geq1-\eps$. Letting $n\rightarrow\infty$ proves
\eqref{eq:strong-correlation-iterated-limit}.
\end{proof}
\paragraph*{Statement on AI use} The authors used ChatGPT 5.6/6.1 in preparation of this manuscript to check the proofs, write code and produce figures, and polish the writing. The authors assume responsibility for the contents of this manuscript.

\bibliographystyle{alpha}
\bibliography{PR_amp}
\appendix
\section{Proofs for Section \ref{sec:main_proof}}\label{Sec:Error_Proofs}
This appendix proves the results of Section \ref{sec:main_proof}. We begin with the proof showing that our phase retrieval example satisfies the nonlinearity assumption.

\begin{proof}[Proof of Lemma \ref{lem:phase_retrieval_parameter_lipschitz}]
For $z\in\mathbb R$, we have the identity
\begin{align}\label{eq:tanh_parameter_integral}
\frac{\tanh(\mu z)}{\mu} = z\int_0^1\sech^2(s\mu z)\,ds.
\end{align}
Write
\begin{align*}
K(\mu,u,y)=(1+\mu^2)\sqrt y\, \frac{\tanh(\mu u\sqrt y)}{\mu}-u.
\end{align*}
The integral representation gives
\begin{align*}
|K(\mu,u,y)| + |\partial_\mu K(\mu,u,y)| \leq C(1+\mu^2) \big(|u|+|u|y+u^2y^{3/2}\big).
\end{align*}
For $1\leq j\leq3$, direct differentiation gives
\begin{align*}
\partial_u^jK(\mu,u,y) = (1+\mu^2)\mu^{j-1}y^{(j+1)/2} \tanh^{(j)}(\mu u\sqrt y)-\bm 1\{j=1\}.
\end{align*}
The derivatives of $\tanh$ of orders one through four are bounded. Differentiating once in $\mu$, and using Young's inequality for the terms $|u|y^{5/2}$ and $u^2y^{3/2}$, therefore yields, for $a\in\{0,1\}$ and $0\leq j\leq3$,
\begin{align}\label{eq:tanh_parameter_derivative_bound}
\Big|\partial_\mu^a\partial_u^jK(\mu,u,y)\Big|
\leq C(1+|\mu|^5)(1+|u|^5+y^4).
\end{align}
Equation \eqref{eq:h_parameterization}, the product rule, and $\delta>1/2$ now give
\begin{align*}
\Big|\partial_\mu\partial_u^j\mathsf h(u,y;\mu,\pi)\Big|
+\Big|\partial_\pi\partial_u^j\mathsf h(u,y;\mu,\pi)\Big|
\leq C_Dn^{C_D}(1+|u|^5+y^4).
\end{align*}
Integrating along the line segment from $(\mu,\pi)$ to $(\tilde\mu,\tilde\pi)$ proves the claim.
\end{proof}

We next give the proof of the coupling.

\begin{proof}[Proof of Lemma \ref{lem:X_coupling}]
(1) Since $\btheta^\star=\sqrt d\,\q^{-1}$ and all the other $\q^\ell$ are orthogonal to $\q^{-1}$, \eqref{eq:intro_X} gives
\begin{align}\label{eq:coupling_signal_identity}
\X^{(t)}\btheta^\star=\g^{-1}.
\end{align}
Since $\y=\varphi(\g^{-1})$ by construction \eqref{eq:intro_X}, where $\varphi$ uses the randomness $\cR$, this gives $\y=\varphi(\X^{(t)}\btheta^\star)$. Hence $(\X^{(t)},\y)$ follows the single-index model \eqref{def:single_index} with design matrix $\X^{(t)}$, and the identity in law in part~1 follows once we show that $\X^{(t)}$ has i.i.d.\ $\N(0,1/d)$ entries and is independent of $(\v^0,\cR)$.

For $t\geq0$, let
\begin{align*}
\tilde\X^{(t)}
&=\sum_{\ell=-1}^tP_{\bR^{[0:\ell-1]}}^\perp\frac{\g^\ell\q^{\ell\top}}{\sqrt d}
+\sum_{\ell=0}^{t-1}\frac{\r^\ell\s^{\ell\top}}{\sqrt d}P_{\Q^{[-1:\ell]}}^\perp
+P_{\bR^{[0:t-1]}}^\perp\W P_{\Q^{[-1:t]}}^\perp.
\end{align*}
We start with
\begin{align*}
\X^{(-1)}=\frac{\g^{-1}\q^{-1\top}}{\sqrt d}+\W P_{\q^{-1}}^\perp
\end{align*}
and induct along
\begin{align*}
\X^{(-1)}\rightarrow\tilde\X^{(0)}\rightarrow\X^{(0)}\rightarrow\tilde\X^{(1)}\rightarrow\X^{(1)}\rightarrow\cdots\rightarrow\X^{(t)}.
\end{align*}
The initial matrix has independent $\N(0,1/d)$ entries and is independent of $(\v^0,\cR)$.

To check the inductive step, suppose the claim holds for $\tilde\X^{(t)}$. We now check that $\tilde\X^{(t+1)}$ has $\N(0,1/d)$ entries and is independent of $(\v^0,\cR).$ To simplify notation, define
\begin{align*}
\cF_t^U &:= \cF\Big(\cR,\v^0,\{\g^j\}_{j=-1}^t,\{\s^j\}_{j=0}^{t-1}\Big),\qquad \cF_t^V := \cF\Big(\cR,\v^0,\{\g^j\}_{j=-1}^t,\{\s^j\}_{j=0}^t\Big).
\end{align*}

Since $\r^t\perp\bR^{[0:t-1]}$, we have
\begin{align*}
\X^{(t)}
&=\sum_{\ell=-1}^tP_{\bR^{[0:\ell-1]}}^\perp\frac{\g^\ell\q^{\ell\top}}{\sqrt d}
+\sum_{\ell=0}^{t-1}\frac{\r^\ell\s^{\ell\top}}{\sqrt d}P_{\Q^{[-1:\ell]}}^\perp+P_{\bR^{[0:t-1]}}^\perp
\underbrace{\Big(\frac{\r^t\s^{t\top}}{\sqrt d}+P_{\r^t}^\perp\W\Big)}_{\W^{(t)}}
P_{\Q^{[-1:t]}}^\perp.
\end{align*}
Conditional on $\cF_t^U$, $\r^t$ is fixed and $(\s^t,\W)$ is independent of $\cF_t^U$. Thus, we have
\begin{align*}
\sP\Big(\W^{(t)}\Big|\cF_t^U\Big)=\sP(\W).
\end{align*}
Replacing $\W$ by $\W^{(t)}$ in $\tilde\X^{(t)}$ therefore preserves its joint law with $(\v^0,\cR)$, proving the claim for $\X^{(t)}$.

The induction from $\X^{(t)}$ to $\tilde\X^{(t+1)}$ is similar. Conditional on $\cF_t^V$, $\q^{t+1}$ is fixed and $(\g^{t+1},\W)$ is independent of $\cF_t^V$. The replacement
\begin{align*}
\W\rightarrow\frac{\g^{t+1}\q^{(t+1)\top}}{\sqrt d}+\W P_{\q^{t+1}}^\perp
\end{align*}
preserves the conditional Gaussian law and transforms $\X^{(t)}$ into $\tilde\X^{(t+1)}$. This completes the induction and proves (1).

(2) Fix $0\leq\ell\leq t$. Since $\z^\ell$ belongs to the span of $\q^{-1},\ldots,\q^\ell$, multiplying \eqref{eq:intro_X} by $\z^\ell$ gives
\begin{align}\label{eq:coupling_Xz}
\X^{(t)}\z^\ell =\sum_{j=-1}^\ell P_{\bR^{[0:j-1]}}^\perp\g^j\frac{\q^{j\top}\z^\ell}{\sqrt d} + \sum_{j=0}^{\ell-1}\r^j\frac{\big(P_{\Q^{[-1:j]}}^\perp\s^j\big)^\top\z^\ell}{\sqrt d} = \u^\ell+\hat b_\ell\w^{\ell-1}.
\end{align}
At $\ell=0$ this is the initialization $\u^0=\X^{(t)}f_0(\v^0)$; for $\ell\geq1$ it is the update for $\u^\ell$ in \eqref{eq:main_amp}.

Similarly, $\w^\ell$ belongs to the span of $\r^0,\ldots,\r^\ell$, so
\begin{align}\label{eq:coupling_Xtw}
\X^{(t)\top}\w^\ell = \sum_{j=-1}^\ell\q^j\frac{\big(P_{\bR^{[0:j-1]}}^\perp\g^j\big)^\top\w^\ell}{\sqrt d} + \sum_{j=0}^\ell P_{\Q^{[-1:j]}}^\perp\s^j\frac{\r^{j\top}\w^\ell}{\sqrt d} = \v^{\ell+1}+\hat c_\ell\z^\ell.
\end{align}
These are the AMP updates through $\v^{t+1}$, using only the directions through $\q^t$ and $\r^t$.
\end{proof}

The remainder of this section gives the proofs for the second-order error recursion. The starting point is the following lemma which gives a useful decomposition of these errors.

\begin{lemma}
\label{lem:error_decompose}
For $t\geq0$, define the coefficients
\begin{align}\label{def:AB_tl}
\begin{aligned}
A_t^\ell
&\coloneqq \frac{1}{d}\iprod{\s^\ell}{f_t(\hat\v^t)}
-\av{f_t'(\hat\v^t)}\lambda_{t-1,\ell}
-\sum_{k=\ell+1}^{t-1}\alpha_{t-1}^{k-1}\lambda_{k-1,\ell},
&&0\leq\ell\leq t-1,\quad t\geq1,\\
B_t^\ell
&\coloneqq \frac{1}{d}\iprod{\g^\ell}{h_t(\hat\u^t,\y)}
-\delta\av{h_t'(\hat\u^t,\y)}\gamma_{t,\ell}
-\sum_{k=\ell}^{t-1}\beta_t^k\gamma_{k,\ell},
&&0\leq\ell\leq t.
\end{aligned}
\end{align}
Then $\Delta_{U,0}'=0$, and for $t\geq1$,
\begin{align}\label{def:Delta_Ut}
\begin{aligned}
\Delta_{U,t}'
&=P_{\w^{t-1}}^\perp
\sum_{\ell=0}^{t-1}\frac{\r^\ell}{\sqrt d}
\iprod{\s^\ell}{f_t(\v^t)-f_t(\hat\v^t)}
+\sum_{\ell=0}^{t-1}\sqrt d\,\r^\ell A_t^\ell
+\Delta_{U,t}^{\low},\\
\Delta_{V,t+1}'
&=P_{\bar\z^t}^\perp
\sum_{\ell=0}^{t}\frac{\q^\ell}{\sqrt d}
\iprod{\g^\ell}{h_t(\u^t,\y)-h_t(\hat\u^t,\y)}
+\sum_{\ell=0}^{t}\sqrt d\,\q^\ell B_t^\ell
+\Delta_{V,t+1}^{\low},
\end{aligned}
\end{align}
where the second identity also holds at $t=0$.
\end{lemma}

\begin{proof}
We first prove the decomposition for $\u^t$. The case $t=0$ gives $\Delta_{U,0}'=0$, so suppose $t\geq1$. By (\ref{eq:construct_uv}) and the definition of $\tilde\u^t$ in (\ref{eq:tilde_u_intro}), we have
\begin{align*}
\u^t &= \tilde{\u}^t - \sum_{\ell=-1}^t P_{\bR^{[0:\ell-1]}}\g^\ell\gamma_{t,\ell} + \sum_{\ell=0}^{t-1} \r^\ell \frac{1}{\sqrt{d}} (P_{\Q^{[-1:\ell]}}^\perp\s^\ell)^\top\z^t - \hat b_t \w^{t-1}\\
&= \tilde{\u}^t + \sum_{\ell=0}^{t-1} \r^\ell\Big(\frac{1}{\sqrt{d}}\s^{\ell\top}\z^t - \hat b_t \r^{\ell\top}\w^{t-1}\Big) \underbrace{-\sum_{\ell=-1}^t P_{\bR^{[0:\ell-1]}}\g^\ell\gamma_{t,\ell} - \sum_{\ell=0}^{t-1} \r^\ell\frac{1}{\sqrt{d}}(P_{\Q^{[-1:\ell]}}\s^\ell)^\top \z^t}_{\Delta^{\low}_{U,t}}\\
&= \tilde{\u}^t + \sum_{\ell=0}^{t-1} \sqrt{d}\r^\ell\Big(\frac{1}{d}\s^{\ell\top}\z^t - \hat b_t \lambda_{t-1,\ell}\Big) + \Delta^{\low}_{U,t}.
\end{align*}
To further analyze the middle term, using $\hat b_t = \av{f_t'(\v^t)}$, we have
\begin{align*}
&\sum_{\ell=0}^{t-1} \sqrt{d}\r^\ell\Big(\frac{1}{d}\s^{\ell\top}\z^t - \hat b_t \lambda_{t-1,\ell}\Big) = \sum_{\ell=0}^{t-1} (\sqrt{d}\r^\ell) \cdot \Big(\frac{1}{d}\s^{\ell\top} f_t(\v^t) - \av{f_t'(\v^t)} \lambda_{t-1,\ell}\Big)\\
&= \underbrace{\sum_{\ell=0}^{t-1} \sqrt{d}\r^\ell \cdot \frac{1}{d}\s^{\ell\top} \big(f_t(\v^t) - f_t(\hat\v^t)\big)}_{(I_t)} + \underbrace{\sum_{\ell=0}^{t-1} \sqrt{d}\r^\ell \cdot \big(\frac{1}{d}\s^{\ell\top}f_t(\hat\v^t) - \av{f_t'(\hat\v^t)} \lambda_{t-1,\ell}\big)}_{(I\!I_t)} + \underbrace{\sum_{\ell=0}^{t-1} \sqrt{d}\r^\ell \cdot \av{f_t'(\hat\v^t) - f_t'(\v^t)} \lambda_{t-1,\ell}}_{(I\!I\!I_t)}.
\end{align*}
For $(I_t)$, suppose first that $\w^{t-1}\neq0$. Using the fact that $\pnorm{\w^{t-1}}{}^2/d = \sum_{\ell=0}^{t-1} (\r^{\ell\top} \w^{t-1}/\sqrt{d})^2 = \sum_{\ell=0}^{t-1}\lambda_{t-1,\ell}^2 = \pnorm{\lambda_{t-1,[0:t-1]}}{}^2$, we have
\begin{align*}
(I_t) &= (P_{\w^{t-1}} + P_{\w^{t-1}}^\perp) (I_t)\\
&= \frac{1}{d\pnorm{\lambda_{t-1,[0:t-1]}}{}^2}\w^{t-1}\w^{t-1\top} \sum_{\ell=0}^{t-1} \frac{\r^\ell}{\sqrt{d}} \s^{\ell\top} \big(f_t(\v^t) - f_t(\hat\v^t)\big) + P_{\w^{t-1}}^\perp (I_t)\\
&= \frac{1}{d\pnorm{\lambda_{t-1,[0:t-1]}}{}^2} \w^{t-1} \sum_{\ell=0}^{t-1} \lambda_{t-1,\ell} \cdot \s^{\ell\top}\big(f_t(\v^t) - f_t(\hat\v^t)\big) + P_{\w^{t-1}}^\perp \Big(\sum_{\ell=0}^{t-1} \frac{\r^\ell}{\sqrt{d}} \s^{\ell\top} \big(f_t(\v^t) - f_t(\hat\v^t)\big)\Big).
\end{align*}
For $(I\!I_t)$, recall $A_t^\ell$ defined in (\ref{def:AB_tl}), then using $\sum_{\ell=0}^t \sqrt{d}\r^\ell \lambda_{t,\ell} = h_t(\u^t,\y) = \w^t$ for any $t\geq0$, we have
\begin{align*}
(I\!I_t) &= \sum_{\ell=0}^{t-1} \sqrt{d}\r^\ell \Big(\sum_{k=\ell+1}^{t-1}\alpha_{t-1}^{k-1} \lambda_{k-1,\ell} + A_t^\ell\Big)\\
&= \sum_{k=1}^{t-1} \alpha_{t-1}^{k-1}\Big(\sum_{\ell=0}^{k-1} \sqrt{d}\r^\ell \lambda_{k-1,\ell}\Big) + \sum_{\ell=0}^{t-1} \sqrt{d}\r^\ell A_t^\ell\\
&= \sum_{k=1}^{t-1} \alpha_{t-1}^{k-1} \w^{k-1} + \sum_{\ell=0}^{t-1} \sqrt{d}\r^\ell A_t^\ell\\
&= \sum_{k=0}^{t-2} \alpha_{t-1}^k \w^k + \sum_{\ell=0}^{t-1} \sqrt{d}\r^\ell A_t^\ell.
\end{align*}
Lastly, we have $(I\!I\!I_t) = \av{f_t'(\hat\v^t) - f_t'(\v^t)} \w^{t-1}$, so combining the expressions for $(I_t), (I\!I_t), (I\!I\!I_t)$ yields
\begin{align*}
&\sum_{\ell=0}^{t-1} (\sqrt{d}\r^\ell) \cdot \Big(\frac{1}{d}\s^{\ell\top} f_t(\v^t) - \av{f_t'(\v^t)} \lambda_{t-1,\ell}\Big)\\
&= \sum_{k=0}^{t-2} \alpha^k_{t-1}\w^k + \w^{t-1}\underbrace{\Big(\frac{1}{d\pnorm{\lambda_{t-1,[0:t-1]}}{}^2} \sum_{\ell=0}^{t-1} \lambda_{t-1,\ell} \cdot \s^{\ell\top}\big(f_t(\v^t) - f_t(\hat\v^t)\big) + \av{f_t'(\hat\v^t) - f_t'(\v^t)}\Big)}_{\alpha_{t-1}^{t-1}}\\
&\quad + P_{\w^{t-1}}^\perp \Big(\sum_{\ell=0}^{t-1} \frac{\r^\ell}{\sqrt{d}} \s^{\ell\top} \big(f_t(\v^t) - f_t(\hat\v^t)\big)\Big) + \sum_{\ell=0}^{t-1} \sqrt{d}\r^\ell A_t^\ell\\
&= \sum_{k=0}^{t-1} \alpha^k_{t-1}\w^k + P_{\w^{t-1}}^\perp \Big(\sum_{\ell=0}^{t-1} \frac{\r^\ell}{\sqrt{d}} \s^{\ell\top} \big(f_t(\v^t) - f_t(\hat\v^t)\big)\Big) + \sum_{\ell=0}^{t-1} \sqrt{d}\r^\ell A_t^\ell.
\end{align*}
In summary, we have as desired that
\begin{align*}
\Delta_{U,t}' = \u^t - \Big(\tilde{\u}^t + \sum_{k=0}^{t-1} \alpha^k_{t-1}\w^k\Big) = P_{\w^{t-1}}^\perp \Big(\sum_{\ell=0}^{t-1} \frac{\r^\ell}{\sqrt{d}} \s^{\ell\top} \big(f_t(\v^t) - f_t(\hat\v^t)\big)\Big) + \sum_{\ell=0}^{t-1} \sqrt{d}\r^\ell A_t^\ell + \Delta^{\low}_{U,t}.
\end{align*}
If $\w^{t-1}=0$, the same identity follows from $P_{\bm0}=0$, $P_{\bm0}^\perp=I$, and $\alpha_{t-1}^{t-1}=0$.

Next we establish the expansion for $\v^{t+1}$, for $t\geq0$. Again by (\ref{eq:construct_uv}) and the definition of $\tilde\v^{t+1}$ in (\ref{eq:tilde_u_intro}), we have
\begin{align*}
\v^{t+1} &= \tilde{\v}^{t+1} - \sum_{\ell=0}^t P_{\Q^{[-1:\ell]}}\s^\ell \lambda_{t,\ell} + \sum_{\ell=0}^t \q^\ell\frac{1}{\sqrt{d}}(P_{\bR^{[0:\ell-1]}}^\perp \g^\ell)^\top \w^t - \hat c_t \sum_{\ell=0}^t \q^\ell\q^{\ell\top}\z^t\\
&= \tilde\v^{t+1} + \sum_{\ell=0}^t \q^\ell \Big(\frac{1}{\sqrt{d}}\g^{\ell\top}\w^t - \hat c_t \q^{\ell\top}\z^t\Big) \underbrace{- \sum_{\ell=0}^t P_{\Q^{[-1:\ell]}}\s^\ell \lambda_{t,\ell} - \sum_{\ell=0}^t \q^\ell \frac{1}{\sqrt{d}} (P_{\bR^{[0:\ell-1]}} \g^\ell)^\top \w^t}_{\Delta^{\low}_{V,t+1}}\\
&= \tilde\v^{t+1} + \sum_{\ell=0}^t \sqrt{d}\q^\ell \Big(\frac{1}{d}\g^{\ell\top}\w^t - \hat c_t \gamma_{t,\ell}\Big) + \Delta^{\low}_{V,t+1}.
\end{align*}
To further analyze the middle term, using $\hat c_t = \delta \av{h_t'(\u^t,\y)}$, we have
\begin{align*}
&\sum_{\ell=0}^t \sqrt{d}\q^\ell \Big(\frac{1}{d}\g^{\ell\top}\w^t - \delta \av{h_t'(\u^t,\y)} \gamma_{t,\ell}\Big)\\
&= \underbrace{\sum_{\ell=0}^t \sqrt{d}\q^\ell \frac{1}{d}\g^{\ell\top}\Big(h_t(\u^t,\y) - h_t(\hat\u^t,\y)\Big)}_{(I_t)} + \underbrace{\sum_{\ell=0}^t \sqrt{d}\q^\ell \cdot \Big(\frac{1}{d}\g^{\ell\top}h_t(\hat\u^t,\y) - \delta \av{h_t'(\hat\u^t,\y)}{}\gamma_{t,\ell}\Big)}_{(I\!I_t)}\\
&\quad + \underbrace{\sum_{\ell=0}^t \sqrt{d}\q^\ell \cdot\delta\av{h_t'(\hat \u^t,\y) - h_t'(\u^t,\y)}\gamma_{t,\ell}}_{(I\!I\!I_t)}.
\end{align*}
For $(I_t)$, let $\bar\z^t := P_{\z^{-1}}^\perp \z^t = \z^t - \gamma_{t,-1}\btheta^\star$. When $\bar\z^t\neq0$, using the fact that $\pnorm{\bar\z^t}{}^2/d = \sum_{\ell=0}^t (\q^{\ell\top}\z^t/\sqrt{d})^2 = \sum_{\ell=0}^t \gamma_{t,\ell}^2 = \pnorm{\gamma_{t,[0:t]}}{}^2$, we have
\begin{align*}
(I_t) &= \Big(P_{\bar\z^t} + P_{\bar\z^t}^\perp\Big)(I_t)\\
&= \frac{1}{d\pnorm{\gamma_{t,[0:t]}}{}^2} \bar\z^t\bar\z^{t\top} \sum_{\ell=0}^t \sqrt{d}\q^\ell \frac{1}{d}\g^{\ell\top}\Big(h_t(\u^t,\y) - h_t(\hat\u^t,\y)\Big) + P_{\bar\z^t}^\perp (I_t)\\
&= \frac{1}{d\pnorm{\gamma_{t,[0:t]}}{}^2} \bar\z^t \sum_{\ell=0}^t \gamma_{t,\ell}\g^{\ell\top}(h_t(\u^t,\y) - h_t(\hat\u^t,\y)) + P_{\bar\z^t}^\perp \Big(\sum_{\ell=0}^t \frac{\q^\ell}{\sqrt{d}} \g^{\ell\top}\Big(h_t(\u^t,\y) - h_t(\hat\u^t,\y)\Big)\Big),
\end{align*}
where in the last line we apply $\frac{1}{\sqrt{d}}\q^{\ell\top}\bar\z^t = \frac{1}{\sqrt{d}}\q^{\ell\top}\z^t = \gamma_{t,\ell}$ for $0\leq\ell\leq t$. For $(I\!I_t)$, recall $B_t^\ell$ defined in (\ref{def:AB_tl}). Changing the order of summation and using $\bar\z^k=\sum_{\ell=0}^k\sqrt d\,\q^\ell\gamma_{k,\ell}$, we have
\begin{align*}
(I\!I_t) &= \sum_{\ell=0}^t \sqrt{d}\q^\ell \cdot \Big(\sum_{k=\ell}^{t-1} \beta_t^k \gamma_{k,\ell} + B_t^\ell\Big)\\
&= \sum_{k=0}^{t-1} \beta_t^k \sum_{\ell=0}^k \sqrt{d}\q^\ell \gamma_{k,\ell} + \sum_{\ell=0}^t \sqrt{d}\q^\ell B_t^\ell = \sum_{k=0}^{t-1} \beta_t^k\bar\z^k + \sum_{\ell=0}^t \sqrt{d}\q^\ell B_t^\ell,
\end{align*}
where $\bar\z^k = P_{\z^{-1}}^\perp\z^k = \z^k - \gamma_{k,-1}\btheta^\star$. Lastly, we have
\begin{align*}
(I\!I\!I_t) = \delta\av{h_t'(\hat \u^t,\y) - h_t'(\u^t,\y)}\cdot \sum_{\ell=0}^t \sqrt{d}\q^\ell \gamma_{t,\ell} = \delta\av{h_t'(\hat \u^t,\y) - h_t'(\u^t,\y)} \bar\z^t.
\end{align*}
Combining the expressions for $(I_t), (I\!I_t), (I\!I\!I_t)$ yields that
\begin{align*}
&\sum_{\ell=0}^t \sqrt{d}\q^\ell \Big(\frac{1}{d}\g^{\ell\top}\w^t - \delta \av{h_t'(\u^t,\y)} \gamma_{t,\ell}\Big)\\
&= \sum_{k=0}^{t-1}\beta_t^k \bar\z^k + \bar\z^t\underbrace{\Big(\frac{1}{d\pnorm{\gamma_{t,[0:t]}}{}^2} \sum_{\ell=0}^t \gamma_{t,\ell}\g^{\ell\top}(h_t(\u^t,\y) - h_t(\hat\u^t,\y)) + \delta\av{h_t'(\hat \u^t,\y) - h_t'(\u^t,\y)} \Big)}_{\beta_t^t}\\
&\quad + P_{\bar\z^t}^\perp \Big(\sum_{\ell=0}^t \frac{\q^\ell}{\sqrt{d}}\g^{\ell\top}\Big(h_t(\u^t,\y) - h_t(\hat\u^t,\y)\Big)\Big) + \sum_{\ell=0}^t \sqrt{d}\q^\ell B_t^\ell\\
&= \sum_{k=0}^t\beta_t^k \bar\z^k + P_{\bar\z^t}^\perp \Big(\sum_{\ell=0}^t \frac{\q^\ell}{\sqrt{d}} \g^{\ell\top}\Big(h_t(\u^t,\y) - h_t(\hat\u^t,\y)\Big)\Big) + \sum_{\ell=0}^t \sqrt{d}\q^\ell B_t^\ell.
\end{align*}
In summary, we have as desired that
\begin{align*}
\Delta_{V,t+1}' = \v^{t+1} - \Big(\tilde\v^{t+1} + \sum_{k=0}^t\beta_t^k \bar\z^k\Big) = P_{\bar\z^t}^\perp \Big(\sum_{\ell=0}^t \frac{\q^\ell}{\sqrt{d}} \g^{\ell\top}\Big(h_t(\u^t,\y) - h_t(\hat\u^t,\y)\Big)\Big) + \sum_{\ell=0}^t \sqrt{d}\q^\ell B_t^\ell + \Delta^{\low}_{V,t+1}.
\end{align*}
If $\bar\z^t=0$, the same identity follows from $P_{\bm0}=0$, $P_{\bm0}^\perp=I$, and $\beta_t^t=0$.
\end{proof}

We prove Proposition~\ref{prop:general_error_induction} using the decomposition in Lemma \ref{lem:error_decompose}. We start by giving some further auxiliary bounds in Appendix \ref{subsec:aux_bound_proof} and complete the proof in Appendix \ref{subsec:ind_err_proof}.

\subsection{Auxiliary Bounds}
\label{subsec:aux_bound_proof}

The next lemma compares the first- and second-order errors.
\begin{lemma}
\label{lem:mutual_close}
For every $t\geq0$,
\begin{align}
\pnorm{\Delta_{U,t}}{}
&\leq\pnorm{\Delta_{U,t}'}{}
+\bar\lambda_{t-1}\sum_{k=0}^{t-1}|\alpha_{t-1}^k|,
\label{eq:diff_u_tilde}\\
\pnorm{\u^t-\hat\u^t}{}
&\leq\pnorm{\Delta_{U,t}'}{}
+\sum_{k=0}^{t-1}|\alpha_{t-1}^k|
\pnorm{h_k'}{\infty}\pnorm{\Delta_{U,k}}{},
\label{eq:diff_u_hat}\\
\pnorm{\Delta_{V,t+1}}{}
&\leq\pnorm{\Delta_{V,t+1}'}{}
+\bar\gamma_t\sum_{k=0}^t|\beta_t^k|,
\label{eq:diff_v_tilde}\\
\pnorm{\v^{t+1}-\hat\v^{t+1}}{}
&\leq\pnorm{\Delta_{V,t+1}'}{}
+\sum_{k=0}^t|\beta_t^k|
\pnorm{f_k'}{\infty}\pnorm{\Delta_{V,k}}{}.
\label{eq:diff_v_hat}
\end{align}
At $t=0$, the two $U$-side bounds are understood as $\Delta_{U,0}=\Delta_{U,0}'=\u^0-\hat\u^0=0$, with their correction terms set to zero.
\end{lemma}
\begin{proof}
Recall the definitions of $\bar\lambda_t$ and $\bar\gamma_t$ in (\ref{def:induction_constant}). At $t=0$, the two $U$-side bounds follow from $\u^0=\tilde\u^0=\hat\u^0$. For $t\geq1$, by Lemma \ref{lem:error_decompose}, we have
\begin{align*}
\Delta_{U,t} = \u^t - \tilde\u^t = \sum_{k=0}^{t-1} \alpha_{t-1}^k \w^k + \Delta_{U,t}',
\end{align*}
hence by the fact that $\pnorm{\w^k}{} = \sqrt{d}\pnorm{\lambda_{k,[0:k]}}{} \leq \bar\lambda_{t-1}$ for $k\leq t-1$, we have
\begin{align*}
\pnorm{\Delta_{U,t}}{} &= \Bigpnorm{\sum_{k=0}^{t-1}\alpha_{t-1}^k \w^k + \Delta_{U,t}'}{} \leq \sum_{k=0}^{t-1}|\alpha_{t-1}^k|\pnorm{\w^k}{} + \pnorm{\Delta_{U,t}'}{} \leq \bar\lambda_{t-1}\sum_{k=0}^{t-1}|\alpha_{t-1}^k| + \pnorm{\Delta_{U,t}'}{}.
\end{align*}
Next, to bound $\pnorm{\u^t - \hat\u^t}{}$, recall from (\ref{eq:second_order_error}) and (\ref{eq:def_u_v_hat}) that
\begin{align*}
\u^t = \tilde\u^t + \sum_{k=0}^{t-1} \alpha_{t-1}^k h_k(\u^k,\y) + \Delta_{U,t}', \quad \hat\u^t = \tilde\u^t + \sum_{k=0}^{t-1} \alpha_{t-1}^k h_k(\tilde\u^k,\y),
\end{align*}
hence we have
\begin{align*}
\pnorm{\u^t - \hat\u^t}{} &\leq \sum_{k=0}^{t-1} |\alpha_{t-1}^k|\pnorm{h_k(\u^k,\y) - h_k(\tilde\u^k,\y)}{} + \pnorm{\Delta_{U,t}'}{}\\
&\leq \sum_{k=0}^{t-1} |\alpha_{t-1}^k|\pnorm{h_k'}{\infty}\pnorm{\Delta_{U,k}}{} + \pnorm{\Delta_{U,t}'}{}.
\end{align*}

Let us now derive the analogous bounds for $\pnorm{\v^{t+1} - \tilde\v^{t+1}}{}$ and $\pnorm{\v^{t+1} - \hat\v^{t+1}}{}$, for $t\geq0$. By Lemma \ref{lem:error_decompose}, we have
\begin{align*}
\Delta_{V,t+1} = \v^{t+1} - \tilde\v^{t+1} = \sum_{k=0}^t \beta_t^k \bar\z^k + \Delta_{V,t+1}'.
\end{align*}
Since $\bar\z^k = \z^k - \gamma_{k,-1}\btheta^\star = \sum_{\ell=0}^k \gamma_{k,\ell} (\sqrt{d}\q^\ell)$, we have $\pnorm{\bar\z^k}{} = \sqrt{d}\pnorm{\gamma_{k,[0:k]}}{} \leq \bar\gamma_t$ for $k\leq t$, and hence
\begin{align*}
\pnorm{\Delta_{V,t+1}}{} &= \Bigpnorm{\sum_{k=0}^t \beta_t^k \bar\z^k + \Delta_{V,t+1}'}{} \leq \sum_{k=0}^t |\beta_t^k|\pnorm{\bar\z^k}{} + \pnorm{\Delta_{V,t+1}'}{} \leq \bar\gamma_t \sum_{k=0}^t |\beta_t^k| + \pnorm{\Delta_{V,t+1}'}{}.
\end{align*}
Lastly, recall from (\ref{eq:second_order_error}) and (\ref{eq:def_u_v_hat}) that
\begin{align*}
\v^{t+1} &= \tilde\v^{t+1} + \sum_{\ell=0}^t \beta_t^\ell \Big(f_\ell(\v^\ell) - \gamma_{\ell,-1}\btheta^\star\Big) + \Delta_{V,t+1}',\\
\hat\v^{t+1} &= \tilde\v^{t+1} + \sum_{\ell=0}^t \beta_t^\ell \Big(f_\ell(\tilde\v^\ell) - \gamma_{\ell,-1}\btheta^\star\Big).
\end{align*}
Hence we may bound the difference of $\v^{t+1}$ and $\hat\v^{t+1}$ by
\begin{align*}
\pnorm{\v^{t+1} - \hat\v^{t+1}}{} &\leq \pnorm{\Delta_{V,t+1}'}{} + \biggpnorm{\sum_{\ell=0}^t \beta_t^\ell \big(f_\ell(\v^\ell) - f_\ell(\tilde\v^\ell)\big)}{}\\
&\leq \pnorm{\Delta_{V,t+1}'}{} + \sum_{\ell=0}^t|\beta_t^\ell|\pnorm{f_\ell'}{\infty}\pnorm{\Delta_{V,\ell}}{}.
\end{align*}
\end{proof}

We next bound the projection errors in \eqref{eq:amp_low_d}.
\begin{lemma}
\label{lem:bound_lowd_projection}
For every fixed $t\geq0$, the following bounds hold with probability at least $1-2n^{-20}$
\begin{align*}
\pnorm{\Delta_{U,t}^{\low}}{}
&\leq C\big(\sqrt{t+1}+\sqrt{\log n}\big)
\pnorm{\gamma_{t,[-1:t]}}{},\\
\pnorm{\Delta_{V,t+1}^{\low}}{}
&\leq C\big(\sqrt{t+1}+\sqrt{\log n}\big)
\pnorm{\lambda_{t,[0:t]}}{}.
\end{align*}
\end{lemma}
\begin{proof}
For $t=0$, $\Delta_{U,0}^{\low}=0$. Letting $g_s^\ell := \r^{s\top}\g^\ell$ for $\ell\geq s+1$ and $s_j^\ell := \q^{j\top}\s^\ell$ for $-1\leq j\leq\ell$, we have, for $t\geq1$,
\begin{align*}
-\Delta^{\low}_{U,t} &= \sum_{\ell=-1}^t P_{\bR^{[0:\ell-1]}}\g^\ell\gamma_{t,\ell} + \sum_{\ell=0}^{t-1} \frac{\r^\ell (P_{\Q^{[-1:\ell]}}\s^\ell)^\top}{\sqrt{d}}\z^t\\
&= \sum_{\ell=1}^t \sum_{s=0}^{\ell-1} \r^s g_s^\ell \gamma_{t,\ell} + \sum_{\ell=0}^{t-1} \sum_{j=-1}^\ell \r^\ell s_j^\ell \frac{1}{\sqrt{d}} \q^{j\top}\z^t\\
&= \sum_{\ell=0}^{t-1} \r^\ell\sum_{j=\ell+1}^t g_\ell^j \gamma_{t,j} + \sum_{\ell=0}^{t-1} \sum_{j=-1}^\ell \r^\ell s_j^\ell \gamma_{t,j}\\
&= \sum_{\ell=0}^{t-1} \r^\ell \Big(\sum_{j=\ell+1}^t g_\ell^j \gamma_{t,j} + \sum_{j=-1}^\ell s_j^\ell \gamma_{t,j}\Big).
\end{align*}
Since $\{\r^\ell\}_{\ell=0}^{t-1}$ are orthonormal, we have
\begin{align}\label{eq:Delta-low-U}
\pnorm{\Delta^{\low}_{U,t}}{} &= \Bigg[\sum_{\ell=0}^{t-1} \Big(\sum_{j=\ell+1}^t g_\ell^j \gamma_{t,j} + \sum_{j=-1}^\ell s_j^\ell \gamma_{t,j}\Big)^2\Bigg]^{1/2}.
\end{align}
Define a random matrix $\H\in\R^{t\times(t+2)}$ by
\begin{align*}
H_{\ell j} =
\begin{cases}
g_\ell^j, & j\geq\ell+1,\\
s_j^\ell, & j\leq\ell,
\end{cases}
\end{align*}
where $0\leq\ell\leq t-1$ and $-1\leq j\leq t$. All the $g_s^\ell$ and $s_j^\ell$ used here are mutually independent $\N(0,1)$ variables. In particular, $\H$ has i.i.d. $\N(0,1)$ entries. We then rewrite \eqref{eq:Delta-low-U} to see
\begin{align*}
\pnorm{\Delta^{\low}_{U,t}}{}
&= \pnorm{\H\gamma_{t,[-1:t]}}{}
\leq \pnorm{\H}{\mathrm{op}}\cdot\pnorm{\gamma_{t,[-1:t]}}{}\\
&\leq C\pnorm{\gamma_{t,[-1:t]}}{}\big(\sqrt{t+1}+\sqrt{\log n}\big),
\end{align*}
with probability at least $1-n^{-20}$ for some universal constant $C>0$, where the last inequality is the standard Gaussian operator-norm bound. 

For the bound on $\pnorm{\Delta_{V,t+1}^{\low}}{}$, with $t\geq0$, we have
\begin{align*}
-\Delta_{V,t+1}^{\low} = \sum_{\ell=0}^t P_{\Q^{[-1:\ell]}}\s^\ell \lambda_{t,\ell} + \sum_{\ell=0}^t \frac{\q^\ell}{\sqrt{d}} (P_{\bR^{[0:\ell-1]}}\g^\ell)^\top\w^t.
\end{align*}
The first term can be equivalently written as
\begin{align*}
\sum_{\ell=0}^t P_{\Q^{[-1:\ell]}}\s^\ell \lambda_{t,\ell}
&= \sum_{\ell=0}^t \sum_{j=-1}^\ell \q^j s_j^\ell\lambda_{t,\ell}
= \sum_{j=-1}^t \q^j \sum_{\ell=j\vee0}^t s_j^\ell\lambda_{t,\ell}\\
&= \q^{-1}\cdot\sum_{\ell=0}^t s_{-1}^\ell\lambda_{t,\ell}
+ \q^0\cdot\sum_{\ell=0}^t s_0^\ell\lambda_{t,\ell}
+ \sum_{j=1}^t \q^j\sum_{\ell=j}^t s_j^\ell\lambda_{t,\ell}.
\end{align*}
For the second term, note that the summand is $0$ for $\ell=0$, hence
\begin{align*}
\sum_{\ell=0}^t \frac{\q^\ell}{\sqrt{d}} (P_{\bR^{[0:\ell-1]}}\g^\ell)^\top\w^t
= \sum_{\ell=1}^t \sum_{s=0}^{\ell-1} \q^\ell\lambda_{t,s}g_s^\ell
= \sum_{j=1}^t \q^j\sum_{\ell=0}^{j-1}\lambda_{t,\ell}g_\ell^j.
\end{align*}
Combining the two terms yields
\begin{align*}
-\Delta_{V,t+1}^{\low}
= \q^{-1}\cdot\sum_{\ell=0}^t s_{-1}^\ell\lambda_{t,\ell}
+ \q^0\cdot\sum_{\ell=0}^t s_0^\ell\lambda_{t,\ell}
+ \sum_{j=1}^t \q^j\Big(\sum_{\ell=0}^{j-1}\lambda_{t,\ell}g_\ell^j+\sum_{\ell=j}^t\lambda_{t,\ell}s_j^\ell\Big),
\end{align*}
which further implies that
\begin{align*}
\pnorm{\Delta_{V,t+1}^{\low}}{}
\leq \big(\pnorm{s_{-1}^{[0:t]}}{}+\pnorm{s_0^{[0:t]}}{}\big)\pnorm{\lambda_{t,[0:t]}}{}
+ \Bigpnorm{\sum_{j=1}^t \q^j\Big(\sum_{\ell=0}^{j-1}\lambda_{t,\ell}g_\ell^j+\sum_{\ell=j}^t\lambda_{t,\ell}s_j^\ell\Big)}{}.
\end{align*}
Each of the two Gaussian row norms satisfies
\begin{align*}
\pnorm{s_j^{[0:t]}}{}\leq C\big(\sqrt{t+1}+\sqrt{\log n}\big),
\qquad j\in\{-1,0\},
\end{align*}
with probability $1-n^{-20}/3$. The last term can be bounded in a similar way as above by $C\big(\sqrt{t+1}+\sqrt{\log n}\big)\pnorm{\lambda_{t,[0:t]}}{}$, with the same probability. At $t=0$, this last term is zero. A union bound gives both assertions with probability at least $1-2n^{-20}$, concluding the proof.
\end{proof}

We next bound the coefficient errors $A_t^\ell$ and $B_t^\ell$. We will make use of the following version of Bernstein's inequality proved in \cite[Lemma 3]{li2024non}.

\begin{lemma}[Lemma 3, \cite{li2024non}]\label{lem:modified_bernstein}
Suppose $Z_1,\ldots,Z_n$ are independent random variables satisfying
\begin{align*}
\E[Z_i] = 0, \quad \P(|Z_i| > B\log(1/\delta)) \leq \delta 
\end{align*}
for some $B > 0$ and every $0 < \delta < 1/\poly(n)$. Then with probability $1-\delta$,
\begin{align*}
\Big|\sum_{i=1}^n Z_i\Big| \leq C\sqrt{\sum_{i=1}^n \Big[\Var(Z_i) + \Big(\frac{B\log n}{n}\Big)^2\Big]\log\frac{1}{\delta}} + B\log n\log \frac{1}{\delta}.
\end{align*}
\end{lemma}

\begin{lemma}
\label{lem:error_main_lemma}
Fix an integer $t\geq0$ such that $t\leq c\sqrt d/\log^3n$. With probability $1-O(n^{-12})$, the following implications hold. When $t\geq1$, on $\cA_{t-1}^V$,
\begin{align*}
\Bigpnorm{\sum_{\ell=0}^{t-1}\sqrt{d}\,\r^\ell A_t^\ell}{}
&\leq C\big(\bar\gamma_t+\bar\lambda_{t-1}\big)
\big(\pnorm{f_{[0:t]}'}{\infty}\vee1\big)^2
\big(\pnorm{\rho_{[0:t]}}{\infty}\vee1\big)
t\sqrt{\frac{\log n}{d}}.
\end{align*}
In all cases, on $\cA_t^U$,
\begin{align*}
\Bigpnorm{\sum_{\ell=0}^t\sqrt d\,\q^\ell B_t^\ell}{}
&\leq C(1+\sqrt\delta)\big(\bar\lambda_t+\bar\gamma_t\big)
\big(\pnorm{h_{[0:t]}'}{\infty}\vee1\big)^2
\big(\pnorm{\gamma_{[0:t],-1}}{\infty}\vee1\big)
(t+1)\sqrt{\frac{\log n}{d}}.
\end{align*}
\end{lemma}
\begin{proof}
We first prove the first assertion for $t\geq1$. Condition on $(\v^0,\cR)$, so that $\xi_0$ and $f_0(\v^0)$ are fixed. We simplify notation by setting
\begin{align*}
K_t\coloneqq(\bar\gamma_t+\bar\lambda_{t-1})\big(\pnorm{f_{[0:t]}'}{\infty}\vee1\big)^2\big(\pnorm{\rho_{[0:t]}}{\infty}\vee1\big).
\end{align*}
Because $\{\r^\ell\}_{\ell=0}^{t-1}$ is orthonormal, we can write
\begin{align}\label{eq:error-main-duality}
\Bigpnorm{\sum_{\ell=0}^{t-1}\sqrt d\,\r^\ell A_t^\ell}{}
=\sqrt d\,\pnorm{A_t^{[0:t-1]}}{}
=\sqrt d\sup_{\omega\in\sS^{t-1}}\sum_{\ell=0}^{t-1}\omega_\ell A_t^\ell.
\end{align}
To bound the right-hand side, we treat $\omega$ and the random coefficients entering $A_t^\ell$ as parameters. We first prove the bound for fixed values of these parameters, collected in \eqref{def:Theta_uniform} below, and then extend it uniformly over all values satisfying the bounds \eqref{eqs:AtV} of $\cA_{t-1}^V$.

To condense notation, set
\begin{align*}
\m^0&\coloneqq\v^0,\qquad \m^k\coloneqq\rho_k\btheta^\star\quad(1\leq k\leq t),\qquad
\b_t\coloneqq\rho_t\btheta^\star+\sum_{k=0}^{t-1}\beta_{t-1}^k\big(f_k(\m^k)-\gamma_{k,-1}\btheta^\star\big).
\end{align*}
Define the parameter set
\begin{align}\label{def:Theta_uniform}
\Theta_t\coloneqq\Big(\omega,\{\lambda_{k-1,[-1:k-1]}\}_{k=1}^t,\beta_{t-1}^{[0:t-1]},\gamma_{[0:t-1],-1},\xi_{[1:t]}\Big).
\end{align}
With the parameters in \eqref{def:Theta_uniform} fixed, $\b_t$ is fixed and independent of $\{\s^\ell\}_{\ell=0}^{t-1}$. By the definition of $A_t^\ell$,
\begin{align}\label{eq:error-main-coordinate-decomposition}
\sum_{\ell=0}^{t-1}\omega_\ell A_t^\ell
=\sum_{j=1}^dM_j,\qquad M_j=W_j+X_j+Y_j+Z_j,
\end{align}
where, writing $(\bS\omega)_j\coloneqq\sum_{\ell=0}^{t-1}\omega_\ell s_j^\ell$,
\begin{gather*}
W_j\coloneqq\frac1d(\bS\omega)_j[f_t(\b_t)]_j,\qquad
X_j\coloneqq\frac1d(\bS\omega)_j\big\{f_t(\hat v_j^t)-[f_t(\b_t)]_j\big\},\\
Y_j\coloneqq-\frac1d f_t'(\hat v_j^t)\sum_{\ell=0}^{t-1}\lambda_{t-1,\ell}\omega_\ell,\qquad
Z_j\coloneqq-\frac1d\sum_{k=1}^{t-1}\beta_{t-1}^k f_t'(\hat v_j^t)f_k'(\tilde v_j^k)\sum_{\ell=0}^{k-1}\lambda_{k-1,\ell}\omega_\ell.
\end{gather*}
This choice of $W_j$ makes its second factor independent of $\{s_j^\ell\}_{\ell=0}^{t-1}$.

Let $\E_{\s}$ denote expectation with respect to $\{\s^\ell\}_{\ell=0}^{t-1}$, while treating
\eqref{def:Theta_uniform} as fixed. Then $\E_{\s}W_j=0$. Moreover,
\begin{align*}
\partial_{s_j^\ell}\hat v_j^t
=\lambda_{t-1,\ell}+\sum_{k=\ell+1}^{t-1}\beta_{t-1}^k f_k'(\tilde v_j^k)\lambda_{k-1,\ell}.
\end{align*}
Gaussian integration by parts consequently yields
\begin{align*}
\E_{\s}X_j
&=\frac1d\sum_{\ell=0}^{t-1}\omega_\ell\E_{\s}\big[s_j^\ell f_t(\hat v_j^t)\big]\\
&=\frac1d\E_{\s}\big[f_t'(\hat v_j^t)\big]\sum_{\ell=0}^{t-1}\lambda_{t-1,\ell}\omega_\ell+\frac1d\sum_{k=1}^{t-1}\beta_{t-1}^k\E_{\s}\big[f_t'(\hat v_j^t)f_k'(\tilde v_j^k)\big]\sum_{\ell=0}^{k-1}\lambda_{k-1,\ell}\omega_\ell\\
&=-\E_{\s}(Y_j+Z_j).
\end{align*}
Thus, $\E_{\s}M_j=0$ for every $j$.
Let $\mathcal I_t$ be the set of pairs $(\Theta_t,K_t)$, where $\Theta_t$ is defined in \eqref{def:Theta_uniform}, which satisfy the conditions of $\cA_{t-1}^V$.

Equip $\mathcal I_t$ with Euclidean metric. The pair $(\Theta_t,K_t)$ contains $O((t+1)^2)$ components, each  polynomially bounded in $n$. Standard Euclidean nets therefore give an $\varepsilon$-net $\mathcal N_\varepsilon$ with centers in $\mathcal I_t$, where $\varepsilon=n^{-C_\varepsilon}$ for a sufficiently large universal constant $C_\varepsilon$, such that
\begin{align}\label{eq:error-main-net-size}
|\mathcal N_\varepsilon|\leq\exp\big\{C(t+1)^2\log n\big\}.
\end{align}
At each net point, the bound $|f_k(x)-f_k(z)|\leq\pnorm{f_k'}{\infty}|x-z|$, and \eqref{eq:cond_beta_sum} give
\begin{align}
\pnorm{\b_t}{}&\leq C(\bar\gamma_t+\bar\lambda_{t-1})\big(\pnorm{f_{[0:t]}'}{\infty}\vee1\big)\big(\pnorm{\rho_{[0:t]}}{\infty}\vee1\big),\qquad
\pnorm{f_t(\b_t)}{}\leq CK_t.
\label{eq:error-main-baseline-bounds}
\end{align}
This shows $\sum_jW_j$ is a centered Gaussian with
\begin{align}\label{eq:error-main-W-variance}
\Var_{\s}\Big(\sum_{j=1}^dW_j\Big)\leq\frac{CK_t^2}{d^2}.
\end{align}
The Gaussian maximal inequality and \eqref{eq:error-main-net-size} imply
\begin{align}\label{eq:error-main-W-bound}
\max_{(\Theta_t,K_t)\in\mathcal N_\varepsilon}
\frac{1}{K_t}\Big|\sum_{j=1}^dW_j\Big|
\leq\frac{Ct\sqrt{\log n}}{d}
\end{align}
with probability $1-O(n^{-12})$.

For $X_j$, let $F_j=(\bS\omega)_j f_t(\hat v_j^t)$, so that $d^{-1}F_j=W_j+X_j$. Gaussian Poincar\'e gives
\begin{gather*}
\Var_{\s}\Big(\sum_{j=1}^d(W_j+X_j)\Big)
\leq\frac1{d^2}\sum_{j=1}^d\sum_{\ell=0}^{t-1}\E_{\s}\big[(\partial_{s_j^\ell}F_j)^2\big],\\
\partial_{s_j^\ell}F_j
=\omega_\ell f_t(\hat v_j^t)+(\bS\omega)_j f_t'(\hat v_j^t)\Big(\lambda_{t-1,\ell}+\sum_{k=\ell+1}^{t-1}\beta_{t-1}^k f_k'(\tilde v_j^k)\lambda_{k-1,\ell}\Big).
\end{gather*}
By Cauchy--Schwarz and the inequalities defining $\cA_{t-1}^V$ in \eqref{eqs:AtV},
\begin{align}\label{eq:error-main-coefficient-bound}
&\sum_{\ell=0}^{t-1}\Big(\sum_{k=\ell+1}^{t-1}\beta_{t-1}^k f_k'(\tilde v_j^k)\lambda_{k-1,\ell}\Big)^2\leq \big(\pnorm{f_{[0:t]}'}{\infty}\vee1\big)^2\Big(\sum_{k=1}^{t-1}|\beta_{t-1}^k|\Big)^2\max_{0\leq k\leq t-1}\pnorm{\lambda_{k,[0:k]}}{}^2
\leq\frac{C\bar\lambda_{t-1}^2}{d}.
\end{align}
To bound the remaining term, Minkowski's inequality and \eqref{eqs:AtV} give
\begin{align*}
\Big(\E_{\s}\pnorm{\hat\v^t-\b_t}{}^2\Big)^{1/2}
\leq\bar\lambda_{t-1}+\big(\pnorm{f_{[0:t]}'}{\infty}\vee1\big)\bar\lambda_{t-1}\sum_{k=0}^{t-1}|\beta_{t-1}^k|
\leq C\bar\lambda_{t-1}.
\end{align*}
Together with \eqref{eq:error-main-baseline-bounds}, this implies
\begin{align}\label{eq:error-main-f-moment}
\sum_{j=1}^d\E_{\s}\big[f_t(\hat v_j^t)^2\big]\leq CK_t^2.
\end{align}
Substitution into the Poincar\'e bound, followed by $\Var(U-V)\leq2\Var(U)+2\Var(V)$ and \eqref{eq:error-main-W-variance}, yields
\begin{align}\label{eq:var_X_j}
\Var_{\s}\Big(\sum_{j=1}^dX_j\Big)\leq\frac{CK_t^2}{d^2}.
\end{align}

For the tail envelope required by Lemma \ref{lem:modified_bernstein}, observe that
\begin{align*}
|X_j|\leq\frac{\big(\pnorm{f_{[0:t]}'}{\infty}\vee1\big)}{d}|(\bS\omega)_j|\Big(|\tilde v_j^t-m_j^t|+\sum_{k=0}^{t-1}|\beta_{t-1}^k|\big(\pnorm{f_{[0:t]}'}{\infty}\vee1\big)|\tilde v_j^k-m_j^k|\Big).
\end{align*}
 Hence, for every $0<\eta<1/\poly(n)$,
\begin{align}\label{eq:error-main-X-envelope}
\P_{\s}\Big(|X_j|>\frac{C\big(\pnorm{f_{[0:t]}'}{\infty}\vee1\big)^2\bar\lambda_{t-1}}{d^{3/2}}\log\Big(\frac1\eta\Big)\Big)\leq\eta.
\end{align}
Applying Lemma \ref{lem:modified_bernstein} with $B=C\big(\pnorm{f_{[0:t]}'}{\infty}\vee1\big)^2\bar\lambda_{t-1}/d^{3/2}$ and using \eqref{eq:var_X_j}, we find that, for each fixed parameter value,
\begin{align}\label{eq:error-main-X-fixed}
\Big|\sum_{j=1}^d(X_j-\E_{\s}X_j)\Big|
\leq\frac{CK_t}{d}\sqrt{\log\Big(\frac1\eta\Big)}
+\frac{C\big(\pnorm{f_{[0:t]}'}{\infty}\vee1\big)^2\bar\lambda_{t-1}\log n}{d^{3/2}}\log\Big(\frac1\eta\Big)
\end{align}
with probability at least $1-\eta$.

The remaining two terms are simpler. Uniformly over $(\Theta_t,K_t)\in\mathcal I_t$,
\begin{align}\label{eq:error-main-YZ-envelope}
|Y_j|\leq\frac{C\big(\pnorm{f_{[0:t]}'}{\infty}\vee1\big)\bar\lambda_{t-1}}{d^{3/2}},
\qquad |Z_j|\leq\frac{C\big(\pnorm{f_{[0:t]}'}{\infty}\vee1\big)^2\bar\lambda_{t-1}}{d^{3/2}}.
\end{align}
Their coordinatewise independence, Lemma \ref{lem:modified_bernstein}, and a union bound over $\mathcal N_\varepsilon$ therefore imply
\begin{align}\label{eq:error-main-XYZ-net}
\max_{(\Theta_t,K_t)\in\mathcal N_\varepsilon}
\frac{1}{K_t}\Big|\sum_{j=1}^d\big[(X_j+Y_j+Z_j)-\E_{\s}(X_j+Y_j+Z_j)\big]\Big|
\leq\frac{Ct\sqrt{\log n}}{d}
\end{align}
with probability $1-O(n^{-12})$ where in \eqref{eq:error-main-X-fixed} we take $\eta=\exp\{-C_0t^2\log n\}$, with $C_0$ sufficiently large after fixing the net. Combining these computations gives
\begin{align*}
\sqrt d\Bigg|\sum_{\ell=0}^{t-1}\omega_\ell A_t^\ell\Bigg|
\leq CK_t\,t\sqrt{\frac{\log n}{d}},
\qquad \text{for every }(\Theta_t,K_t)\in\mathcal N_\varepsilon.
\end{align*}

Lastly, we extend to the entire parameter space. By standard Gaussian operator-norm concentration and the horizon $t\leq \frac{c\sqrt{d}}{\log^3 n}$, the event
\begin{align}\label{eq:error-main-f-gradient-event}
\mathcal G_t^f\coloneqq\big\{\pnorm{\bS}{\op}\leq C(\sqrt d+\sqrt t)\big\}.
\end{align}
holds with probability $1-n^{-20}$. For $1\leq k\leq t$,
\begin{align*}
\tilde v_j^k=\lambda_{k-1,-1}\theta_j^\star+\sum_{\ell=0}^{k-1}\lambda_{k-1,\ell}s_j^\ell, \qquad m_j^k=\lambda_{k-1,-1}\theta_j^\star
\end{align*}
give, on $\mathcal G_t^f$, that
\begin{align}\label{eq:error-main-f-input-gradients}
|\tilde v_j^k| + |m_j^k| \leq n^{C_1},\qquad |\tilde v_j^k-\tilde v_j^{k,\prime}|+|m_j^k-m_j^{k,\prime}|
\leq C\sqrt d\,\pnorm{\Theta-\Theta'}{}.
\end{align}
 For $1\leq k\leq t$ and $q\in\{0,1\}$, \eqref{eq:parameter_lipschitz_assumption} and \eqref{eq:cond_derivative_linfty_f} imply
\begin{align}\label{eq:error-main-f-chain-rule}
\big|\partial_v^q\mathsf f_k(x;\xi_k)-\partial_v^q\mathsf f_k(x';\xi_k')\big|
&\leq C\log^5n\,|x-x'|+n^{C_{\rm par}}(1+|x'|^{C_{\rm par}})\pnorm{\xi_k-\xi_k'}{}.
\end{align}
Applying these estimates to the formulas for $\hat v_j^t$ and $b_{t,j}$ shows that 
\begin{align}
|\hat v_j^t|+|b_{t,j}|
&\leq n^{C_1},\\
|\hat v_j^t-\hat v_j^{t,\prime}|+|b_{t,j}-b_{t,j}'|
&\leq n^{C_1}\pnorm{\Theta-\Theta'}{}.
\end{align}
We next use these bounds to bound the difference for the $W,X,Y,Z$ terms. For example, the definition of $W_j$ gives
\begin{align*}
d\big|W_j(\Theta)-W_j(\Theta')\big|
&\leq \big|(\bS(\omega-\omega'))_j\big|
\big|\mathsf f_t(b_{t,j};\xi_t)\big|+\big|(\bS\omega')_j\big|
\big|\mathsf f_t(b_{t,j};\xi_t)-\mathsf f_t(b_{t,j}';\xi_t')\big|.
\end{align*}
On $\mathcal G_t^f$,
\begin{align*}
\big|(\bS(\omega-\omega'))_j\big|
&\leq C\sqrt d\,\pnorm{\Theta-\Theta'}{},\qquad
\big|(\bS\omega')_j\big|\leq C\sqrt d.
\end{align*}
The bounds on $b_{t,j}$ and \eqref{eq:error-main-f-chain-rule} therefore give
\begin{align*}
\big|W_j(\Theta)-W_j(\Theta')\big|
\leq \frac{n^{C_1}}{d}\pnorm{\Theta-\Theta'}{}.
\end{align*}
Applying the same strategy to $X_j,Y_j,Z_j$, using the bounds on $\hat v_j^t,\tilde v_j^k$ gives
\begin{align*}
\big|X_j(\Theta)-X_j(\Theta')\big|
+\big|Y_j(\Theta)-Y_j(\Theta')\big|
+\big|Z_j(\Theta)-Z_j(\Theta')\big|
\leq \frac{n^{C_1}}{d}\pnorm{\Theta-\Theta'}{}.
\end{align*}
Summing over $j$ gives
\begin{align}\label{eq:error-main-f-parameter-gradient}
\Big|\sum_{j=1}^dM_j(\Theta)-\sum_{j=1}^dM_j(\Theta')\Big|
\leq n^{C_1}\pnorm{\Theta-\Theta'}{}.
\end{align}
For $(\Theta_t,K_t)\in\mathcal I_t$, choose a net point $(\Theta_t',K_t')$ within distance $\varepsilon$. Taking $C_\varepsilon\geq C_1+20$ gives
\begin{align}\label{eq:error-main-net-extension}
\Big|\sum_{j=1}^dM_j(\Theta_t)-\sum_{j=1}^dM_j(\Theta_t')\Big|
\leq n^{C_1}\pnorm{\Theta_t-\Theta_t'}{}\leq n^{-20}.
\end{align}
Moreover, $|K_t-K_t'|\leq\varepsilon$ and $K_t\geq1$, so $K_t'\leq2K_t$. Combining \eqref{eq:error-main-W-bound} and \eqref{eq:error-main-XYZ-net}, we obtain
\begin{align*}
\Big|\sum_{j=1}^dM_j(\Theta_t)\Big|
\leq\Big|\sum_{j=1}^dM_j(\Theta_t')\Big|+n^{-20}\leq\frac{CK_t't\sqrt{\log n}}d+n^{-20}
\leq\frac{CK_tt\sqrt{\log n}}d.
\end{align*} Therefore,
\begin{align}\label{eq:error-main-unscaled-conclusion}
\sup_{\omega\in\sS^{t-1}}\Big|\sum_{\ell=0}^{t-1}\omega_\ell A_t^\ell\Big|
\leq\frac{CK_tt\sqrt{\log n}}d
\end{align}
with probability $1-O(n^{-12}).$ Multiplying \eqref{eq:error-main-unscaled-conclusion} by $\sqrt d$ in \eqref{eq:error-main-duality} proves the first statement.

A similar uniform concentration argument can be used to show the second statement so we will omit it.
\end{proof}

We next bound the weighted and unweighted derivative averages.
\begin{lemma}
\label{lem:F_concentration}
Fix an integer $t\geq0$ satisfying $t+1\leq c\sqrt d/\log^3n$. With probability $1-O(n^{-12})$, the following bounds hold uniformly over $\{\rho_k,\lambda_{k-1,[0:k-1]}\}_{k=1}^{t+1}$ and $\xi_{[0:t+1]}$ satisfying the corresponding bounds in \eqref{eqs:AtV}:
\begin{align}
&\sup_{\substack{\omega\in\sS^t\\ \omega\perp\lambda_{t,[0:t]}}}
\Bigg|\frac{1}{d}\Bigpnorm{\Big(\sum_{\ell=0}^t\omega_\ell\s^\ell\Big)\circ f_{t+1}'(\tilde\v^{t+1})}{}^2
-F_{t+1}\big(\rho_{t+1},\pnorm{\lambda_{t,[0:t]}}{};\btheta^\star\big)\Bigg|
\leq C\big(\pnorm{f_{t+1}'}{\infty}\vee1\big)^2(t+1)\sqrt{\frac{\log n}{d}},
\label{eq:F_concentration}\\
&\max_{1\leq k\leq t+1}
\Bigg|\frac{1}{d}\pnorm{f_k'(\tilde\v^k)}{}^2
-F_k\big(\rho_k,\pnorm{\lambda_{k-1,[0:k-1]}}{};\btheta^\star\big)\Bigg|
\leq C\big(\pnorm{f_{[1:t+1]}'}{\infty}\vee1\big)^2(t+1)\sqrt{\frac{\log n}{d}}.
\label{eq:F_unweighted_concentration}
\end{align}
Similarly, with probability $1-O(n^{-12})$, the following bounds hold uniformly over $\{\gamma_{k,[-1:k]}\}_{k=0}^t$ and $\vartheta_{[0:t]}$ satisfying the corresponding bounds in \eqref{eqs:AtU},
\begin{align}
&\sup_{\substack{\omega\in\sS^t\\ \omega\perp\gamma_{t,[0:t]}}}
\Bigg|\frac{1}{d}\Bigpnorm{\Big(\sum_{\ell=0}^t\omega_\ell\g^\ell\Big)\circ h_t'(\tilde\u^t,\y)}{}^2
-H_t\big(\gamma_{t,-1},\pnorm{\gamma_{t,[0:t]}}{};\bbeta^\star\big)\Bigg|
\leq C\delta\big(\pnorm{h_t'}{\infty}\vee1\big)^2(t+1)\sqrt{\frac{\log n}{n}},
\label{eq:G_concentration}\\
&\max_{0\leq k\leq t}
\Bigg|\frac{1}{d}\pnorm{h_k'(\tilde\u^k,\y)}{}^2
-H_k\big(\gamma_{k,-1},\pnorm{\gamma_{k,[0:k]}}{};\bbeta^\star\big)\Bigg|
\leq C\delta\big(\pnorm{h_{[0:t]}'}{\infty}\vee1\big)^2(t+1)\sqrt{\frac{\log n}{n}}.
\label{eq:H_unweighted_concentration}
\end{align}
\end{lemma}
\begin{proof}
To prove \eqref{eq:F_concentration}, fix deterministic $\omega\in\sS^t$, $a=(a_0,\ldots,a_t)$ with $\omega\perp a$, and $(\rho,\xi_{t+1})$, and write
\begin{align*}
\zeta\coloneqq\sum_{\ell=0}^t\omega_\ell\s^\ell,
\qquad
\bar\v\coloneqq\rho\btheta^\star+\sum_{\ell=0}^ta_\ell\s^\ell.
\end{align*}
Then $\zeta\sim\N(0,\I_d)$ and, by Gaussian orthogonality, $\zeta$ is independent of $\bar\v$, hence
\begin{align*}
\frac{1}{d}\E\pnorm{\zeta\circ f_{t+1}'(\bar\v)}{}^2
=F_{t+1}(\rho,\pnorm{a}{};\btheta^\star).
\end{align*}
We wish to apply Bernstein to
\begin{align*}
Z_j\coloneqq\frac{1}{d}\Big\{\zeta_j^2 f_{t+1}'(\bar v_j)^2-\E\big[\zeta_j^2 f_{t+1}'(\bar v_j)^2\big]\Big\},
\qquad 1\leq j\leq d.
\end{align*}
The $Z_j$ are independent and satisfy
\begin{align*}
|Z_j|\leq\frac{\pnorm{f_{t+1}'}{\infty}^2}{d}(\zeta_j^2+1).
\end{align*}
Bernstein's inequality then gives
\begin{align*}
\Big|\frac{1}{d}\pnorm{\zeta\circ f_{t+1}'(\bar\v)}{}^2-F_{t+1}(\rho,\pnorm{a}{};\btheta^\star)\Big|
&\leq C\pnorm{f_{t+1}'}{\infty}^2\Bigg[\sqrt{\frac{(t+1)\log n}{d}}+\frac{(t+1)\log n}{d}\Bigg]\\
&\qquad\leq C\big(\pnorm{f_{t+1}'}{\infty}\vee1\big)^2(t+1)\sqrt{\frac{\log n}{d}}
\end{align*}
with probability at least $1-2\exp\{-C_1(t+1)\log n\}$. 

Let $\Theta = (\omega,a,\rho,\xi_{t+1})$ be a tuple satisfying \eqref{eqs:AtV}, $\pnorm{\omega}{}=1,$ and $\omega\perp a$. We cover the parameter space of pairs
\begin{align*}
(\Theta,L),\quad \text{where} \quad L\coloneqq\pnorm{\mathsf f_{t+1}'(\cdot;\xi_{t+1})}{\infty}\vee1.
\end{align*}
Choose an $n^{-C_0}$-net $\mathcal N$ of these pairs. By \eqref{eq:parameter_cover_assumption} and \eqref{eqs:AtV}, the parameters are polynomially bounded and $1\leq L\leq C\log^5n$, so we can choose $\mathcal N$ with cardinality
\begin{align*}
|\mathcal N|\leq\exp\big\{C_2(t+1)\log n\big\}.
\end{align*}

On the event $\pnorm{[\s^0,\ldots,\s^t]}{\op}\leq C(\sqrt d+\sqrt{t+1})$, which holds with probability $1-n^{-20}$ by standard Gaussian operator-norm concentration and the horizon $t+1\leq c\sqrt d/\log^3n$, we claim that
\begin{align}\label{eq:F_parameter_gradient}
|\Phi_f(\Theta)-\Phi_f(\Theta')|\leq n^C\pnorm{\Theta-\Theta'}{}, \qquad \Phi_f(\Theta) := \frac{1}{d}\pnorm{\zeta\circ f_{t+1}'(\bar\v)}{}^2-F_{t+1}(\rho,\pnorm{a}{};\btheta^\star).
\end{align}
To verify \eqref{eq:F_parameter_gradient}, let $\bar\v',\zeta'$ correspond to $\Theta'$ which satisfy,
\begin{align*}
\pnorm{\bar\v-\bar\v'}{}+\pnorm{\zeta-\zeta'}{}
&\leq C\big(\sqrt d+\sqrt{t+1}\big)\pnorm{\Theta-\Theta'}{}.
\end{align*}
The nonlinearity satisfies the Lipschitz condition
\begin{align}
\big|\mathsf f_{t+1}'(v;\xi)-\mathsf f_{t+1}'(v';\xi')\big|&\leq\big|\mathsf f_{t+1}'(v;\xi)-\mathsf f_{t+1}'(v';\xi)\big|
+\big|\mathsf f_{t+1}'(v';\xi)-\mathsf f_{t+1}'(v';\xi')\big|\notag\\
&\leq C\log^5n\,|v-v'|
+n^{C_{\rm par}}\big(1+|v'|^{C_{\rm par}}\big)\pnorm{\xi-\xi'}{}.\label{eq:inter_lip}
\end{align}
Together with $|x^2-y^2|\leq(|x|+|y|)|x-y|$ we get
\begin{align*}
\frac{1}{d}\Bigg|\Bigpnorm{\zeta\circ\mathsf f_{t+1}'(\bar\v;\xi)}{}^2 - \Bigpnorm{\zeta'\circ\mathsf f_{t+1}'(\bar\v';\xi')}{}^2\Bigg| &\leq\frac{1}{d} \Big(\Bigpnorm{\zeta\circ\mathsf f_{t+1}'(\bar\v;\xi)}{} + \Bigpnorm{\zeta'\circ\mathsf f_{t+1}'(\bar\v';\xi')}{}\Big) \Bigpnorm{\zeta\circ\mathsf f_{t+1}'(\bar\v;\xi)-\zeta'\circ\mathsf f_{t+1}'(\bar\v';\xi')}{}\\
&\leq n^C\pnorm{\Theta-\Theta'}{}.
\end{align*} Moreover, using \eqref{eq:inter_lip} again and taking expectations gives
\begin{align*}
&\frac{1}{d}\Big|\E\pnorm{\mathsf f_{t+1}'(\bar\v;\xi_{t+1})}{}^2
-\E\pnorm{\mathsf f_{t+1}'(\bar\v';\xi_{t+1}')}{}^2\Big|\leq n^C\pnorm{\Theta-\Theta'}{}.
\end{align*}

For $(\Theta,L)$, choose $(\bar\Theta,\bar L)\in\mathcal N$ within distance $n^{-C_0}$. Then
\begin{align*}
\pnorm{\Theta-\bar\Theta}{}\leq n^{-C_0},
\qquad |L-\bar L|\leq n^{-C_0},
\qquad \bar L\leq2L.
\end{align*}
By \eqref{eq:F_parameter_gradient},
\begin{align}\label{eq:f_net_ext}
|\Phi_f(\Theta)-\Phi_f(\bar\Theta)|
\leq n^C\pnorm{\Theta-\bar\Theta}{}
\leq n^{C-C_0}\leq n^{-20},
\end{align}
for sufficiently large $C_0$.

After fixing $C_0$, choose $C_1\geq C_2+30$. Applying the preceding Bernstein bound at each net point and taking a union bound gives
\begin{align*}
\P\big(\text{some net-point bound fails}\big)
&\leq2|\mathcal N|\exp\big\{-C_1(t+1)\log n\big\}
\leq2n^{-30}.
\end{align*}
Thus, simultaneously for every $(\bar\Theta,\bar L)\in\mathcal N$, where $\bar\Theta=(\bar\omega,\bar a,\bar\rho,\bar\xi_{t+1})$,
\begin{align*}
&\Bigg|\frac{1}{d}\Bigpnorm{\Big(\sum_{\ell=0}^t\bar\omega_\ell\s^\ell\Big)\circ\mathsf f_{t+1}'(\bar\v;\bar\xi_{t+1})}{}^2
-F_{t+1}\big(\bar\rho,\pnorm{\bar a}{};\btheta^\star\big)\Bigg|\leq C\bar L^2(t+1)\sqrt{\frac{\log n}{d}}.
\end{align*}

Consequently, by \eqref{eq:f_net_ext},
\begin{align*}
|\Phi_f(\Theta)|
\leq|\Phi_f(\bar\Theta)|+n^{-20}
\leq C\bar L^2(t+1)\sqrt{\frac{\log n}{d}}+n^{-20}
\leq C\big(\pnorm{f_{t+1}'}{\infty}\vee1\big)^2(t+1)\sqrt{\frac{\log n}{d}}.
\end{align*}

For \eqref{eq:F_unweighted_concentration}, omit the Gaussian weight and use a net of $(a,\rho,\xi_k,L)$, with $a=(a_0,\ldots,a_{k-1})$ and $L=\pnorm{\mathsf f_k'(\cdot;\xi_k)}{\infty}\vee1$. The same argument, followed by a union bound over $1\leq k\leq t+1$, gives the result.

The proofs of \eqref{eq:G_concentration}-\eqref{eq:H_unweighted_concentration} are similar so we will omit them.
\end{proof}

We next bound the projected nonlinear difference in \eqref{def:Delta_Ut}. The orthogonal projection allows the derivative norms in a direct Lipschitz bound to be replaced by the averaged quantities $F_t$ and $H_t$.

\begin{lemma}
\label{lem:key_estimate_3}
Fix $t+1\leq c\sqrt d/\log^3n$. Except on an event of probability $O(n^{-12})$, the following implications hold: on $\cA_t^V$,
\begin{align}
\notag&\biggpnorm{P_{\w^t}^\perp\Big(\sum_{\ell=0}^t\frac{\r^\ell}{\sqrt d}\s^{\ell\top}\big(f_{t+1}(\v^{t+1})-f_{t+1}(\hat\v^{t+1})\big)\Big)}{}\\
&\quad\leq\Big(\sqrt{F_{t+1}\big(\rho_{t+1},\pnorm{\lambda_{t,[0:t]}}{};\btheta^\star\big)}+\cE^f_{t+1}\Big)\Big(\pnorm{\Delta'_{V,t+1}}{}+\sum_{k=0}^t|\beta_t^k|\pnorm{f_k'}{\infty}\pnorm{\Delta_{V,k}}{}\Big),\label{eq:perp_bound_v}
\end{align}
and, on $\cA_t^U$,
\begin{align}
\notag&\Bigpnorm{P_{\bar\z^t}^\perp\Big(\sum_{\ell=0}^t\frac{\q^\ell}{\sqrt d}\g^{\ell\top}\big(h_t(\u^t,\y)-h_t(\hat\u^t,\y)\big)\Big)}{}\\
&\quad\leq\Big(\sqrt{H_t\big(\gamma_{t,-1},\pnorm{\gamma_{t,[0:t]}}{};\bbeta^\star\big)}+\cE_t^h\Big)\Big(\pnorm{\Delta'_{U,t}}{}+\sum_{k=0}^{t-1}|\alpha_{t-1}^k|\pnorm{h_k'}{\infty}\pnorm{\Delta_{U,k}}{}\Big).\label{eq:perp_bound_u}
\end{align}
\end{lemma}
\begin{proof}
We first prove \eqref{eq:perp_bound_v}. The vector inside the projection belongs to $\spa\{\r^0,\ldots,\r^t\}$. Write
\begin{align*}
\delta\v=\v^{t+1}-\hat\v^{t+1},
\qquad
\a=\int_0^1f_{t+1}'\big(\hat\v^{t+1}+s\delta\v\big)\,ds.
\end{align*}
The integral is coordinatewise, and $f_{t+1}(\v^{t+1})-f_{t+1}(\hat\v^{t+1})=\a\circ\delta\v$. By duality, the left-hand side of \eqref{eq:perp_bound_v} equals
\begin{align}\label{eq:proj_out_1}
&\sup_{\substack{\bomega\in\spa\{\r^0,\ldots,\r^t\}\\ \pnorm{\bomega}{}=1,\ \bomega\perp\w^t}}
\frac{1}{\sqrt d}\iprod{\sum_{\ell=0}^t\omega_\ell\s^\ell}{\a\circ\delta\v}\leq\sup_{\substack{\omega\in\sS^t\\ \omega\perp\lambda_{t,[0:t]}}}
\frac{1}{\sqrt d}\Bigpnorm{\Big(\sum_{\ell=0}^t\omega_\ell\s^\ell\Big)\circ\a}{}\pnorm{\delta\v}{},
\end{align}
where $\omega_\ell=\iprod{\bomega}{\r^\ell}$. The orthonormality of the $\r^\ell$'s gives $\pnorm{\omega}{}=1$, and $\w^t=\sqrt d\sum_{\ell=0}^t\lambda_{t,\ell}\r^\ell$ gives $\omega\perp\lambda_{t,[0:t]}$.

Under $\cA_t^V$, we have $d\leq Cn^{10}$. Standard Gaussian concentration gives, with probability $1-O(n^{-20})$,
\begin{align*}
\sup_{\omega\in\sS^t}\Bigpnorm{\sum_{\ell=0}^t\omega_\ell\s^\ell}{\infty}
\leq C\big(\sqrt{t+1}+\sqrt{\log n}\big).
\end{align*}
Also, the Lipschitz continuity of $f_{t+1}'$ implies
\begin{align*}
\pnorm{\a-f_{t+1}'(\tilde\v^{t+1})}{}
\leq\pnorm{f_{t+1}''}{\infty}\big(\pnorm{\Delta_{V,t+1}}{}+\pnorm{\hat\v^{t+1}-\tilde\v^{t+1}}{}\big).
\end{align*}
Combining these bounds yields
\begin{align}\label{eq:f_prime_approx}
&\sup_{\omega\in\sS^t}\frac{1}{\sqrt d}\Bigpnorm{\Big(\sum_{\ell=0}^t\omega_\ell\s^\ell\Big)\circ\big(\a-f_{t+1}'(\tilde\v^{t+1})\big)}{}\notag\\
&\qquad\leq C\pnorm{f_{t+1}''}{\infty}\frac{\sqrt{t+1}+\sqrt{\log n}}{\sqrt d}\big(\pnorm{\Delta_{V,t+1}}{}+\pnorm{\hat\v^{t+1}-\tilde\v^{t+1}}{}\big).
\end{align}
Using $|\sqrt a-\sqrt b|\leq\sqrt{|a-b|}$,  \eqref{eq:f_prime_approx}, and the definition \eqref{eq:E_multiplier} give
\begin{align*}
\sup_{\substack{\omega\in\sS^t\\ \omega\perp\lambda_{t,[0:t]}}}\frac{1}{\sqrt d}\Bigpnorm{\Big(\sum_{\ell=0}^t\omega_\ell\s^\ell\Big)\circ\a}{}
\leq\sqrt{F_{t+1}\big(\rho_{t+1},\pnorm{\lambda_{t,[0:t]}}{};\btheta^\star\big)}+\cE^f_{t+1}.
\end{align*}
Finally, Lemma~\ref{lem:mutual_close} gives
\begin{align*}
\pnorm{\delta\v}{}\leq\pnorm{\Delta'_{V,t+1}}{}+\sum_{k=0}^t|\beta_t^k|\pnorm{f_k'}{\infty}\pnorm{\Delta_{V,k}}{}.
\end{align*}
The term $k=0$ is zero because $\Delta_{V,0}=0$. Substituting the last two bounds into \eqref{eq:proj_out_1} proves \eqref{eq:perp_bound_v}.

For \eqref{eq:perp_bound_u}, the case $t=0$ follows from $\u^0=\hat\u^0$. For $t\geq1$, write
\begin{align*}
\delta\u=\u^t-\hat\u^t,
\qquad
\a_h=\int_0^1h_t'(\hat\u^t+s\delta\u,\y)\,ds.
\end{align*}
As above, duality and $\bar\z^t=\sqrt d\sum_{\ell=0}^t\gamma_{t,\ell}\q^\ell$ give
\begin{align*}
&\Bigpnorm{P_{\bar\z^t}^\perp\Big(\sum_{\ell=0}^t\frac{\q^\ell}{\sqrt d}\g^{\ell\top}\big(h_t(\u^t,\y)-h_t(\hat\u^t,\y)\big)\Big)}{}\leq\sup_{\substack{\omega\in\sS^t\\ \omega\perp\gamma_{t,[0:t]}}}\frac{1}{\sqrt d}\Bigpnorm{\Big(\sum_{\ell=0}^t\omega_\ell\g^\ell\Big)\circ\a_h}{}\pnorm{\delta\u}{}.
\end{align*}
The Gaussian row-norm bound and the Lipschitz continuity of $h_t'$ imply
\begin{align*}
&\sup_{\omega\in\sS^t}\frac{1}{\sqrt d}\Bigpnorm{\Big(\sum_{\ell=0}^t\omega_\ell\g^\ell\Big)\circ\big(\a_h-h_t'(\tilde\u^t,\y)\big)}{}\leq C\pnorm{h_t''}{\infty}\frac{\sqrt{t+1}+\sqrt{\log n}}{\sqrt d}\big(\pnorm{\Delta_{U,t}}{}+\pnorm{\hat\u^t-\tilde\u^t}{}\big).
\end{align*}
Applying $|\sqrt a-\sqrt b|\leq\sqrt{|a-b|}$ to \eqref{eq:G_concentration}, and then using the preceding bound and the definition \eqref{eq:E_multiplier}, we obtain
\begin{align*}
\sup_{\substack{\omega\in\sS^t\\ \omega\perp\gamma_{t,[0:t]}}}
\frac{1}{\sqrt d}\Bigpnorm{\Big(\sum_{\ell=0}^t\omega_\ell\g^\ell\Big)\circ\a_h}{}
\leq\sqrt{H_t\big(\gamma_{t,-1},\pnorm{\gamma_{t,[0:t]}}{};\bbeta^\star\big)}+\cE_t^h.
\end{align*}
By \eqref{eq:diff_u_hat},
\begin{align*}
\pnorm{\delta\u}{}\leq\pnorm{\Delta'_{U,t}}{}+\sum_{k=0}^{t-1}|\alpha_{t-1}^k|\pnorm{h_k'}{\infty}\pnorm{\Delta_{U,k}}{}.
\end{align*}
Multiplying the last two bounds proves \eqref{eq:perp_bound_u}.
\end{proof}
\subsection{Proof of Proposition \ref{prop:general_error_induction}}
\label{subsec:ind_err_proof}
\begin{proof}[Proof of Proposition \ref{prop:general_error_induction}]
We first verify $\cH_0^U$. By construction, $\u^0=\tilde\u^0=\hat\u^0$, so $\Delta_{U,0}'=0$.

Fix a deterministic time $t$ in the stated range. We use the $U$-side bounds of Lemmas~\ref{lem:bound_lowd_projection} and \ref{lem:error_main_lemma} at time $t+1$, their $V$-side bounds at time $t$, Lemma~\ref{lem:F_concentration} at time $t$ and also at time $t-1$ when $t\geq1$, and Lemma~\ref{lem:key_estimate_3} at time $t$. The Gaussian operator-norm bounds also give
\begin{align*}
\pnorm{[\s^0,\ldots,\s^t]}{\op}\leq2\sqrt d,
\qquad
\pnorm{[\g^0,\ldots,\g^t]}{\op}\leq C\sqrt n.
\end{align*}
These statements hold together with probability $1-O(n^{-12})$. We work on this event and verify the two implications.

We first prove the bounds defining $\cH_{t+1}^U$, assuming that $\cA_t^V$ and $\cA_t^U$ hold.

\medskip
\noindent\textbf{Bound of $\pnorm{\Delta_{U,t+1}'}{}$ in \eqref{eq:induction_Delta_U}.} By Lemma~\ref{lem:error_decompose},
\begin{align}\label{eq:Delta_U_decompose}
\pnorm{\Delta_{U,t+1}'}{}
&\leq\Bigg\|P_{\w^t}^\perp\sum_{\ell=0}^t\frac{\r^\ell}{\sqrt d}\iprod{\s^\ell}{f_{t+1}(\v^{t+1})-f_{t+1}(\hat\v^{t+1})}\Bigg\|+\Bigpnorm{\sum_{\ell=0}^t\sqrt d\,\r^\ell A_{t+1}^\ell}{}+\pnorm{\Delta_{U,t+1}^{\low}}{}.
\end{align}
Since $\cA_t^V$ holds, Lemma~\ref{lem:error_main_lemma} gives
\begin{align}\label{eq:A_bound}
\Bigpnorm{\sum_{\ell=0}^t\sqrt d\,\r^\ell A_{t+1}^\ell}{}
&\leq C\big(\bar\gamma_{t+1}+\bar\lambda_t\big)\big(\pnorm{f_{[0:t+1]}'}{\infty}\vee1\big)^2\big(\pnorm{\rho_{[0:t+1]}}{\infty}\vee1\big)(t+1)\sqrt{\frac{\log n}{d}}.
\end{align}
On the other hand, \eqref{eq:perp_bound_v} of Lemma~\ref{lem:key_estimate_3} gives
\begin{align}\label{eq:perp_bound}
&\Bigg\|P_{\w^t}^\perp\sum_{\ell=0}^t\frac{\r^\ell}{\sqrt d}\iprod{\s^\ell}{f_{t+1}(\v^{t+1})-f_{t+1}(\hat\v^{t+1})}\Bigg\|\notag\\
&\quad\leq\Big(\sqrt{F_{t+1}\big(\rho_{t+1},\pnorm{\lambda_{t,[0:t]}}{};\btheta^\star\big)}+\cE_{t+1}^f\Big)\Big(\pnorm{\Delta_{V,t+1}'}{}+\sum_{k=0}^t|\beta_t^k|\pnorm{f_k'}{\infty}\pnorm{\Delta_{V,k}}{}\Big).
\end{align}
Finally, Lemma~\ref{lem:bound_lowd_projection} and the definition of $\bar\gamma_{t+1}$ give
\begin{align*}
\pnorm{\Delta_{U,t+1}^{\low}}{}
&\leq C\big(\sqrt{t+2}+\sqrt{\log n}\big)\pnorm{\gamma_{t+1,[-1:t+1]}}{}\leq C\bar\gamma_{t+1}(t+1)\sqrt{\frac{\log n}{d}}.
\end{align*}
Substituting these three bounds into \eqref{eq:Delta_U_decompose} proves \eqref{eq:induction_Delta_U} at time $t+1$.

\medskip
\noindent\textbf{Bound of $|\alpha_t^t|$ in \eqref{eq:induction_alpha_last}.} By \eqref{eq:alpha_correct_def}, we have
\begin{align*}
\alpha_t^t
&=\underbrace{\av{f_{t+1}'(\hat\v^{t+1})-f_{t+1}'(\v^{t+1})}}_{(I)}+\underbrace{\frac{1}{d\pnorm{\lambda_{t,[0:t]}}{}^2}\Big\langle\sum_{\ell=0}^t\lambda_{t,\ell}\s^\ell,f_{t+1}(\v^{t+1})-f_{t+1}(\hat\v^{t+1})\Big\rangle}_{(I\!I)}.
\end{align*}
For the first term,
\begin{align*}
|(I)|\leq\frac{\pnorm{f_{t+1}''}{\infty}}{\sqrt d}\pnorm{\v^{t+1}-\hat\v^{t+1}}{}.
\end{align*}
For the second term, the operator-norm bound gives
\begin{align*}
\Bigpnorm{\sum_{\ell=0}^t\lambda_{t,\ell}\s^\ell}{}\leq2\sqrt d\,\pnorm{\lambda_{t,[0:t]}}{},
\end{align*}
and hence
\begin{align}\label{eq:alpha_t_II}
|(I\!I)|\leq\frac{2\pnorm{f_{t+1}'}{\infty}}{\sqrt d\,\pnorm{\lambda_{t,[0:t]}}{}}\pnorm{\v^{t+1}-\hat\v^{t+1}}{}.
\end{align}
Combining the two estimates and applying Lemma~\ref{lem:mutual_close}, we obtain
\begin{align*}
|\alpha_t^t|
&\leq\Big(\frac{\pnorm{f_{t+1}''}{\infty}}{\sqrt d}+\frac{2\pnorm{f_{t+1}'}{\infty}}{\sqrt d\,\pnorm{\lambda_{t,[0:t]}}{}}\Big)\Big(\pnorm{\Delta_{V,t+1}'}{}+\sum_{k=0}^t|\beta_t^k|\pnorm{f_k'}{\infty}\pnorm{\Delta_{V,k}}{}\Big).
\end{align*}
Here we use
\begin{align*}
\sum_{k=0}^t|\beta_t^k|\pnorm{f_k'}{\infty}\pnorm{\Delta_{V,k}}{}
&\leq\pnorm{f_{[0:t]}'}{\infty}
\Big(\max_{0\leq k\leq t}\pnorm{\Delta_{V,k}}{}\Big)
\sum_{k=0}^t|\beta_t^k|,
\end{align*}
which proves \eqref{eq:induction_alpha_last} at time $t+1$.

\medskip
\noindent\textbf{Bounds of $\{|\alpha_t^k|\}_{0\leq k\leq t-1}$.} By \eqref{eq:alpha_correct_def} and \eqref{eq:beta_correct_def},
\begin{align}
\alpha_t^k&=\beta_t^{k+1}\av{f_{t+1}'(\hat\v^{t+1})\circ f_{k+1}'(\tilde\v^{k+1})},\qquad 0\leq k\leq t-1,\label{eq:alpha_def_1}\\
\beta_t^k&=\alpha_{t-1}^k\delta\av{h_t'(\hat\u^t,\y)\circ h_k'(\tilde\u^k,\y)},\qquad 0\leq k\leq t-1.\label{eq:beta_def_1}
\end{align}
For $t\geq1$ and $k=t-1$, the first identity gives
\begin{align*}
|\alpha_t^{t-1}|\leq\pnorm{f_{t+1}'}{\infty}\pnorm{f_t'}{\infty}|\beta_t^t|,
\end{align*}
which is \eqref{eq:induction_alpha_middle_last}. For $0\leq k\leq t-2$, combining the two identities gives
\begin{align*}
\alpha_t^k
&=\alpha_{t-1}^{k+1}\av{f_{t+1}'(\hat\v^{t+1})\circ f_{k+1}'(\tilde\v^{k+1})}\delta\av{h_t'(\hat\u^t,\y)\circ h_{k+1}'(\tilde\u^{k+1},\y)}.
\end{align*}
We bound these products. For every $0\leq k\leq t$, adding and subtracting $f_{t+1}'(\tilde\v^{t+1})$ and applying Cauchy--Schwarz gives
\begin{align*}
\Big|\av{f_{t+1}'(\hat\v^{t+1})\circ f_{k+1}'(\tilde\v^{k+1})}\Big|
&\leq\frac{\pnorm{f_{k+1}'}{\infty}\pnorm{f_{t+1}''}{\infty}}{\sqrt d}\pnorm{\hat\v^{t+1}-\tilde\v^{t+1}}{}+\frac{1}{d}\pnorm{f_{t+1}'(\tilde\v^{t+1})}{}\pnorm{f_{k+1}'(\tilde\v^{k+1})}{}.
\end{align*}
By \eqref{eq:F_unweighted_concentration}, for $1\leq j\leq t+1$,
\begin{align*}
\frac{1}{\sqrt d}\pnorm{f_j'(\tilde\v^j)}{}
&\leq\sqrt{F_j\big(\rho_j,\pnorm{\lambda_{j-1,[0:j-1]}}{};\btheta^\star\big)}+C\big(\pnorm{f_{[0:t+1]}'}{\infty}\vee1\big)\sqrt{t+2}\Big(\frac{\log n}{d}\Big)^{1/4}.
\end{align*}
For $1\leq j\leq t+1$, the definition of $\pnorm{f_j'}{\infty}$ and \eqref{eq:F_def} give
\begin{align*}
\frac{1}{\sqrt d}\pnorm{f_j'(\tilde\v^j)}{}
\leq\pnorm{f_{[0:t+1]}'}{\infty},\qquad
\sqrt{F_j\big(\rho_j,\pnorm{\lambda_{j-1,[0:j-1]}}{};\btheta^\star\big)}
\leq\pnorm{f_{[0:t+1]}'}{\infty}.
\end{align*} 
Using the preceding bounds, we have
\begin{align*}
&\frac{1}{d}\pnorm{f_{t+1}'(\tilde\v^{t+1})}{}\pnorm{f_{k+1}'(\tilde\v^{k+1})}{}
-\mathfrak F_{t+1,k+1}\\
&=\Bigg[\frac{\pnorm{f_{t+1}'(\tilde\v^{t+1})}{}}{\sqrt d}
-\sqrt{F_{t+1}\big(\rho_{t+1},\pnorm{\lambda_{t,[0:t]}}{};\btheta^\star\big)}\Bigg]
\frac{\pnorm{f_{k+1}'(\tilde\v^{k+1})}{}}{\sqrt d}\\
&\quad+\sqrt{F_{t+1}\big(\rho_{t+1},\pnorm{\lambda_{t,[0:t]}}{};\btheta^\star\big)}
\Bigg[\frac{\pnorm{f_{k+1}'(\tilde\v^{k+1})}{}}{\sqrt d}
-\sqrt{F_{k+1}\big(\rho_{k+1},\pnorm{\lambda_{k,[0:k]}}{};\btheta^\star\big)}\Bigg]\\
&\leq 2C\pnorm{f_{[0:t+1]}'}{\infty}
\big(\pnorm{f_{[0:t+1]}'}{\infty}\vee1\big)
\sqrt{t+2}\Big(\frac{\log n}{d}\Big)^{1/4}\\
&\leq C\big(\pnorm{f_{[0:t+1]}'}{\infty}\vee1\big)^2
\sqrt{t+2}\Big(\frac{\log n}{d}\Big)^{1/4}.
\end{align*}

Thus, we have
\begin{align*}
\frac{1}{d}\pnorm{f_{t+1}'(\tilde\v^{t+1})}{}\pnorm{f_{k+1}'(\tilde\v^{k+1})}{}
&\leq\mathfrak F_{t+1,k+1}+C\big(\pnorm{f_{[0:t+1]}'}{\infty}\vee1\big)^2\sqrt{t+2}\Big(\frac{\log n}{d}\Big)^{1/4}.
\end{align*}
The two error terms are bounded by $\cE_{t+1}^f$ in \eqref{eq:E_multiplier}, so
\begin{align}\label{eq:f_prime_concentration}
\Big|\av{f_{t+1}'(\hat\v^{t+1})\circ f_{k+1}'(\tilde\v^{k+1})}\Big|
\leq\mathfrak F_{t+1,k+1}+\cE_{t+1}^f,\qquad 0\leq k\leq t.
\end{align}

Similarly, on $\cA_t^U$, for $0\leq k\leq t$,
\begin{align*}
\Big|\delta\av{h_t'(\hat\u^t,\y)\circ h_k'(\tilde\u^k,\y)}\Big|
&\leq\frac{\sqrt\delta\,\pnorm{h_k'}{\infty}\pnorm{h_t''}{\infty}}{\sqrt d}\pnorm{\hat\u^t-\tilde\u^t}{}+\frac{1}{d}\pnorm{h_t'(\tilde\u^t,\y)}{}\pnorm{h_k'(\tilde\u^k,\y)}{}.
\end{align*}
By \eqref{eq:H_unweighted_concentration} and $|\sqrt a-\sqrt b|\leq\sqrt{|a-b|}$, for $0\leq k\leq t,$
\begin{align*}
&\Bigg|\frac{\pnorm{h_k'(\tilde\u^k,\y)}{}}{\sqrt d}
-\sqrt{H_k\big(\gamma_{k,-1},\pnorm{\gamma_{k,[0:k]}}{};\bbeta^\star\big)}\Bigg|\leq C\delta^{1/4}\big(\pnorm{h_{[0:t]}'}{\infty}\vee1\big)
\sqrt{t+1}\Big(\frac{\log n}{d}\Big)^{1/4}.
\end{align*}
For $0\leq j\leq t$, the definition of $\pnorm{h_j'}{\infty}$ and \eqref{eq:F_def} give
\begin{align*}
\frac{1}{\sqrt d}\pnorm{h_j'(\tilde\u^j,\y)}{}
\leq\sqrt\delta\,\pnorm{h_{[0:t]}'}{\infty},\qquad
\sqrt{H_j\big(\gamma_{j,-1},\pnorm{\gamma_{j,[0:j]}}{};\bbeta^\star\big)}
\leq\sqrt\delta\,\pnorm{h_{[0:t]}'}{\infty}.
\end{align*}

The same argument as on the $f$-side then gives
\begin{align*}
\frac{1}{d}\pnorm{h_t'(\tilde\u^t,\y)}{}\pnorm{h_k'(\tilde\u^k,\y)}{}
&\leq\mathfrak H_{t,k}+C\delta^{3/4}\big(\pnorm{h_{[0:t]}'}{\infty}\vee1\big)^2\sqrt{t+1}\Big(\frac{\log n}{d}\Big)^{1/4}.
\end{align*}
Both error terms are bounded by $\cE_t^h$, yielding
\begin{align}\label{eq:h_prime_concentration}
\Big|\delta\av{h_t'(\hat\u^t,\y)\circ h_k'(\tilde\u^k,\y)}\Big|
\leq\mathfrak H_{t,k}+\cE_t^h,\qquad 0\leq k\leq t.
\end{align}
Applying \eqref{eq:f_prime_concentration} and \eqref{eq:h_prime_concentration} to the formula for $\alpha_t^k$ proves \eqref{eq:induction_alpha_middle} at time $t+1$.

Next we prove the bounds defining $\cH_t^V$, assuming that $\cA_t^U$ and $\cA_{t-1}^V$ hold.

\medskip
\noindent\textbf{Bound of $\pnorm{\Delta_{V,t+1}'}{}$ in \eqref{eq:induction_Delta_V}.} The proof is similar to the $U$-side bound above. By Lemma~\ref{lem:error_decompose},
\begin{align*}
\pnorm{\Delta_{V,t+1}'}{}
&\leq\Bigg\|P_{\bar\z^t}^\perp\sum_{\ell=0}^t\frac{\q^\ell}{\sqrt d}\iprod{\g^\ell}{h_t(\u^t,\y)-h_t(\hat\u^t,\y)}\Bigg\|+\Bigpnorm{\sum_{\ell=0}^t\sqrt d\,\q^\ell B_t^\ell}{}+\pnorm{\Delta_{V,t+1}^{\low}}{}.
\end{align*}
On $\cA_t^U$, Lemma~\ref{lem:error_main_lemma} gives
\begin{align}\label{eq:B_bound}
\Bigpnorm{\sum_{\ell=0}^t\sqrt d\,\q^\ell B_t^\ell}{}
&\leq C(1+\sqrt\delta)\big(\bar\lambda_t+\bar\gamma_t\big)\big(\pnorm{h_{[0:t]}'}{\infty}\vee1\big)^2\big(\pnorm{\gamma_{[0:t],-1}}{\infty}\vee1\big)(t+1)\sqrt{\frac{\log n}{d}}.
\end{align}
The projected term is bounded by \eqref{eq:perp_bound_u} of Lemma~\ref{lem:key_estimate_3}. For the remaining term, Lemma~\ref{lem:bound_lowd_projection} gives
\begin{align*}
\pnorm{\Delta_{V,t+1}^{\low}}{}
&\leq C\big(\sqrt{t+1}+\sqrt{\log n}\big)\pnorm{\lambda_{t,[0:t]}}{}\leq C\bar\lambda_t(t+1)\sqrt{\frac{\log n}{d}}.
\end{align*}
Combining these bounds proves \eqref{eq:induction_Delta_V}.

\medskip
\noindent\textbf{Bound of $|\beta_t^t|$ in \eqref{eq:induction_beta_last}.}  Using \eqref{eq:beta_correct_def} we get
\begin{align*}
\beta_t^t
&=\underbrace{\delta\av{h_t'(\hat\u^t,\y)-h_t'(\u^t,\y)}}_{(I)}+\underbrace{\frac{1}{d\pnorm{\gamma_{t,[0:t]}}{}^2}\Big\langle\sum_{\ell=0}^t\gamma_{t,\ell}\g^\ell,h_t(\u^t,\y)-h_t(\hat\u^t,\y)\Big\rangle}_{(I\!I)}.
\end{align*}
For the first term,
\begin{align*}
|(I)|\leq\frac{\sqrt n}{d}\pnorm{h_t''}{\infty}\pnorm{\u^t-\hat\u^t}{}=\frac{\sqrt\delta}{\sqrt d}\pnorm{h_t''}{\infty}\pnorm{\u^t-\hat\u^t}{}.
\end{align*}
For the second term, the operator-norm bound gives
\begin{align*}
\Bigpnorm{\sum_{\ell=0}^t\gamma_{t,\ell}\g^\ell}{}\leq C\sqrt n\,\pnorm{\gamma_{t,[0:t]}}{},
\end{align*}
so, as in \eqref{eq:alpha_t_II},
\begin{align*}
|(I\!I)|\leq\frac{C\sqrt\delta\,\pnorm{h_t'}{\infty}}{\sqrt d\,\pnorm{\gamma_{t,[0:t]}}{}}\pnorm{\u^t-\hat\u^t}{}.
\end{align*}
Combining the estimates and applying Lemma~\ref{lem:mutual_close} yields
\begin{align*}
|\beta_t^t|
&\leq\sqrt\delta\Big(\frac{\pnorm{h_t''}{\infty}}{\sqrt d}+\frac{C\pnorm{h_t'}{\infty}}{\sqrt d\,\pnorm{\gamma_{t,[0:t]}}{}}\Big)\Big(\pnorm{\Delta_{U,t}'}{}+\sum_{k=0}^{t-1}|\alpha_{t-1}^k|\pnorm{h_k'}{\infty}\pnorm{\Delta_{U,k}}{}\Big).
\end{align*}
Taking maxima inside the sum proves \eqref{eq:induction_beta_last}.

\medskip
\noindent\textbf{Bounds of $\{|\beta_t^k|\}_{0\leq k\leq t-1}$.} For $t\geq1$, \eqref{eq:beta_def_1} gives
\begin{align*}
|\beta_t^{t-1}|\leq\delta\pnorm{h_t'}{\infty}\pnorm{h_{t-1}'}{\infty}|\alpha_{t-1}^{t-1}|,
\end{align*}
which is \eqref{eq:induction_beta_middle_last}. For $0\leq k\leq t-2$, substituting \eqref{eq:alpha_def_1} at time $t-1$ into \eqref{eq:beta_def_1} yields
\begin{align*}
\beta_t^k
&=\beta_{t-1}^{k+1}\delta\av{h_t'(\hat\u^t,\y)\circ h_k'(\tilde\u^k,\y)}\av{f_t'(\hat\v^t)\circ f_{k+1}'(\tilde\v^{k+1})}.
\end{align*}
Applying \eqref{eq:h_prime_concentration} at time $t$ and \eqref{eq:f_prime_concentration} at time $t-1$, under $\cA_t^U$ and $\cA_{t-1}^V$, respectively, gives
\begin{align*}
|\beta_t^k|\leq|\beta_{t-1}^{k+1}|\big(\mathfrak H_{t,k}+\cE_t^h\big)\big(\mathfrak F_{t,k+1}+\cE_t^f\big),
\end{align*}
as required by \eqref{eq:induction_beta_middle}.
\end{proof}

\section{Proofs for Phase I}\label{sec:proof_phase_1}
We first fix two high-probability events used throughout this appendix. Define the event
\begin{align}\label{eq:global_norm_event}
\mathcal E_{\rm norm} :=
\Big\{\pnorm{\X}{\op}\leq C(1+\sqrt\delta),\quad \pnorm{\bbeta^\star}{}\leq C(\sqrt n+\sqrt d)\Big\}.
\end{align}
Gaussian operator-norm and chi-square concentration give $\P(\mathcal E_{\rm norm}^c)=O(n^{-12})$. Since $h_s(\bm0_n,\bm y)=0_n$ and $\sqrt d\|\lambda_{s,[0:s]}\|=\|\bm w^s\|=\sqrt d$, this event gives
\begin{align}\label{eq:global_current_lambda_control}
\bar\lambda_s\leq C(\sqrt n+\sqrt d).
\end{align}
We also define the Gaussian moment event
\begin{align}\label{eq:gaussian_moment_event}
\mathcal E_{\rm mom}\coloneqq \big\{\pnorm{(\y-1)\circ\g^{-1}}{}^2\leq Cn, \quad \pnorm{(\y-1)^{\circ2}\circ\g^{-1}}{}^2\leq Cn\big\}.
\end{align}
Since $\y=(\g^{-1})^{\circ2}$, both quantities are sums of $n$ i.i.d.\ polynomials of standard Gaussian variables, so Gaussian polynomial concentration gives $\P(\mathcal E_{\rm mom}^c)=O(n^{-12})$. Throughout the remainder, we will work on $\cE_{\rm norm}\cap\cE_{\rm mom}$.

We will also use the following weighted near-isometry estimate.
\begin{lemma}\label{lem:near_isometry}
With probability $1-O(n^{-12})$, the following holds simultaneously for every integer $0\leq t\leq\log^2n$ and every $a\in\R^{t+1}$:
\begin{align*}
\Big|\frac{1}{\sqrt{d}}\Bigpnorm{\sum_{\ell=0}^t a_\ell
\frac{\g^\ell\circ (\y-1)}{\sqrt{2\delta}}}{} - \pnorm{a}{}\Big| \leq C_\delta\Big\{\sqrt{\frac{t+\log n}{n}}+\frac{(t+\log n)\log^2n}{n}\Big\}
\pnorm{a}{}.
\end{align*}
\end{lemma}

\begin{proof}
For a fixed $t$, define
\begin{align*}
z_i=\frac{(y_i-1)^2}{2},\qquad
\G_i=(g_i^0,\ldots,g_i^t)^\top,\qquad
\H=\frac{1}{n}\sum_{i=1}^nz_i\G_i\G_i^\top.
\end{align*}
The weights $\{z_i\}$ are independent of the vectors $\G_i$ and have mean one. Since $\delta d=n$,
\begin{align*}
\frac{1}{d}\Big\|\sum_{\ell=0}^ta_\ell
\frac{\g^\ell\circ(\y-1)}{\sqrt{2\delta}}\Big\|^2
=a^\top\H a.
\end{align*}
We will use the following  bounds which hold with probability $1-O(n^{-14})$ via standard concentration arguments:
\begin{align}\label{eq:y_basic_bounds}
\max_{i\leq n}z_i \leq C\log^2n,\qquad
\frac{1}{n}\sum_{i=1}^nz_i^2 \leq C, \qquad \Big|\frac{1}{n}\sum_{i=1}^nz_i-1\Big| \leq C\Big(\sqrt{\frac{\log n}{n}}+\frac{\log^3n}{n}\Big).
\end{align}

Condition on weights satisfying \eqref{eq:y_basic_bounds}. For any fixed $\bm x\in\mathbb S^t$, the variables $z_i\{\langle\G_i,\bm x\rangle^2-1\}$ are independent and centered. Their conditional variance sum is $O(n)$, and their largest sub-exponential scale is $O(\log^2n)$. Hence
\begin{align*}
\mathbb P\Big(\Big|\bm x^\top\Big(\H-\frac{1}{n}\sum_{i=1}^nz_i\I_{t+1}\Big)\bm x\Big|>s\,\Bigm|\,z_1,\ldots,z_n\Big)
\leq2\exp\Big\{-cn\min\Big(s^2,\frac{s}{\log^2n}\Big)\Big\}.
\end{align*}
A $1/4$-net of $\mathbb S^t$ has at most $9^{t+1}$ points. Taking
\begin{align*}
s=C\Big(\sqrt{\frac{t+\log n}{n}}+
\frac{(t+\log n)\log^2n}{n}\Big)
\end{align*}
with a sufficiently large numerical $C$ makes the union of the net failures $O(n^{-14})$. The quadratic-form net inequality then bounds the operator norm by twice the bound on the net. Combining this with the empirical mean of the weights yields
\begin{align}\label{eq:near_isometry_conc}
\|\H-\I_{t+1}\|_{\op}
\leq C\Big(\sqrt{\frac{t+\log n}{n}}+
\frac{(t+\log n)\log^2n}{n}\Big).
\end{align}
For every $a\ne0$,
\begin{align*}
\Big|\sqrt{a^\top\H a}-\pnorm a{}\Big|
=\frac{|a^\top(\H-\I_{t+1})a|}{\sqrt{a^\top\H a}+\pnorm a{}}
\leq\pnorm{\H-\I_{t+1}}{\op}\pnorm a{}.
\end{align*}
\end{proof}

\subsection{Proof of Lemma \ref{lem:uv_norm}}

\begin{proof}[Proof of Lemma \ref{lem:uv_norm}]
Since the bound of Lemma \ref{lem:gaussian_near_isometry} holds simultaneously for every coefficient vector, it applies to $(\rho_{t+1},\lambda_{t,[0:t]})$.  Taking $\a=\btheta^\star/\sqrt d$, $\bphi_{\ell+1}=\s^\ell/\sqrt d$, and $\eta=n^{-14}$, and taking a union bound of size $n$, gives
\begin{align*}
\frac{1}{\sqrt{d}}\pnorm{\tilde\v^{t+1}}{}
=\Bigg(1+O\Big(\sqrt{\frac{t+\log n}{d}}\Big)\Bigg)
\sqrt{\rho_{t+1}^2+\pnorm{\lambda_{t,[0:t]}}{}^2}.
\end{align*}
By the normalization in \eqref{eq:rescaled_nonlinearity} and the definition of the Gram--Schmidt coordinates,
\begin{align*}
\pnorm{\lambda_{t,[0:t]}}{}
=\frac{1}{\sqrt d}\pnorm{\w^t}{}=1,
\end{align*}
which proves \eqref{eq:lambda_norm_main} and completes the proof of \eqref{eq:tilde_v_norm}.  The triangle inequality then yields \eqref{eq:v_norm}.  Similarly, since $\tilde\u^t=\sum_{\ell=-1}^t\gamma_{t,\ell}\g^\ell$, \eqref{eq:claim_near_isometry} gives
\begin{align}\label{eq:u_tilde_norm}
\frac{1}{\sqrt{n}}\pnorm{\tilde\u^t}{}
=\Bigg(1+O\Big(\sqrt{\frac{t+\log n}{n}}\Big)\Bigg)
\pnorm{\gamma_{t,[-1:t]}}{}
=\Bigg(1+O\Big(\sqrt{\frac{t+\log n}{n}}\Big)\Bigg)
\frac{1}{\sqrt d}\pnorm{\v^t}{},
\end{align}
where the last equality uses $\z^t=f_t(\v^t)=\v^t$.  Combining this display with \eqref{eq:v_norm} (at time $t-1$) for $t\geq1$ proves \eqref{eq:tilde_u_norm}; the triangle inequality gives \eqref{eq:u_norm}.  The identity $\pnorm{\gamma_{t,[-1:t]}}{}=\pnorm{\z^t}{}/\sqrt d$ also gives \eqref{eq:gamma_norm}.

It remains to prove \eqref{eq:gamma_norm_loo}.  By orthogonal projection,
\begin{align*}
\Big|\pnorm{\gamma_{t,[0:t]}}{}
-\frac{1}{\sqrt d}\Bigpnorm{P_{\btheta^\star}^\perp\tilde\v^t}{}\Big|
\leq\frac{1}{\sqrt d}\pnorm{\Delta_{V,t}}{}.
\end{align*}
Moreover, \eqref{eq:claim_near_isometry} gives
\begin{align*}
\frac{1}{\sqrt{d}}\Bigpnorm{P_{\btheta^\star}^\perp \tilde\v^t}{}
= \frac{1}{\sqrt d}\Bigpnorm{P_{\btheta^\star}^\perp
\sum_{\ell=0}^{t-1} \lambda_{t-1,\ell}\s^\ell}{}
=1+O\Big(\sqrt{\frac{t+\log n}{d}}\Big),
\end{align*}
which proves \eqref{eq:gamma_norm_loo} for $t\geq1$. At $t=0$, the same bounds follow from Gaussian concentration for $\v^0$ and $P_{\btheta^\star}^\perp\v^0$, together with $\Delta_{V,0}=0$ and $\rho_0=d^{-1}\iprod{\btheta^\star}{\v^0}$. The horizon $t\leq \log^2 n$ and $\delta>1/2$ ensure the dimension conditions in the near-isometry estimates for $n$ larger than a universal constant.
\end{proof}

\subsection{Proof of Proposition \ref{prop:phase_1_tilde_close}}\label{subsec:proof_approximation_phase1}
The key step is to apply the induction of Section \ref{subsec:main_ind}, where the recursive bounds \eqref{eq:induction_Delta_U}--\eqref{eq:induction_beta_middle} and the assumptions $\cA^U_t, \cA^V_t$ of Definition \ref{assump:inductive_assumption} are stated for general nonlinearities. To lighten the exposition, we specialize both to Phase I below. We start with a slightly stronger set of assumptions.
\begin{definition}
\label{assump:inductive_assumption_phase_1}
Let $\cA^U_{I,t}$ be the event on which
\begin{subequations}
\begin{align}
\sum_{k=0}^{t-1}|\alpha_{t-1}^k|
&\leq C_\delta\frac{t\log^{9}n}{\sqrt{d}}, \label{eq:phase_1_assump_U_5}
\\
\pnorm{\Delta_{U,t}}{}\vee\pnorm{\hat \u^t-\tilde \u^t}{}
&\leq C_\delta t\log^9n, \label{eq:phase_1_assump_U_2}
\\
\bar\gamma_t &\leq C_\delta\sqrt{d},\label{eq:phase_1_assump_U_3}
\\
\bar\lambda_t&\leq C_\delta\sqrt d,
\label{eq:phase_1_current_lambda}\\
\pnorm{h_t'}{\infty}\leq C_\delta\log n,&\qquad \pnorm{h_t''}{\infty}\leq C_\delta
\label{eq:phase_1_assump_U_1}.
\end{align}
\end{subequations}
Similarly, let $\cA_{I,t}^V$ be the event on which
\begin{subequations}
\begin{align}
\sum_{k=0}^{t}|\beta_t^k|
&\leq C_\delta\frac{(t+1)\log^{9}n}{\sqrt{d}}, \label{eq:phase_1_assump_V_5}\\
    \pnorm{\Delta_{V,t+1}}{} \vee \pnorm{\hat \v^{t+1}-\tilde \v^{t+1}}{}
    &\leq C_\delta(t+1)\log^9 n, \label{eq:uv_closeness_phase_1}\\
    \bar\lambda_t &\leq C_\delta\sqrt d, \label{eq:lambda_main_norm_phase_1}\\
    \bar\gamma_{t+1}&\leq Cn^5,
    \label{eq:phase_1_next_gamma}\\
    \pnorm{f'_{t+1}}{\infty}=1,&\qquad
    \pnorm{f_{t+1}''}{\infty}=0.
    \label{eq:phase_1_assump_V_1}
\end{align}
\end{subequations}
\end{definition}

Throughout this subsection, we say that the refined assumptions hold up to $\cA_{I,t-1}^V$ if $\cA_{I,s}^U$ and $\cA_{I,s}^V$ hold for every $0\leq s\leq t-1$, and that they hold up to $\cA_{I,t}^U$ if, in addition, $\cA_{I,t}^U$ holds. Next, we give a few additional consequences of these assumptions. The proofs of Lemmas~\ref{lem:phase_1_assump_conseq} and \ref{lem:phase_1_simp} will be given after the proof of Proposition~\ref{prop:phase_1_tilde_close}.

\begin{lemma}\label{lem:phase_1_assump_conseq}
Let $0\leq t<\tau_I$ with $t\leq\log^2n$. Suppose the refined assumptions hold up to $\cA_{I,t-1}^V$, and that the bounds \eqref{eq:phase_1_assump_U_5}, \eqref{eq:phase_1_assump_U_2}, and \eqref{eq:phase_1_assump_U_3} hold at time $t$. Then \eqref{eq:phase_1_mu_est}, \eqref{eq:phase_1_norm}, \eqref{eq:ht_first}, and \eqref{eq:phase_1_assump_U_1} hold at time $t$. Moreover,
\begin{align}\label{eq:phase_1_HF}
F_{t+1}\Big(\rho_{t+1},\|\lambda_{t,[0:t]}\|;\btheta^\star\Big)=1,\qquad
H_t\Big(\rho_t,\|\gamma_{t,[0:t]}\|;\bbeta^\star\Big)=1+O_\delta\Big(\frac{\log^{20}n}{\sqrt n}\Big).
\end{align}
\end{lemma}

Using the refined assumptions $\cA_{I,t}^U, \cA_{I,t}^V$ and their consequences in Lemma \ref{lem:phase_1_assump_conseq}, we can readily simplify the recursive bounds $\cH_t^U$ (\ref{eq:induction_Delta_U})-(\ref{eq:induction_alpha_middle}) and $\cH_t^V$ (\ref{eq:induction_Delta_V})-(\ref{eq:induction_beta_middle}) as specified in the following lemma.

\begin{lemma}
\label{lem:phase_1_simp}
Let $0\leq t<\tau_I$ with $t\leq\log^2n$. Suppose $t\geq 1$ and the refined assumptions hold up to $\cA^V_{I,t-1}$. Then the bounds $\cH_{t}^U$ in (\ref{eq:induction_Delta_U})-(\ref{eq:induction_alpha_middle}) imply that
\begin{align}
\pnorm{\Delta_{U,t}'}{} &\leq \Big(1+C_\delta\frac{\log^2n}{n^{1/4}}\Big)\pnorm{\Delta_{V,t}'}{} + C_\delta t\sqrt{\log n}, \label{eq:induce_U_1}\\
|\alpha_{t-1}^{t-1}| &\leq C_\delta\Big( \frac{\pnorm{\Delta_{V,t}'}{}}{\sqrt{d}} + \frac{t^2\log^{18}n}{d}\Big), \label{eq:induce_U_2}\\
|\alpha_{t-1}^k| &\leq 
\begin{cases}
|\beta_{t-1}^{t-1}| & k=t-2\\
\Big(1+C_\delta\frac{\log^4 n}{n^{1/4}}\Big)|\alpha_{t-2}^{k+1}| & 0\leq k \leq t-3.
\end{cases} \label{eq:induce_U_3}
\end{align}
Suppose the refined assumptions hold up to $\cA^U_{I,t}$. Then the bounds $\cH_{t}^V$ in (\ref{eq:induction_Delta_V})-(\ref{eq:induction_beta_middle}) imply that
\begin{align}
\pnorm{\Delta_{V,t+1}'}{} &\leq \Big(1+C_\delta\frac{\log^4n}{n^{1/4}}\Big)\pnorm{\Delta_{U,t}'}{} + C_\delta(t+1)\log^{5/2}n, \label{eq:induce_V_1}\\
|\beta_{t}^{t}| &\leq C_\delta\Big(\frac{\log n\pnorm{\Delta'_{U,t}}{}}{\sqrt{d}} + \frac{t^2\log^{21}n}{d}\Big), \label{eq:induce_V_2}\\
|\beta_{t}^k| &\leq 
\begin{cases}
C\log^2 n|\alpha_{t-1}^{t-1}| & k=t-1\\
\Big(1+C_\delta\frac{\log^4n}{n^{1/4}}\Big)|\beta_{t-1}^{k+1}| & 0\leq k \leq t-2.
\end{cases} \label{eq:induce_V_3}
\end{align}
\end{lemma}

With these preliminary results, we are now ready to give a proof of Proposition \ref{prop:phase_1_tilde_close}.

\begin{proof}[Proof of Proposition \ref{prop:phase_1_tilde_close}]
Observe that, by the definition of the refined assumptions and Lemma~\ref{lem:phase_1_assump_conseq}, it suffices to show that $\cA_{I,t}^U$ and $\cA_{I,t}^V$ hold simultaneously for all $0\leq t<\tau_I$ with $t\leq\log^2n$, with probability $1-O(n^{-10})$. The key idea is to establish the following chain of claims:
\begin{align}\label{eq:phase_1_chain}
\cH_0^U
&\rightarrow\cA_{I,0}^U
\overset{\checkmark}{\rightarrow}\cH_0^V
\rightarrow\cA_{I,0}^V
\overset{\checkmark}{\rightarrow}\cH_1^U\rightarrow\cA_{I,1}^U
\overset{\checkmark}{\rightarrow}\cH_1^V
\rightarrow\cA_{I,1}^V
\overset{\checkmark}{\rightarrow}\cdots.
\end{align}
The arrows marked with a $\checkmark$ follow from Proposition~\ref{prop:general_error_induction}. Specifically, the refined assumptions through $\cA_{I,t}^U$ imply the assumptions $\cA_t^U\cap\cA_{t-1}^V$ of Definition~\ref{assump:inductive_assumption}, and those through $\cA_{I,t}^V$ imply $\cA_t^U\cap\cA_t^V$ whenever $t+1<\tau_I$ and $t+1\leq\log^2n$.

We inductively verify the remaining arrows:
\begin{itemize}
\item (Claim 1) Suppose the chain \eqref{eq:phase_1_chain} holds up to and including $\cH_t^U$. Then $\cA_{I,t}^U$ holds.
\item (Claim 2) Suppose the chain \eqref{eq:phase_1_chain} holds up to and including $\cH_t^V$. Then $\cA_{I,t}^V$ holds.
\end{itemize}
The constants in the refined assumptions are fixed sufficiently large in terms of $\delta$, and then we take $n$ large enough.

\medskip
\noindent\textbf{Proof of Claim 1.}
At $t=0$, the initialization gives
\begin{align}\label{eq:phase1_U_zero_base}
\Delta'_{U,0}=\Delta_{U,0}=\hat\u^0-\tilde\u^0=0,
\qquad \sum_{k=0}^{-1}|\alpha_{-1}^k|=0.
\end{align}
Since $f_0=\operatorname{id}$, \eqref{def:induction_constant} and Gaussian concentration give
\begin{align*}
\bar\gamma_0
&=\max\Big\{\sqrt d,\sqrt d\,\pnorm{\gamma_{0,[-1:0]}}{}\Big\}
=\max\Big\{\sqrt d,\pnorm{\v^0}{}\Big\}
\leq C\sqrt d,
\end{align*}
which proves \eqref{eq:phase_1_assump_U_3} at $t=0$. This shows the three bounds required by Lemma~\ref{lem:phase_1_assump_conseq} hold, and so
\eqref{eq:phase_1_norm} and \eqref{eq:ht_first}
hold at time $0$. Also, \eqref{eq:global_current_lambda_control} gives the bound for $\bar\lambda_0$, proving
$\cA_{I,0}^U$.

Now let $t\geq1$ and suppose the chain holds up to $\cH_t^U$. For every $1\leq s\leq t$, combining \eqref{eq:induce_V_1} at time $s-1$ with \eqref{eq:induce_U_1} at time $s$ gives
\begin{align*}
\pnorm{\Delta'_{U,s}}{}
&\leq\Big(1+C_\delta\frac{\log^4n}{n^{1/4}}\Big)
\pnorm{\Delta'_{U,s-1}}{}+C_\delta s\log^{5/2}n.
\end{align*}
Since $s\leq\log^2n$, we have
\begin{align}
\label{eq:mult}
\Big(1+C_\delta\frac{\log^4n}{n^{1/4}}\Big)^{\log^2n}
\leq\exp\Big\{C_\delta\frac{\log^6n}{n^{1/4}}\Big\}\leq2,
\end{align}
where the last inequality follows from $n$ large enough. Iterating from $\Delta'_{U,0}=0$ therefore gives
\begin{align*}
\pnorm{\Delta'_{U,s}}{}
\leq C_\delta\log^{5/2}n\sum_{j=1}^s j
\leq C_\delta s^2\log^{5/2}n.
\end{align*}
Equation \eqref{eq:induce_V_1} gives the same bound for $\Delta'_{V,s}$, so
\begin{align}\label{eq:phase1_second_order_residual}
\pnorm{\Delta'_{U,s}}{}\vee\pnorm{\Delta'_{V,s}}{}
\leq C_\delta s^2\log^{5/2}n
\leq C_\delta\log^{13/2}n,
\qquad 1\leq s\leq t.
\end{align}

\medskip
\noindent\textbf{Proof of \eqref{eq:phase_1_assump_U_5}.}
Equations \eqref{eq:induce_U_2} and \eqref{eq:phase1_second_order_residual} give
\begin{align*}
|\alpha_{t-1}^{t-1}|
\leq C_\delta\Big(\frac{\log^{13/2}n}{\sqrt d}
+\frac{t^2\log^{18}n}{d}\Big).
\end{align*}
For $t\geq2$, the first line of \eqref{eq:induce_U_3}, together with \eqref{eq:induce_V_2} at time $t-1$ and \eqref{eq:phase1_second_order_residual}, gives
\begin{align*}
|\alpha_{t-1}^{t-2}| \leq |\beta_{t-1}^{t-1}| \leq C_\delta\Big(\frac{\log^{15/2}n}{\sqrt d} + \frac{t^2\log^{21}n}{d}\Big).
\end{align*}
For $0\leq k\leq t-3$, the second line of \eqref{eq:induce_U_3} gives $|\alpha_{t-1}^k|\leq(1+C_\delta\log^4n/n^{1/4})|\alpha_{t-2}^{k+1}|$.
Applying the recursion repeatedly gives, for some $1\leq s\leq t$,
\begin{align*}
|\alpha_{t-1}^k|
&\leq\Big(1+C_\delta\frac{\log^4n}{n^{1/4}}\Big)^{t-s}
|\alpha_{s-1}^{k+t-s}|\leq\Big(1+C_\delta\frac{\log^4n}{n^{1/4}}\Big)^t
|\alpha_{s-1}^{k+t-s}|
\leq2|\alpha_{s-1}^{k+t-s}|,
\end{align*}
where $k+t-s\in\{s-2,s-1\}$, so the last coefficient is either $\alpha_{s-1}^{s-2}$ or $\alpha_{s-1}^{s-1}$. The last inequality uses the bound \eqref{eq:mult}.
Summing over $k$ therefore yields
\begin{align*}
\sum_{k=0}^{t-1}|\alpha_{t-1}^k| \leq C_\delta\sum_{s=1}^t
\Big(\frac{\log^{15/2}n}{\sqrt d}+\frac{s^2\log^{21}n}{d}\Big) \leq C_\delta\Big(\frac{t\log^{15/2}n}{\sqrt d}
+\frac{t^3\log^{21}n}{d}\Big) \leq C_\delta\frac{t\log^{15/2}n}{\sqrt d},
\end{align*}
where the last step follows from $t\leq\log^2n$ and $n$ large enough. This proves \eqref{eq:phase_1_assump_U_5}.

\medskip
\noindent\textbf{Proof of \eqref{eq:phase_1_assump_U_2}.}
For every $k<t$, the bounds \eqref{eq:phase_1_assump_U_1} and \eqref{eq:phase_1_assump_U_2} at time $k$ give
\begin{align*}
\pnorm{h_k(\tilde\u^k,\y)}{} \leq\pnorm{\w^k}{}+\pnorm{h_k'}{\infty}\pnorm{\Delta_{U,k}}{} \leq\sqrt d+C_\delta k\log^{10}n \leq C_\delta\sqrt d.
\end{align*}
By the definition of $\hat\u^t$ and Lemma~\ref{lem:mutual_close},
\begin{align*}
\pnorm{\hat\u^t-\tilde\u^t}{}
&\leq C_\delta\sqrt d\sum_{k=0}^{t-1}|\alpha_{t-1}^k|,\qquad
\pnorm{\Delta_{U,t}}{}
\leq\pnorm{\Delta_{U,t}'}{}+C_\delta\sqrt d\sum_{k=0}^{t-1}|\alpha_{t-1}^k|.
\end{align*}
The bound \eqref{eq:phase_1_assump_U_5} just proved and \eqref{eq:phase1_second_order_residual} show that both right-hand sides are at most $C_\delta t\log^{15/2}n$. This proves \eqref{eq:phase_1_assump_U_2}.

\medskip
\noindent\textbf{Proof of \eqref{eq:phase_1_assump_U_3}.}
For $0\leq s\leq t$, \eqref{eq:gamma_norm} and the bound \eqref{eq:uv_closeness_phase_1} on $\pnorm{\Delta_{V,s}}{}$ at time $s-1$ give
\begin{align*}
\sqrt d\,\pnorm{\gamma_{s,[-1:s]}}{}
&\leq\sqrt d\Big(1+C\sqrt{\frac{s+\log n}{d}}\Big)\sqrt{1+\rho_s^2}
+C\pnorm{\Delta_{V,s}}{}
\leq C\sqrt d.
\end{align*}
At $s=0$, use $\Delta_{V,0}=0$. Since $f_s(0_d)=0_d$ and $\pnorm{\btheta^\star}{}=\sqrt d$, taking the maximum in the definition of $\bar\gamma_t$ proves \eqref{eq:phase_1_assump_U_3}.

\medskip
\noindent\textbf{Proof of \eqref{eq:phase_1_assump_U_1} and \eqref{eq:phase_1_current_lambda}.}
The bounds \eqref{eq:phase_1_assump_U_5}, \eqref{eq:phase_1_assump_U_2} and \eqref{eq:phase_1_assump_U_3} at time $t$, just proved, are the hypotheses of Lemma~\ref{lem:phase_1_assump_conseq}, which therefore gives \eqref{eq:phase_1_norm} and \eqref{eq:ht_first}, and hence
\begin{align*}
\pnorm{h_t'}{\infty}\leq C_\delta\log n,
\qquad
\pnorm{h_t''}{\infty}\leq C_\delta|\hat\mu_t|\log^{3/2}n\leq C_\delta.
\end{align*}
This proves \eqref{eq:phase_1_assump_U_1}. Then $N_t$ of (\ref{eq:pi_parameter}) is positive by \eqref{eq:phase_1_norm}, so $\w^t$ is defined and \eqref{eq:global_current_lambda_control} gives $\bar\lambda_t\leq C_\delta\sqrt d$. This proves \eqref{eq:phase_1_current_lambda}, and completes Claim 1.

\medskip
\noindent\textbf{Proof of Claim 2.}
Suppose the chain holds up to $\cH_t^V$. Equation \eqref{eq:induce_V_1} and the bound \eqref{eq:phase1_second_order_residual} on $\pnorm{\Delta'_{U,t}}{}$ give
\begin{align}\label{eq:phase1_second_order_V_next}
\pnorm{\Delta'_{V,t+1}}{} \leq C_\delta(t+1)^2\log^{5/2}n.
\end{align}
For $\beta_t^t$ and $\beta_t^{t-1}$, \eqref{eq:induce_V_2}--\eqref{eq:induce_V_3} give
\begin{align*}
|\beta_t^t| & \leq C_\delta\Big(\frac{\log^{15/2}n}{\sqrt d} + \frac{t^2\log^{21}n}{d}\Big),\\
|\beta_t^{t-1}| & \leq C\log^2n|\alpha_{t-1}^{t-1}|
\leq C_\delta\Big(\frac{\log^{17/2}n}{\sqrt d} + \frac{t^2\log^{20}n}{d}\Big),\qquad t\geq1.
\end{align*}
For $0\leq k\leq t-2$, \eqref{eq:induce_V_3} gives $|\beta_t^k|\leq(1+C_\delta\log^4n/n^{1/4})|\beta_{t-1}^{k+1}|$. Applying this repeatedly reduces $\beta_t^k$ to one of the coefficients $\beta_s^s$, $\beta_s^{s-1}$ at an earlier time $s$, at the cost of at most $t$ factors $1+C_\delta\log^4n/n^{1/4}$, whose product is bounded as in Claim 1. Summing over $k$ therefore gives
\begin{align*}
\sum_{k=0}^t|\beta_t^k|
&\leq C_\delta\Big(\frac{(t+1)\log^{17/2}n}{\sqrt d}
+\frac{(t+1)^3\log^{21}n}{d}\Big)
\leq C_\delta\frac{(t+1)\log^{17/2}n}{\sqrt d}.
\end{align*}
Since $17/2<9$, this proves \eqref{eq:phase_1_assump_V_5} for large enough $n$. At $t=0$ the only coefficient is $\beta_0^0$, and the bound above applies to it directly.

Since $f_k=\operatorname{id}$ and $\gamma_{k,-1}=\rho_k$, we have $f_k(\tilde\v^k)-\gamma_{k,-1}\btheta^\star=\tilde\v^k-\rho_k\btheta^\star$. For $1\leq k\leq t$, \eqref{eq:claim_near_isometry} and $\pnorm{\lambda_{k-1,[0:k-1]}}{}=1$ give $\pnorm{\tilde\v^k-\rho_k\btheta^\star}{}\leq C\sqrt d$; for $k=0$ this norm is at most $\pnorm{\v^0}{}\leq C\sqrt d$. The definition of $\hat\v^{t+1}$, Lemma~\ref{lem:mutual_close}, the bound on $\sum_k|\beta_t^k|$ just proved, and \eqref{eq:phase1_second_order_V_next} therefore imply
\begin{align*}
\pnorm{\hat\v^{t+1}-\tilde\v^{t+1}}{} \vee\pnorm{\Delta_{V,t+1}}{} \leq C_\delta(t+1)\log^{17/2}n,
\end{align*}
which proves \eqref{eq:uv_closeness_phase_1}.

The bound \eqref{eq:lambda_main_norm_phase_1} on $\bar\lambda_t$ is \eqref{eq:phase_1_current_lambda}, proved in Claim 1. For $\bar\gamma_{t+1}$, the update \eqref{eq:simplified_amp} gives
\begin{align*}
\sqrt d\,\pnorm{\gamma_{t+1,[-1:t+1]}}{} = \pnorm{\v^{t+1}}{} \leq\pnorm{\X}{\mathrm{op}}\pnorm{\w^t}{} + |\hat c_t|\pnorm{\v^t}{} \leq C_\delta\sqrt d\log n\leq Cn^5,
\end{align*}
where $|\hat c_t|\leq\delta\pnorm{h_t'}{\infty}$ and the last inequality holds for $n$ large enough. Since $\bar\gamma_t\leq C_\delta\sqrt d$ by \eqref{eq:phase_1_assump_U_3}, this proves \eqref{eq:phase_1_next_gamma}. Finally, \eqref{eq:phase_1_assump_V_1} holds because $f_{t+1}$ is the identity. Thus $\cA_{I,t}^V$ holds, proving Claim 2 and closing the induction.

Finally, $(t+1)\log^9n\leq C\log^{11}n$ on this horizon, so \eqref{eq:phase_1_assump_U_2} and \eqref{eq:uv_closeness_phase_1} imply \eqref{eq:Delta_bound_phase1}. Lemma~\ref{lem:phase_1_assump_conseq} gives the remaining conclusions. 
\end{proof}

\begin{proof}[Proof of Lemma \ref{lem:phase_1_assump_conseq}]
We prove the inequalities in order.

\medskip
\noindent\textbf{Proof of \eqref{eq:phase_1_mu_est}.}
By the definition of $\hat\mu_t$, it suffices to bound $|\pnorm{\u^t}{}^2/n-1|$. Recall that $\u^t=\tilde\u^t+\Delta_{U,t}$, so
\begin{align*}
\Big|\frac{\pnorm{\u^t}{}^2}{n}-1\Big|
&\leq\Big|\frac{\pnorm{\tilde\u^t}{}^2}{n}-1\Big|
+\frac{\pnorm{\Delta_{U,t}}{}^2}{n}
+\frac{2\pnorm{\tilde\u^t}{}\pnorm{\Delta_{U,t}}{}}{n}.
\end{align*}
Using \eqref{eq:tilde_u_norm}, the bound \eqref{eq:uv_closeness_phase_1} on $\pnorm{\Delta_{V,t}}{}$ at time $t-1$, and \eqref{eq:phase_1_assump_U_2}, we obtain
\begin{align*}
\Big|\frac{\pnorm{\u^t}{}^2}{n}-1\Big|
&\leq C_\delta\Big(\rho_t^2+\frac{\log^{11}n}{\sqrt n}\Big)
\leq C_\delta\frac{\log^{20}n}{\sqrt n},
\end{align*}
where we used $t\leq\log^2n$ and $|\rho_t|\leq n^{-1/4}\log^{10}n$. At $t=0$, the same argument applies with $\Delta_{V,0}=0$. Since
\begin{align*}
\hat\mu_t=\sqrt{\Big(\frac{\pnorm{\u^t}{}^2}{n}-1\Big)\vee n^{-1/2}},
\end{align*}
this proves \eqref{eq:phase_1_mu_est}.

\medskip
\noindent\textbf{Proof of \eqref{eq:phase_1_norm} and \eqref{eq:ht_first}.}
Taylor's theorem and $\pnorm{\tanh^{(5)}}{\infty}\leq C$ give $\tanh(x)=x-x^3/3+O(|x|^5)$ for every $x\in\R$. This implies that
\begin{align}\label{eq:ht_taylor}
h_t^\star(u,y) =\big((1+\hat\mu_t^2)y-1\big)\hat\mu_tu
-\frac{1}{3}(1+\hat\mu_t^2)\hat\mu_t^3y^2u^3 + O\big((1+\hat\mu_t^2)\hat\mu_t^5y^3|u|^5\big).
\end{align}
By \eqref{eq:phase_1_mu_est}, we have $\hat\mu_t\leq1$ for large enough $n$. Hence
\begin{align}\label{eq:linearization_main}
\Big|\frac{\pnorm{h_t^\star(\u^t,\y)}{}}{\sqrt d}-\sqrt{2\delta}\hat\mu_t\Big|
&\leq\hat\mu_t\Big|\frac{\pnorm{(\y-1)\circ\u^t}{}}{\sqrt d}-\sqrt{2\delta}\Big|+\frac{C\hat\mu_t^3}{\sqrt d} \Big(\pnorm{\y\circ\u^t}{} + \pnorm{\y^{\circ2}\circ(\u^t)^{\circ3}}{}\Big) + \frac{C\hat\mu_t^5}{\sqrt d}\pnorm{\y^{\circ3}\circ(\u^t)^{\circ5}}{}.
\end{align}
We first claim that
\begin{align}\label{eq:claim_delta}
\Big|\frac{\pnorm{(\y-1)\circ\u^t}{}}{\sqrt d}-\sqrt{2\delta}\Big|
&\leq C_\delta\Big(\rho_t^2+\frac{\log^{12}n}{\sqrt n}\Big).
\end{align}
To see this, recall that $\u^t=\tilde\u^t+\Delta_{U,t}$. Since $\pnorm{\y-1}{\infty}\leq C\log n$, we have
\begin{align}\label{eq:claim_comp_1}
\Big|\frac{\pnorm{(\y-1)\circ\u^t}{}}{\sqrt d}
-\frac{\pnorm{(\y-1)\circ\tilde\u^t}{}}{\sqrt d}\Big|
&\leq\frac{\pnorm{\y-1}{\infty}\pnorm{\Delta_{U,t}}{}}{\sqrt d}
\leq C_\delta\frac{\log^{12}n}{\sqrt n}.
\end{align}
On the other hand, using $\tilde\u^t=\rho_t\bbeta^\star+\sum_{\ell=0}^t\gamma_{t,\ell}\g^\ell$, we have
\begin{align}\label{eq:claim_comp_2}
\frac{\pnorm{(\y-1)\circ\tilde\u^t}{}^2}{d} =\frac{\rho_t^2}{d}\pnorm{(\y-1)\circ\bbeta^\star}{}^2
+\frac{1}{d}\Bigpnorm{(\y-1)\circ\sum_{\ell=0}^t\gamma_{t,\ell}\g^\ell}{}^2 + \frac{2\rho_t}{d}\Big\langle(\y-1)^{\circ2}\circ\bbeta^\star,
\sum_{\ell=0}^t\gamma_{t,\ell}\g^\ell\Big\rangle.
\end{align}
The first term is $O_\delta(\rho_t^2)$ by Gaussian moment concentration. For the second term, Lemma \ref{lem:near_isometry} and \eqref{eq:gamma_norm_loo} give
\begin{align}\label{eq:fixed_para_concentration}
\frac{1}{d}\Bigpnorm{(\y-1)\circ\sum_{\ell=0}^t\gamma_{t,\ell}\g^\ell}{}^2
&=2\delta+O_\delta\Big(\frac{\log^{11}n}{\sqrt n}\Big).
\end{align}
Lastly, conditional on $\bbeta^\star=\g^{-1}$, the variables $d^{-1}\langle(\y-1)^{\circ2}\circ\bbeta^\star,\g^\ell\rangle$, $0\leq\ell\leq t$, are independent centered Gaussians with variance at most $C_\delta/n$ on $\cE_{\rm mom}$ defined in (\ref{eq:gaussian_moment_event}). A Gaussian norm bound and Cauchy--Schwarz therefore give, uniformly over the coefficient vector,
\begin{align*}
\Big|\frac{1}{d}\Big\langle(\y-1)^{\circ2}\circ\bbeta^\star,
\sum_{\ell=0}^t\gamma_{t,\ell}\g^\ell\Big\rangle\Big|
&\leq C_\delta\pnorm{\gamma_{t,[0:t]}}{}\sqrt{\frac{t+\log n}{n}}.
\end{align*}
Combining these estimates and taking square roots proves \eqref{eq:claim_delta}, since $\pnorm{\gamma_{t,[0:t]}}{}\leq2$. The Phase-I bound on $|\rho_t|$ then gives
\begin{align}\label{eq:claim_consequence}
\Big|\frac{\pnorm{(\y-1)\circ\u^t}{}}{\sqrt d}-\sqrt{2\delta}\Big|
&\leq C_\delta\frac{\log^{20}n}{\sqrt n},
\end{align}
concluding the bound for the main term in \eqref{eq:linearization_main}.

Next we bound the remainder terms in \eqref{eq:linearization_main}. We claim that
\begin{align}\label{eq:phase1_weighted_moments}
\frac{1}{\sqrt d}\Big(
\pnorm{\y\circ\u^t}{}+
\pnorm{\y^{\circ2}\circ(\u^t)^{\circ3}}{}+
\pnorm{\y^{\circ3}\circ(\u^t)^{\circ5}}{}\Big)
&\leq C\sqrt\delta.
\end{align}
By \eqref{eq:gamma_norm} and \eqref{eq:uv_closeness_phase_1} at time $t-1$,
\begin{align}\label{eq:phase1_gamma_universal_radius}
\pnorm{\gamma_{t,[-1:t]}}{}
&\leq\Big(1+C\sqrt{\frac{t+\log n}{d}}\Big)\sqrt{1+\rho_t^2}
+C\frac{\pnorm{\Delta_{V,t}}{}}{\sqrt d}
\leq2.
\end{align}
For $c\in\mathcal K_t\coloneqq\{c\in\R^{t+2}:\pnorm{c}{}\leq2\}$, write $U_i(c)=c_{-1}g_i^{-1}+\sum_{\ell=0}^tc_\ell g_i^\ell$. For each deterministic $c\in\mathcal K_t$ and $(q,p)\in\{(0,2),(1,1),(2,3),(3,5)\}$, Gaussian moments and independence across rows give
\begin{align*}
\E\big[y_i^{2q}U_i(c)^{2p}\big]&\leq C,
\qquad
\Var\Big(\frac1n\sum_{i=1}^ny_i^{2q}U_i(c)^{2p}\Big)\leq\frac Cn.
\end{align*}
Since $y_i=(g_i^{-1})^2$, the empirical average is a polynomial of degree $4q+2p\leq22$ in the underlying Gaussian variables. Gaussian hypercontractivity, followed by Markov's inequality, therefore gives
\begin{align*}
\P\Big(\frac1n\sum_{i=1}^ny_i^{2q}U_i(c)^{2p}>C\Big)
\leq2\exp\Big\{-\frac{n^{1/22}}{C}\Big\},
\end{align*}
for $n$ greater than a sufficiently large constant; see \cite[Theorem 2.1]{diakonikolas2010bounding}.

Take a $1/2$-net $\mathcal N$ of the unit sphere in $\R^{t+2}$, with $|\mathcal N|\leq5^{t+2}$. A union bound over the net and the four pairs $(q,p)$ shows that the preceding empirical averages are all bounded by $C$, except on an event of probability at most
\begin{align*}
C\exp\Big\{C(t+2)-\frac{n^{1/22}}{C}\Big\}\leq Cn^{-12},
\end{align*}
where we used $t\leq\log^2n$ and took $n$ sufficiently large.

To extend the bound from the net, observe that, for each fixed sample, $c\mapsto\Big(\frac1n\sum_{i=1}^ny_i^{2q}|U_i(c)|^{2p}\Big)^{1/(2p)}$ is a seminorm. The triangle inequality for the $1/2$-net, followed by homogeneity, gives
\begin{align*}
\sup_{c\in\mathcal K_t}\Big(\frac1n\sum_{i=1}^ny_i^{2q}|U_i(c)|^{2p}\Big)^{1/(2p)}
&\leq4\max_{c\in\mathcal N}\Big(\frac1n\sum_{i=1}^ny_i^{2q}|U_i(c)|^{2p}\Big)^{1/(2p)}.
\end{align*}
Raising to the power $2p$ and using $n/d=\delta$, we conclude that
\begin{align*}
\sup_{c\in\mathcal K_t}\frac1d\sum_{i=1}^ny_i^{2q}U_i(c)^{2p}
\leq C\delta,
\qquad (q,p)\in\{(0,2),(1,1),(2,3),(3,5)\},
\end{align*}
with probability $1-O(n^{-12})$. Substituting $c=\gamma_{t,[-1:t]}$ proves \eqref{eq:phase1_weighted_moments} with $\u^t$ replaced by $\tilde\u^t$. 

To pass from $\tilde\u^t$ to $\u^t$, Gaussian row-norm concentration gives
\begin{align*}
\pnorm{\tilde\u^t}{\infty}
&\leq2\max_{i\leq n}\pnorm{(g_i^{-1},g_i^0,\ldots,g_i^t)}{}
\leq C\sqrt{t+\log n}\leq C\log n.
\end{align*}
Writing $\u^t=\tilde\u^t+\Delta_{U,t}$ and expanding $(\tilde\u^t+\Delta_{U,t})^{\circ p}$ for $p=1,3,5$, the bounds $\pnorm{\Delta_{U,t}}{\infty}\leq\pnorm{\Delta_{U,t}}{}\leq C_\delta\log^{11}n$ and $\pnorm\y{\infty}\leq C\log n$ show that replacing $\tilde\u^t$ by $\u^t$ changes each of the three terms on the left of \eqref{eq:phase1_weighted_moments} by at most
\begin{align}\label{eq:phase1_weighted_moment_replacement}
C_\delta\sum_{r=1}^5\frac{\log^{10r+8}n}{\sqrt d} \leq C_\delta\frac{\log^{58}n}{\sqrt d} \leq 1,
\end{align}
where the last inequality holds for large enough $n$. Since $\delta>1/2$, this proves \eqref{eq:phase1_weighted_moments}.

Plugging these bounds and \eqref{eq:claim_consequence} into \eqref{eq:linearization_main}, we obtain
\begin{align*}
\Big|\frac{\pnorm{h_t^\star(\u^t,\y)}{}}{\sqrt d}-\sqrt{2\delta}\hat\mu_t\Big| \leq C_\delta\hat\mu_t\Big(\frac{\log^{20}n}{\sqrt n}+\hat\mu_t^2+\hat\mu_t^4\Big) \leq C_\delta\hat\mu_t\frac{\log^{20}n}{\sqrt n},
\end{align*}
where the last step follows from \eqref{eq:phase_1_mu_est}. This proves \eqref{eq:phase_1_norm}.

The claim for $h_t'(u,y)$ follows by differentiation, which gives
\begin{align*}
h_t'(u,y) =(1+\pi_t)\frac{\partial_u h_t^\star(u,y)}{\sqrt{2\delta}\hat\mu_t} = \frac{1+\pi_t}{\sqrt{2\delta}} \Big\{(1+\hat\mu_t^2)y\big(1-\tanh^2(\hat\mu_tu\sqrt y)\big)-1\Big\},
\end{align*}
where $|\pi_t|\leq C_\delta\log^{20}n/\sqrt n$ by \eqref{eq:phase_1_norm}. This proves \eqref{eq:ht_first}. The same formula and its derivative, together with $\pnorm\y{\infty}\leq C\log n$, give $\pnorm{h_t'}{\infty}\leq C_\delta\log n$ and $\pnorm{h_t''}{\infty}\leq C_\delta\hat\mu_t\log^{3/2}n\leq C_\delta$, proving \eqref{eq:phase_1_assump_U_1}.

\medskip
\noindent\textbf{Proof of \eqref{eq:phase_1_HF}.}
Since $f_{t+1}(x)=x$, we have $F_{t+1}(\rho,\sigma;\btheta^\star)=1$ for every $\rho,\sigma$. To bound $H_t$, write
\begin{align*}
\bar\u^t=\rho_t\bbeta^\star+\pnorm{\gamma_{t,[0:t]}}{}\tilde\g,
\qquad \tilde\g\sim\N(0,\I_n).
\end{align*}
The expectation defining $H_t$ is over $\tilde\g$. By \eqref{eq:ht_first} and $\tanh^2x\leq x^2$,
\begin{align*}
\Big|h_t'(\bar u_i^t,y_i)-\frac{y_i-1}{\sqrt{2\delta}}\Big|
&\leq C_\delta\Big\{(|\pi_t|+\hat\mu_t^2)(1+y_i)
+\hat\mu_t^2y_i^2(\bar u_i^t)^2\Big\}.
\end{align*}
Since $|\rho_t|\leq1$ and $\pnorm{\gamma_{t,[0:t]}}{}\leq2$, Gaussian moments in $\tilde\g$, together with the bounds $n^{-1}\sum_{i=1}^n(g_i^{-1})^{2k}\leq C$ for $1\leq k\leq6$, which hold with probability $1-O(n^{-12})$ by Gaussian polynomial concentration, give
\begin{align*}
\frac{1}{d}\E_{\tilde\g}\Bigpnorm{h_t'(\bar\u^t,\y)-\frac{\y-1}{\sqrt{2\delta}}}{}^2
&\leq C_\delta\big(|\pi_t|^2+\hat\mu_t^4\big)
\leq C_\delta\frac{\log^{40}n}{n}.
\end{align*}
Moreover,
\begin{align*}
\frac{1}{d}\Bigpnorm{\frac{\y-1}{\sqrt{2\delta}}}{}^2
&=\frac{1}{2n}\sum_{i=1}^n(y_i-1)^2
=1+O\Big(\frac{\log^4n}{\sqrt n}\Big).
\end{align*}
Expanding the squared norm in the definition of $H_t$ and applying Cauchy--Schwarz, we therefore obtain
\begin{align*}
\Big|H_t\Big(\rho_t,\pnorm{\gamma_{t,[0:t]}}{};\bbeta^\star\Big)-1\Big| \leq C_\delta\Big(\frac{\log^4n}{\sqrt n}
+\frac{\log^{20}n}{\sqrt n}+\frac{\log^{40}n}{n}\Big) \leq C_\delta\frac{\log^{20}n}{\sqrt n},
\end{align*}
proving \eqref{eq:phase_1_HF}.
\end{proof}

\begin{proof}[Proof of Lemma \ref{lem:phase_1_simp}]
We first simplify the bounds in $\cH_t^U$, assuming that $t\geq1$, $\cH_t^U$ holds, and the refined assumptions hold up to $\cA_{I,t-1}^V$. We first note that
\begin{subequations}
\begin{gather}
\pnorm{f_t'}{\infty}=1,\qquad \pnorm{f_t''}{\infty}=0,\label{eq:temp_phase1_1}\\
\sqrt{F_t\big(\rho_t,\pnorm{\lambda_{t-1,[0:t-1]}}{};\btheta^\star\big)}=1,\label{eq:temp_phase1_2}\\
\sqrt{H_{t-1}\big(\rho_{t-1},\pnorm{\gamma_{t-1,[0:t-1]}}{};\bbeta^\star\big)}
=1+O_\delta\Big(\frac{\log^{20}n}{\sqrt n}\Big),\label{eq:temp_phase1_5}\\
\sum_{k=0}^{t-1}|\beta_{t-1}^k|
\leq C_\delta\frac{t\log^9n}{\sqrt d},\label{eq:temp_phase1_3}\\
\bar\gamma_t\vee\bar\lambda_{t-1}\leq C_\delta\sqrt d.\label{eq:temp_phase1_4}
\end{gather}
\end{subequations}
Here \eqref{eq:temp_phase1_1}, \eqref{eq:temp_phase1_2} and \eqref{eq:temp_phase1_5} follow from $f_k=\operatorname{id}$ and Lemma~\ref{lem:phase_1_assump_conseq} applied at time $t-1$. The bound \eqref{eq:temp_phase1_3} is \eqref{eq:phase_1_assump_V_5} at time $t-1$. Finally, \eqref{eq:gamma_norm} and \eqref{eq:uv_closeness_phase_1} at time $t-1$ give the bound on $\bar\gamma_t$ in \eqref{eq:temp_phase1_4}, while \eqref{eq:global_current_lambda_control} gives the bound on $\bar\lambda_{t-1}$.

To see \eqref{eq:induce_U_1}, recall from \eqref{eq:induction_Delta_U} and the five bounds
above that
\begin{align*}
\pnorm{\Delta_{U,t}'}{}
&\leq\big(1+\cE_t^f\big)\Big(\pnorm{\Delta_{V,t}'}{}+\sum_{k=0}^{t-1}|\beta_{t-1}^k|\pnorm{\Delta_{V,k}}{}\Big)
+C_\delta t\sqrt{\log n}.
\end{align*}
Since $f_t''=0$, the definition \eqref{eq:E_multiplier} gives
\begin{align}
\label{eq:phase_1_E_f}
\cE_t^f=C\sqrt{t+1}\Big(\frac{\log n}{d}\Big)^{1/4}
\leq C_\delta\frac{\log^2n}{n^{1/4}},
\end{align}
using $t\leq\log^2n$. Moreover, \eqref{eq:temp_phase1_3} and the bounds \eqref{eq:uv_closeness_phase_1} on $\pnorm{\Delta_{V,k}}{}$ for $k<t$ give
\begin{align*}
\sum_{k=0}^{t-1}|\beta_{t-1}^k|\pnorm{\Delta_{V,k}}{}
\leq C_\delta\frac{t^2\log^{18}n}{\sqrt d}.
\end{align*}
Substituting these bounds, we obtain
\begin{align*}
\pnorm{\Delta_{U,t}'}{}
&\leq\Big(1+C_\delta\frac{\log^2n}{n^{1/4}}\Big)
\Big(\pnorm{\Delta_{V,t}'}{}+C_\delta\frac{t^2\log^{18}n}{\sqrt d}\Big)
+C_\delta t\sqrt{\log n}\\
&\leq\Big(1+C_\delta\frac{\log^2n}{n^{1/4}}\Big)\pnorm{\Delta_{V,t}'}{}
+C_\delta t\sqrt{\log n},
\end{align*}
where the last inequality follows from $t\leq\log^2n$ and large enough $n$. This proves \eqref{eq:induce_U_1}.

Next, \eqref{eq:induction_alpha_last}, $f_t''=0$, and $\pnorm{\lambda_{t-1,[0:t-1]}}{}=1$ give
\begin{align*}
|\alpha_{t-1}^{t-1}| \leq\frac{2}{\sqrt d}\Big(\pnorm{\Delta_{V,t}'}{}+\max_{0\leq\ell\leq t-1}\pnorm{\Delta_{V,\ell}}{}\sum_{k=0}^{t-1}|\beta_{t-1}^k|\Big) \leq C_\delta\Big(\frac{\pnorm{\Delta_{V,t}'}{}}{\sqrt d}+\frac{t^2\log^{18}n}{d}\Big),
\end{align*}
which proves \eqref{eq:induce_U_2}.

To prove \eqref{eq:induce_U_3}, first note that, for $t\geq2$, \eqref{eq:induction_alpha_middle_last} gives
\begin{align*}
|\alpha_{t-1}^{t-2}|\leq\pnorm{f_t'}{\infty}\pnorm{f_{t-1}'}{\infty}|\beta_{t-1}^{t-1}|=|\beta_{t-1}^{t-1}|.
\end{align*}
For the remaining coefficients, we also need to bound $\cE_s^h$ for $0\leq s<t$. Equation \eqref{eq:ht_first} and the refined assumptions give
\begin{align*}
\pnorm{h_s'}{\infty}\leq\frac{C\log n}{\sqrt\delta},
\qquad
\pnorm{h_s''}{\infty}\leq C_\delta,
\qquad
\pnorm{\Delta_{U,s}}{}+\pnorm{\hat\u^s-\tilde\u^s}{}\leq C_\delta s\log^9n.
\end{align*}
Substituting into \eqref{eq:E_multiplier}, and using $s\leq\log^2n$, yields
\begin{align}\label{eq:phase1_E_bounds}
\cE_s^h
\leq C_\delta\Big(\frac{\log^4n}{n^{1/4}}+\frac{\log^{13}n}{\sqrt n}\Big)
\leq C_\delta\frac{\log^4n}{n^{1/4}}.
\end{align}
Applying \eqref{eq:phase1_E_bounds}, along with the bounds on $H_s,F_s$ in \eqref{eq:phase_1_HF} to \eqref{eq:induction_alpha_middle} gives, for $0\leq k\leq t-3$,
\begin{align*}
|\alpha_{t-1}^k| \leq|\alpha_{t-2}^{k+1}|\big(1+\cE_t^f\big) \Big(1+C_\delta\frac{\log^{20}n}{\sqrt n}+\cE_{t-1}^h\Big) \leq \Big(1+C_\delta\frac{\log^4n}{n^{1/4}}\Big)|\alpha_{t-2}^{k+1}|,
\end{align*}
proving \eqref{eq:induce_U_3}. Here the last inequality uses \eqref{eq:phase_1_E_f} and the bound $C_\delta\log^{20}n/\sqrt n\leq C_\delta\log^4n/n^{1/4}$ for sufficiently large $n$.

The $V$ bounds follow similarly. Suppose that $t\geq0$, $\cH_t^V$ holds, and the refined assumptions hold up to $\cA_{I,t}^U$. Lemma~\ref{lem:phase_1_assump_conseq} now applies at time $t$, so the derivative bounds on $h_s$ above and the bound \eqref{eq:phase1_E_bounds} on $\cE_s^h$ also hold for $s=t$. By \eqref{eq:phase_1_assump_U_5}, \eqref{eq:phase_1_assump_U_1} and \eqref{eq:phase_1_assump_U_2}, the sum over $k$ in \eqref{eq:induction_Delta_V} satisfies
\begin{align*}
\sum_{k=0}^{t-1}|\alpha_{t-1}^k|\pnorm{h_k'}{\infty}\pnorm{\Delta_{U,k}}{}
\leq C_\delta\frac{t^2\log^{19}n}{\sqrt d}.
\end{align*}
Using this bound, $\bar\lambda_t+\bar\gamma_t\leq C_\delta\sqrt d$, and \eqref{eq:phase_1_HF} in \eqref{eq:induction_Delta_V} gives
\begin{align*}
\pnorm{\Delta_{V,t+1}'}{}
&\leq\Big(1+C_\delta\frac{\log^4n}{n^{1/4}}\Big)
\Big(\pnorm{\Delta_{U,t}'}{}+C_\delta\frac{t^2\log^{19}n}{\sqrt d}\Big)
+C_\delta(t+1)\log^{5/2}n\\
&\leq\Big(1+C_\delta\frac{\log^4n}{n^{1/4}}\Big)\pnorm{\Delta_{U,t}'}{}
+C_\delta(t+1)\log^{5/2}n,
\end{align*}
where the last step again uses $t\leq\log^2n$ and $n$ large enough. This proves \eqref{eq:induce_V_1}.

By \eqref{eq:gamma_norm_loo} and \eqref{eq:uv_closeness_phase_1} at time $t-1$, $\pnorm{\gamma_{t,[0:t]}}{}\geq1/2$. Substituting this lower bound, together with \eqref{eq:phase_1_assump_U_5}, \eqref{eq:phase_1_assump_U_1} and \eqref{eq:phase_1_assump_U_2}, into \eqref{eq:induction_beta_last} gives
\begin{align*}
|\beta_t^t| \leq C_\delta\Big(\frac{\log n\pnorm{\Delta_{U,t}'}{}}{\sqrt d}+\frac{t^2\log^{21}n}{d}\Big),
\end{align*}
proving \eqref{eq:induce_V_2}. For $t\geq1$, \eqref{eq:ht_first} and $\pnorm{\y}{\infty}\leq C\log n$ give $\pnorm{h_s'}{\infty}\leq C\log n/\sqrt\delta$ for $s\in\{t-1,t\}$, with a universal constant $C$. The factor $\delta$ therefore cancels, and
\begin{align*}
|\beta_t^{t-1}|
\leq\delta\pnorm{h_t'}{\infty}\pnorm{h_{t-1}'}{\infty}|\alpha_{t-1}^{t-1}|
\leq C\log^2n|\alpha_{t-1}^{t-1}|.
\end{align*}
For $0\leq k\leq t-2$, \eqref{eq:induction_beta_middle} and the same $F$, $H$, and multiplier estimates give
\begin{align*}
|\beta_t^k|\leq\Big(1+C_\delta\frac{\log^4n}{n^{1/4}}\Big)|\beta_{t-1}^{k+1}|.
\end{align*}
This proves \eqref{eq:induce_V_3}. 
\end{proof}
\subsection{Proof of Lemma \ref{lem:linearization}}\label{subsec:h_linearization_phase1}
\begin{proof}
First note that
\begin{align*}
h_t(u,y)
&=\frac{h_t^\star(u,y)}{\sqrt{2\delta}\hat\mu_t}
\frac{\sqrt{2\delta}\hat\mu_t}{\pnorm{h_t^\star(\u^t,\y)}{}/\sqrt d}=(1+\pi_t)\frac{h_t^\star(u,y)}{\sqrt{2\delta}\hat\mu_t},
\end{align*}
where $|\pi_t|\leq C_\delta\log^{20}n/\sqrt n$ by \eqref{eq:ht_first}.  The Taylor expansion \eqref{eq:ht_taylor} gives
\begin{align*}
\frac{h_t^\star(u,y)}{\hat\mu_t}
&=(y-1)u+\hat\mu_t^2yu
-\frac{1}{3}(1+\hat\mu_t^2)\hat\mu_t^2u^3y^2
+O\big((1+\hat\mu_t^2)\hat\mu_t^4u^5y^3\big).
\end{align*}
It follows that
\begin{align*}
\frac{h_t^\star(u,y)}{\sqrt{2\delta}\hat\mu_t}
&=\frac{(y-1)u}{\sqrt{2\delta}}+r_t^{(1)}=\frac{1}{\sqrt{2\delta}}
\Big\{(y-1)u+\hat\mu_t^2
\Big(yu-\frac{1}{3}u^3y^2\Big)\Big\}
+r_t^{(2)},
\end{align*}
where
\begin{align*}
r_t^{(1)} &= \frac{1}{\sqrt{2\delta}}\Big(\hat\mu_t^2 yu - \frac{1}{3}(1+\hat\mu_t^2)\hat\mu_t^2 u^3y^2 + O\Big((\hat\mu_t^2+1)\hat\mu_t^4u^5y^3\Big)\Big),\\
r_t^{(2)} &= \frac{1}{\sqrt{2\delta}}\Big(-\frac{1}{3}\hat\mu_t^4u^3y^2 + O\Big((\hat\mu_t^2+1)\hat\mu_t^4u^5y^3\Big)\Big).
\end{align*}
For large enough $n$, \eqref{eq:phase_1_mu_est} gives $|\hat\mu_t|\leq1$, so these expressions give \eqref{eq:taylor_error_bound}.  Using \eqref{eq:phase_1_mu_est} and the weighted moment bound \eqref{eq:phase1_weighted_moments}, we further obtain
\begin{align*}
\frac{1}{\sqrt{d}}\pnorm{\r_t^{(1)}}{}
&\leq C_\delta(|\hat\mu_t|^2+|\hat\mu_t|^4)
\leq C_\delta\frac{\log^{20}n}{\sqrt n},\\
\frac{1}{\sqrt{d}}\pnorm{\r_t^{(2)}}{}
&\leq C_\delta|\hat\mu_t|^4
\leq C_\delta\frac{\log^{40}n}{n},
\end{align*}
which proves \eqref{eq:r_vec_bound}. The same bounds hold for the remainders evaluated entrywise at $\tilde\u^t$ or $\g^t$, since the uniform weighted moment estimate used in \eqref{eq:phase1_weighted_moments} also covers these inputs.
\end{proof}

\subsection{Proof of Lemma \ref{lem:lambda_almost_orthogonal}}\label{subsec:proof_orthogonality_2}

\begin{proof}
By definition,
\begin{align*}
\pnorm{\gamma_{t+1,[-1:t]}}{} =\frac{1}{\sqrt d}\Bigpnorm{P_{\Q^{[-1:t]}}\z^{t+1}}{} = \frac{1}{\sqrt d}\Bigpnorm{P_{\Q^{[-1:t]}}\v^{t+1}}{} \leq \frac{1}{\sqrt d}\Bigpnorm{P_{\Q^{[-1:t]}}\tilde\v^{t+1}}{} + \frac{1}{\sqrt d}\pnorm{\Delta_{V,t+1}}{}.
\end{align*}
Recalling the definition of $\tilde\v^{t+1}$ in (\ref{eq:tilde_u_intro}), we have
\begin{align*}
\frac{1}{\sqrt{d}}\Bigpnorm{P_{\Q^{[-1:t]}}\tilde\v^{t+1}}{} &= \frac{1}{\sqrt{d}}\Bigpnorm{P_{\Q^{[-1:t]}} \Big(\rho_{t+1} \btheta^\star + \sum_{\ell=0}^t \lambda_{t,\ell}\s^\ell\Big)}{}\\
&\leq |\rho_{t+1}| + \frac{1}{\sqrt{d}}\Bigpnorm{\sum_{\ell=0}^{t-1}\lambda_{t,\ell}\s^\ell}{} + |\lambda_{t,t}|\frac{1}{\sqrt{d}}\Bigpnorm{P_{\Q^{[-1:t]}} \s^t}{}.
\end{align*}
By \eqref{eq:uv_dependence}, $\s^t$ is independent of $\Q^{[-1:t]}$. Conditional on this matrix, the projection has rank at most $t+2$, so Gaussian concentration gives $\|P_{\Q^{[-1:t]}}\s^t\|/\sqrt d\leq C\sqrt{(t+\log n)/d}$ with failure probability $O(n^{-12})$. The $d$-dimensional near-isometry estimate \eqref{eq:claim_near_isometry}, which is uniform in the coefficients, similarly gives
\begin{align}\label{eq:concentration_lambda}
\frac{1}{\sqrt{d}}\Bigpnorm{\sum_{\ell=0}^{t-1}\lambda_{t,\ell}\s^\ell}{}
= \Big(1+O\Big(\sqrt{\frac{t+\log n}{d}}\Big)\Big)
\pnorm{\lambda_{t,[0:t-1]}}{}.
\end{align}
Lastly, using $|\lambda_{t,t}| \leq \pnorm{\w^t}{}/\sqrt{d} = 1$ (see Lemma \ref{lem:uv_norm}), we conclude that
\begin{align}\label{eq:lambda_2_gamma}
\pnorm{\gamma_{t+1,[-1:t]}}{}
\leq \Big(1+O\Big(\sqrt{\frac{t+\log n}{d}}\Big)\Big)
\pnorm{\lambda_{t,[0:t-1]}}{}
+\frac{\pnorm{\Delta_{V,t+1}}{}}{\sqrt d}+|\rho_{t+1}|
+O\Big(\sqrt{\frac{t+\log n}{d}}\Big).
\end{align}

Similarly, with $\tilde\u^t$ defined in \eqref{eq:tilde_u_intro},
\begin{align}\label{eq:lambda_norm}
\pnorm{\lambda_{t,[0:t-1]}}{} = \frac{1}{\sqrt{d}}\Bigpnorm{P_{\bR^{[0:t-1]}} \w^t}{} \leq  \underbrace{\frac{1}{\sqrt{d}}\Bigpnorm{h_t(\u^t,\y) - h_t(\tilde\u^t,\y)}{}}_{(I)} + \underbrace{\frac{1}{\sqrt{d}}\Bigpnorm{P_{\bR^{[0:t-1]}}h_t(\tilde\u^t,\y)}{}}_{(I\!I)}.
\end{align}
For $(I)$, using the bound $|h_t'(u,y)| \leq C_\delta(1+|y|)$ deduced from (\ref{eq:ht_first}) and $\pnorm{\y}{\infty} = O(\log n)$ with probability $1-O(n^{-12})$, we have
\begin{align}\label{eq:term_I_main}
(I)\leq \frac{C_\delta}{\sqrt{d}}\Bigpnorm{(|\y|+1)\circ (\u^t - \tilde\u^t)}{}\leq C_\delta(\pnorm{\y}{\infty}+1)\frac{1}{\sqrt{d}}\pnorm{\Delta_{U,t}}{}\leq C_\delta\frac{\log n}{\sqrt{d}}\pnorm{\Delta_{U,t}}{}.
\end{align}
For $(I\!I)$, we apply (\ref{eq:h_first_order}) to obtain
\begin{align}\label{eq:term_II_main}
\notag(I\!I) &\leq |\pi_t|\frac{1}{\sqrt{d}}\Bigpnorm{\frac{(\y-1)\circ \tilde\u^t}{\sqrt{2\delta}} + \r_t^{(1)}(\tilde\u^t,\y)}{} + \frac{1}{\sqrt{d}}\pnorm{\r_t^{(1)}(\tilde\u^t,\y)}{} + \frac{1}{\sqrt{2\delta}}\frac{1}{\sqrt{d}}\Bigpnorm{P_{\bR^{[0:t-1]}}(\y-1)\circ\tilde \u^t}{}\\
&\leq \underbrace{(1+|\pi_t|)\frac{1}{\sqrt{d}}\pnorm{\r_t^{(1)}(\tilde\u^t,\y)}{}}_{(I\!I_1)} + \underbrace{\frac{|\pi_t|}{\sqrt{2\delta}}\frac{1}{\sqrt{d}}\Bigpnorm{(\y-1)\circ \tilde\u^t}{}}_{(I\!I_2)} + \underbrace{\frac{1}{\sqrt{2\delta}}\frac{1}{\sqrt{d}}\Bigpnorm{P_{\bR^{[0:t-1]}}(\y-1)\circ\tilde \u^t}{}}_{(I\!I_3)},
\end{align}
where the remainder is evaluated entrywise at $\widetilde{\bm u}^t$, as justified after \eqref{eq:r_vec_bound}.  Using $|\pi_t|\leq C_\delta\log^{20}n/\sqrt n$ and the bound on $\r_t^{(1)}(\tilde\u^t,\y)$, we obtain
\begin{align}\label{eq:term_II_3}
(I\!I_1)\leq C_\delta\frac{\log^{20}n}{\sqrt n}.
\end{align}
For $(I\!I_2)$, Lemma \ref{lem:uv_norm} implies that
\begin{align}\label{eq:term_II_2}
\frac{1}{\sqrt{d}}\pnorm{(\y-1)\circ \tilde\u^t}{}
\leq C_\delta \log n.
\end{align}
Consequently,
\begin{align*}
(I\!I_1)+(I\!I_2)
&\leq
C_\delta\Big(\frac{\log^{20}n}{\sqrt n}+\frac{\log^{21}n}{\sqrt n}\Big)
\leq
C_\delta\frac{\log^{10}n}{n^{1/4}}.
\end{align*}
For $(I\!I_3)$, note that
\begin{align*}
&\frac{1}{\sqrt{d}}\Bigpnorm{P_{\bR^{[0:t-1]}} (\y-1)\circ\tilde\u^t}{} = \frac{1}{\sqrt{d}}\Bigpnorm{P_{\bR^{[0:t-1]}}(\y-1)\circ\Big(\gamma_{t,-1}\bbeta^\star + \sum_{\ell=0}^t \gamma_{t,\ell}\g^\ell\Big)}{}\\
&\leq |\rho_t|\frac{1}{\sqrt{d}}\pnorm{(\y-1)\circ \bbeta^\star}{} + \frac{1}{\sqrt{d}}\Bigpnorm{\sum_{\ell=0}^{t-1}\gamma_{t,\ell}(\y-1)\circ\g^\ell}{} + |\gamma_{t,t}|\frac{1}{\sqrt{d}}\Bigpnorm{P_{\bR^{[0:t-1]}}(\y-1)\circ\g^t}{}.
\end{align*}
The first term is at most $C_\delta|\rho_t|$.  Lemma~\ref{lem:near_isometry} gives
\begin{align*}
\frac{1}{\sqrt{d}}\Bigpnorm{\sum_{\ell=0}^{t-1}\gamma_{t,\ell}(\y-1) \circ \g^\ell}{}
= \Big(\sqrt{2\delta}+O_\delta\Big(\sqrt{\frac{t\log^3n}{n}}\Big)\Big)
\pnorm{\gamma_{t,[0:t-1]}}{}.
\end{align*}
Lastly,
\begin{align*}
\|P_{\bR^{[0:t-1]}}((\y-1)\circ\g^t)\|
=\|\{P_{\bR^{[0:t-1]}}\diag(\y-1)\}\g^t\|.
\end{align*}
By \eqref{eq:uv_dependence}, $\g^t$ is independent of $(\y,\bR^{[0:t-1]})$. Conditional on these variables, the squared norm is a Gaussian quadratic form with matrix $\A=\diag(\y-1)P_{\bR^{[0:t-1]}}\diag(\y-1)$. This matrix has rank at most $t$ and operator norm at most $C\log^2n$. For $t\geq1$, the Hanson--Wright inequality therefore gives
\begin{align*}
\frac{1}{\sqrt d}\pnorm{P_{\bR^{[0:t-1]}}(\y-1)\circ\g^t}{}
\leq C_\delta\sqrt{\frac{t\log^3n}{d}}
\end{align*}
with probability $1-O(n^{-12})$; at $t=0$, the projection is zero. Also $|\gamma_{t,t}|\leq C_\delta$ by Lemma~\ref{lem:uv_norm}. Therefore
\begin{align}
\frac{1}{\sqrt{d}}\Bigpnorm{P_{\bR^{[0:t-1]}}(\y-1)\circ\tilde\u^t}{}
\leq \Big(\sqrt{2\delta}+O_\delta\Big(\sqrt{\frac{t\log^3n}{n}}\Big)\Big)
\pnorm{\gamma_{t,[0:t-1]}}{}
+O_\delta\Big(\frac{\log^{10}n}{n^{1/4}}+\sqrt{\frac{t\log^3n}{d}}\Big).
\end{align}
Substitution into \eqref{eq:term_II_main}, using $|\pi_t|\leq C_\delta\log^{20}n/\sqrt n$, gives
\begin{align*}
(I\!I) \leq
\Big(1+O_\delta\Big(\sqrt{\frac{t\log^3n}{n}}\Big)\Big)
\pnorm{\gamma_{t,[0:t-1]}}{}
+O_\delta\Big(\frac{\log^{10}n}{n^{1/4}}+\sqrt{\frac{t\log^3n}{d}}\Big).
\end{align*}
Substitution of this bound and \eqref{eq:term_I_main} into \eqref{eq:lambda_norm} gives
\begin{align}\label{eq:gamma_2_lambda}
\notag\pnorm{\lambda_{t,[0:t-1]}}{}
&\leq \Big(1+O_\delta\Big(\sqrt{\frac{t\log^3n}{n}}\Big)\Big)
\pnorm{\gamma_{t,[0:t-1]}}{}
+O_\delta\Big(\frac{\log^{10}n}{n^{1/4}}+\sqrt{\frac{t\log^3n}{d}}
+\frac{\log n}{\sqrt d}\pnorm{\Delta_{U,t}}{}\Big)\\
&\leq \Big(1+O_\delta\Big(\sqrt{\frac{t\log^3n}{n}}\Big)\Big)
\pnorm{\gamma_{t,[-1:t-1]}}{}
+O_\delta\Big(\frac{\log^{10}n}{n^{1/4}}+\sqrt{\frac{t\log^3n}{d}}
+\frac{\log n}{\sqrt d}\pnorm{\Delta_{U,t}}{}\Big).
\end{align}

Combining \eqref{eq:lambda_2_gamma} and \eqref{eq:gamma_2_lambda}, we arrive at the recursion
\begin{align*}
\pnorm{\gamma_{t+1,[-1:t]}}{}
&\leq \Big(1+O_\delta\Big(\sqrt{\frac{t\log^3n}{n}}
+\sqrt{\frac{t+\log n}{d}}\Big)\Big)
\pnorm{\gamma_{t,[-1:t-1]}}{}\\
&\quad+C_\delta|\rho_{t+1}|
+C_\delta\Big(\frac{\log^{10}n}{n^{1/4}}+\sqrt{\frac{t\log^3n}{d}}
+\frac{\log n}{\sqrt d}\big(\pnorm{\Delta_{V,t+1}}{}+\pnorm{\Delta_{U,t}}{}\big)\Big)\\
&\leq \Big(1+C_\delta\frac{\log^{5/2}n}{\sqrt n}\Big)
\pnorm{\gamma_{t,[-1:t-1]}}{}+C_\delta|\rho_{t+1}|+C_\delta\Big(\frac{\log^{10}n}{n^{1/4}}
+\frac{\log n}{\sqrt d}\big(\pnorm{\Delta_{V,t+1}}{}+\pnorm{\Delta_{U,t}}{}\big)\Big)\\
&\leq \Big(1+C_\delta\frac{\log^{5/2}n}{\sqrt n}\Big)
\pnorm{\gamma_{t,[-1:t-1]}}{}+C_\delta\Big(\frac{\log^{10}n}{n^{1/4}}
+\frac{\log n}{\sqrt d}\big(\pnorm{\Delta_{V,t+1}}{}+\pnorm{\Delta_{U,t}}{}\big)\Big).
\end{align*}
initialized at $|\gamma_{0,-1}|\leq C_\delta\sqrt{\log n/d}$ with probability $1-O(n^{-12})$. Here, the last line uses $|\rho_s|\leq n^{-1/4}\log^{10}n$ which holds throughout Phase I. Iteration therefore yields
\begin{align*}
\pnorm{\lambda_{t,[0:t-1]}}{}
\vee \pnorm{\gamma_{t,[-1:t-1]}}{}
\leq C_\delta \log^2n\Big[\frac{\log^{10}n}{n^{1/4}}
+\max_{0\leq\tau\leq t}\frac{\log n}{\sqrt d}
(\pnorm{\Delta_{V,\tau}}{}+\pnorm{\Delta_{U,\tau}}{})\Big]
\leq C_\delta\frac{\log^{12}n}{n^{1/4}}.
\end{align*}
The last inequality follows from Proposition~\ref{prop:phase_1_tilde_close}.
\end{proof}

\subsection{Proof of Lemma \ref{lem:approx_orthogonal_1}}\label{subsec:proof_orthogonality_1}

\begin{proof}[Proof of Lemma \ref{lem:approx_orthogonal_1}]
By the triangle inequality,
\begin{align}\label{eq:temp1}
\Big|\frac{1}{\sqrt d}\Bigpnorm{\sum_{\ell=0}^t a_\ell\w^\ell}{}
-\pnorm a{}\Big|
&\leq
\frac{1}{\sqrt d}\Big|
\Bigpnorm{\sum_{\ell=0}^ta_\ell\w^\ell}{}
-\Bigpnorm{\sum_{\ell=0}^ta_\ell
\frac{\g^\ell\circ(\y-1)}{\sqrt{2\delta}}}{}
\Big|\notag\\
&\quad+
\Big|\frac{1}{\sqrt d}\Bigpnorm{\sum_{\ell=0}^ta_\ell
\frac{\g^\ell\circ(\y-1)}{\sqrt{2\delta}}}{}
-\pnorm a{}\Big|.
\end{align}
Lemma~\ref{lem:near_isometry} bounds the second term by
\begin{align}\label{eq:temp2}
\Big|\frac{1}{\sqrt d}\Bigpnorm{\sum_{\ell=0}^ta_\ell
\frac{\g^\ell\circ(\y-1)}{\sqrt{2\delta}}}{}
-\pnorm a{}\Big|
\leq C_\delta\Big\{\sqrt{\frac{t+\log n}{n}}+\frac{(t+\log n)\log^2n}{n}\Big\}
\pnorm{a}{}.
\end{align}

By \eqref{eq:h_first_order},
\begin{align*}
\Bigpnorm{\sum_{\ell=0}^t a_\ell
\Big(h_\ell(\g^\ell,\y)
-\frac{\g^\ell\circ(\y-1)}{\sqrt{2\delta}}\Big)}{}
&\leq \sqrt{t+1}\,\pnorm a{}
\max_{0\leq\ell\leq t}
\Big\{|\pi_\ell|
\Bigpnorm{\frac{(\y-1)\circ\g^\ell}{\sqrt{2\delta}}}{}
+(1+|\pi_\ell|)\pnorm{\r_\ell^{(1)}(\g^\ell,\y)}{}\Big\}.
\end{align*}
The bounds for $|\pi_\ell|$ and $\pnorm{\r_\ell^{(1)}(\g^\ell,\y)}{}$ in Lemma~\ref{lem:linearization}, Lemma~\ref{lem:near_isometry}, and $t\leq \log^2n$ imply
\begin{align}\label{eq:bound_temp_I}
\frac{1}{\sqrt d}
\Bigpnorm{\sum_{\ell=0}^t a_\ell
\Big(h_\ell(\g^\ell,\y)
-\frac{\g^\ell\circ(\y-1)}{\sqrt{2\delta}}\Big)}{}
\leq C_\delta\frac{\log^{21}n}{\sqrt n}\pnorm{a}{}.
\end{align}

By the mean value theorem and \eqref{eq:ht_first},
\begin{align}\label{eq:bound_2}
\Bigpnorm{\sum_{\ell=0}^ta_\ell
\{h_\ell(\u^\ell,\y)-h_\ell(\g^\ell,\y)\}}{}
&\leq\sum_{\ell=0}^t|a_\ell|
\pnorm{h_\ell(\u^\ell,\y)-h_\ell(\g^\ell,\y)}{}\notag\\
&\leq C_\delta\sqrt{t+1}\,\pnorm{a}{}
\max_{0\leq\ell\leq t}
\pnorm{(\y+1)\circ(\u^\ell-\g^\ell)}{}\notag\\
&\leq C_\delta\log n\sqrt{t+1}\,\pnorm{a}{}
\max_{0\leq\ell\leq t}\pnorm{\u^\ell-\g^\ell}{},
\end{align}
where $|h_\ell'(u,y)|\leq C_\delta(1+y)$.  Moreover,
\begin{align*}
\frac{1}{\sqrt d}\pnorm{\u^\ell-\g^\ell}{}
&\leq\frac{1}{\sqrt d}\pnorm{\tilde\u^\ell-\g^\ell}{}
+\frac{1}{\sqrt d}\pnorm{\Delta_{U,\ell}}{},\\
\frac{1}{\sqrt d}\pnorm{\tilde\u^\ell-\g^\ell}{}
&\leq
\frac{1}{\sqrt d}\Bigpnorm{\sum_{j=-1}^{\ell-1}
\gamma_{\ell,j}\g^j}{}
+\frac{1}{\sqrt d}\pnorm{\g^\ell}{}
|\gamma_{\ell,\ell}-1|.
\end{align*}
Equation \eqref{eq:claim_near_isometry} and Lemma~\ref{lem:lambda_almost_orthogonal} give
\begin{align*}
\frac{1}{\sqrt{d}}\Bigpnorm{\sum_{j=-1}^{\ell-1}\gamma_{\ell,j}\g^j}{}
\leq C_\delta\Big(1+O\Big(\sqrt{\frac{t+\log n}{n}}\Big)\Big)
\pnorm{\gamma_{\ell,[-1:\ell-1]}}{}
\leq C_\delta\frac{\log^{12}n}{n^{1/4}}.
\end{align*}
Since $\gamma_{\ell,\ell} \geq 0$ by the Gram-Schmidt construction, \eqref{eq:gamma_norm} gives
\begin{align}
\big|1-\gamma_{\ell,\ell}\big|
&\leq \pnorm{\gamma_{\ell,[-1:\ell-1]}}{}
+ \big|1 - \pnorm{\gamma_{\ell,[-1:\ell]}}{}\big|\\
&\leq C_\delta\frac{\log^{12}n}{n^{1/4}}
+O_\delta\Big(\rho_\ell^2+\sqrt{\frac{\ell+\log n}{d}}
+\frac{\pnorm{\Delta_{V,\ell}}{}}{\sqrt d}\Big)
\notag\\
&\leq C_\delta\frac{\log^{12}n}{n^{1/4}},
\label{eq:last_gamma_diff}
\end{align}
where in the last line we used $|\rho_\ell|\leq \log^{10}n n^{-1/4}$, $\ell<\tau_I$, and Proposition~\ref{prop:phase_1_tilde_close}. Combining these estimates with Proposition~\ref{prop:phase_1_tilde_close} gives, for $\ell\leq t$,
\begin{align*}
\frac{1}{\sqrt{d}}\pnorm{\u^\ell-\g^\ell}{}
\leq C_\delta\frac{\log^{12}n}{n^{1/4}}.
\end{align*}
Substitution into \eqref{eq:bound_2} gives
\begin{align}\label{eq:bound_temp_II}
\frac{1}{\sqrt d}
\Bigpnorm{\sum_{\ell=0}^ta_\ell
\{h_\ell(\u^\ell,\y)-h_\ell(\g^\ell,\y)\}}{}
\leq C_\delta\frac{\log^{14}n}{n^{1/4}}\pnorm{a}{}.
\end{align}

Since $\w^\ell=h_\ell(\u^\ell,\y)$, adding \eqref{eq:bound_temp_I} and \eqref{eq:bound_temp_II} gives, uniformly in $a$,
\begin{align*}
\frac{1}{\sqrt d}\Bigpnorm{\sum_{\ell=0}^ta_\ell
\Big(\w^\ell-\frac{\g^\ell\circ(\y-1)}{\sqrt{2\delta}}\Big)}{}
\leq C_\delta\frac{\log^{14}n}{n^{1/4}}\pnorm a{}.
\end{align*}
Taking a single nonzero coefficient proves \eqref{eq:w_approx}. The reverse triangle inequality, followed by \eqref{eq:temp2}, proves \eqref{eq:w_isometry}.
\end{proof}
\subsection{Proof of Proposition \ref{prop:phase_1_main} (Recursion)}\label{subsec:proof_recursion_phase1}
\begin{proof}
Fix $0\leq t<\tau_I$ with $t\leq\log^2n$. Applying Lemma~\ref{lem:conjugate_signal} and the linearization \eqref{eq:h_second_order}, we obtain
\begin{align}\label{eq:rhot_main}
\rho_{t+1}
&=\frac{1+\pi_t}{\sqrt{2\delta}\,d}
  \iprod{\g^{-1}}{(\y-1)\circ\u^t}
  +\frac{(1+\pi_t)\hat\mu_t^2}{\sqrt{2\delta}\,d}
  \iprod{\g^{-1}}{\y\circ\u^t-\frac{1}{3}\y^{\circ2}\circ(\u^t)^{\circ3}}
\notag\\
&\quad+\frac{1+\pi_t}{d}\iprod{\g^{-1}}{\r_t^{(2)}}
-\hat c_t\rho_t .
\end{align}
Here $\r_t^{(2)}=(r_t^{(2)}(u_i^t,y_i))_{i=1}^n$ is the remainder in \eqref{eq:h_second_order}. To identify the leading signal and fresh noise, substitute the Gaussian approximation of $\u^t$ into the first term:
\begin{align}
\frac{1}{\sqrt{2\delta}\,d}
\iprod{\g^{-1}}{(\y-1)\circ\u^t}
={}&
\frac{\rho_t}{\sqrt{2\delta}\,d}
\iprod{\g^{-1}}{(\y-1)\circ\g^{-1}}
+\varphi_t \notag\\
&+\frac{1}{\sqrt{2\delta}\,d}
\iprod{(\y-1)\circ\g^{-1}}
{\sum_{\ell=0}^{t-1}\gamma_{t,\ell}\g^\ell}
+\frac{1}{\sqrt{2\delta}\,d}
\iprod{(\y-1)\circ\g^{-1}}{\Delta_{U,t}}\notag\\
&+\frac{\gamma_{t,t}-1}{\sqrt{2\delta}\,d}
\iprod{(\y-1)\circ\g^{-1}}{\g^t},
\label{eq:phase1_signal_decomposition}
\end{align}
where
\begin{align*}
\varphi_t=\frac{1}{\sqrt{2\delta}\,d}
\iprod{(\y-1)\circ\g^{-1}}{\g^t}.
\end{align*}
Since $\y=(\g^{-1})^{\circ2}$ and $\E[G^4-G^2]=2$ for a standard Gaussian $G$, the coefficient of $\rho_t$ has mean $\sqrt{2\delta}$. Truncating at $|g_i^{-1}|\leq C\sqrt{\log n}$ and applying Bernstein's inequality gives, with failure probability $O(n^{-12})$,
\begin{align}
\big|
\frac{1}{\sqrt{2\delta}\,d}
\iprod{\g^{-1}}{(\y-1)\circ\g^{-1}}
-\sqrt{2\delta}
\big|
\leq C_\delta\sqrt{\frac{\log n}{n}}.
\label{eq:phase1_signal_coefficient}
\end{align}

We first bound the Taylor remainder and the Onsager term in \eqref{eq:rhot_main}. For large enough $n$,  \eqref{eq:r_vec_bound} gives
\begin{align}
\frac{1+|\pi_t|}{d}
\big|\iprod{\g^{-1}}{\r_t^{(2)}}\big|
\leq C_\delta\frac{\log^{40}n}{n}
\leq C_\delta\frac{\log^{33}n}{n^{3/4}}.
\label{eq:phase1_second_order_term}
\end{align}
For the Onsager term, average the derivative formula \eqref{eq:ht_first} and use $\tanh^2x\leq x^2$:
\begin{align*}
|\widehat c_t|
&\leq C_\delta\Big\{\Big|\frac{1}{d}\sum_{i=1}^n(y_i-1)\Big|
+\widehat\mu_t^2\frac{1}{d}\sum_{i=1}^n\big(y_i+y_i^2(u_i^t)^2\big)\Big\}\\
&\leq C_\delta\Big(\sqrt{\frac{\log n}{n}}+\widehat\mu_t^2\log^2n\Big)
\leq C_\delta\frac{\log^{22}n}{\sqrt n}.
\end{align*}
The second line uses Gaussian concentration and \eqref{eq:phase1_weighted_moments}. Since $|\rho_t|\leq n^{-1/4}\log^{10}n$, it follows that
\begin{align}
|\hat c_t\rho_t|
\leq C_\delta\frac{\log^{32}n}{n^{3/4}}.
\label{eq:phase1_onsager_signal}
\end{align}

For the sum over $0\leq\ell<t$ in \eqref{eq:phase1_signal_decomposition}, Cauchy--Schwarz gives
\begin{align*}
\frac{1}{d}\big|
\iprod{(\y-1)\circ\g^{-1}}
{\sum_{\ell=0}^{t-1}\gamma_{t,\ell}\g^\ell}
\big|
&\leq\sqrt t\,\pnorm{\gamma_{t,[0:t-1]}}{}
\max_{0\leq\ell<t}
\frac{1}{d}\big|\iprod{(\y-1)\circ\g^{-1}}{\g^\ell}\big|.
\end{align*}
Conditional on $\g^{-1}$, the $t$ inner products in the maximum are independent centered Gaussians with variance $\pnorm{(\y-1)\circ\g^{-1}}{}^2$. On $\cE_{\rm mom}$ of \eqref{eq:gaussian_moment_event}, the conditional Gaussian maximum bound gives
\begin{align*}
\max_{0\leq\ell<t}\frac{1}{d}
\big|\iprod{(\y-1)\circ\g^{-1}}{\g^\ell}\big|
\leq C_\delta\sqrt{\frac{\log n}{n}}
\end{align*}
with probability $1-O(n^{-12})$.  Therefore
\begin{align}
\frac{1}{d}\big|
\iprod{(\y-1)\circ\g^{-1}}
{\sum_{\ell=0}^{t-1}\gamma_{t,\ell}\g^\ell}
\big|
\leq C_\delta\sqrt{\frac{t\log n}{n}}
\pnorm{\gamma_{t,[0:t-1]}}{}
\leq C_\delta\frac{\log^{14}n}{n^{3/4}},
\label{eq:phase1_old_gaussian_terms}
\end{align}
where the last inequality uses Lemma~\ref{lem:lambda_almost_orthogonal}. 

We next bound $\frac{1}{d}|\iprod{(\y-1)\circ\g^{-1}}{\Delta_{U,t}}|$. Using only the simple Cauchy-Schwarz bound gives
\begin{align*}
\frac{1}{d}\big|\iprod{(\y-1)\circ\g^{-1}}{\Delta_{U,t}}\big|
&\leq\frac{\pnorm{(\y-1)\circ\g^{-1}}{}}{d}\pnorm{\Delta_{U,t}}{}
\leq C_\delta\frac{\log^{11}n}{\sqrt n},
\end{align*}
which is not sufficiently small. The key is to use approximate orthogonality to get an improvement. Consider the decomposition of $\Delta_{U,t}'$ in \eqref{def:Delta_Ut} of Lemma~\ref{lem:error_decompose}. Every term in this decomposition, including $\Delta_{U,t}^{\low}$ from \eqref{eq:amp_low_d}, belongs to $\spa(\r^0,\ldots,\r^{t-1})$. By \eqref{eq:w_isometry}, the vectors $\w^0,\ldots,\w^{t-1}$ are linearly independent for large enough $n$, and hence
\begin{align*}
\spa(\r^0,\ldots,\r^{t-1})
=\spa(\w^0,\ldots,\w^{t-1}).
\end{align*}
The identity $\Delta_{U,t}=\Delta'_{U,t}+\sum_{k=0}^{t-1}\alpha_{t-1}^k\w^k$ then gives
\begin{align*}
\Delta_{U,t}\in\spa(\w^0,\ldots,\w^{t-1}).
\end{align*}
Write
\begin{align*}
\Delta_{U,t}=\sum_{\ell=0}^{t-1}a_{t-1,\ell}\w^\ell,\qquad \bm a_{t-1}=(a_{t-1,\ell})_{\ell=0}^{t-1}.
\end{align*}
By \eqref{eq:w_isometry} and Proposition \ref{prop:phase_1_tilde_close},
\begin{align*}
\pnorm{\bm a_{t-1}}{}
\leq C_\delta\frac{\pnorm{\Delta_{U,t}}{}}{\sqrt d}
\leq C_\delta\frac{\log^{11}n}{\sqrt n}.
\end{align*}
Add and subtract $(\y-1)\circ\g^\ell/\sqrt{2\delta}$ inside each $\w^\ell$.  Equation \eqref{eq:w_approx}, the displayed bound on $\pnorm{\bm a_{t-1}}{}$, and Cauchy--Schwarz give
\begin{align}
\frac{1}{\sqrt{2\delta}\,d}
\big|\iprod{(\y-1)\circ\g^{-1}}{\Delta_{U,t}}\big|
&\leq
C_\delta\sqrt t\,\pnorm{\bm a_{t-1}}{}
\max_{\ell<t}\frac{1}{\sqrt d}
\Bigpnorm{\w^\ell-\frac{(\y-1)\circ\g^\ell}{\sqrt{2\delta}}}{}
\notag\\
&\quad+
C_\delta\sqrt t\,\pnorm{\bm a_{t-1}}{}
\max_{\ell<t}\frac{1}{d}
\big|\iprod{(\y-1)^{\circ2}\circ\g^{-1}}{\g^\ell}\big|
\notag\\
&\leq
C_\delta\frac{\log^{26}n}{n^{3/4}}
+C_\delta\frac{\log^{13}n}{n}
\leq C_\delta\frac{\log^{26}n}{n^{3/4}}.
\label{eq:phase1_delta_projection}
\end{align}
For the second term, condition on $\g^{-1}$ and apply the Gaussian maximum bound to $d^{-1}\langle(\y-1)^{\circ2}\circ\g^{-1},\g^\ell\rangle$. Its conditional variance is $O_\delta(n^{-1})$ on $\cE_{\rm mom}$ of (\ref{eq:gaussian_moment_event}).

The same conditional Gaussian estimate, together with \eqref{eq:last_gamma_diff}, gives
\begin{align}
|\gamma_{t,t}-1|
\Big|\frac{1}{\sqrt{2\delta}d}
\iprod{(\y-1)\circ\g^{-1}}{\g^t}\Big|
\leq C_\delta\frac{\log^{14}n}{n^{3/4}}.
\label{eq:phase1_last_gaussian_coefficient}
\end{align}

The remaining two terms in \eqref{eq:rhot_main} have the factor $\widehat\mu_t^2\leq C_\delta\log^{20}n/\sqrt n$. For the linear term, substituting $\u^t=\tilde\u^t+\Delta_{U,t}$ and applying Cauchy--Schwarz gives
\begin{align*}
\Big|\frac{1}{d}\iprod{\y\circ\g^{-1}}{\u^t}\Big|
&\leq C_\delta\Big(|\rho_t|+\sqrt{\frac{(t+1)\log n}{n}}+\frac{\pnorm{\Delta_{U,t}}{}}{\sqrt n}\Big)
\leq C_\delta\frac{\log^{10}n}{n^{1/4}},
\end{align*}
where we used Lemma~\ref{lem:uv_norm} and the same conditional Gaussian maximum bound, now for $d^{-1}\iprod{\y\circ\g^{-1}}{\g^\ell}$, $0\leq\ell\leq t$.
Since $\hat\mu_t^2\leq C_\delta\log^{20}n/\sqrt n$,
\begin{align}
\frac{(1+|\pi_t|)\hat\mu_t^2}{\sqrt{2\delta}\,d}
\big|\iprod{\y\circ\g^{-1}}{\u^t}\big|
\leq C_\delta\frac{\log^{30}n}{n^{3/4}}.
\label{eq:phase1_mu_linear}
\end{align}

For the cubic term, the leading contribution from $(\g^t)^{\circ3}$ is centered conditional on $\g^{-1}$. We compare $(\u^t)^{\circ3}$ with this vector in two stages. With $U_i(c)$ and $\mathcal K_t$ as in the proof of \eqref{eq:phase1_weighted_moments}, let
\begin{align*}
c_t=(\rho_t,\gamma_{t,0},\ldots,\gamma_{t,t}),
\qquad
e_t=(0,\ldots,0,1)\in\mathbb R^{t+2}.
\end{align*}
Then $\widetilde u_i^t=U_i(c_t)$, $g_i^t=U_i(e_t)$, and \eqref{eq:gamma_norm}, \eqref{eq:last_gamma_diff}, and Lemma~\ref{lem:lambda_almost_orthogonal} give
\begin{align}
\|c_t-e_t\|
&\leq
|\rho_t|+\|\gamma_{t,[0:t-1]}\|+|\gamma_{t,t}-1|
\leq C_\delta\frac{\log^{12}n}{n^{1/4}}.
\label{eq:phase1_cubic_coefficient_distance}
\end{align}

The comparison uses the uniform polynomial bound
\begin{align}
\sup_{\substack{a,b\in\mathcal K_t\\a\ne b}}
\frac{1}{d\|a-b\|^2}
\sum_{i=1}^n
y_i^4(g_i^{-1})^2
\{U_i(a)^3-U_i(b)^3\}^2
\leq C_\delta.
\label{eq:phase1_cubic_uniform_polynomial}
\end{align}
To prove it, the uniform moment bounds in the proof of \eqref{eq:phase1_weighted_moments} and H\"older's inequality give, for every $c\in\mathcal K_t$,
\begin{align*}
\frac{1}{d}\sum_{i=1}^ny_i^5|U_i(c)|^6
&\leq\Big(\frac{1}{d}\sum_{i=1}^ny_i^6|U_i(c)|^{10}\Big)^{3/5}
\Big(\frac{1}{d}\sum_{i=1}^ny_i^{7/2}\Big)^{2/5}
\leq C_\delta.
\end{align*}
The second average is bounded by the same moment estimate with $U_i(c)=g_i^{-1}$, since $y_i^{7/2}\leq1+y_i^{11}$. For $a,b\in\mathcal K_t$ with $a\ne b$, take $v=(a-b)/\|a-b\|\in\mathcal K_t$ and write
\begin{align*}
U_i(a)^3-U_i(b)^3
=\|a-b\|U_i(v)\{U_i(a)^2+U_i(a)U_i(b)+U_i(b)^2\}.
\end{align*}
Since $y_i^4(g_i^{-1})^2=y_i^5$, another application of H\"older's inequality, with exponents $3$ and $3/2$, gives
\begin{align*}
\frac{1}{d\|a-b\|^2}\sum_{i=1}^ny_i^4(g_i^{-1})^2\{U_i(a)^3-U_i(b)^3\}^2
&\leq\frac{C}{d}\sum_{i=1}^ny_i^5U_i(v)^2\big(U_i(a)^4+U_i(b)^4\big)
\leq C_\delta
\end{align*}
proving \eqref{eq:phase1_cubic_uniform_polynomial}.

Cauchy--Schwarz, \eqref{eq:phase1_cubic_coefficient_distance}, and \eqref{eq:phase1_cubic_uniform_polynomial} now give
\begin{align}
\frac{1}{d}\big|
\iprod{\y^{\circ2}\circ\g^{-1}}
{(\widetilde\u^t)^{\circ3}-(\g^t)^{\circ3}}
\big|
\leq C_\delta\|c_t-e_t\|
\leq C_\delta\frac{\log^{12}n}{n^{1/4}}.
\label{eq:phase1_cubic_tilde_difference}
\end{align}
It remains to bound
$
\frac{1}{d}\Big|\Big\langle\y^{\circ2}\circ\g^{-1},(\g^t)^{\circ3}\Big\rangle\Big|.
$ Conditional on $\g^{-1}$, the variables
\begin{align*}
\xi_i\coloneqq d^{-1}y_i^2g_i^{-1}(g_i^t)^3
\end{align*}
are independent and centered, with total conditional variance at most $C_\delta/n$ on the event $n^{-1}\sum_i(g_i^{-1})^{10}\leq C$.  Bernstein's inequality, followed by the Gaussian truncation bound, gives, with probability $1-O(n^{-12})$,
\begin{align*}
\big|\frac{1}{d}\iprod{\y^{\circ2}\circ\g^{-1}}
{(\g^t)^{\circ3}}\big|
\leq C_\delta\sqrt{\frac{\log n}{n}}.
\end{align*}

Finally,
\begin{align*}
(\u^t)^{\circ3}-(\widetilde\u^t)^{\circ3}
=3(\widetilde\u^t)^{\circ2}\circ\Delta_{U,t}
+3\widetilde\u^t\circ\Delta_{U,t}^{\circ2}
+\Delta_{U,t}^{\circ3}.
\end{align*}
The Gaussian maximum bound gives $\|\y^{\circ2}\circ\g^{-1}\|_\infty\leq C\log^{5/2}n$.
The proof of \eqref{eq:phase1_weighted_moments} also gives
\begin{align*}
\pnorm{(\tilde\u^t)^{\circ2}}{}&\leq C\sqrt n,
\qquad
\pnorm{\tilde\u^t}{\infty}\leq C\log n.
\end{align*}
These bounds, together with $\|\Delta_{U,t}\|\leq C_\delta \log^{11}n$, yield
\begin{align}
\frac{1}{d}\big|
\iprod{\y^{\circ2}\circ\g^{-1}}
{(\u^t)^{\circ3}-(\widetilde\u^t)^{\circ3}}
\big|
&\leq
\frac{C_\delta\log^{5/2}n}{d}\Big\{
\|(\widetilde\u^t)^{\circ2}\|\|\Delta_{U,t}\|
+\|\widetilde\u^t\|_\infty\|\Delta_{U,t}\|^2
+\|\Delta_{U,t}\|^3
\Big\}
\notag\\
&\leq C_\delta\Big\{
\frac{\log^{14}n}{\sqrt n}
+\frac{\log^{26}n}{n}
+\frac{\log^{36}n}{n}
\Big\}
\leq C_\delta\frac{\log^{14}n}{\sqrt n}.
\label{eq:phase1_cubic_delta_terms}
\end{align}
Combining \eqref{eq:phase1_cubic_tilde_difference} and \eqref{eq:phase1_cubic_delta_terms} gives
\begin{align*}
\frac{1}{d}\big|\iprod{\y^{\circ2}\circ\g^{-1}}
{(\u^t)^{\circ3}-(\g^t)^{\circ3}}\big|
&\leq C_\delta\Big(\frac{\log^{12}n}{n^{1/4}}
+\frac{\log^{14}n}{\sqrt n}\Big)
\leq C_\delta\frac{\log^{13}n}{n^{1/4}}.
\end{align*}
Consequently,
\begin{align}
\frac{(1+|\pi_t|)\hat\mu_t^2}{3\sqrt{2\delta}\,d}
\big|\iprod{\y^{\circ2}\circ\g^{-1}}{(\u^t)^{\circ3}}\big|
\leq C_\delta\frac{\log^{33}n}{n^{3/4}}.
\label{eq:phase1_mu_cubic}
\end{align}

Finally, we bound the remaining terms involving $\pi_t$. Conditional on $\g^{-1}$, $\varphi_t$ has variance at most $C/d$ on the event $\mathcal E_{\rm mom}$ in \eqref{eq:gaussian_moment_event} so, with probability $1-O(n^{-12})$,
\begin{align}
|\varphi_t|
\leq C\sqrt{\frac{\log n}{d}}.
\label{eq:phase1_varphi_upper}
\end{align}
Together with \eqref{eq:phase1_signal_coefficient}, the bounds
\begin{align*}
|\rho_t|\leq \log^{10}n n^{-1/4},
\qquad
|\pi_t|\leq C_\delta\frac{\log^{20}n}{\sqrt n}
\end{align*}
imply
\begin{align}
\big|
\big\{(1+\pi_t)
\frac{1}{\sqrt{2\delta}\,d}
\iprod{\g^{-1}}{(\y-1)\circ\g^{-1}}
-\sqrt{2\delta}\big\}\rho_t
+\pi_t\varphi_t
\big|
\leq C_\delta\frac{\log^{30}n}{n^{3/4}}.
\label{eq:phase1_coefficient_and_normalization}
\end{align}
Substituting \eqref{eq:phase1_signal_decomposition} into \eqref{eq:rhot_main} and combining \eqref{eq:phase1_second_order_term}--\eqref{eq:phase1_coefficient_and_normalization} gives the desired scalar recursion
\begin{align*}
\big|\rho_{t+1}-\sqrt{2\delta}\rho_t-\varphi_t\big|
\leq C_\delta\frac{\log^{33}n}{n^{3/4}}.
\end{align*}
\end{proof}

\subsection{Proof of Proposition \ref{prop:phase_1_main} (Crossing time)}

\begin{proof}
The proof is similar to that of \cite[claim (42)]{li2023approximate}, and we present the details for the convenience of the reader. Let $\mathcal E_I$ be the event on which the bounds of Proposition~\ref{prop:phase_1_tilde_close} and the recursion \eqref{eq:recursion_phase_1} hold simultaneously for $0\leq t<\tau_I$ with $t\leq\log^2n$, and $|\rho_0|\leq n^{-1/4}\log^{10}n$. We have shown that $\P(\mathcal E_I^c)=O(n^{-10})$ for large enough $n$.

We first note that, on $\mathcal E_I$, for any integers $0\leq k\leq t<\tau_I$ with $t\leq\log^2n$, iterating \eqref{eq:recursion_phase_1} for $t+1-k$ steps yields
\begin{align}\label{eq:phase_k_recursion}
\rho_{t+1}
&=\big(\sqrt{2\delta}\big)^{t+1-k}\rho_k+\sum_{j=1}^{t+1-k}\big(\sqrt{2\delta}\big)^{j-1}\varphi_{t+1-j}+O_\delta\Big(\frac{\log^{80}n}{n^{3/4}}\big(\sqrt{2\delta}\big)^{t-k}\Big).
\end{align}
Here the accumulated error is at most $C_\delta(t+1-k)\log^{34}n\,n^{-3/4}(\sqrt{2\delta})^{t-k}$, so the displayed bound follows from $t+1-k\leq\log^2n+1$. 

We also use the event
\begin{align}\label{eq:phase1_noise_variance_event}
\mathcal G_I\coloneqq\Bigg\{\frac{\pnorm{(\y-1)\circ\g^{-1}}{}^2}{2\delta d^2}\geq\frac{4}{d}\Bigg\}.
\end{align}
Since $\y=\g^{-1}\circ\g^{-1}$ and $\E[(G^2-1)^2G^2]=10$ for $G\sim\N(0,1)$, Gaussian polynomial concentration gives $\P(\mathcal G_I^c)=O(n^{-12})$.

Let $T$ be a positive integer to be specified later, and let $m$ be a positive integer such that $mT\leq\log^2n$. On $\mathcal E_I\cap\{\tau_I>mT\}$, applying \eqref{eq:phase_k_recursion} with $t=iT-1$ and $k=(i-1)T$ gives, for every $1\leq i\leq m$,
\begin{align*}
\rho_{iT}
&=\big(\sqrt{2\delta}\big)^T\rho_{(i-1)T}+\sum_{j=1}^T\big(\sqrt{2\delta}\big)^{j-1}\varphi_{iT-j}+O_\delta\Big(\frac{\log^{80}n}{n^{3/4}}\big(\sqrt{2\delta}\big)^T\Big).
\end{align*}
Since $|\rho_{iT}|\leq n^{-1/4}\log^{10}n$ on this event, it implies each of the events
\begin{align}\label{eq:phase1_block_small_ball}
B_i\coloneqq\Bigg\{
&\Big|\big(\sqrt{2\delta}\big)^T\rho_{(i-1)T}+\sum_{j=1}^T\big(\sqrt{2\delta}\big)^{j-1}\varphi_{iT-j}\Big|\leq\frac{\log^{10}n}{n^{1/4}}+C_\delta\frac{\log^{80}n}{n^{3/4}}\big(\sqrt{2\delta}\big)^T\Bigg\},
\end{align}
where $C_\delta$ is chosen large enough to bound the remainder. Consequently,
\begin{align}\label{eq:cross_main}
\P(\tau_I>mT) \leq\P\Big(\mathcal G_I\cap\bigcap_{i=1}^m B_i\Big)+\P(\mathcal E_I^c)+\P(\mathcal G_I^c) \leq \P\Big(\mathcal G_I\cap\bigcap_{i=1}^m B_i\Big)+O(n^{-10}).
\end{align}

Let us bound the probability of $B_i$ conditional on the full history before block $i$:
\begin{align*}
\mathcal F_{i-1}\coloneqq\cF\Big(\g^{-1},\v^0,\{\g^\ell\}_{\ell=0}^{(i-1)T-1},\{\s^\ell\}_{\ell=0}^{(i-1)T-1}\Big).
\end{align*}
By the dependencies in \eqref{eq:uv_dependence}, $\rho_{(i-1)T}$ is $\mathcal F_{i-1}$-measurable, and conditional on $\g^{-1}$, the variables $\{\varphi_{iT-j}\}_{j=1}^T$ are independent of $\mathcal F_{i-1}$. In particular, conditional on $\mathcal F_{i-1}$,
\begin{align*}
\sum_{j=1}^T\big(\sqrt{2\delta}\big)^{j-1}\varphi_{iT-j}
\sim\N(0,\sigma^2),\qquad
\sigma^2
=\frac{\pnorm{(\y-1)\circ\g^{-1}}{}^2}{2\delta d^2}\sum_{j=1}^T(2\delta)^{j-1}.
\end{align*}

Next, we use the fact that, for $\varphi\sim\N(0,\sigma^2)$ with $\sigma>0$, every $x\in\R$ and $u\geq0$ satisfy
\begin{align}\label{eq:fact_gauss_mass}
\P(|x+\varphi|\leq u)\leq\P(|\varphi|\leq u)\leq\sqrt{\frac{2}{\pi}}\frac{u}{\sigma}.
\end{align}
The first inequality holds because the Gaussian density is symmetric and decreases away from zero; the second follows by bounding the density by its maximum. On $\mathcal G_I$, the variance formula gives
\begin{align}\label{eq:phase1_block_variance}
\sigma^2\geq\frac{4}{d}\sum_{j=1}^T(2\delta)^{j-1}\geq\frac{4}{d}(2\delta)^{T-1}.
\end{align}
Applying \eqref{eq:fact_gauss_mass} conditional on $\mathcal F_{i-1}$, and using $d=n/\delta$, we obtain on $\mathcal G_I$ that
\begin{align*}
\P(B_i\mid\mathcal F_{i-1}) \leq\frac{C\sqrt d}{(\sqrt{2\delta})^{T-1}}\Big(\frac{\log^{10}n}{n^{1/4}}+C_\delta\frac{\log^{80}n}{n^{3/4}}\big(\sqrt{2\delta}\big)^T\Big)\leq C_\delta\Big(\frac{n^{1/4}\log^{10}n}{(\sqrt{2\delta})^T}+\frac{\log^{80}n}{n^{1/4}}\Big).
\end{align*}

Choose
\begin{align}\label{eq:phase1_block_length}
T=T_\delta(n)\coloneqq\frac{\log n+20\log\log n}{\log(2\delta)}.
\end{align}
Then $(\sqrt{2\delta})^T\geq n^{1/2}\log^{10}n$. Thus, for sufficiently large $n$, on $\mathcal G_I$,
\begin{align}\label{eq:phase1_qn_conditional}
\P(B_i\mid\mathcal F_{i-1})\leq C_\delta\frac{\log^{80}n}{n^{1/4}}\leq n^{-1/5}.
\end{align}
Take $m=60$ and $n$ large enough so that
\begin{align}\label{eq:phase1_crossing_finite_size}
60T\leq\log^2n.
\end{align}
Since $\mathcal G_I$ and $B_1,\ldots,B_{i-1}$ are $\mathcal F_{i-1}$-measurable, iterated conditioning gives
\begin{align*}
\P\Big(\mathcal G_I\cap\bigcap_{i=1}^{60}B_i\Big)\leq n^{-12}.
\end{align*}
Combining this with \eqref{eq:cross_main}, and recalling that $\tau_I\geq1$ on $\mathcal E_I$, we conclude that, with probability $1-O(n^{-10})$,
\begin{align}\label{eq:phase1_explicit_length}
1\leq\tau_I\leq60T\leq C_\delta\log n,
\end{align}
as desired.

\end{proof}

\section{Proofs for Phase II}\label{sec:proof_phase_2}
Recall the endpoint $c_{\weak}$ from Section~\ref{subsec:analysis_phase2}. Phase II tracks the correlation starting from $\tau_I$, where the Phase-I crossing bound gives
\begin{align}\label{eq:phase2_entrance_size}
|\rho_{\tau_I}|>n^{-1/4}\log^{10}n
\end{align}
until it reaches $c_{\weak}.$
\subsection{Proof of Proposition \ref{prop:phase_2_exp_growth}}

The proof proceeds in the same manner as the proof of Proposition~\ref{prop:phase_1_tilde_close}. We start by defining the analogous refined assumptions for Phase II.

\begin{definition}\label{assump:inductive_assumption_phase_2}
Let $\cA^U_{I\!I,t}$ be the event on which
\begin{subequations}
\begin{align}
\sum_{k=0}^{t-1}|\alpha_{t-1}^k|
&\leq C_\delta\frac{t\log^{10}n}{\sqrt d},\label{eq:alpha_sum_phase_2}\\
\pnorm{\Delta_{U,t}}{}\vee\pnorm{\hat\u^t-\tilde\u^t}{}
&\leq C_\delta t\log^{10}n,\label{eq:u_closeness_phase_2}\\
\bar\gamma_t&\leq C_\delta\sqrt d,\label{eq:lambda_main_norm_phase_2}\\
\bar\lambda_t&\leq C_\delta\sqrt d,\label{eq:phase_2_current_lambda}\\
\pnorm{h_t'}{\infty}\leq C_\delta\log n,&\qquad
\pnorm{h_t''}{\infty}\leq C_\delta\log^{3/2}n.\label{eq:phase_2_assump_U_1}
\end{align}
\end{subequations}
Similarly, let $\cA_{I\!I,t}^V$ be the event on which
\begin{subequations}
\begin{align}
\sum_{k=0}^t|\beta_t^k|
&\leq C_\delta\frac{(t+1)\log^{10}n}{\sqrt d},\label{eq:beta_sum_phase_2}\\
\pnorm{\Delta_{V,t+1}}{}\vee\pnorm{\hat\v^{t+1}-\tilde\v^{t+1}}{}
&\leq C_\delta(t+1)\log^{10}n,\label{eq:uv_closeness_phase_2}\\
\bar\lambda_t&\leq C_\delta\sqrt d,\\
\bar\gamma_{t+1}&\leq Cn^5,\label{eq:phase_2_next_gamma}\\
\pnorm{f'_{t+1}}{\infty}=1,&\qquad
\pnorm{f''_{t+1}}{\infty}=0.\label{eq:phase_2_assump_V_1}
\end{align}
\end{subequations}
\end{definition}
 The following lemma gives the analogues of the Phase-I parameter and derivative estimates, together with the desired exponential growth. The proofs follow the same expansions, but now track the errors in terms of $\rho_t$. 

\begin{lemma}\label{lem:phase_2_assump_conseq}
Suppose $\tau_I\leq t<\tau_{I\!I}$, $t+1\leq\log^2n$, and $|\rho_t|\geq n^{-1/4}\log^{10}n$. Suppose the refined assumptions hold up to $\cA_{I\!I,t-1}^V$, and \eqref{eq:alpha_sum_phase_2}, \eqref{eq:u_closeness_phase_2}, and \eqref{eq:lambda_main_norm_phase_2} hold at time $t$. Then
\begin{gather}
\big|\hat\mu_t-|\rho_t|\big|
\leq C_\delta\frac{|\rho_t|}{\log^8n},\label{eq:phase_2_mu_est}\\
\Big|\frac{\pnorm{h_t^\star(\u^t,\y)}{}}{\sqrt d}-\sqrt{2\delta}\hat\mu_t\Big|
\leq C_\delta|\hat\mu_t|\rho_t^2,\label{eq:phase_2_norm}\\
h_t'(u,y)
=\frac{1+\pi_t}{\sqrt{2\delta}}\Big((1+\hat\mu_t^2)y\big(1-\tanh^2(\hat\mu_tu\sqrt y)\big)-1\Big),\label{eq:ht_second}
\end{gather}
where $|\pi_t|\leq C_\delta\rho_t^2$. Moreover,
\begin{gather}\label{eq:phase_2_HF}
F_{t+1}\Big(\rho_{t+1},\pnorm{\lambda_{t,[0:t]}}{};\btheta^\star\Big)=1,\\
\Big|H_t\Big(\rho_t,\pnorm{\gamma_{t,[0:t]}}{};\bbeta^\star\Big)-1\Big|
\leq C_\delta\rho_t^2,\notag\\
|\rho_{t+1}-\sqrt{2\delta}\rho_t|
\leq C_\delta|\rho_t|^3+C_\delta\frac{\log^{14}n}{\sqrt d},\label{eq:phase2_cubic_recursion}\\
|\rho_{t+1}|\geq\Big(\frac{1+\sqrt{2\delta}}{2}\Big)|\rho_t|.\label{eq:phase_2_exp_growth}
\end{gather}
\end{lemma}

As in Phase I, the refined assumptions simplify the general induction inequalities $\cH_t^U$ and $\cH_t^V$. The following lemma gives the Phase II consequences of the refined assumptions.

\begin{lemma}\label{lem:phase_2_simp}
Let $\tau_I\leq t<\tau_{I\!I}$ with $t+1\leq\log^2n$, and suppose the growth bound \eqref{eq:phase_2_exp_growth} holds at all earlier Phase-II times. Suppose the refined assumptions hold up to $\cA_{I\!I,t-1}^V$. For $t\geq1$, $\cH_t^U$ implies
\begin{align}
\pnorm{\Delta'_{U,t}}{}
&\leq\Big(1+C_\delta\frac{\log^{16}n}{n^{1/4}}\Big)\pnorm{\Delta'_{V,t}}{}
+C_\delta t\log^{5/2}n,\label{eq:phase_2_simp_1}\\
|\alpha_{t-1}^{t-1}|
&\leq C_\delta\Big(\frac{\pnorm{\Delta'_{V,t}}{}}{\sqrt d}+\frac{t^2\log^{20}n}{d}\Big),\label{eq:phase_2_simp_2}\\
|\alpha_{t-1}^k|
&\leq\begin{cases}
|\beta_{t-1}^{t-1}|,&k=t-2,\\
\Big(1+C_\delta(\rho_{k+1}^2+\rho_{t-1}^2)+C_\delta\frac{\log^{16}n}{n^{1/4}}\Big)|\alpha_{t-2}^{k+1}|,&0\leq k\leq t-3.
\end{cases}\label{eq:phase_2_simp_3}
\end{align}
Suppose instead that the refined assumptions hold up to $\cA_{I\!I,t}^U$. Then $\cH_t^V$ implies
\begin{align*}
\pnorm{\Delta'_{V,t+1}}{}
&\leq\Big(1+C_\delta\rho_t^2+C_\delta\frac{\log^{16}n}{n^{1/4}}\Big)\pnorm{\Delta'_{U,t}}{}
+C_\delta(t+1)\log^{5/2}n,\\
|\beta_t^t|
&\leq C_\delta\Big(\frac{\log^{3/2}n\pnorm{\Delta'_{U,t}}{}}{\sqrt d}+\frac{(t+1)^2\log^{23}n}{d}\Big),\\
|\beta_t^k|
&\leq\begin{cases}
C_\delta\log^2n|\alpha_{t-1}^{t-1}|,&k=t-1,\\
\Big(1+C_\delta(\rho_t^2+\rho_{k+1}^2)+C_\delta\frac{\log^{16}n}{n^{1/4}}\Big)|\beta_{t-1}^{k+1}|,&0\leq k\leq t-2.
\end{cases}
\end{align*}
\end{lemma}

With these results, we are ready to prove the main proposition.

\begin{proof}[Proof of Proposition \ref{prop:phase_2_exp_growth}]
By \eqref{eq:phase_2_exp_growth} of Lemma \ref{lem:phase_2_assump_conseq}, it is sufficient to show that $\cA_{I\!I,t}^U$ and $\cA_{I\!I,t}^V$ hold for all $\tau_I\leq t<\tau_{I\!I}$ with $t+1\leq\log^2n$ with probability $1-O(n^{-10})$, and then show that the crossing occurs within this horizon. Similar to the proof of Proposition \ref{prop:phase_1_tilde_close}, we show the following chain of claims
\begin{align}\label{eq:phase_2_chain}
\cH^U_{\tau_I}
&\rightarrow\cA^U_{I\!I,\tau_I}
\overset{\checkmark}{\rightarrow}\cH^V_{\tau_I}
\rightarrow\cA^V_{I\!I,\tau_I}
\overset{\checkmark}{\rightarrow}\cdots\overset{\checkmark}{\rightarrow}\cH^U_t
\rightarrow\cA^U_{I\!I,t}
\overset{\checkmark}{\rightarrow}\cH^V_t
\rightarrow\cA^V_{I\!I,t}.
\end{align}
The $\checkmark$ follow by direct application of Proposition \ref{prop:general_error_induction}.  The first component of the chain, $\cH^U_{\tau_I}$, holds by Proposition \ref{prop:general_error_induction}, because we have verified Assumptions $\cA^U_{\tau_I-1}$ and $\cA^V_{\tau_I-1}$ hold in our Phase I analysis. Furthermore, by the proof of Proposition \ref{prop:phase_1_tilde_close}, the Phase-II refined assumptions hold at $\tau_I$ (it can be checked that the phase I assumptions imply the Phase II assumptions). This provides the base case for the Phase II induction. We inductively verify the remaining arrows in the chain, which consists of the following two steps:
\begin{itemize}
\item (Claim 1) Suppose the chain \eqref{eq:phase_2_chain} holds up to (and including) $\cH_t^U$, and \eqref{eq:phase_2_exp_growth} holds at all earlier Phase-II times. We check that $\cA_{I\!I,t}^U$ and \eqref{eq:phase_2_exp_growth} hold at time $t$.
\item (Claim 2) Suppose the chain \eqref{eq:phase_2_chain} holds up to (and including) $\cH_t^V$. We check that $\cA_{I\!I,t}^V$ holds.
\end{itemize}
These two claims are similar, so we give the details for the first claim and briefly verify the second. We first derive a consequence of the preceding exponential growth inequalities \eqref{eq:phase_2_exp_growth}. For any $\tau_I\leq r\leq s\leq t<\tau_{I\!I}$, we have
\begin{align}\label{eq:phase_2_O(1)_sum}
\sum_{k=r}^s\rho_k^2
&\leq c_{\weak}^2\sum_{j=0}^{\infty}\Big(\frac{2}{1+\sqrt{2\delta}}\Big)^{2j}
\leq\frac{c_{\weak}^2(1+\sqrt{2\delta})}{\sqrt{2\delta}-1}
\leq C,
\end{align}
for sufficiently small $c_{\weak}$. Consequently,
\begin{align}\label{eq:phase2_product_bound}
\prod_{k=r}^s\Big(1+C\rho_k^2+C_\delta\frac{\log^{16}n}{n^{1/4}}\Big)
&\leq\exp\Big\{C\sum_{k=r}^s\rho_k^2+C_\delta\frac{\log^{18}n}{n^{1/4}}\Big\}
\leq C.
\end{align}

Next, we derive bounds on $\pnorm{\Delta_{U,s}'}{}$ and $\pnorm{\Delta_{V,s}'}{}$ for $s\leq t$. Using the simplified bounds in Lemma \ref{lem:phase_2_simp}, we have the one-step recursions
\begin{align*}
\pnorm{\Delta'_{V,s}}{}
&\leq\Big(1+C\rho_{s-1}^2+C_\delta\frac{\log^{16}n}{n^{1/4}}\Big)\pnorm{\Delta'_{U,s-1}}{}
+C_\delta s\log^{5/2}n,
\qquad \tau_I<s\leq t,\\
\pnorm{\Delta'_{U,s}}{}
&\leq\Big(1+C_\delta\frac{\log^{16}n}{n^{1/4}}\Big)\pnorm{\Delta'_{V,s}}{}
+C_\delta s\log^{5/2}n,
\qquad \tau_I\leq s\leq t.
\end{align*}
Iterating, we get for all $\tau_I\leq s\leq t$
\begin{align}\label{eq:phase2_second_order_residual}
\pnorm{\Delta'_{U,s}}{}
&\leq\Bigg(\prod_{k=\tau_I}^{s-1}\Big(1+C\rho_k^2+C_\delta\frac{\log^{16}n}{n^{1/4}}\Big)\Bigg)\pnorm{\Delta'_{U,\tau_I}}{}\notag\\
&\quad+C_\delta\sum_{r=\tau_I+1}^s\Bigg(\prod_{k=r}^{s-1}\Big(1+C\rho_k^2+C_\delta\frac{\log^{16}n}{n^{1/4}}\Big)\Bigg)r\log^{5/2}n\notag\\
&\overset{(i)}{\leq} C\pnorm{\Delta'_{U,\tau_I}}{}+C_\delta s^2\log^{5/2}n
\overset{(ii)}{\leq} C_\delta s^2\log^{5/2}n\leq C_\delta\log^7n,
\end{align}
where $(i)$ follows from \eqref{eq:phase2_product_bound}, and in $(ii)$ we apply \eqref{eq:phase_2_simp_1} with the Phase-I bound \eqref{eq:phase1_second_order_V_next} to obtain $\pnorm{\Delta'_{U,\tau_I}}{}\leq C_\delta\tau_I^2\log^{5/2}n$. Similarly, we get $\pnorm{\Delta'_{V,s}}{}\leq C_\delta s^2\log^{5/2}n\leq C_\delta\log^7n$ for all $\tau_I\leq s\leq t$. Now we check the assumption $\cA_{I\!I,t}^U$.

\noindent\textbf{Proof of (\ref{eq:alpha_sum_phase_2})} Since the refined assumptions hold up to $\cA_{I\!I,t-1}^V$, we have by \eqref{eq:phase_2_simp_2} that
\begin{align*}
|\alpha_{t-1}^{t-1}|
&=O_\delta\Big(\frac{\pnorm{\Delta_{V,t}'}{}}{\sqrt d}+\frac{\log^{30}n}{d}\Big)
=O_\delta\Big(\frac{\log^7n}{\sqrt d}\Big).
\end{align*}
The remaining $\alpha$'s satisfy, by \eqref{eq:phase_2_simp_3} and the bound for $|\beta_{t-1}^{t-1}|$ in Lemma~\ref{lem:phase_2_simp}, using \eqref{eq:induce_V_2} at $t=\tau_I$,
\begin{align*}
|\alpha_{t-1}^{t-2}|
&\leq|\beta_{t-1}^{t-1}|
=O_\delta\Big(\frac{\log^{3/2}n\pnorm{\Delta_{U,t-1}'}{}}{\sqrt d}+\frac{\log^{30}n}{d}\Big)
=O_\delta\Big(\frac{\log^9n}{\sqrt d}\Big),\qquad t\geq2,\\
|\alpha_{t-1}^k|
&\leq\Big(1+C\rho_{t-1}^2+C_\delta\frac{\log^{16}n}{n^{1/4}}\Big)|\alpha_{t-2}^{k+1}|,
\qquad 0\leq k\leq t-3.
\end{align*}
In the last line we used that $\rho_{k+1}^2\leq\rho_{t-1}^2+n^{-1/2}\log^{20}n$ for $k\leq t-3$: this follows from the growth bound \eqref{eq:phase_2_exp_growth} when $k+1\geq\tau_I$, and from $|\rho_{k+1}|\leq n^{-1/4}\log^{10}n$ when $k+1<\tau_I$. The extra term is absorbed into $C_\delta\log^{16}n/n^{1/4}$.
This implies the following recursion of the sequence $t\mapsto\sum_{k=0}^{t-1}|\alpha_{t-1}^k|$:
\begin{align*}
\sum_{k=0}^{t-1}|\alpha_{t-1}^k|
&\leq\Big(1+C\rho_{t-1}^2+C_\delta\frac{\log^{16}n}{n^{1/4}}\Big)\sum_{k=0}^{t-2}|\alpha_{t-2}^k|
+C_\delta\frac{\log^9n}{\sqrt d}.
\end{align*}
Iterating this recursion gives
\begin{align*}
\sum_{k=0}^{t-1}|\alpha_{t-1}^k|
&\leq\Bigg[\prod_{j=\tau_I-1}^{t-1}
\Big(1+C\rho_j^2+C_\delta\frac{\log^{16}n}{n^{1/4}}\Big)\Bigg]
\sum_{k=0}^{\tau_I-2}|\alpha_{\tau_I-2}^k|+C_\delta\frac{\log^9n}{\sqrt d}
\sum_{s=\tau_I}^t\prod_{j=s}^{t-1}
\Big(1+C\rho_j^2+C_\delta\frac{\log^{16}n}{n^{1/4}}\Big)\\
&\leq C\sum_{k=0}^{\tau_I-2}|\alpha_{\tau_I-2}^k|
+C_\delta(t-\tau_I+1)\frac{\log^9n}{\sqrt d}\\
&\leq C_\delta\frac{(\tau_I-1)\log^9n}{\sqrt d}
+C_\delta\frac{(t-\tau_I+1)\log^9n}{\sqrt d}
=C_\delta\frac{t\log^9n}{\sqrt d}.
\end{align*}
The second inequality uses \eqref{eq:phase2_product_bound} and the Phase-I bound $|\rho_{\tau_I-1}|\leq n^{-1/4}\log^{10}n$. The last inequality uses \eqref{eq:phase_1_assump_U_5} at time $\tau_I-1$.

\noindent\textbf{Proof of \eqref{eq:u_closeness_phase_2}}
By the same argument as in the proof of \eqref{eq:phase_1_assump_U_2} in Proposition \ref{prop:phase_1_tilde_close}, we have
\begin{align*}
\pnorm{\hat\u^t-\tilde\u^t}{}
&\leq C_\delta\sqrt d\sum_{k=0}^{t-1}|\alpha_{t-1}^k|
=O_\delta(t\log^9n),
\end{align*}
using $\pnorm{h_k(\tilde\u^k,\y)}{}\leq\sqrt d+\pnorm{h_k'}{\infty}\pnorm{\Delta_{U,k}}{}\leq C_\delta\sqrt d$ for $k<t$. In addition, by Lemma \ref{lem:mutual_close} and the bound $\bar\lambda_{t-1}\leq C_\delta\sqrt d$ in $\cA_{I\!I,t-1}^V$, we have
\begin{align*}
\pnorm{\Delta_{U,t}}{}
&\leq\pnorm{\Delta'_{U,t}}{}+C_\delta\sqrt d\sum_{k=0}^{t-1}|\alpha_{t-1}^k|=O_\delta(\log^7n)+O_\delta(t\log^9n)
=O_\delta(t\log^9n).
\end{align*}

\noindent\textbf{Proof of \eqref{eq:lambda_main_norm_phase_2}}
This follows from \eqref{eq:gamma_norm} of Lemma~\ref{lem:uv_norm} and the bounds on $\pnorm{\Delta_{V,s}}{}$ in \eqref{eq:uv_closeness_phase_1} and \eqref{eq:uv_closeness_phase_2}, taking the maximum over $0\leq s\leq t$.

\noindent\textbf{Proof of \eqref{eq:phase_2_assump_U_1}} The bounds \eqref{eq:alpha_sum_phase_2}, \eqref{eq:u_closeness_phase_2}, and \eqref{eq:lambda_main_norm_phase_2} allow us to apply Lemma~\ref{lem:phase_2_assump_conseq}, which gives \eqref{eq:ht_second}. Note that
\begin{align*}
h_t'(u,y)
&=\frac{1+O_\delta(\rho_t^2)}{\sqrt{2\delta}}\Big((\hat\mu_t^2+1)y\big(1-\tanh^2(\hat\mu_tu\sqrt y)\big)-1\Big),
\end{align*}
and by Lemma \ref{lem:h_deriv_comp},
\begin{align*}
h_t''(u,y)
&=\frac{1+O_\delta(\rho_t^2)}{\sqrt{2\delta}}\Big(-2(\hat\mu_t^2+1)\hat\mu_t y^{3/2}\tanh(\hat\mu_tu\sqrt y)\big(1-\tanh^2(\hat\mu_tu\sqrt y)\big)\Big).
\end{align*}
The claim follows from $\pnorm{\y}{\infty}=O(\log n)$ and $\hat\mu_t=|\rho_t|(1+O_\delta(\log^{-8}n))$.

\medskip
\noindent\textbf{Proof of Claim 2.}
The first bound of Lemma \ref{lem:phase_2_simp} for $\cH_t^V$, together with the bound on $\pnorm{\Delta'_{U,t}}{}$ from Claim 1, gives
\begin{align}\label{eq:phase2_second_order_V_next}
\pnorm{\Delta'_{V,t+1}}{} \leq C_\delta(t+1)^2\log^{5/2}n\leq C_\delta\log^7n.
\end{align}
We next bound the coefficients $\beta_t^k$ using the other two bounds of Lemma \ref{lem:phase_2_simp}. For $k=t$, the bound $\pnorm{\Delta'_{U,t}}{}\leq C_\delta\log^7n$ gives $|\beta_t^t|=O_\delta(\log^9n/\sqrt d)$. For $k=t-1$, the bound $|\alpha_{t-1}^{t-1}|=O_\delta(\log^7n/\sqrt d)$ from the proof of \eqref{eq:alpha_sum_phase_2} gives 
\begin{align*}
|\beta_t^{t-1}|\leq C_\delta\log^2n\,|\alpha_{t-1}^{t-1}|=O_\delta(\log^9n/\sqrt d).
\end{align*}
For $0\leq k\leq t-2$, since $|\rho_j|\leq|\rho_t|$ for $j\leq t$, we have $|\beta_t^k|\leq\big(1+C_\delta\rho_t^2+C_\delta\log^{16}n/n^{1/4}\big)|\beta_{t-1}^{k+1}|$. Summing over $k$ and iterating as in the proof of \eqref{eq:alpha_sum_phase_2}, with
\eqref{eq:phase2_product_bound}, gives
\begin{align*}
\sum_{k=0}^t|\beta_t^k| =O_\delta\Big(\frac{(t+1)\log^9n}{\sqrt d}\Big), \pnorm{\hat\v^{t+1}-\tilde\v^{t+1}}{}\vee\pnorm{\Delta_{V,t+1}}{}
&=O_\delta((t+1)\log^9n),
\end{align*}
where the second bound follows from Lemma \ref{lem:mutual_close} and the definition of $\hat\v^{t+1}$. The bound $\bar\lambda_t\leq C_\delta\sqrt d$ is \eqref{eq:global_current_lambda_control}.
For $\bar\gamma_{t+1}$, the update \eqref{eq:simplified_amp} gives
\begin{align*}
\sqrt d\,\pnorm{\gamma_{t+1,[-1:t+1]}}{} = \pnorm{\v^{t+1}}{} \leq\pnorm{\X}{\mathrm{op}}\sqrt d+|\hat c_t|\pnorm{\v^t}{} \leq C_\delta\sqrt d\log n\leq Cn^5.
\end{align*}
Since $\bar\gamma_t\leq C_\delta\sqrt d$ by \eqref{eq:lambda_main_norm_phase_2}, this proves \eqref{eq:phase_2_next_gamma}. Finally, \eqref{eq:phase_2_assump_V_1} holds because $f_{t+1}$ is the identity. This completes Claim 2.

To bound the crossing time, the growth inequalities give, until the crossing,
\begin{align*}
|\rho_{\tau_I+s}|>\Big(\frac{1+\sqrt{2\delta}}{2}\Big)^s n^{-1/4}\log^{10}n.
\end{align*}
Since $c_{\weak}\leq1$, at most $1+\log n/\big[4\log\big((1+\sqrt{2\delta})/2\big)\big]$ steps suffice. Taking $n$ large enough that
\begin{align}\label{eq:phase2_horizon_condition}
60T_\delta(n)+\frac{C\log n}{\log\Big(\frac{1+\sqrt{2\delta}}{2}\Big)}\leq\log^2n
\end{align}
ensures that $\tau_{I\!I}\leq \log^2 n$. This proves \eqref{eq:phase2_explicit_length}.
\end{proof}

\subsection{Proof of Lemma \ref{lem:phase_2_assump_conseq}}
\begin{proof}
\noindent\textbf{Proof of (\ref{eq:phase_2_mu_est})}
Recall from (\ref{eq:hat_mu_t}) that $\hat\mu_t=\sqrt{\big(\frac{1}{n}\pnorm{\u^t}{}^2-1\big)\vee n^{-1/2}}$. By Lemma \ref{lem:uv_norm}, $\u^t=\tilde\u^t+\Delta_{U,t}$, and $|\rho_t|<c_{\weak}\leq1$, we have
\begin{align}
\frac{1}{n}\pnorm{\u^t}{}^2
&=1+\rho_t^2+O_\delta\Bigg(
\sqrt{\frac{t+\log n}{n}}
+\frac{\pnorm{\Delta_{V,t}}{}}{\sqrt d}
+\frac{\pnorm{\Delta_{U,t}}{}}{\sqrt n}+\Big(\frac{\pnorm{\Delta_{V,t}}{}}{\sqrt d}
+\frac{\pnorm{\Delta_{U,t}}{}}{\sqrt n}\Big)^2
\Bigg)\notag\\
&=1+\rho_t^2+O_\delta\Big(
\sqrt{\frac{t+\log n}{n}}
+\frac{\pnorm{\Delta_{V,t}}{}}{\sqrt d}
+\frac{\pnorm{\Delta_{U,t}}{}}{\sqrt n}\Big),
\label{eq:ut_norm}
\end{align}
where the second line holds for sufficiently large $n$. For $\tau_I\leq t<\tau_{I\!I}$ with $t+1\leq\log^2n$, we then get
\begin{align*}
\frac{1}{n}\pnorm{\u^t}{}^2-1
&=\rho_t^2+O_\delta\Big(\frac{\log^{12}n}{\sqrt n}\Big)
=\rho_t^2+O_\delta\Big(\frac{\rho_t^2}{\log^8n}\Big),
\end{align*}
where the first equality uses \eqref{eq:u_closeness_phase_2} and \eqref{eq:uv_closeness_phase_2}, and the second uses $|\rho_t|\geq n^{-1/4}\log^{10}n$ which holds in the current phase. For sufficiently large $n$, this is at least $\rho_t^2/2>n^{-1/2}$, so the truncation in $\hat\mu_t$ is inactive. Taking square roots gives the desired conclusion and, in particular,
\begin{align}\label{eq:phase2_mu_comparable}
\frac12|\rho_t|\leq\hat\mu_t\leq\frac32|\rho_t|.
\end{align}

\noindent\textbf{Proof of (\ref{eq:phase_2_norm}) and (\ref{eq:ht_second})}
Using that $|\tanh x-x|\leq C|x|^3$ for all $x\in\R$ and $\hat\mu_t\leq1$, we have
\begin{align*}
h_t^\star(u,y)=\hat\mu_t(y-1)u+\hat\mu_t^3yu+O\big(\hat\mu_t^3y^2|u|^3\big),
\end{align*}
which implies that
\begin{align*}
\Big|\frac{\pnorm{h_t^\star(\u^t,\y)}{}}{\sqrt d}-\sqrt{2\delta}\hat\mu_t\Big|
&\leq\underbrace{\hat\mu_t\Big|\frac{\pnorm{(\y-1)\circ\u^t}{}}{\sqrt d}-\sqrt{2\delta}\Big|}_{(I)}+\underbrace{\frac{C\hat\mu_t^3}{\sqrt d}\Big(\pnorm{\y\circ\u^t}{}+\pnorm{\y^{\circ2}\circ(\u^t)^{\circ3}}{}\Big)}_{(I\!I)}.
\end{align*}
To bound $(I)$, first note that we can replace $\u^t$ by $\tilde\u^t$, i.e.,
\begin{align*}
\Big|\frac{\pnorm{(\y-1)\circ\u^t}{}}{\sqrt d}-\sqrt{2\delta}\Big|
&\leq\Big|\frac{\pnorm{(\y-1)\circ\tilde\u^t}{}}{\sqrt d}-\sqrt{2\delta}\Big|+O\Big(\frac{\log n}{\sqrt d}\pnorm{\Delta_{U,t}}{}\Big),
\end{align*}
by the same computation as in \eqref{eq:claim_comp_1}. Also, \eqref{eq:gamma_norm_loo} gives
\begin{align}\label{eq:phase2_gamma_universal_radius}
\Big|\pnorm{\gamma_{t,[0:t]}}{}-1\Big|\leq C_\delta\frac{\log^{12}n}{\sqrt n},
\qquad \frac12\leq\pnorm{\gamma_{t,[0:t]}}{}\leq2.
\end{align}
Using that $\y=(\bbeta^\star)^{\circ2}$ and $\tilde\u^t=\rho_t\bbeta^\star+\sum_{\ell=0}^t\gamma_{t,\ell}\g^\ell$, the same calculation as in \eqref{eq:claim_comp_2} gives
\begin{align*}
\frac{\pnorm{(\y-1)\circ\tilde\u^t}{}^2}{d} = 10\delta\rho_t^2+2\delta\pnorm{\gamma_{t,[0:t]}}{}^2+O_\delta\Big(\sqrt{\frac{t+\log n}{n}}\Big) = 2\delta\bigg(1+5\rho_t^2+O_\delta\Big(\frac{\log^{12}n}{\sqrt n}\Big)\bigg).
\end{align*}
Here the first equality uses Gaussian polynomial concentration, Lemma \ref{lem:near_isometry}, and the same conditional Gaussian bound for the cross term as in Phase I, uniformly over $\pnorm{\gamma_{t,[0:t]}}{}\leq2$. The second equality applies \eqref{eq:gamma_norm_loo}.

Thus, for sufficiently large $n$, we have
\begin{align}
\Big|\frac{\pnorm{(\y-1)\circ\u^t}{}}{\sqrt d}-\sqrt{2\delta}\Big|
&\leq\frac{\Big|\frac{1}{d}\pnorm{(\y-1)\circ\tilde\u^t}{}^2-2\delta\Big|}
{\frac{1}{\sqrt d}\pnorm{(\y-1)\circ\tilde\u^t}{}+\sqrt{2\delta}}
+\frac{C\log n}{\sqrt d}\pnorm{\Delta_{U,t}}{}\notag\\
&\leq C_\delta\rho_t^2 + C_\delta\frac{\log^{13}n}{\sqrt n} \leq C_\delta\rho_t^2.
\label{eq:phase2_weighted_norm_deviation}
\end{align}
In summary, we may bound $(I)$ by $C_\delta\hat\mu_t\rho_t^2$.

For $(I\!I)$, applying the net argument used to prove \eqref{eq:phase1_weighted_moments}, now over the set $\{c\in\R^{t+2}:\pnorm{c}{}\leq3\}$, gives
\begin{align}\label{eq:phase2_tilde_weighted_moments}
\frac{\pnorm{\y\circ\tilde\u^t}{}+\pnorm{\y^{\circ2}\circ(\tilde\u^t)^{\circ3}}{}}{\sqrt d}\leq C_\delta, \qquad \pnorm{\tilde\u^t}{\infty}\leq C\log n.
\end{align}
The last bound follows from Gaussian row-norm concentration. Expanding, the cost of replacing $\u^t$ by $\tilde \u^t$ is at most 
\begin{align}
\frac{1}{\sqrt d}\Bigpnorm{\y^{\circ2}\circ\big((\u^t)^{\circ3}-(\tilde\u^t)^{\circ3}\big)}{}&=\frac{1}{\sqrt d}\Bigpnorm{\y^{\circ2}\circ\Big(3(\tilde\u^t)^{\circ2}\circ\Delta_{U,t}
+3\tilde\u^t\circ\Delta_{U,t}^{\circ2}
+\Delta_{U,t}^{\circ3}\Big)}{}\notag\\
&\quad\leq\frac{C}{\sqrt d}\Big(
\log^4n\pnorm{\Delta_{U,t}}{}
+\log^3n\pnorm{\Delta_{U,t}}{}^2
+\log^2n\pnorm{\Delta_{U,t}}{}^3\Big)\notag\\
&\quad\leq C_\delta\frac{\log^{38}n}{\sqrt d}\leq1.
\label{eq:phase2_weighted_moment_replacement}
\end{align}
Thus, since $\delta>1/2$,
\begin{align}\label{eq:phase2_weighted_moments}
\frac{\pnorm{\y\circ\u^t}{}+\pnorm{\y^{\circ2}\circ(\u^t)^{\circ3}}{}}{\sqrt d}\leq C_\delta,
\end{align}
and hence $(I\!I)\leq C_\delta\hat\mu_t^3\leq C_\delta\hat\mu_t\rho_t^2$, applying (\ref{eq:phase_2_mu_est}) in the last step. Combining the estimates of $(I)$ and $(I\!I)$ gives
\begin{align*}
\Big|\frac{\pnorm{h_t^\star(\u^t,\y)}{}}{\sqrt d}-\sqrt{2\delta}\hat\mu_t\Big|\leq C_\delta\hat\mu_t\rho_t^2,
\end{align*}
as desired. For sufficiently small $c_{\weak}$, $\pnorm{h_t^\star(\u^t,\y)}{}/\sqrt d$ is at least $\sqrt{2\delta}\hat\mu_t/2$, and
\begin{align}\label{eq:phase2_pi_sharp}
|\pi_t|=\Big|\frac{\sqrt{2\delta}\hat\mu_t}{\pnorm{h_t^\star(\u^t,\y)}{}/\sqrt d}-1\Big|\leq C\rho_t^2.
\end{align}
Equation \eqref{eq:ht_second} follows in the same way as \eqref{eq:ht_first}.

\noindent\textbf{Proof of (\ref{eq:phase_2_HF})}
$F_{t+1}\big(\rho_{t+1},\pnorm{\lambda_{t,[0:t]}}{};\btheta^\star\big)=1$ follows by the same proof as (\ref{eq:phase_1_HF}). For $H_t$, let $\bar\u^t=\rho_t\bbeta^\star+\pnorm{\gamma_{t,[0:t]}}{}\tilde\g$, where $\tilde\g\sim\N(0,\I_n)$. The same calculation as in the proof of (\ref{eq:phase_1_HF}) gives
\begin{align}\label{eq:claim_h_phase2}
h_t'(\bar\u^t,\y)=\frac{\y-1}{\sqrt{2\delta}}+\e^t,
\qquad \frac{1}{d}\E_{\tilde\g}\pnorm{\e^t}{}^2\leq C\big(\pi_t^2+\hat\mu_t^4\big)\leq C\rho_t^4.
\end{align}
Here we use $\tanh^2x\leq x^2$, $|\pi_t|\leq C\rho_t^2$, $\hat\mu_t\leq3|\rho_t|/2$, and $\pnorm{\gamma_{t,[0:t]}}{}\leq2$; the remaining averages are bounded by empirical Gaussian polynomial moments of $\bbeta^\star$, with universal constants.

This entails that
\begin{align*}
H_t\big(\rho_t,\pnorm{\gamma_{t,[0:t]}}{};\bbeta^\star\big)
&=\frac{1}{d}\E_{\tilde\g}\pnorm{h_t'(\bar\u^t,\y)}{}^2\\
&=\frac{1}{2n}\pnorm{\y-1}{}^2+\frac{1}{d}\E_{\tilde\g}\pnorm{\e^t}{}^2
+\frac{2}{d}\E_{\tilde\g}\Big\langle\frac{\y-1}{\sqrt{2\delta}},\e^t\Big\rangle\\
&=1+O\Big(\frac{\log^4n}{\sqrt n}\Big)+O(\rho_t^4)+O(\rho_t^2)
=1+O(\rho_t^2),
\end{align*}
where we use Cauchy--Schwarz for the cross term and the Phase-II lower bound on $\rho_t$ in the last equality.

\noindent\textbf{Proof of (\ref{eq:phase_2_exp_growth})}
We start with
\begin{align*}
h_t(u,y)
&=\frac{h_t^\star(u,y)}{\sqrt{2\delta}\hat\mu_t}\frac{\sqrt{2\delta}\hat\mu_t}{\pnorm{h_t^\star(\u^t,\y)}{}/\sqrt d}
=(1+O(\rho_t^2))\frac{h_t^\star(u,y)}{\sqrt{2\delta}\hat\mu_t},
\end{align*}
where we use \eqref{eq:phase_2_norm}. By first-order Taylor expansion, we know that
\begin{align*}
h_t^\star(u,y)=\hat\mu_t(y-1)u+O\Big(\hat\mu_t^3\big(y|u|+y^2|u|^3\big)\Big).
\end{align*}
Thus, noting that $\hat\mu_t=O(|\rho_t|)$ by \eqref{eq:phase_2_mu_est} and applying the weighted moment bounds above, we get
\begin{align*}
\frac{h_t^\star(\u^t,\y)}{\sqrt{2\delta}\hat\mu_t}
=\frac{(\y-1)\circ\u^t}{\sqrt{2\delta}}+\e_0^t,
\qquad \frac{\pnorm{\e_0^t}{}}{\sqrt d}=O(\rho_t^2).
\end{align*}
This further implies that
\begin{align}\label{eq:ht_expansion_phase2}
h_t(\u^t,\y)=\frac{(\y-1)\circ\u^t}{\sqrt{2\delta}}+\tilde\e^t,
\qquad \frac{\pnorm{\tilde\e^t}{}}{\sqrt d}=O(\rho_t^2).
\end{align}
Using this computation and the definition of $\rho_t$ in (\ref{eq:key_signal}), we have
\begin{align}
\rho_{t+1} =\frac{1}{d}\iprod{\g^{-1}}{h_t(\u^t,\y)}-\hat c_t\rho_t =\underbrace{\frac{1}{\sqrt{2\delta}\,d}\iprod{\g^{-1}}{(\y-1)\circ\u^t}}_{(i)}
+\underbrace{\frac{1}{d}\iprod{\g^{-1}}{\tilde\e^t}}_{(ii)}
-\underbrace{\hat c_t\rho_t}_{(iii)}.
\end{align}
To bound $(i)$, we replace $\u^t$ by $\tilde\u^t$, then use uniform Gaussian concentration as in the proof of Proposition \ref{prop:phase_1_main}. As a result, we get
\begin{align*}
(i)
&=\frac{1}{\sqrt{2\delta}\,d}\iprod{\g^{-1}}{(\y-1)\circ\tilde\u^t}+O_\delta\Big(\frac{\log^{12}n}{\sqrt d}\Big)\\
&=\frac{\rho_t}{\sqrt{2\delta}\,d}\iprod{\g^{-1}}{(\y-1)\circ\g^{-1}}
+\frac{1}{\sqrt{2\delta}\,d}\Big\langle\g^{-1},(\y-1)\circ\sum_{\ell=0}^t\gamma_{t,\ell}\g^\ell\Big\rangle
+O_\delta\Big(\frac{\log^{12}n}{\sqrt d}\Big)\\
&=\sqrt{2\delta}\rho_t+O_\delta\Big(|\rho_t|\sqrt{\frac{\log n}{n}}+\sqrt{\frac{t+\log n}{n}}+\frac{\log^{12}n}{\sqrt d}\Big)\\
&=\sqrt{2\delta}\rho_t+O_\delta\Big(\frac{\log^{12}n}{\sqrt d}\Big).
\end{align*}
For the replacement, use $\pnorm{(\y-1)\circ\g^{-1}}{}\leq C\sqrt n$. Conditional on $\g^{-1}$, the inner products in the second term form a Gaussian vector, so its bound holds uniformly over $\pnorm{\gamma_{t,[0:t]}}{}\leq2$.

For $(ii)$, we use a cancellation to improve the bound in \eqref{eq:ht_expansion_phase2}. For fixed $\mu,\pi$, write
\begin{align*}
\mathsf e_{\mu,\pi}(u,y)
\coloneqq\frac{1+\pi}{\sqrt{2\delta}\,\mu}h_\mu^\star(u,y)-\frac{(y-1)u}{\sqrt{2\delta}},
\end{align*}
where $h_\mu^\star$ is $h_t^\star$ with $\hat\mu_t$ replaced by $\mu$. Then $\tilde\e^t=\mathsf e_{\hat\mu_t,\pi_t}(\u^t,\y)$. Differentiating, we get
\begin{align}
\label{eq:e_deriv}
\partial_u\mathsf e_{\mu,\pi}(u,y)
&=\frac{1}{\sqrt{2\delta}}\Big(\pi(y-1)+(1+\pi)\mu^2y
-(1+\pi)(1+\mu^2)y\tanh^2(\mu u\sqrt y)\Big),\\
|\partial_u\mathsf e_{\mu,\pi}(u,y)|
&\leq\frac{C}{\sqrt\delta}\Big(|\pi|(1+y)+\mu^2y+\mu^2y^2u^2\Big),\notag
\end{align}
for $0\leq\mu\leq1$ and $|\pi|\leq1$. For independent standard Gaussians $G,W$, symmetry gives $\E[G\mathsf e_{\mu,\pi}(\sigma W,G^2)]=0$. Keeping $\mu,\pi,\sigma$ fixed and integrating the derivative in $\rho$ therefore yields
\begin{align*}
\delta\Big|\E\big[G\mathsf e_{\mu,\pi}(\rho G+\sigma W,G^2)\big]\Big| =\delta\Big|\int_0^\rho\E\big[G^2\partial_u\mathsf e_{\mu,\pi}(rG+\sigma W,G^2)\big]\,dr\Big| \leq C_\delta(|\pi|+\mu^2)|\rho|,
\end{align*}
uniformly for $|\rho|\leq1$ and $0\leq\sigma\leq2$. 

To apply this estimate to the random parameters, take an $n^{-20}$-net of $\Theta = (\rho,\mu,\pi,\gamma)$ satisfying
\begin{align*}
|\rho|\leq c_{\weak},\qquad 0\leq\mu\leq2c_{\weak},\qquad
|\pi|\leq Cc_{\weak}^2,\qquad \pnorm\gamma{}\leq2.
\end{align*}
Its logarithmic cardinality is $O((t+1)\log n)=O(\log^3n)$. Consider the summands
\begin{align*}
Z_i(\Theta):= g_i^{-1}\mathsf e_{\mu,\pi}\Big(\rho g_i^{-1}+\sum_{\ell=0}^t\gamma_\ell g_i^\ell,y_i\Big)
\end{align*}
and the truncated version
\begin{align*}
\tilde Z_i(\Theta)
&\coloneqq Z_i\mathbf 1_{\{\max_{-1\leq\ell\leq t}|g_i^\ell|\leq K\sqrt{\log n}\}}
\end{align*}
where $K$ is a sufficiently large universal constant. At any fixed net point $\Theta$, Bernstein's inequality gives, with probability at least $1-2\exp\{-C_1\log^3n\}$,
\begin{align*}
\Bigg|\frac{1}{d}\sum_{i=1}^n(\tilde Z_i(\Theta)-\E\tilde Z_i(\Theta))\Bigg|
\leq C_\delta\Big(\frac{\log^{3/2}n}{\sqrt n}+\frac{\log^{10}n}{n}+n^{-10}\Big).
\end{align*}
Moreover, a direct computation shows that, for any two parameter tuples $\Theta,\bar\Theta$ with $\pnorm{\Theta-\bar\Theta}{}\leq n^{-20}$,
\begin{align*}
\Bigg|\frac{1}{d}\sum_{i=1}^n
\Big[
\tilde Z_i(\Theta)-\E\tilde Z_i(\Theta)
-\tilde Z_i(\bar\Theta)+\E\tilde Z_i(\bar\Theta)
\Big]\Bigg|
\leq C_\delta n^{-10}.
\end{align*}
Taking $C_1$ large enough, a union bound over the net, together with this comparison between nearby tuples, gives, uniformly over all $\Theta$,
\begin{align*}
\Bigg|\frac{1}{d}\sum_{i=1}^n(\tilde Z_i(\Theta)-\E\tilde Z_i(\Theta))\Bigg|
\leq C_\delta\Big(\frac{\log^{3/2}n}{\sqrt n}+\frac{\log^{10}n}{n}+n^{-10}\Big).
\end{align*}

With probability $1-O(n^{-20})$, the truncation is inactive by standard Gaussian maximal inequalities. Moreover, since $\mathsf e_{\mu,\pi}(0,y)=0$ and $|\partial_u\mathsf e_{\mu,\pi}(u,y)|\leq C_\delta(1+y)$ by \eqref{eq:e_deriv}, Gaussian moments give $\E Z_i(\Theta)^2\leq C_\delta$ uniformly over $\Theta$. Cauchy--Schwarz therefore gives
\begin{align*}
|\E Z_i-\E\tilde Z_i| \leq \sqrt{\E Z_i^2}\;\P\Big(\max_{-1\leq\ell\leq t}|g_i^\ell|>K\sqrt{\log n}\Big)^{1/2} \leq C_\delta n^{-10}
\end{align*}
for $K$ large enough. On this event where the truncation is inactive, we have
\begin{align*}
\Bigg|\frac{1}{d}\sum_{i=1}^n(Z_i-\E Z_i)\Bigg|
&\leq\Bigg|\frac{1}{d}\sum_{i=1}^n(\tilde Z_i-\E\tilde Z_i)\Bigg|
+\frac{1}{d}\sum_{i=1}^n\big|\E\tilde Z_i-\E Z_i\big|\\
&\leq C_\delta\Big(\frac{\log^{3/2}n}{\sqrt n}+\frac{\log^{10}n}{n}+n^{-10}\Big) \leq C_\delta\frac{\log^{12}n}{\sqrt d}.
\end{align*}

Therefore, with probability $1-O(n^{-12})$, uniformly over the parameter $\Theta$,
\begin{align}
&\Bigg|\frac{1}{d}\sum_{i=1}^n
g_i^{-1}\mathsf e_{\mu,\pi}\Big(\rho g_i^{-1}+\sum_{\ell=0}^t\gamma_\ell g_i^\ell,y_i\Big)-\delta\E\big[G\mathsf e_{\mu,\pi}(\rho G+\pnorm{\gamma}{}W,G^2)\big]\Bigg|
\leq C_\delta\frac{\log^{12}n}{\sqrt d}.
\label{eq:phase2_uniform_signed_remainder}
\end{align}

It remains to replace $\tilde\u^t$ by $\u^t$. The derivative computation \eqref{eq:e_deriv} also gives $|\partial_u\mathsf e_{\mu,\pi}(u,y)|\leq C_\delta(1+y)$, so
\begin{align}
\label{eq:replace_u_tilde}
\frac{1}{d}\Big|\Big\langle\g^{-1},\mathsf e_{\hat\mu_t,\pi_t}(\u^t,\y)-\mathsf e_{\hat\mu_t,\pi_t}(\tilde\u^t,\y)\Big\rangle\Big| \leq\frac{C_\delta(1+\pnorm\y{\infty})\pnorm{\g^{-1}}{}\pnorm{\Delta_{U,t}}{}}{d}\leq C_\delta\frac{\log^{13}n}{\sqrt d}.
\end{align}
Applying \eqref{eq:phase2_uniform_signed_remainder} with $\mu=\hat\mu_t$, $\pi=\pi_t$, $\rho=\rho_t$, and $\gamma=\gamma_{t,[0:t]}$, and using \eqref{eq:replace_u_tilde} to replace $\tilde\u^t$ by $\u^t$, gives
\begin{align}\label{eq:phase2_signed_taylor_remainder}
|(ii)|=\Big|\frac{1}{d}\iprod{\g^{-1}}{\tilde\e^t}\Big| \leq C_\delta|\rho_t|^3+C_\delta\frac{\log^{13}n}{\sqrt d}.
\end{align}

It remains to bound the Onsager correction term. Using \eqref{eq:ht_second}, we have
\begin{align}\label{eq:phase2_onsager_bound}
|\hat c_t| &= \delta\big|\av{h_t'(\u^t,\y)}\big| \leq C_\delta\Big(\big|\av{\y-1}\big|+\hat\mu_t^2\av\y + \av{\y\circ\tanh^2(\hat\mu_t\u^t\circ\sqrt{\y})}\Big)\notag\\
&\leq C_\delta\Big(\frac{\log n}{\sqrt n}+\hat\mu_t^2 + \hat\mu_t^2\av{\y^{\circ2}\circ(\u^t)^{\circ2}}\Big)
\leq C_\delta\rho_t^2.
\end{align}
Here we use $\tanh^2x\leq x^2$, $\hat\mu_t=O(|\rho_t|)$, the weighted moment bound \eqref{eq:phase2_weighted_moments}, and the lower bound on $\rho_t$ in Phase II. Then we conclude $|(iii)|\leq  C_\delta|\rho_t|^3$.

Combining the bounds for $(i),(ii),(iii)$, we get the recursion
\begin{align*}
\big|\rho_{t+1}-\sqrt{2\delta}\rho_t\big| \leq  C_\delta|\rho_t|^3 + C_\delta\frac{\log^{13}n}{\sqrt d},
\end{align*}
which implies \eqref{eq:phase2_cubic_recursion}. For $c_{\weak}$ small enough, we have $ C_\delta c_{\weak}^2\leq(\sqrt{2\delta}-1)/4$. Taking $n$ sufficiently large also gives
\begin{align*}
C_\delta\frac{\log^{13}n}{\sqrt d}
\leq\frac{\sqrt{2\delta}-1}{4}n^{-1/4}\log^{10}n
\leq\frac{\sqrt{2\delta}-1}{4}|\rho_t|.
\end{align*}
Consequently, we have $|\rho_{t+1}|\geq \Big(\frac{1+\sqrt{2\delta}}{2}\Big)|\rho_t|$, as desired.
\end{proof}

\subsection{Proof of Lemma \ref{lem:phase_2_simp}}
\begin{proof}
We will simplify the bounds in $\cH_t^U$ in (\ref{eq:induction_Delta_U})--(\ref{eq:induction_alpha_middle}) under the refined assumptions up to $\cA_{I\!I,t-1}^V$. The corresponding simplification under $\cH_t^V$ and the refined assumptions up to $\cA_{I\!I,t}^U$ is similar and thus omitted. We first note that the following claims hold:
\begin{align}
\pnorm{f_t'}{\infty}&=1,\qquad \pnorm{f_t''}{\infty}=0,\label{eq:temp_phase2_1}\\
\sqrt{F_t\big(\rho_t,\pnorm{\lambda_{t-1,[0:t-1]}}{};\btheta^\star\big)}&=1,\qquad \pnorm{\lambda_{t-1,[0:t-1]}}{}=1,\label{eq:temp_phase2_2}\\
\sqrt{H_s\big(\rho_s,\pnorm{\gamma_{s,[0:s]}}{};\bbeta^\star\big)}
&=1+O(\rho_s^2)+O_\delta\Big(\frac{\log^{20}n}{\sqrt n}\Big),\qquad 0\leq s<t,\label{eq:temp_phase2_5}\\
\sum_{k=0}^{t-1}|\beta_{t-1}^k|&\leq C_\delta\frac{t\log^{10}n}{\sqrt d},\label{eq:temp_phase2_3}\\
\bar\gamma_t\vee\bar\lambda_{t-1}&\leq C_\delta\sqrt d.\label{eq:temp_phase2_4}
\end{align}
Here (i): The refined assumptions $\cA_{I,s}^U$ for $0\leq s<\tau_I$ and $\cA_{I\!I,s}^U$ for $\tau_I\leq s<t$ hold by the induction hypothesis, so (\ref{eq:temp_phase2_1})--(\ref{eq:temp_phase2_5}) follow from $f_s=\operatorname{id}$, $\pnorm{h_s(\u^s,\y)}{}=\sqrt d$ from \eqref{eq:rescaled_nonlinearity}, and Lemmas \ref{lem:phase_1_assump_conseq} and \ref{lem:phase_2_assump_conseq}. (ii) Claim (\ref{eq:temp_phase2_3}) holds by (\ref{eq:beta_sum_phase_2}), with the Phase-I bound used at $t=\tau_I$. (iii) The bound on $\bar\gamma_t$ in \eqref{eq:temp_phase2_4} follows from \eqref{eq:gamma_norm} of Lemma~\ref{lem:uv_norm}, applied to $\gamma_{s,[-1:s]}$ for $0\leq s\leq t$, using the bounds on $\pnorm{\Delta_{V,s}}{}$ in \eqref{eq:uv_closeness_phase_1} and \eqref{eq:uv_closeness_phase_2}. The identity $\pnorm{h_s(\u^s,\y)}{}=\sqrt d$ from \eqref{eq:rescaled_nonlinearity}, for $0\leq s<t$, together with \eqref{eq:global_current_lambda_control}, gives $\bar\lambda_{t-1}\leq C_\delta\sqrt d$.

As in the proof of Lemma \ref{lem:phase_1_simp}, we have
\begin{align*}
\max_{0\leq k\leq t-1}\pnorm{\Delta_{V,k}}{}\sum_{k=0}^{t-1}|\beta_{t-1}^k|
\leq C_\delta\frac{t^2\log^{20}n}{\sqrt d}.
\end{align*}
To bound $\cE_t^f$ and $\cE_s^h$ in \eqref{eq:E_multiplier}, use \eqref{eq:phase_1_assump_U_1} and \eqref{eq:phase_2_assump_U_1}, which give $\pnorm{h_{[0:s]}'}{\infty}\leq C_\delta\log n$ and $\pnorm{h_s''}{\infty}\leq C_\delta\log^{3/2}n$ for $0\leq s<t$ together with $\pnorm{\Delta_{U,s}}{}+\pnorm{\hat\u^s-\tilde\u^s}{}\leq C_\delta s\log^{10}n$.  Substitution into (\ref{eq:E_multiplier}), using $f_t''=0$ and $t+1\leq\log^2n$, gives
\begin{align*}
\cE_t^f&\leq C_\delta\frac{\log^2n}{n^{1/4}},\qquad
\cE_s^h\leq C_\delta\Big(\frac{\log^4n}{n^{1/4}}+\frac{\log^{16}n}{\sqrt n}\Big),\qquad 0\leq s<t.
\end{align*}
Thus
\begin{align}\label{eq:phase2_multiplier_error}
\cE_t^f\vee\max_{0\leq s<t}\cE_s^h
\leq C_\delta\frac{\log^{16}n}{n^{1/4}}.
\end{align}
Substituting these estimates into (\ref{eq:induction_Delta_U}) and (\ref{eq:induction_alpha_last}), and using further
\begin{align*}
\sum_{k=0}^{t-1}|\beta_{t-1}^k|\pnorm{\Delta_{V,k}}{} \leq C_\delta\frac{t^2\log^{20}n}{\sqrt d} \leq C_\delta t\log^{5/2}n
\end{align*}
as in the proof of Lemma \ref{lem:phase_1_simp}, proves \eqref{eq:phase_2_simp_1} and \eqref{eq:phase_2_simp_2}. 

To prove \eqref{eq:phase_2_simp_3}, we bound $\alpha_{t-1}^{t-2}$, when $t\geq2$, by
\begin{align*}
|\alpha_{t-1}^{t-2}|
&\leq\pnorm{f_t'}{\infty}\pnorm{f_{t-1}'}{\infty}|\beta_{t-1}^{t-1}|
=|\beta_{t-1}^{t-1}|.
\end{align*}
For $0\leq k\leq t-3$, applying (\ref{eq:induction_alpha_middle}) with (\ref{eq:temp_phase2_2}), (\ref{eq:temp_phase2_5}), and \eqref{eq:phase2_multiplier_error}, we obtain
\begin{align*}
|\alpha_{t-1}^k|
&\leq\Bigg(1+C(\rho_{k+1}^2+\rho_{t-1}^2)
+C_\delta\frac{\log^{20}n}{\sqrt n}
+C_\delta\frac{\log^{16}n}{n^{1/4}}\Bigg)|\alpha_{t-2}^{k+1}|\\
&\leq\Big(1+C(\rho_{k+1}^2+\rho_{t-1}^2)
+C_\delta\frac{\log^{16}n}{n^{1/4}}\Big)|\alpha_{t-2}^{k+1}|,
\end{align*}
proving (\ref{eq:phase_2_simp_3}).
\end{proof}

\section{{Proofs for intermediate recovery}}
\label{sec:proof_phase_3}

\subsection{Proof of Proposition \ref{prop:phase_3_asymp}}\label{subsec:proof_fixed_point}
Define the function
\begin{align}\label{eq:scalar_G_definition}
G(\mu)\coloneqq\frac{\mu\delta}{F_\delta(\mu)}
=\frac{\mu}{(\mu+1)\big((\mu+1)m(\mu)-\mu\big)},\qquad \mu>0.
\end{align}
The Bayes identity \eqref{eq:bayes_fact_1}, proved below, shows that the denominator is positive.

\begin{lemma}\label{lem:scalar_G_geometry}
The function $G$ has a unique critical point $\mu_\star>0$, with
\begin{gather}\label{eq:scalar_G_sign}
G'(\mu)
\begin{cases}
>0 &  0<\mu<\mu_\star\\
= 0 & \mu = \mu_\star\\
< 0 & \mu > \mu_\star
\end{cases},
\qquad G''(\mu_\star)<0.
\end{gather}
Numerically, $\mu_\star\approx5.52$ and $G(\mu_\star)=\delta_{\str}\approx1.13$.
\end{lemma}
\begin{proof}
The uniqueness, derivative signs, and nondegeneracy of the maximum are supported by numerical evaluation of $G$ and its derivatives. We omit the numerical details.
\end{proof}

\begin{proof}[Proof of Proposition \ref{prop:phase_3_asymp}]
It can be readily checked that $G(\cdot)$ satisfies the following properties (these can be visually checked from Figure \ref{fig:G(mu)}):
\begin{itemize}
\item Using $m'(0_+)=3$ and $1-m(\mu)=O(\mu^{-3/2})$, proved below, we have $G(0_+)=1/2$ and $G(\mu)\rightarrow1$ as $\mu\rightarrow\infty$. The same calculation at zero gives $G'(0_+)=5/2$.
\item $G(\cdot)$ is twice continuously differentiable on $(0,\infty)$ and, by Lemma \ref{lem:scalar_G_geometry}, has a unique critical point (i.e., global maximum) at $\mu_\star\approx5.52$ with function value $G(\mu_\star)=\delta_{\str}\approx1.13$.
\item For every strictly positive fixed point $\mu_{\fix}$ of $F_\delta(\mu)$, it holds that
\begin{align}\label{eq:F_derivative_G}
F_\delta'(\mu_{\fix})=1-\frac{\mu_{\fix}G'(\mu_{\fix})}{G(\mu_{\fix})}.
\end{align}
To see this, note that $F_\delta(\mu) = \mu\delta/G(\mu)$, hence
\begin{align*}
F_\delta'(\mu)
&=\frac{\delta}{G(\mu)}-\frac{\mu\delta G'(\mu)}{G(\mu)^2}
=\frac{\delta}{G(\mu)}\Big(1-\frac{\mu G'(\mu)}{G(\mu)}\Big),
\end{align*}
and (\ref{eq:F_derivative_G}) follows since $G(\mu_{\fix}) = \delta$ for any positive fixed point $\mu_{\fix}$.
\end{itemize}

\begin{figure}[H]
\centering
\includegraphics[width=0.6\textwidth]{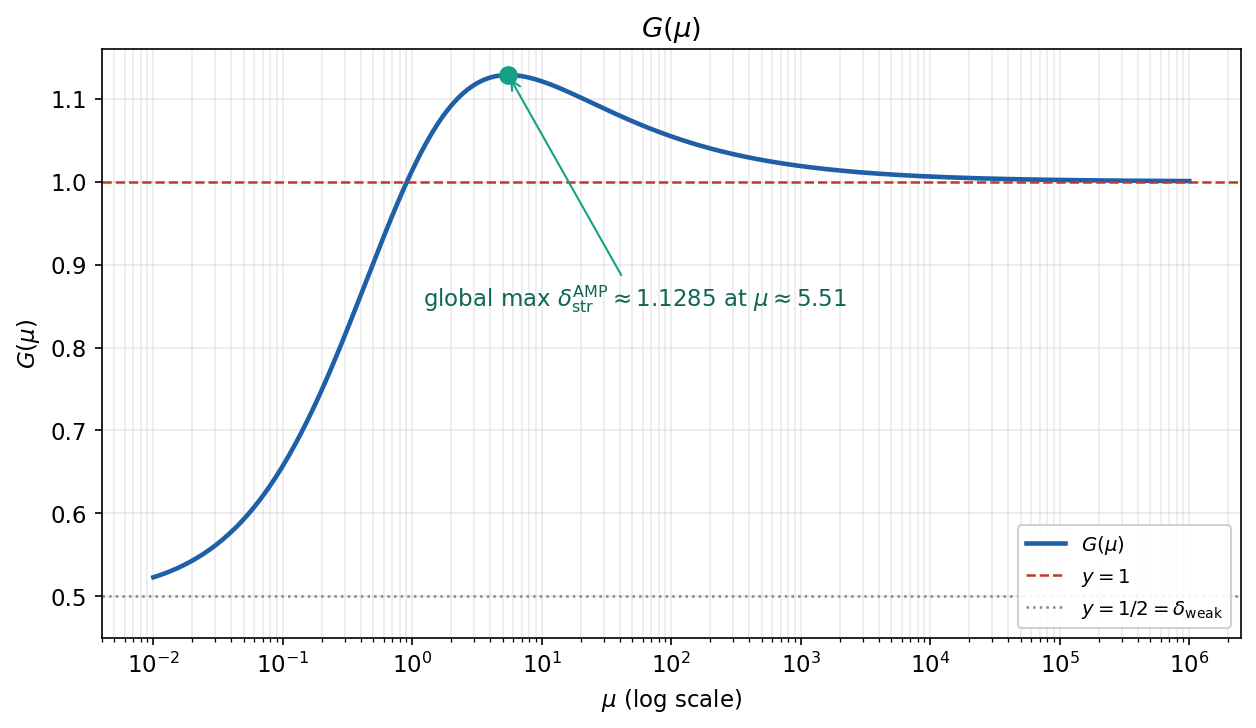}
\caption{Plot of $G(\mu)$}
\label{fig:G(mu)}
\end{figure}

We note that any strictly positive fixed point $\mu_{\fix}$ of $\mu\mapsto F_\delta(\mu)$ satisfies that
\begin{align*}
G(\mu_{\fix})=\delta.
\end{align*}
Conversely, any positive solution of this equation is a fixed point. The Bayes calculation below also gives $F_\delta'(\mu)\geq0$.

\noindent\textbf{Case $\delta\leq\delta_{\weak}$}
Since $G(\mu)>1/2$ for any $\mu>0$, $F_\delta$ admits no strictly positive fixed point. At the fixed point $0$, we have
\begin{align*}
F_\delta'(0)=\delta\big(m'(0_+)-1\big)=2\delta\leq1,
\end{align*}
yielding that $0$ is a stable fixed point. In fact, $0<F_\delta(\mu)<\mu$ for every $\mu>0$ in this regime.

\noindent\textbf{Case $\delta>\delta_{\str}$}
Using the same reasoning as above, $F_\delta$ has no strictly positive fixed point, and $F_\delta'(0)=2\delta>1$, hence $0$ is unstable. Moreover, by definition,
\begin{align*}
\frac{F_\delta(\mu)}{\mu}=\frac{\delta}{G(\mu)}\geq\frac{\delta}{\delta_{\str}},\qquad \mu>0,
\end{align*}
yielding the desired geometric growth. The inequality $F_\delta(\mu)\geq(\delta/\delta_{\str})\mu$ also holds at $\mu=0$.

\noindent\textbf{Case $\delta\in(\delta_{\weak},\delta_{\str}]$}
When $\delta\in(\delta_{\weak},1]$, $G(\mu)=\delta$ has a single root $\mu_{\fix}$ in $(0,\mu_\star)$ (recall that $\mu_\star$ is the global argmax of $G$), hence $F_\delta$ has a unique positive fixed point $\mu_\infty(\delta)=\mu_{\fix}$. Moreover, $\mu_{\fix}$ is stable since $G'(\mu_{\fix})>0$, and hence $0\leq F_\delta'(\mu_{\fix})<1$ by (\ref{eq:F_derivative_G}).

When $\delta\in(1,\delta_{\str})$, $G(\mu)=\delta$ has two roots $\mu_\infty(\delta)<\widetilde\mu_\infty(\delta)$ on either side of $\mu_\star$. Moreover, $G'(\mu_\infty(\delta))>0$ while $G'(\widetilde\mu_\infty(\delta))<0$, hence using again (\ref{eq:F_derivative_G}), $\mu_\infty(\delta)$ is stable while $\widetilde\mu_\infty(\delta)$ is unstable. As $\delta\rightarrow\delta_{\str}$ from below, the roots merge at $\mu_\star\approx5.52$. At $\delta=\delta_{\str}$, this is the unique positive fixed point and $F_\delta'(\mu_\star)=1$, so it is neutral. Zero remains unstable throughout this regime since $F_\delta'(0)=2\delta>1$.

Lastly, we establish the connection with $\mathcal B(\rho)$. We claim that, for $\rho\geq0$,
\begin{align}
\delta\E\big[h_\rho^\star(U,Y)^2\big]&=F_\delta(\rho^2),\label{eq:bayes_fact_1}\\
\delta\E\big[\{\partial_u h_\rho^\star(U,Y)\}^2\big]&=\rho^2F_\delta'(\rho^2),\label{eq:bayes_fact_2}
\end{align}
where $U = \rho G + W$ and $Y = G^2$, with $G,W$ being independent $\N(0,1)$ variables. Assuming these claims are true, for $\rho_{\fix} = \sqrt{\mu_{\fix}}$ with $\mu_{\fix}$ a positive fixed point, we have
\begin{align*}
\mathcal B(\rho_{\fix})
=\frac{\rho_{\fix}^2F_\delta'(\rho_{\fix}^2)}{F_\delta(\rho_{\fix}^2)}
=F_\delta'(\mu_{\fix})
=1-\frac{\mu_{\fix}G'(\mu_{\fix})}{G(\mu_{\fix})},
\end{align*}
where the last step follows from (\ref{eq:F_derivative_G}). This establishes (\ref{eq:B_claim}). More generally, for every $\rho>0$,
\begin{align*}
\mathcal B(\rho)
=\frac{\rho^2F_\delta'(\rho^2)}{F_\delta(\rho^2)}
=1-\frac{\rho^2G'(\rho^2)}{G(\rho^2)}.
\end{align*}
Together with Lemma \ref{lem:scalar_G_geometry}, this shows that $\mathcal B(\rho)=1$ has the unique positive solution $\rho=\sqrt{\mu_\star}\approx2.35$.

It remains to prove (\ref{eq:bayes_fact_1}) and (\ref{eq:bayes_fact_2}). To this end, let
\begin{align*}
\eta_\rho(u,y)=\E[G\mid U=u,Y=y],\qquad V_\rho(u,y)=\Var(G\mid U=u,Y=y)
\end{align*}
be the posterior mean and variance functions, respectively. Recall from \cite[Eq.~(64)]{guo2011estimation}, applied conditional on $Y$, that
\begin{align}\label{eq:e_relation}
m'(\mu)=\E\big[V_{\sqrt\mu}(U,Y)^2\big].
\end{align}
Here and below, expectations indexed by $\mu$ use $U=\sqrt\mu G+W$. To see (\ref{eq:bayes_fact_1}), note that
\begin{align*}
h_\rho^\star(u,y)=(1+\rho^2)\eta_\rho(u,y)-\rho u,
\end{align*}
hence with $m(\rho^2) = \E[\eta_\rho^2(U,Y)]$ and $\E[\eta_\rho(U,Y)U]=\E[GU]=\rho$ by the tower property,
\begin{align*}
\E\big[h_\rho^\star(U,Y)^2\big]
&=(1+\rho^2)^2\underbrace{\E[\eta_\rho(U,Y)^2]}_{m(\rho^2)}
+\rho^2\underbrace{\E[U^2]}_{\rho^2+1}-2(1+\rho^2)\rho\underbrace{\E[\eta_\rho(U,Y)U]}_{\rho}\\
&=(1+\rho^2)^2m(\rho^2)-(1+\rho^2)\rho^2
=\frac{F_\delta(\rho^2)}{\delta},
\end{align*}
as desired. Lastly, to see (\ref{eq:bayes_fact_2}), let us compute the derivatives on both sides. Using $\partial_u \eta_\rho(u,y) = \rho V_\rho(u,y)$ from (\ref{eq:post_mean_derivative}), we have
\begin{align*}
\partial_u h_\rho^\star(u,y)=(1+\rho^2)\rho V_\rho(u,y)-\rho,
\end{align*}
implying that
\begin{align*}
\E\big[\{\partial_u h_\rho^\star(U,Y)\}^2\big]
=\rho^2\E\Big[\big((1+\rho^2)V_\rho(U,Y)-1\big)^2\Big].
\end{align*}
For the right side, since $F_\delta(\mu) = \delta(1+\mu)\Big((1+\mu)m(\mu) - \mu\Big)$, we have
\begin{align*}
F_\delta'(\mu)
&=\delta\Big[2(1+\mu)m(\mu)+(1+\mu)^2m'(\mu)-(1+2\mu)\Big]\\
&\overset{(\star)}{=}\delta\Big[2(1+\mu)\big(1-\E[V_{\sqrt\mu}(U,Y)]\big)
+(1+\mu)^2\E[V_{\sqrt\mu}(U,Y)^2]-(1+2\mu)\Big]\\
&=\delta\Big[1-2(1+\mu)\E[V_{\sqrt\mu}(U,Y)]
+(1+\mu)^2\E[V_{\sqrt\mu}(U,Y)^2]\Big]\\
&=\delta\E\Big[\big(1-(1+\mu)V_{\sqrt\mu}(U,Y)\big)^2\Big]\geq0,
\end{align*}
where we apply (\ref{eq:e_relation}) and $m(\mu) = 1 - \E[V_{\sqrt\mu}(U,Y)]$ in $(\star)$, the latter following from
\begin{align*}
1=\E[G^2] =\E\big[\E[G^2\mid U,Y]\big] =\E\big[\E[G\mid U,Y]^2\big]+\E\big[\Var(G\mid U,Y)\big] =m(\mu)+\E[V_{\sqrt\mu}(U,Y)].
\end{align*}
This entails that, for $\rho>0$,
\begin{align*}
F_\delta'(\rho^2)
&=\delta\E\Big[\big(1-(1+\rho^2)V_\rho(U,Y)\big)^2\Big]
=\frac{\delta}{\rho^2}\E\big[\{\partial_u h_\rho^\star(U,Y)\}^2\big],
\end{align*}
proving (\ref{eq:bayes_fact_2}). At $\rho=0$, both sides of (\ref{eq:bayes_fact_2}) vanish.

For $\rho>0$, $h_\rho^\star(u,y)$ is not identically zero in $u$, and the conditional law of $U$ given $Y$ has a positive density on $\R$. Thus the expectation in (\ref{eq:bayes_fact_1}) is positive, proving $F_\delta(\mu)>0$ for $\mu>0$.

Finally, we verify the endpoint estimates used above. Since $V_{\sqrt\mu}(u,y)=y\big(1-\tanh^2(\sqrt\mu u\sqrt y)\big)$, equation (\ref{eq:e_relation}) gives
\begin{align}\label{eq:scalar_m_derivatives}
m'(\mu) = \E\Big[G^4\operatorname{sech}^4\big(\mu G|G|+\sqrt\mu W|G|\big)\Big] = 3-30\mu+O(\mu^2),\qquad \mu\downarrow 0.
\end{align}
Here we use $\operatorname{sech}^4z=1-2z^2+O(z^4)$, $\E[G^4]=3$, and $\E[G^6]=15$; finite Gaussian moments bound the expected remainder by $O(\mu^2)$. Since $m(0)=0$, integration gives the first expansion below, and substitution into the formulas for $G$ and $G'$ gives the others:
\begin{align}\label{eq:scalar_G_zero_asymptotics}
m(\mu)&=3\mu-15\mu^2+O(\mu^3),\notag\\
G(\mu)&=\frac12+\frac52\mu+O(\mu^2),\qquad
G'(\mu)=\frac52+O(\mu).
\end{align}
Differentiation under the Gaussian expectations also shows that $m$ and $G$ are twice continuously differentiable on $(0,\infty)$.

At the other endpoint, the posterior mean minimizes mean squared error. Comparing it with $\sqrt Y\operatorname{sign}(U)$ gives
\begin{align*}
1-m(\mu) \leq\E\Big[\big(G-|G|\operatorname{sign}(\sqrt\mu G+W)\big)^2\Big] =4\E\Big[G^2\P\big(W<-\sqrt\mu|G|\mid G\big)\Big] \leq 4\E\big[G^2e^{-\mu G^2/2}\big] = \frac{4}{(1+\mu)^{3/2}}.
\end{align*}
Consequently,
\begin{align}\label{eq:scalar_G_infinity_asymptotics}
\frac{1}{G(\mu)} =\Big(1+\frac{1}{\mu}\Big)\Big(1-(1+\mu)\big(1-m(\mu)\big)\Big) = 1+O(\mu^{-1/2}),\qquad \mu\rightarrow\infty,
\end{align}
so $G(\mu)\rightarrow1$. The proof is complete.
\end{proof}
\subsection{Proof of Proposition \ref{prop:phase_3_weak_main}}
\label{subsec:proof_phase_3_weak}
Fix $\delta\in(\delta_{\weak},\delta_{\str})$. For technical reasons, we define a stopping time
\begin{align}\label{eq:phase3_stopping_time}
\varrho\coloneqq\inf\big\{t\geq\tau_{I\!I}:\rho_t^2\notin I_\delta\big\},
\end{align}
where
\begin{align*}
I_\delta := \Big[\frac{c_{\weak}^2}{4},\frac{\mu_\infty(\delta)+\mu_\star}{2}\Big].
\end{align*}
We further define
\begin{align*}
    \chi_\delta := \min_{\mu\in I_\delta}\Big\{1-\mathcal B(\sqrt\mu)\Big\}
\end{align*}
and set the intermediate regime horizon as
\begin{align}\label{eq:phase3_gap_parameters}
T_{I\!I\!I}(n,\delta)\coloneqq\frac{c\,\chi_\delta^2\,n^{1/3}}{\log^{20}n},
\end{align}
where $c>0$ is a sufficiently small universal constant. This is the horizon in
Proposition~\ref{prop:phase_3_weak_main}, with $c_\delta=c\chi_\delta^2$.
By \eqref{eq:phase3_exact_bolthausen_gap} and Lemma \ref{lem:scalar_G_geometry}, we have
\begin{align*}
\chi_\delta
=\min_{\mu\in I_\delta}\frac{\mu G'(\mu)}{G(\mu)}>0.
\end{align*}
Later we will show that $\varrho$ is larger than the desired $T_{I\!I\!I}(n,\delta)$ horizon. We start by stating the AMP approximation result whose proof will be deferred to Appendix \ref{subsec:phase_3_approx_weak}.

\begin{proposition}\label{prop:phase_3_approx_weak}
With probability $1-O(n^{-10})$, simultaneously for $\tau_{I\!I}\leq t<T_{I\!I\!I}(n,\delta)\wedge\varrho$,
\begin{align}\label{eq:phase3_approx_direct}
\pnorm{\Delta_{U,t}}{}\vee\pnorm{\Delta_{V,t+1}}{}
\vee\pnorm{\hat\u^t-\tilde\u^t}{}\vee\pnorm{\hat\v^{t+1}-\tilde\v^{t+1}}{} \leq C_\delta(t+1)\log^{10}n.
\end{align}
\end{proposition}

We now give two technical lemmas. The first gives uniform concentration results of the terms in $\rho_{t+1}$ and establishes the main recursion. The second controls the state evolution approximation.

\begin{lemma}\label{lem:phase_3_weak_unif_conc}
With probability $1-O(n^{-10})$, the following holds simultaneously for every $\tau_{I\!I}\leq t<T_{I\!I\!I}(n,\delta)\wedge\varrho$. If
\begin{align}\label{eq:phase3_local_approximation}
\pnorm{\Delta_{U,t}}{}\vee\pnorm{\Delta_{V,t}}{} \leq C_\delta(t+1)\log^{10}n,
\end{align}
then
\begin{gather}
\big|\hat\mu_t-|\rho_t|\big| \leq C_\delta \frac{(t+1)\log^{10}n}{\sqrt n},\label{eq:phase3_mu_direct}\\
\Big|\frac{1}{d}\pnorm{h_t^\star(\u^t,\y)}{}^2 -F_\delta(\rho_t^2)\Big| \leq C_\delta\frac{(t+1)\log^{12}n}{\sqrt n},\label{eq:phase3_denom_direct}\\
\Big|\frac{1}{d}\iprod{\g^{-1}}{h_t^\star(\u^t,\y)} -\delta\rho_t\av{\partial_u h_t^\star(\u^t,\y)} -\operatorname{sign}(\rho_t)F_\delta(\rho_t^2)\Big|
\leq C_\delta\frac{(t+1)\log^{12}n}{\sqrt n}.\label{eq:phase3_numer_direct}
\end{gather}
Thus, with the same probability, we have
\begin{gather}\label{eq:phase3_signal_direct}
\big|\rho_{t+1}^2-F_\delta(\rho_t^2)\big|\leq C_\delta \frac{(t+1)\log^{12}n}{\sqrt n},\\
F_{t+1}\Big(\rho_{t+1},\pnorm{\lambda_{t,[0:t]}}{};\btheta^\star\Big)=1,\qquad
H_t\Big(\rho_t,\pnorm{\gamma_{t,[0:t]}}{};\bbeta^\star\Big)\leq1-\frac{\chi_\delta}{2}.\label{eq:phase3_HF_direct}
\end{gather}
\end{lemma}

\begin{lemma}\label{lem:phase_3_weak_SE}
Initialize $\rho_{\SE,\tau_{I\!I}}^2=\rho_{\tau_{I\!I}}^2$. With probability $1-O(n^{-10})$, simultaneously for $\tau_{I\!I}\leq t\leq T_{I\!I\!I}(n,\delta)\wedge\varrho$, we have
\begin{align}\label{eq:phase3_SE_direct}
\big|\rho_t^2-\rho_{\SE,t}^2\big| \leq C_\delta\frac{(t+1)\log^{12}n}{\sqrt n}.
\end{align}
\end{lemma}

With these results, we complete the proof of Proposition \ref{prop:phase_3_weak_main}.

\begin{proof}[Proof of Proposition \ref{prop:phase_3_weak_main}]
We first check that $c_{\weak}^2\leq\mu_\infty(\delta)$ and
\begin{align}\label{eq:phase3_entrance_interval}
c_{\weak}^2\leq\rho_{\tau_{I\!I}}^2
\leq\mu_\infty(\delta)+\frac{\mu_\star-\mu_\infty(\delta)}{8}.
\end{align}
 The lower bound is the definition of $\tau_{I\!I}$. For the upper bound, the recursion
\eqref{eq:phase2_cubic_recursion} at $t=\tau_{I\!I}-1$ gives
$|\rho_{\tau_{I\!I}}|\leq\frac{3\sqrt{2\delta}-1}{2}|\rho_{\tau_{I\!I}-1}| < 2\sqrt{2\delta}\,c_{\weak}$, so $\rho_{\tau_{I\!I}}^2<8\delta c_{\weak}^2\leq\mu_\infty(\delta)$ for sufficiently small $c_{\weak}$. The same condition gives
$c_{\weak}^2\leq\mu_\infty(\delta)$.

Recall that
\begin{align*}
\rho_{\SE,t+1}^2=F_\delta(\rho_{\SE,t}^2),
\qquad
\mu_\infty(\delta)=F_\delta\big(\mu_\infty(\delta)\big).
\end{align*}
By \eqref{eq:bayes_fact_2}, $F_\delta$ is nondecreasing. Since $G$ is increasing on $(0,\mu_\star)$ and $F_\delta(\mu)/\mu=\delta/G(\mu)$, each population update moves towards $\mu_\infty(\delta)$ without crossing it. Thus the population sequence is monotone and
\begin{align}\label{eq:phase3_SE_invariant_interval}
\rho_{\SE,t}^2\in
\Big[\min\{\rho_{\tau_{I\!I}}^2,\mu_\infty(\delta)\},
\max\{\rho_{\tau_{I\!I}}^2,\mu_\infty(\delta)\}\Big]
\subset I_\delta,
\qquad t\geq\tau_{I\!I}.
\end{align}
Its limit must be a fixed point, and $\mu_\infty(\delta)$ is the only positive fixed point in this interval. Hence
\begin{align}\label{eq:phase3_SE_limit}
\rho_{\SE,t}^2\rightarrow\mu_\infty(\delta)
\qquad\text{as }t\rightarrow\infty.
\end{align}

By \eqref{eq:phase3_exact_bolthausen_gap}, the map $z\mapsto\log F_\delta(e^z)$ has derivative between $0$ and $1-\chi_\delta$ whenever $e^z\in I_\delta$. Then by the mean value theorem,
\begin{align*}
\big|\log\rho_{\SE,t+1}^2-\log\mu_\infty(\delta)\big|
\leq(1-\chi_\delta)
\big|\log\rho_{\SE,t}^2-\log\mu_\infty(\delta)\big|.
\end{align*}
Iterating this bound and using $c_{\weak}^2\leq\rho_{\SE,t}^2,\mu_\infty(\delta)\leq C$ gives
\begin{align*}
\big|\rho_{\SE,t}^2-\mu_\infty(\delta)\big|
\leq C\log\Big(\frac{C}{c_{\weak}^2}\Big)
\exp\big\{-\chi_\delta(t-\tau_{I\!I})\big\},
\qquad t\geq\tau_{I\!I}.
\end{align*}
This proves \eqref{eq:phase3_SE_geometric}.

Proposition~\ref{prop:phase_3_approx_weak} bounds $\pnorm{\Delta_{U,t}}{}$ and $\pnorm{\Delta_{V,t}}{}$ by $C_\delta(t+1)\log^{10}n$. Lemma \ref{lem:phase_3_weak_unif_conc} then gives \eqref{eq:phase_3_recursion} for $\tau_{I\!I}\leq t<T_{I\!I\!I}(n,\delta)\wedge\varrho$. The triangle inequality and Lemma \ref{lem:phase_3_weak_SE} then show that
\begin{align*}
|\rho_t^2-\mu_\infty(\delta)| \leq|\rho_t^2-\rho_{\SE,t}^2|
+|\rho_{\SE,t}^2-\mu_\infty(\delta)| \leq\frac{C}{c_{\weak}^2\chi_\delta^3}
\frac{(t+1)\log^{12}n}{\sqrt n}
+|\rho_{\SE,t}^2-\mu_\infty(\delta)|.
\end{align*}
Since $\rho_\infty^2(\delta)=\mu_\infty(\delta)$, this gives \eqref{eq:convergence_decomposition} for $t\leq T_{I\!I\!I}(n,\delta) \wedge \varrho$. Together with Proposition \ref{prop:phase_3_approx_weak}, these estimates prove the claims up to horizon $\varrho\wedge T_{I\!I\!I}(n,\delta)$.

It remains to check that they hold up to horizon $T_{I\!I\!I}(n,\delta)$. We do this by showing $\varrho>T_{I\!I\!I}(n,\delta)$. By \eqref{eq:phase3_entrance_interval} and \eqref{eq:phase3_SE_invariant_interval}, for all $t\geq\tau_{I\!I}$,
\begin{align*}
c_{\weak}^2\leq\rho_{\SE,t}^2
\leq\mu_\infty(\delta)+\frac{\mu_\star-\mu_\infty(\delta)}{8}.
\end{align*}
We also have by Lemma \ref{lem:phase_3_weak_SE} that
\begin{align*}
|\rho_t^2-\rho_{\SE,t}^2|
\leq\frac{C}{c_{\weak}^2\chi_\delta^3}
\frac{(t+1)\log^{12}n}{\sqrt n}
<\min\Big\{\frac{c_{\weak}^2}{2},
\frac{\mu_\star-\mu_\infty(\delta)}{8}\Big\},
\end{align*}
for all $\tau_{I\!I}\leq t\leq T_{I\!I\!I}(n,\delta)\wedge\varrho$ for sufficiently large $n$. Thus, for all such $t$,
\begin{align}\label{eq:phase3_rho_range}
\frac{c_{\weak}^2}{2}<\rho_t^2
<\mu_\infty(\delta)+\frac{\mu_\star-\mu_\infty(\delta)}{4}.
\end{align}
Both bounds hold at $t=\varrho$ if $\varrho\leq T_{I\!I\!I}(n,\delta)$, contradicting the definition of $\varrho$. We conclude that $\varrho>T_{I\!I\!I}(n,\delta)$ as desired.
\end{proof}

It remains to prove Lemmas \ref{lem:phase_3_weak_unif_conc} and \ref{lem:phase_3_weak_SE}.

\begin{proof}[Proof of Lemma \ref{lem:phase_3_weak_unif_conc}]
By \eqref{eq:u_norm} and the hypothesis \eqref{eq:phase3_local_approximation}, we get
\begin{align*}
\Big|\frac{\pnorm{\u^t}{}^2}{n}-1-\rho_t^2\Big|
\leq\frac{C}{\chi_\delta^2}\frac{(t+1)\log^{10}n}{\sqrt n}.
\end{align*}

Since $\rho_t^2\geq c_{\weak}^2/4$, for a large enough $n$ the truncation in \eqref{eq:hat_mu_t} is inactive and we have
\begin{align}
\label{eq:phase3_mu_product_bound}
\frac{|\rho_t|}{2}\leq\hat\mu_t\leq2|\rho_t|,
\qquad
|\rho_t|\big|\hat\mu_t-|\rho_t|\big|
\leq\frac{C}{\chi_\delta^2}\frac{(t+1)\log^{10}n}{\sqrt n}.
\end{align}
This proves \eqref{eq:phase3_mu_direct}. Using the bound of $|\pnorm{\gamma_{t,[0:t]}}{}-1|$ in Lemma \ref{lem:uv_norm}, we can further bound
\begin{align}\label{eq:phase3_gamma_norm}
\big|\pnorm{\gamma_{t,[0:t]}}{}-1\big| \leq C\Big(\sqrt{\frac{t+\log n}{d}}+\frac{\pnorm{\Delta_{V,t}}{}}{\sqrt d}\Big) \leq\frac{C}{\chi_\delta^2}\frac{(t+1)\log^{10}n}{\sqrt n}.
\end{align}
In particular, $\pnorm{\gamma_{t,[0:t]}}{}\in[1/2,3/2]$ for sufficiently large $n$.

\noindent\textbf{Bound of $d^{-1}\iprod{\g^{-1}}{h_t^\star(\u^t,\y)}$}
Approximating $\u^t$ by $\tilde\u^t$, we have
\begin{align*}
\Big|\frac{1}{d}\iprod{\g^{-1}}{h_t^\star(\u^t,\y)}
-\frac{1}{d}\iprod{\g^{-1}}{h_t^\star(\tilde\u^t,\y)}\Big|
\leq C\frac{\log n}{\sqrt d}\pnorm{\Delta_{U,t}}{}.
\end{align*}
We next estimate $d^{-1}\iprod{\g^{-1}}{h_t^\star(\tilde\u^t,\y)}$ by uniform concentration.
For convenience, we give the uniform concentration argument once for the four terms \eqref{eq:four_termds} which will be used later on in the proof. For deterministic parameters
\begin{align}
\label{eq:param_range}
|\rho|\leq\sqrt{\mu_\star},\qquad
0\leq\mu\leq2\sqrt{\mu_\star},\qquad
\frac12\leq\sigma\leq\frac32,\qquad
\pnorm{a}{}=1,
\end{align}
write $U_i=\rho g_i^{-1}+\sigma\sum_{\ell=0}^t a_\ell g_i^\ell$ and consider
\begin{gather}
\begin{gathered}
\frac{1}{d}\sum_{i=1}^n g_i^{-1}h_\mu^\star(U_i,(g_i^{-1})^2),
\qquad
\frac{1}{n}\sum_{i=1}^n\partial_u h_\mu^\star(U_i,(g_i^{-1})^2),\\
\frac{1}{d}\sum_{i=1}^n h_\mu^\star(U_i,(g_i^{-1})^2)^2,\qquad\frac{1}{d}\sum_{i=1}^n
\E_W\Big[\big\{\partial_u h_\mu^\star(\rho g_i^{-1}+\sigma W,(g_i^{-1})^2)\big\}^2\Big],\label{eq:four_termds}
\end{gathered}
\end{gather}
where $W$ is an independent standard Gaussian.

Let $Z_i$ denote any of the four row summands. The bounds on $h_\mu^\star$ and $\partial_u h_\mu^\star$, together with Gaussian moments, give $\E Z_i^2\leq C$ uniformly. To obtain bounded summands for Bernstein's inequality, set
\begin{align*}
\tilde Z_i
\coloneqq Z_i\mathbf 1_{\{\max_{-1\leq\ell\leq t}|g_i^\ell|\leq K\sqrt{\log n}\}}.
\end{align*}
For sufficiently large universal $K$, with probability $1-O(n^{-20})$, $\max_{-1\leq\ell\leq t}|g_i^\ell|\leq K\sqrt{\log n}$ so this truncation is inactive. On this event, we have
\begin{align*}
|U_i|\leq C\sqrt{(t+1)\log n},
\qquad
|\tilde Z_i|\leq C(t+1)\log^2n.
\end{align*}
Moreover, Cauchy--Schwarz bounds the change in expectation by
\begin{align*}
|\E Z_i-\E\tilde Z_i|
\leq(\E Z_i^2)^{1/2}\P(Z_i\neq\tilde Z_i)^{1/2}
\leq Cn^{-10}.
\end{align*}

We next pass from fixed parameters to a uniform bound. Take an $n^{-C_0}$-net $\mathcal N$ of the tuples $\Theta=(\rho,\mu,\sigma,a)$ satisfying \eqref{eq:param_range}. We can choose a net such that its cardinality satisfies
\begin{align*}
|\mathcal N|\leq\exp\big\{C_{C_0}(t+1)\log n\big\}.
\end{align*}

For any $\Theta$, choose $\bar\Theta\in\mathcal N$ with $\|\Theta-\bar\Theta\|\leq n^{-C_0}$. The derivative formulas \eqref{eq:h_rho} and \eqref{eq:ht_ast_first}--\eqref{eq:ht_ast_third} give
\begin{align*}
&\Big|\frac1n\sum_{i=1}^n\big(\tilde Z_i(\Theta)-\tilde Z_i(\bar\Theta)\big)\Big|
+\big|\E\tilde Z_i(\Theta)-\E\tilde Z_i(\bar\Theta)\big|\leq n^C\|\Theta-\bar\Theta\|
\leq n^{C-C_0}\leq n^{-10},
\end{align*}
for sufficiently large $C_0$. 

Fix a net point. The truncated summands are independent and satisfy
\begin{align*}
\operatorname{Var}(\tilde Z_i)\leq C,
\qquad
|\tilde Z_i|\leq C(t+1)\log^2n.
\end{align*}
Bernstein's inequality therefore gives, for any $x>0$, with probability at least $1-2e^{-x}$,
\begin{align*}
\Big|\frac1n\sum_{i=1}^n\tilde Z_i-\E\tilde Z_i\Big|
\leq C\Bigg[\sqrt{\frac{x}{n}}+\frac{(t+1)\log^2n\,x}{n}\Bigg].
\end{align*}
Taking $x=C_1(t+1)\log n$, we obtain
\begin{align*}
\Big|\frac1n\sum_{i=1}^n\tilde Z_i-\E\tilde Z_i\Big|
\leq C\Bigg[\sqrt{\frac{(t+1)\log n}{n}}+\frac{(t+1)^2\log^3n}{n}\Bigg]
\end{align*}
with probability $1-2\exp\{-C_1(t+1)\log n\}$.

Replacing $\E\tilde Z_i$ by $\E Z_i$ adds only the previously established $O(n^{-10})$. Since $t+1\leq\sqrt n$,
\begin{align*}
\sqrt{\frac{(t+1)\log n}{n}}
+\frac{(t+1)^2\log^3n}{n}
\leq C\frac{(t+1)\log^3n}{\sqrt n}.
\end{align*}
Thus, on the event that no truncation occurs,
\begin{align*}
\Big|\frac1n\sum_{i=1}^n Z_i-\E Z_i\Big|
\leq C\frac{(t+1)\log^3n}{\sqrt n}.
\end{align*}

Choosing $C_1$ sufficiently large and taking the union bound over the nets shows that, with probability $1-O(n^{-12})$, this bound holds simultaneously for all four terms \eqref{eq:four_termds}, uniformly over all $\Theta=(\rho,\mu,\sigma,a)$ satisfying \eqref{eq:param_range}.

We may now substitute
\begin{align*}
\rho=\rho_t,\qquad
\mu=\hat\mu_t,\qquad
\sigma=\pnorm{\gamma_{t,[0:t]}}{},\qquad
a_\ell=\frac{\gamma_{t,\ell}}{\pnorm{\gamma_{t,[0:t]}}{}}.
\end{align*}
Throughout the rest of this proof, $G,W$ denote independent standard Gaussians, independent of everything else, and $U := \rho_tG+W$.

We also note the bound
\begin{align*}
&\Big|\E\big[Gh_{\hat\mu_t}^\star\big(\rho_tG+\pnorm{\gamma_{t,[0:t]}}{}W,G^2\big)\big]
-\E\big[Gh_{\hat\mu_t}^\star(\rho_tG+W,G^2)\big]\Big|\\
&\quad\leq\Big|\pnorm{\gamma_{t,[0:t]}}{}-1\Big|
\sup_{\frac{1}{2}\leq\sigma\leq\frac{3}{2}}
\Big|\E\big[GW\partial_u h_{\hat\mu_t}^\star(\rho_tG + \sigma W,G^2)\big]\Big|\\ 
&\quad\leq C\Big|\pnorm{\gamma_{t,[0:t]}}{}-1\Big|
\E\big[|GW|(1+G^2)\big] \leq C\Big|\pnorm{\gamma_{t,[0:t]}}{}-1\Big|.
\end{align*}
Applying the concentration estimate and the approximation of $\u^t$ above, we get
\begin{align*}
\frac{1}{d}\iprod{\g^{-1}}{h_t^\star(\u^t,\y)}
&=\frac{1}{d}\iprod{\g^{-1}}{h_t^\star(\tilde\u^t,\y)}
+O\Big(\frac{\log n}{\sqrt d}\pnorm{\Delta_{U,t}}{}\Big)\\
&=\delta\E\Big[Gh_{\hat\mu_t}^\star\Big(\rho_tG+\pnorm{\gamma_{t,[0:t]}}{}W,G^2\Big)\Big]+O\Bigg(\frac{\log n}{\sqrt d}\pnorm{\Delta_{U,t}}{}
+\frac{(t+1)\log^3n}{\sqrt n}\Bigg)\\
&=\delta\E\big[Gh_{\hat\mu_t}^\star(U,G^2)\big]+O\Bigg(\frac{\log n}{\sqrt d}\pnorm{\Delta_{U,t}}{}
+\frac{(t+1)\log^3n}{\sqrt n}
+\Big|\pnorm{\gamma_{t,[0:t]}}{}-1\Big|\Bigg)\\
&=\delta\E\big[Gh_{\hat\mu_t}^\star(U,G^2)\big]
+O\Big(\frac{(t+1)\log^{12}n}{\chi_\delta^2\sqrt n}\Big).
\end{align*}
\noindent\textbf{Bound of $\av{\partial_u h_t^\star(\u^t,\y)}$}
Using $|\partial_u^2h_t^\star(u,y)|\leq C(1+y)^{3/2}$ via Lemma \ref{lem:h_deriv_comp}, we have
\begin{align*}
\av{\partial_u h_t^\star(\u^t,\y)}
=\av{\partial_u h_t^\star(\tilde\u^t,\y)}
+O\Big(\frac{\log^{3/2}n}{\sqrt d}\pnorm{\Delta_{U,t}}{}\Big).
\end{align*}
By the same uniform concentration argument,
\begin{align*}
\av{\partial_u h_t^\star(\tilde\u^t,\y)}
=\E\big[\partial_u h_{\hat\mu_t}^\star(U,G^2)\big]
+O\Big(\frac{(t+1)\log^{12}n}{\chi_\delta^2\sqrt n}\Big).
\end{align*}
Combining the above two pieces, we have
\begin{align}\label{eq:claim_rho_mu2}
\frac{1}{d}\iprod{\g^{-1}}{h_t^\star(\u^t,\y)}
-\delta\rho_t\av{\partial_u h_t^\star(\u^t,\y)} = \delta\E\big[Gh_{\hat\mu_t}^\star(U,G^2)
-\rho_t\partial_u h_{\hat\mu_t}^\star(U,G^2)\big]
+O\Big(\frac{(t+1)\log^{12}n}{\chi_\delta^2\sqrt n}\Big).
\end{align}

It remains to analyze the first quantity on the right side. To this end, recall the following definition in (\ref{eq:h_rho}):
\begin{align*}
h_\rho^\star(u,y)
=(\rho^2+1)\sqrt y\,\tanh(\rho\sqrt y\,u)-\rho u,
\end{align*}
which is a generalization of $h_t^\star$ in (\ref{eq:ht_ast}) such that $h_t^\star=h_{\hat\mu_t}^\star$. We next prove the claim that
\begin{align}\label{eq:claim_rho_mu}
\Big|\E\big[Gh_{\hat\mu_t}^\star(U,G^2)
-\rho_t\partial_u h_{\hat\mu_t}^\star(U,G^2)\big] -\E\big[Gh_{|\rho_t|}^\star(U,G^2)
-\rho_t\partial_u h_{|\rho_t|}^\star(U,G^2)\big]\Big|
\leq C|\rho_t|\big|\hat\mu_t-|\rho_t|\big|.
\end{align}
To see this, note that by integration by parts with respect to $W$, we can write the first term as
\begin{align*}
\E\big[Gh_{\hat\mu_t}^\star(U,G^2)
-\rho_t\partial_u h_{\hat\mu_t}^\star(U,G^2)\big] &=\E\big[(G-\rho_tW)h_{\hat\mu_t}^\star(U,G^2)\big]\\
&=\E\big[\big((1+\rho_t^2)G-\rho_tU\big)
h_{\hat\mu_t}^\star(U,G^2)\big]\\
&=\E\big[\big((1+\rho_t^2)\E[G\mid U,G^2]-\rho_tU\big)
h_{\hat\mu_t}^\star(U,G^2)\big]\\
&=\operatorname{sign}(\rho_t)\E\big[h_{|\rho_t|}^\star(U,G^2)h_{\hat\mu_t}^\star(U,G^2)\big].
\end{align*}
Here $\E[G\mid U,G^2]=\operatorname{sign}(\rho_t)|G|\tanh(|\rho_t||G|U)$. A similar computation shows that the second term equals
\begin{align*}
\E\big[Gh_{|\rho_t|}^\star(U,G^2)-\rho_t\partial_u h_{|\rho_t|}^\star(U,G^2)\big]
=\operatorname{sign}(\rho_t)\E\big[h_{|\rho_t|}^\star(U,G^2)^2\big].
\end{align*}
The absolute difference in \eqref{eq:claim_rho_mu} is therefore
\begin{align*}
\Big|\E\big[h_{|\rho_t|}^\star(U,G^2) \big(h_{\hat\mu_t}^\star(U,G^2)-h_{|\rho_t|}^\star(U,G^2)\big)\big]\Big| \leq \sqrt{\E\big[h_{|\rho_t|}^\star(U,G^2)^2\big]} \sqrt{\E\big[\big(h_{\hat\mu_t}^\star(U,G^2)-h_{|\rho_t|}^\star(U,G^2)\big)^2\big]}.
\end{align*}
By \eqref{eq:bayes_fact_1} and symmetry, the first term can be bounded by
\begin{align*}
\E\big[h_{|\rho_t|}^\star(U,G^2)^2\big] = (\rho_t^2+1)\big((\rho_t^2+1)m(\rho_t^2)-\rho_t^2\big) = \frac{\rho_t^2}{G(\rho_t^2)}
\leq C\rho_t^2.
\end{align*}
For the other term, using the mean value theorem on $\rho\mapsto h_\rho^\star(u,y)$ yields
\begin{align*}
|h_{\hat\mu_t}^\star(U,G^2) - h_{|\rho_t|}^\star(U,G^2)| \leq C\big|\hat\mu_t - |\rho_t|\big|\big(|U|+|G| + G^2|U|\big).
\end{align*}
Taking second moments gives
\begin{align*}
\Big(\E\big[\big(h_{\hat\mu_t}^\star(U,G^2)-h_{|\rho_t|}^\star(U,G^2)\big)^2\big]\Big)^{1/2}
\leq C\big|\hat\mu_t-|\rho_t|\big|.
\end{align*}
Combining the above estimates yields the claim (\ref{eq:claim_rho_mu}).

Combining (\ref{eq:claim_rho_mu2}) and (\ref{eq:claim_rho_mu}), and \eqref{eq:phase3_mu_product_bound}, we have
\begin{align*}
\frac{1}{d}\iprod{\g^{-1}}{h_t^\star(\u^t,\y)}
-\delta\rho_t\av{\partial_u h_t^\star(\u^t,\y)} &=\operatorname{sign}(\rho_t)\delta\E\big[h_{|\rho_t|}^\star(U,G^2)^2\big]
+O\Big(\frac{(t+1)\log^{12}n}{\chi_\delta^2\sqrt n}\Big)\\
&=\operatorname{sign}(\rho_t)F_\delta(\rho_t^2)
+O\Big(\frac{(t+1)\log^{12}n}{\chi_\delta^2\sqrt n}\Big).
\end{align*}
This proves \eqref{eq:phase3_numer_direct}.

\noindent\textbf{Bound of $\frac{1}{d}\pnorm{h_t^\star(\u^t,\y)}{}^2$.}
For deterministic $\rho\in\R$, $\mu\geq0$ and $\sigma>0$, let
\begin{align*}
\Phi(\rho,\mu,\sigma) := \delta\,\E\Big[h_\mu^\star\big(\rho G+\sigma W,G^2\big)^2\Big], \qquad G,W\stackrel{\mathrm{i.i.d.}}{\sim}\N(0,1).
\end{align*}
Since the bound for the third term in \eqref{eq:four_termds} holds uniformly over \eqref{eq:param_range}, we may apply it at the random point $(\rho,\mu,\sigma,a)=\big(\rho_t,\hat\mu_t,\pnorm{\gamma_{t,[0:t]}}{}, \gamma_{t,[0:t]}/\pnorm{\gamma_{t,[0:t]}}{}\big)$. This gives
\begin{align*}
\frac1d\pnorm{h_t^\star(\tilde\u^t,\y)}{}^2
=\Phi\big(\rho_t,\hat\mu_t,\pnorm{\gamma_{t,[0:t]}}{}\big)
+O\Big(\frac{(t+1)\log^3n}{\sqrt n}\Big).
\end{align*}
We next compare $\Phi(\rho_t,\hat \mu_t, \sigma)$ with $\Phi(\rho_t,\hat \mu_t, 1)$. For $\sigma\in[1/2,3/2]$,
\begin{align}
\label{eq:partial_Phi}
\partial_\sigma\Phi(\rho_t,\hat\mu_t,\sigma) = 2\delta\,\E\Big[h_{\hat\mu_t}^\star(\rho_tG + \sigma W,G^2)\, \partial_u h_{\hat\mu_t}^\star(\rho_tG+\sigma W,G^2)\,W\Big].
\end{align}

By \eqref{eq:h_rho} and \eqref{eq:ht_ast_first}, for $0\leq\mu\leq2\sqrt{\mu_\star}$, $u\in\R$, and $y\geq0$,
\begin{align*}
|h_\mu^\star(u,y)|
&\leq(1+\mu^2)\sqrt y+\mu|u|
\leq C(\sqrt y+|u|),\\
|\partial_u h_\mu^\star(u,y)|
&\leq\mu(1+\mu^2)y+\mu
\leq C(1+y).
\end{align*}

Applying these bounds along with Gaussian moment bounds to \eqref{eq:partial_Phi} shows 
\begin{align*}
    |\partial_\sigma\Phi(\rho_t,\hat\mu_t,\sigma)| \leq C.
\end{align*}

The mean value theorem and \eqref{eq:phase3_gamma_norm} therefore give
\begin{align*}
\Big|\Phi\big(\rho_t,\hat\mu_t,\pnorm{\gamma_{t,[0:t]}}{}\big)-\Phi(\rho_t,\hat\mu_t,1)\Big| \leq C\delta\,\big|\pnorm{\gamma_{t,[0:t]}}{}-1\big| \leq\frac{C\delta}{\chi_\delta^2}\frac{(t+1)\log^{10}n}{\sqrt n}.
\end{align*}
Combining these estimates yields
\begin{align*}
\frac1d\pnorm{h_t^\star(\tilde\u^t,\y)}{}^2 = \delta\E\big[h_{\hat\mu_t}^\star(\rho_tG + W,G^2)^2\big]
+O\Big(\frac{(t+1)\log^{12}n}{\chi_\delta^2\sqrt n}\Big).
\end{align*} To replace $\tilde\u^t$ by $\u^t$, factor the difference of squares and use
\begin{align*}
\pnorm{h_t^\star(\u^t,\y)}{}
+\pnorm{h_t^\star(\tilde\u^t,\y)}{}\leq C\sqrt n,
\qquad
\pnorm{\partial_u h_t^\star}{\infty}\leq C\log n.
\end{align*}
The first bound follows from $|h_\mu^\star(u,y)|\leq C(\sqrt y+|u|)$ and the bounds on $\pnorm{\u^t}{}$ and $\pnorm{\tilde\u^t}{}$. The replacement error is thus at most $C\log n\,\pnorm{\Delta_{U,t}}{}/\sqrt n$, and
\begin{align*}
\frac{1}{d}\pnorm{h_t^\star(\u^t,\y)}{}^2
=\delta\E\big[h_{\hat\mu_t}^\star(U,G^2)^2\big]
+O\Big(\frac{(t+1)\log^{12}n}{\chi_\delta^2\sqrt n}\Big).
\end{align*}
Since $\hat\mu_t\in[|\rho_t|/2,2|\rho_t|]$, the same moment bounds used above give
\begin{align*}
\Big|\E\big[h_{\hat\mu_t}^\star(U,G^2)^2\big]
-\E\big[h_{|\rho_t|}^\star(U,G^2)^2\big]\Big|
\leq C|\rho_t|\big|\hat\mu_t-|\rho_t|\big|.
\end{align*}
Together with $\delta\E[h_{|\rho_t|}^\star(U,G^2)^2]=F_\delta(\rho_t^2)$, this proves \eqref{eq:phase3_denom_direct}.

\noindent\textbf{Signal recursion.}
On $I_\delta$,
\begin{align*}
F_\delta(\rho_t^2)
=\frac{\delta\rho_t^2}{G(\rho_t^2)}
\geq c\rho_t^2\geq cc_{\weak}^2.
\end{align*}
Take $n$ large enough so that, throughout the horizon,
\begin{align*}
\frac{C}{\chi_\delta^2}\frac{(t+1)\log^{12}n}{\sqrt n}
\leq c\chi_\delta c_{\weak}^2,
\end{align*}
with a sufficiently small universal $c>0$. Then \eqref{eq:phase3_denom_direct} yields that $\pnorm{h_t^\star(\u^t,\y)}{}^2/d \geq F_\delta(\rho_t^2)/2$. Substituting \eqref{eq:phase3_denom_direct} and \eqref{eq:phase3_numer_direct} into the exact identity \eqref{eq:key_signal} gives
\begin{align*}
\big|\rho_{t+1}^2-F_\delta(\rho_t^2)\big|
\leq\frac{C}{\chi_\delta^2}\frac{(t+1)\log^{12}n}{\sqrt n},
\end{align*}
proving \eqref{eq:phase3_signal_direct}. The numerator estimate also shows that $\rho_{t+1}$ has the sign of $\rho_t$.

Applying again $\pnorm{h_t^\star(\u^t,\y)}{}^2/d \geq F_\delta(\rho_t^2)/2$ and $\hat\mu_t \leq 2|\rho_t|$ gives
\begin{align*}
\frac{\hat\mu_t}{d^{-1/2}\pnorm{h_t^\star(\u^t,\y)}{}}\leq C.
\end{align*}
Thus \eqref{eq:pi_parameter} and \eqref{eq:ht_ast_first}--\eqref{eq:ht_ast_second} imply the derivative bounds needed in the induction:
\begin{align*}
|\pi_t|\leq C,\qquad
\pnorm{h_t'}{\infty}\leq C\log n,\qquad
\pnorm{h_t''}{\infty}\leq C\log^{3/2}n.
\end{align*}

The identity $F_{t+1}=1$ follows from $f_{t+1}(x)=x$. For $H_t$, its definition \eqref{eq:F_def} gives
\begin{align*}
&H_t\Big(\rho_t,\pnorm{\gamma_{t,[0:t]}}{};\bbeta^\star\Big)=
\frac{
d^{-1}\sum_{i=1}^n
\E_W\Big[\big\{\partial_u h_{\hat\mu_t}^\star
\big(\rho_tg_i^{-1}+\pnorm{\gamma_{t,[0:t]}}{}W,y_i\big)\big\}^2\Big]
}{
d^{-1}\pnorm{h_t^\star(\u^t,\y)}{}^2
}.
\end{align*}
The concentration estimate for $\frac{1}{d}\sum_{i=1}^n
\mathbb E_W\Big[
\big\{\partial_u h_\mu^\star
(\rho g_i^{-1}+\sigma W,(g_i^{-1})^2)\big\}^2
\Big]$ and the replacement of $\pnorm{\gamma_{t,[0:t]}}{}$ by $1$ control its numerator. Since $\partial_u h_0^\star=0$ and the parameter derivative of $\partial_u h_\mu^\star$ has polynomial growth, the same difference-of-squares argument gives
\begin{align*}
\Big|\E\big[\{\partial_u h_{\hat\mu_t}^\star(U,G^2)\}^2\big] - \E\big[\{\partial_u h_{|\rho_t|}^\star(U,G^2)\}^2\big]\Big| \leq C|\rho_t|\big|\hat\mu_t-|\rho_t|\big|.
\end{align*}
The numerator therefore equals
\begin{align*}
\delta\E\big[\{\partial_u h_{|\rho_t|}^\star(U,G^2)\}^2\big] + O\Big(\frac{(t+1)\log^{12}n}{\chi_\delta^2\sqrt n}\Big) = \rho_t^2F_\delta'(\rho_t^2) + O\Big(\frac{(t+1)\log^{12}n}{\chi_\delta^2\sqrt n}\Big),
\end{align*}
by symmetry and \eqref{eq:bayes_fact_2}. The denominator estimate and \eqref{eq:phase3_exact_bolthausen_gap} now yield
\begin{align*}
H_t\Big(\rho_t,\pnorm{\gamma_{t,[0:t]}}{};\bbeta^\star\Big) \leq\frac{\rho_t^2F_\delta'(\rho_t^2)}{F_\delta(\rho_t^2)} + \frac{C}{c_{\weak}^2\chi_\delta^2} \frac{(t+1)\log^{12}n}{\sqrt n} \leq1-\chi_\delta+\frac{\chi_\delta}{2} = 1 - \frac{\chi_\delta}{2}.
\end{align*}
 This proves \eqref{eq:phase3_HF_direct} and completes the proof.
\end{proof}

\begin{proof}[Proof of Lemma \ref{lem:phase_3_weak_SE}]
We use \eqref{eq:phase3_signal_direct} through $T_{I\!I\!I} \wedge \varrho$. Recall from \eqref{eq:phase3_SE_invariant_interval} that $c_{\weak}^2\leq\rho_{\SE,t}^2\leq C$ and $\rho_{\SE,t}^2\in I_\delta$ for all $t\geq\tau_{I\!I}$.

Following the proof technique in \cite{li2023approximate}, we form a recursion for $\frac{|\rho_t^2-\rho_{\SE,t}^2|}{\rho_{\SE,t}^2}$. We show inductively that
\begin{align*}
\frac{|\rho_t^2-\rho_{\SE,t}^2|}{\rho_{\SE,t}^2}
\leq\frac{C}{c_{\weak}^2\chi_\delta^3}\frac{(t+1)\log^{12}n}{\sqrt n},
\qquad \tau_{I\!I}\leq t\leq T_{I\!I\!I}(n,\delta)\wedge\varrho.
\end{align*}
Choose $C$ sufficiently large, and then $n$ large enough so that the right-hand side is at most $\chi_\delta/4$ throughout the horizon. The base case holds since $\rho_{\tau_{I\!I}}^2=\rho_{\SE,\tau_{I\!I}}^2$. Suppose the bound holds through time $t<T_{I\!I\!I}(n,\delta)\wedge\varrho$.

By the mean value theorem and \eqref{eq:phase3_signal_direct},
\begin{align}\label{eq:rho_recursion_estimate}
\frac{|\rho_{t+1}^2-\rho_{\SE,t+1}^2|}{\rho_{\SE,t+1}^2}
&\leq\frac{\rho_{\SE,t}^2F_\delta'(\xi_t)}{F_\delta(\rho_{\SE,t}^2)}
\frac{|\rho_t^2-\rho_{\SE,t}^2|}{\rho_{\SE,t}^2}
+\frac{C}{c_{\weak}^2\chi_\delta^2}\frac{(t+1)\log^{12}n}{\sqrt n},
\end{align}
where $\xi_t$ is between $\rho_t^2$ and $\rho_{\SE,t}^2$. Here we used $F_\delta(\rho_{\SE,t}^2)=\rho_{\SE,t+1}^2\geq c_{\weak}^2$.

Both inputs, and hence $\xi_t$, lie in $I_\delta$. The inductive bound gives $\xi_t\geq(1-\chi_\delta/4)\rho_{\SE,t}^2$. Since $F_\delta$ is nondecreasing and $F_\delta(\mu)/\mu$ is decreasing on $I_\delta$, the Bolthausen bound \eqref{eq:phase3_exact_bolthausen_gap} gives
\begin{align}\label{eq:bolthausen_claim}
\frac{\rho_{\SE,t}^2F_\delta'(\xi_t)}{F_\delta(\rho_{\SE,t}^2)}
&=\frac{\rho_{\SE,t}^2}{\xi_t}
\frac{F_\delta(\xi_t)}{F_\delta(\rho_{\SE,t}^2)}
\frac{\xi_tF_\delta'(\xi_t)}{F_\delta(\xi_t)}\leq(1-\chi_\delta)\max\Big\{1,\frac{\rho_{\SE,t}^2}{\xi_t}\Big\}
\leq\frac{1-\chi_\delta}{1-\chi_\delta/4}
\leq1-\frac{\chi_\delta}{2}.
\end{align}

Plugging \eqref{eq:bolthausen_claim} into \eqref{eq:rho_recursion_estimate} and iterating from $\tau_{I\!I}$, we get
\begin{align*}
\frac{|\rho_{t+1}^2-\rho_{\SE,t+1}^2|}{\rho_{\SE,t+1}^2} \leq\frac{C}{c_{\weak}^2\chi_\delta^2} \sum_{s=\tau_{I\!I}}^t \Big(1-\frac{\chi_\delta}{2}\Big)^{t-s} \frac{(s+1)\log^{12}n}{\sqrt n} \leq\frac{C}{c_{\weak}^2\chi_\delta^3} \frac{(t+1)\log^{12}n}{\sqrt n}.
\end{align*}
This closes the induction.  Finally, $\rho_{\SE,t}^2\leq C$ gives \eqref{eq:phase3_SE_direct}.
\end{proof}

\subsection{Proof of Proposition \ref{prop:phase_3_approx_weak}}
\label{subsec:phase_3_approx_weak}
To control the AMP approximation errors, we use the induction of Section \ref{subsec:main_ind}, as in Phases I/II. We start by defining the corresponding Phase III refined assumptions.

\begin{definition}\label{assump:inductive_assumption_phase_3}
For $\tau_{I\!I}\leq t<T_{I\!I\!I}(n,\delta)\wedge\varrho$, we define the following claims as Assumption $\cA^U_{I\!I\!I,t}$:
\begin{subequations}\label{eqs:phase3_refined_U}
\begin{gather}
\sum_{k=0}^{t-1}|\alpha_{t-1}^k|
\leq\frac{C}{\chi_\delta^2}\frac{(t+1)\log^{10}n}{\sqrt d},
\label{eq:phase_3_assump_U_5}\\
\pnorm{\Delta_{U,t}}{}\vee\pnorm{\hat\u^t-\tilde\u^t}{}
\leq\frac{C}{\chi_\delta^2}(t+1)\log^{10}n,
\label{eq:phase_3_assump_U_2}\\
\bar\gamma_t\vee\bar\lambda_t\leq C\sqrt d,
\label{eq:phase_3_assump_U_3}\\
\pnorm{h_t'}{\infty}\leq C\log n,
\qquad\pnorm{h_t''}{\infty}\leq C\log^{3/2}n.
\label{eq:phase_3_assump_U_1}
\end{gather}
\end{subequations}
Similarly, we define the following claims as Assumption $\cA^V_{I\!I\!I,t}$:
\begin{subequations}\label{eqs:phase3_refined_V}
\begin{gather}
\sum_{k=0}^{t}|\beta_t^k|
\leq\frac{C}{\chi_\delta^2}\frac{(t+1)\log^{10}n}{\sqrt d},
\label{eq:phase_3_assump_V_5}\\
\pnorm{\Delta_{V,t+1}}{}\vee\pnorm{\hat\v^{t+1}-\tilde\v^{t+1}}{}
\leq\frac{C}{\chi_\delta^2}(t+1)\log^{10}n,
\label{eq:uv_closeness_phase_3}\\
\bar\lambda_t\vee\bar\gamma_{t+1}\leq C\sqrt d,
\label{eq:lambda_main_norm_phase_3}\\
\pnorm{f_{t+1}'}{\infty}=1,
\qquad\pnorm{f_{t+1}''}{\infty}=0.
\label{eq:phase_3_assump_V_1}
\end{gather}
\end{subequations}
\end{definition}

These assumptions have the following consequences. 

\begin{lemma}\label{lem:phase_3_assump_conseq}
Suppose the refined assumptions hold up to $\cA^V_{I\!I\!I,t-1}$ and the approximation \eqref{eq:phase_3_assump_U_2} holds at time $t$. Then \eqref{eq:phase3_mu_direct}--\eqref{eq:phase3_HF_direct} hold.
\end{lemma}

Similar to the previous phases, the Section \ref{subsec:main_ind} bounds can be simplified as follows.

\begin{lemma}\label{lem:phase_3_simp}
Suppose the refined assumptions hold up to $\cA^V_{I\!I\!I,t-1}$. Then $\cH_t^U$ implies 
\begin{align}\label{eq:phase_3_simp_1}
\pnorm{\Delta'_{U,t}}{} \leq\Big(1+\frac{\chi_\delta}{32}\Big)\pnorm{\Delta'_{V,t}}{} + C(t+1)\log^{5/2}n.
\end{align}
If the refined assumptions hold up to $\cA^U_{I\!I\!I,t}$ and $\cH_t^V$ holds, then
\begin{align}\label{eq:phase_3_simp_V}
\pnorm{\Delta'_{V,t+1}}{} \leq\Big(1-\frac{3\chi_\delta}{16}\Big)\pnorm{\Delta'_{U,t}}{} + C(t+1)\log^{5/2}n.
\end{align}
Under the hypotheses of \eqref{eq:phase_3_simp_1} (resp. \eqref{eq:phase_3_simp_V}), we also have
\begin{align}
\big(\mathfrak F_{k+1,t}+\cE_t^f\big)\big(\mathfrak H_{k+1,t-1}+\cE_{t-1}^h\big)
&\leq1-\frac{\chi_\delta}{8},
\quad \tau_{I\!I}\leq k+1\leq t-2,
\label{eq:phase3_alpha_factor}\\
\big(\mathfrak H_{t,k}+\cE_t^h\big)\big(\mathfrak F_{t,k+1}+\cE_t^f\big)
&\leq1-\frac{\chi_\delta}{8},
\quad \tau_{I\!I}\leq k\leq t-2.
\label{eq:phase3_beta_factor}
\end{align}
\end{lemma}
With these results, we can now complete the proof of Proposition \ref{prop:phase_3_approx_weak}.

\begin{proof}[Proof of Proposition \ref{prop:phase_3_approx_weak}]
By the definition of the refined assumptions, it is sufficient to show that $\cA_{I\!I\!I,t}^U$ and $\cA_{I\!I\!I,t}^V$ hold for all $\tau_{I\!I}\leq t<T_{I\!I\!I}(n,\delta)\wedge\varrho$ with probability $1-O(n^{-10})$.  As in the proofs of Propositions \ref{prop:phase_1_tilde_close} and \ref{prop:phase_2_exp_growth}, we show the following chain of claims
\begin{align}\label{eq:phase_3_chain}
&\cH^U_{\tau_{I\!I}}\rightarrow\cA^U_{I\!I\!I,\tau_{I\!I}}
\overset{\checkmark}{\rightarrow}\cH^V_{\tau_{I\!I}}\rightarrow\cA^V_{I\!I\!I,\tau_{I\!I}}
\overset{\checkmark}{\rightarrow}\cdots\cdots\overset{\checkmark}{\rightarrow}\cH^U_t\rightarrow\cA^U_{I\!I\!I,t}
\overset{\checkmark}{\rightarrow}\cH^V_t\rightarrow\cA^V_{I\!I\!I,t}.
\end{align}
The $\checkmark$ follow by direct application of Proposition \ref{prop:general_error_induction}. The base case of the chain is verified by the previous phases. We inductively verify the remaining arrows in the chain, which consists of the following two steps:
\begin{itemize}
\item (Claim 1) Suppose the chain \eqref{eq:phase_3_chain} holds up to (and including) $\cH_t^U$ for some $t\geq\tau_{I\!I}$, we check that $\cA_{I\!I\!I,t}^U$ holds.
\item (Claim 2) Suppose the chain \eqref{eq:phase_3_chain} holds up to (and including) $\cH_t^V$ for some $t\geq\tau_{I\!I}$, we check that $\cA_{I\!I\!I,t}^V$ holds.
\end{itemize}

\medskip
\noindent\textbf{Proof of Claim 1.}
We derive bounds on $\pnorm{\Delta_{U,s}'}{}$ and $\pnorm{\Delta_{V,s}'}{}$ for $s\leq t$. First, \eqref{eq:phase2_second_order_V_next} at $\tau_{I\!I}-1$, or \eqref{eq:phase1_second_order_V_next} at $\tau_I-1$ if $\tau_{I\!I}=\tau_I$, gives $\pnorm{\Delta'_{V,\tau_{I\!I}}}{}\leq C(\tau_{I\!I}+1)\log^7n$. Equation \eqref{eq:phase_3_simp_1} therefore gives the base bound
\begin{align}\label{eq:phase3_base_U_second_order}
\pnorm{\Delta'_{U,\tau_{I\!I}}}{}\leq C(\tau_{I\!I}+1)\log^7n.
\end{align}
Using the simplified bounds in Lemma \ref{lem:phase_3_simp}, we have the one-step recursions
\begin{align*}
\pnorm{\Delta'_{V,s}}{}
&\leq\Big(1-\frac{3\chi_\delta}{16}\Big)\pnorm{\Delta'_{U,s-1}}{}+Cs\log^{5/2}n,
&&\tau_{I\!I}<s\leq t,\\
\pnorm{\Delta'_{U,s}}{}
&\leq\Big(1+\frac{\chi_\delta}{32}\Big)\pnorm{\Delta'_{V,s}}{}+C(s+1)\log^{5/2}n,
&&\tau_{I\!I}\leq s\leq t.
\end{align*}
Combining these, for large enough $n$ we get
\begin{align*}
\pnorm{\Delta'_{U,s}}{}\leq\Big(1-\frac{\chi_\delta}{8}\Big)\pnorm{\Delta'_{U,s-1}}{}+C(s+1)\log^{5/2}n,
\qquad \tau_{I\!I}<s\leq t.
\end{align*}
Iterating until $\tau_{I\!I}$ and using the base \eqref{eq:phase3_base_U_second_order}, we have
\begin{align}\label{eq:phase3_preliminary_U_residual}
\pnorm{\Delta'_{U,s}}{} \leq\Big(1-\frac{\chi_\delta}{8}\Big)^{s-\tau_{I\!I}}\pnorm{\Delta'_{U,\tau_{I\!I}}}{}
+C\log^{5/2}n\sum_{r=\tau_{I\!I}+1}^s\Big(1-\frac{\chi_\delta}{8}\Big)^{s-r}(r+1) \leq \frac{C}{\chi_\delta}(s+1)\log^7n.
\end{align}
A similar computation shows that $\pnorm{\Delta'_{V,s}}{}\leq C\chi_\delta^{-1}(s+1)\log^7n$ for $\tau_{I\!I}\leq s\leq t$. Now we are ready to show the claims of $\cA_{I\!I\!I,t}^U$.

\noindent\textbf{Proof of \eqref{eq:phase_3_assump_U_5}.}
Since the assumptions hold through $\cA_{I\!I\!I,t-1}^V$ in the chain \eqref{eq:phase_3_chain}, we have by \eqref{eq:induction_alpha_last} and $\pnorm{\lambda_{t-1,[0:t-1]}}{}=1$ that
\begin{align}\label{eq:phase3_alpha_diagonal_sharp}
|\alpha_{t-1}^{t-1}| \leq\frac{2}{\sqrt d}\Bigg[\pnorm{\Delta'_{V,t}}{}
+\max_{k<t}\pnorm{\Delta_{V,k}}{}\sum_{k<t}|\beta_{t-1}^k|\Bigg] \leq \frac{C}{\chi_\delta}\frac{(t+1)\log^7n}{\sqrt d}.
\end{align}
Here the first term contributes at most $C\chi_\delta^{-1}(t+1)\log^7n/\sqrt d$ by the bound on $\pnorm{\Delta'_{V,t}}{}$ just proved, and the second term contributes at most $C\chi_\delta^{-4}(t+1)^2\log^{20}n/d\leq C(t+1)/\sqrt d$ for large enough $n$. 

We next verify \eqref{eq:phase_3_assump_U_5} at the starting time $t=\tau_{I\!I}$, which serves as the base case below. In the sum $\sum_{k=0}^{\tau_{I\!I}-1}|\alpha_{\tau_{I\!I}-1}^k|$, the term $k=\tau_{I\!I}-1$ is bounded by \eqref{eq:phase3_alpha_diagonal_sharp}, the term $k=\tau_{I\!I}-2$ by \eqref{eq:induction_alpha_middle_last}, and the remaining terms by one application of \eqref{eq:induction_alpha_middle}, which reduces them to $\sum_k|\alpha_{\tau_{I\!I}-2}^k|$, bounded in \eqref{eq:alpha_sum_phase_2}. Hence $\sum_{k=0}^{\tau_{I\!I}-1}|\alpha_{\tau_{I\!I}-1}^k|\leq
C(\tau_{I\!I}+1)\log^{10}n/\sqrt d$.

Returning to general $t$, the coefficients other than $\alpha_{t-1}^{t-1}$ satisfy, by \eqref{eq:induction_alpha_middle_last} and \eqref{eq:phase3_alpha_factor},
\begin{align*}
|\alpha_{t-1}^{t-2}| \leq|\beta_{t-1}^{t-1}| \leq\frac{C}{\chi_\delta}\frac{(t+1)\log^9n}{\sqrt d},\qquad  |\alpha_{t-1}^{k}| \leq \Big(1-\frac{\chi_\delta}{8}\Big)|\alpha_{t-2}^{k+1}|, \quad \tau_{I\!I}\leq k\leq t-3.
\end{align*}
This implies the following recursion of the sequence $t\mapsto\sum_{k=\tau_{I\!I}}^{t-1}|\alpha_{t-1}^k|$:
\begin{align*}
\sum_{k=\tau_{I\!I}}^{t-1}|\alpha_{t-1}^k|
\leq\Big(1-\frac{\chi_\delta}{8}\Big)\sum_{k=\tau_{I\!I}}^{t-2}|\alpha_{t-2}^k|
+\frac{C}{\chi_\delta}\frac{(t+1)\log^9n}{\sqrt d}.
\end{align*}
Iterating this to $\tau_{I\!I}$ gives
\begin{align}\label{eq:alpha_sum_part1}
\sum_{k=\tau_{I\!I}}^{t-1}|\alpha_{t-1}^k|
\leq\frac{C}{\chi_\delta^2}\frac{(t+1)\log^9n}{\sqrt d}.
\end{align}
It remains to control the sum $\sum_{k=0}^{\tau_{I\!I}-1}|\alpha_{t-1}^k|$. When $k+1<\tau_{I\!I}$, using the Phase-I/II estimates in \eqref{eq:induction_alpha_middle} gives
\begin{align*}
|\alpha_{t-1}^k|
&\leq\Big(1-\frac{\chi_\delta}{8}\Big)
\Big(1+C\rho_{k+1}^2+C\frac{\log^{20}n}{\sqrt n}\Big)
|\alpha_{t-2}^{k+1}|.
\end{align*}

Each application of \eqref{eq:induction_alpha_middle} replaces $\alpha_{t-1}^k$ by $\alpha_{t-2}^{k+1}$. After $t-s$ applications, the coefficient is $\alpha_{s-1}^{k+t-s}$. We stop when $s=\tau_{I\!I}$ or when this coefficient is one of $\alpha_{s-1}^{s-1},\alpha_{s-1}^{s-2}$, whose bounds were established above. The factors $1+C\rho_{k+1}^2+C\log^{20}n/\sqrt n$ accumulated along this iteration multiply to at most 
\begin{align*}
\prod_{j=0}^{\tau_{I\!I}-1}
\Big(1+C\rho_j^2+C\frac{\log^{20}n}{\sqrt n}\Big)
&\leq\exp\Bigg\{
C\sum_{j=0}^{\tau_{I\!I}-1}\rho_j^2
+C\frac{\tau_{I\!I}\log^{20}n}{\sqrt n}
\Bigg\}
\leq C,
\end{align*}
where we use Phase-II exponential growth \eqref{eq:phase2_recursion} and \eqref{eq:phase_2_O(1)_sum} to control
\begin{align*}
\sum_{j=\tau_I}^{\tau_{I\!I}-1}\rho_j^2
\leq\frac{c_{\weak}^2(1+\sqrt{2\delta})}{\sqrt{2\delta}-1}
\leq C.
\end{align*}

We now sum over $0\leq k<\tau_{I\!I}$. The coefficients reached at $s=\tau_{I\!I}$ are bounded by the sum of $|\alpha_{\tau_{I\!I}-1}^k|$, which was bounded in the base case above. At any later time $s$, the only possible coefficients reached are $\alpha_{s-1}^{s-1}$ and $\alpha_{s-1}^{s-2}$. Each corresponds to at most one starting index $k$, since $k+t-s=s-1$ or $s-2$. Using their bounds gives
\begin{align}\label{eq:alpha_sum_part2}
\sum_{k=0}^{\tau_{I\!I}-1}|\alpha_{t-1}^k|
&\leq C\sum_{k=0}^{\tau_{I\!I}-1}|\alpha_{\tau_{I\!I}-1}^k|
+C\sum_{s=\tau_{I\!I}+1}^t
\Big(1-\frac{\chi_\delta}{8}\Big)^{\max\{t-s-2,0\}}
\big(|\alpha_{s-1}^{s-1}|+|\alpha_{s-1}^{s-2}|\big)\notag\\
&\leq C\frac{(\tau_{I\!I}+1)\log^{10}n}{\sqrt d}
+\frac{C\log^9n}{\chi_\delta\sqrt d}
\sum_{s=\tau_{I\!I}}^t
\Big(1-\frac{\chi_\delta}{8}\Big)^{\max\{t-s-2,0\}}(s+1)\notag\\
&\leq\frac{C}{\chi_\delta^2}\frac{(t+1)\log^{10}n}{\sqrt d}.
\end{align}
Combining \eqref{eq:alpha_sum_part1} and \eqref{eq:alpha_sum_part2} gives the desired result:
\begin{align}\label{eq:phase3_alpha_row_direct}
\sum_{k=0}^{t-1}|\alpha_{t-1}^k|
\leq\frac{C}{\chi_\delta^2}\frac{(t+1)\log^{10}n}{\sqrt d}.
\end{align}

\noindent\textbf{Proof of \eqref{eq:phase_3_assump_U_2}.}
By the same argument as in the proof of \eqref{eq:phase_1_assump_U_2} in Proposition \ref{prop:phase_1_tilde_close}, we have
\begin{align*}
\pnorm{h_k(\tilde\u^k,\y)}{}
&\leq\pnorm{\w^k}{}+\pnorm{h_k'}{\infty}\pnorm{\Delta_{U,k}}{}
\leq\sqrt d+\frac{C}{\chi_\delta^2}(t+1)\log^{11}n
\leq C\sqrt d,\qquad k<t,\\
\pnorm{\hat\u^t-\tilde\u^t}{}
&\leq C\sqrt d\sum_{k=0}^{t-1}|\alpha_{t-1}^k|
\leq\frac{C}{\chi_\delta^2}(t+1)\log^{10}n.
\end{align*}
In addition, by Lemma \ref{lem:mutual_close} and the bound $\bar\lambda_{t-1}=O(\sqrt d)$ in $\cA_{I\!I\!I,t-1}^V$, we have
\begin{align*}
\pnorm{\Delta_{U,t}}{}
\leq\pnorm{\Delta'_{U,t}}{}+C\sqrt d\sum_{k=0}^{t-1}|\alpha_{t-1}^k|
\leq\frac{C}{\chi_\delta^2}(t+1)\log^{10}n.
\end{align*}

\noindent\textbf{Proof of \eqref{eq:phase_3_assump_U_3}.}
The bound on $\bar\gamma_t$ follows from \eqref{eq:gamma_norm} of Lemma~\ref{lem:uv_norm} together with the bounds on $\pnorm{\Delta_{V,s}}{}$ in \eqref{eq:uv_closeness_phase_1}, \eqref{eq:uv_closeness_phase_2}, and \eqref{eq:uv_closeness_phase_3}. For $\bar\lambda_t$, the bound on $\pnorm{\Delta_{U,t}}{}$ in \eqref{eq:phase_3_assump_U_2}, just proved, allows us to apply Lemma~\ref{lem:phase_3_assump_conseq} which gives
\begin{align*}
\hat\mu_t\leq C|\rho_t|,
\qquad \frac{1}{d}\pnorm{h_t^\star(\u^t,\y)}{}^2\geq c\rho_t^2,
\qquad \frac{\hat\mu_t\sqrt d}{\pnorm{h_t^\star(\u^t,\y)}{}}\leq C.
\end{align*}
Thus, \eqref{eq:global_current_lambda_control} gives $\bar\lambda_t\leq C\sqrt d$ and also, \eqref{eq:pi_parameter} gives $|\pi_t|\leq C$.

\noindent\textbf{Proof of \eqref{eq:phase_3_assump_U_1}.} Note that
\begin{align*}
h_t'(u,y)=\frac{1+\pi_t}{\sqrt{2\delta}}\Big((\hat\mu_t^2+1)y\Big(1-\tanh^2(\hat\mu_tu\sqrt y)\Big)-1\Big),
\end{align*}
and by Lemma \ref{lem:h_deriv_comp},
\begin{align*}
h_t''(u,y)=\frac{1+\pi_t}{\sqrt{2\delta}}\Big(-2(\hat\mu_t^2+1)\hat\mu_t y^{3/2}\tanh(\hat\mu_tu\sqrt y)\Big(1-\tanh^2(\hat\mu_tu\sqrt y)\Big)\Big).
\end{align*}
The claim follows from $\pnorm{\y}{\infty}=O(\log n)$, $|\pi_t|\leq C$, and $\hat\mu_t=O(1)$ obtained from Lemma \ref{lem:phase_3_assump_conseq}. This completes the proof of Claim 1.

\medskip
\noindent\textbf{Proof of Claim 2.}
Combining \eqref{eq:phase_3_simp_V} with \eqref{eq:phase_3_simp_1}, we have
\begin{align*}
\|\Delta'_{V,t+1}\|
\leq\Big(1-\frac{\chi_\delta}{8}\Big)\|\Delta'_{V,t}\|
+C(t+1)\log^{5/2}n.
\end{align*}
The same steps as in Claim 1 give
\begin{align}\label{eq:phase3_second_order_direct}
\|\Delta'_{U,t}\|\vee\|\Delta'_{V,t+1}\|
\leq\frac{C}{\chi_\delta}(t+1)\log^7n.
\end{align}

\noindent\textbf{Proof of \eqref{eq:phase_3_assump_V_5}.}
By \eqref{eq:gamma_norm_loo} and the bound on $\pnorm{\Delta_{V,t}}{}$ (from \eqref{eq:uv_closeness_phase_3} at time $t-1$, or from Phase II when $t=\tau_{I\!I}$), we have $\|\gamma_{t,[0:t]}\|\geq1/2$. Using this and the derivative bounds \eqref{eq:phase_3_assump_U_1} in \eqref{eq:induction_beta_last}, we obtain
\begin{align*}
|\beta_t^t| \leq\frac{C\log^{3/2}n}{\sqrt d}\Bigg[\|\Delta'_{U,t}\| + \log n\max_{k<t}\|\Delta_{U,k}\|\sum_{k<t}|\alpha_{t-1}^k|\Bigg] \leq\frac{C}{\chi_\delta}\frac{(t+1)\log^9n}{\sqrt d}.
\end{align*}
Here the term containing the sum over $k<t$ is at most $C\chi_\delta^{-4}(t+1)^2\log^{45/2}n/d\leq C(t+1)\log^{5/2}n/\sqrt d$. For $\beta_t^{t-1}$, \eqref{eq:induction_beta_middle_last}, the derivative bounds \eqref{eq:phase_3_assump_U_1}, and the bound \eqref{eq:phase3_alpha_diagonal_sharp} on $|\alpha_{t-1}^{t-1}|$ give
\begin{align*}
|\beta_t^{t-1}| \leq \delta\|h_t'\|_\infty\|h_{t-1}'\|_\infty|\alpha_{t-1}^{t-1}| \leq \frac{C}{\chi_\delta}\frac{(t+1)\log^9n}{\sqrt d}.
\end{align*}
In particular,
\begin{align}\label{eq:phase_3_diagonal_coefficients}
|\alpha_{t-1}^{t-1}|\vee|\beta_t^t|\vee|\beta_t^{t-1}|
\leq\frac{C}{\chi_\delta}\frac{(t+1)\log^9n}{\sqrt d}.
\end{align}
The remaining coefficients satisfy \eqref{eq:induction_beta_middle}. Iterating this inequality backwards, we eventually reach either a coefficient $\beta_{\tau_{I\!I}-1}^k$ from the preceding phase or one of the coefficients $\beta_s^s,\beta_s^{s-1}$ bounded above. Combining this with \eqref{eq:phase3_beta_factor}, we obtain
\begin{align*}
\sum_{k=0}^t|\beta_t^k|
&\leq C\sum_{k=0}^{\tau_{I\!I}-1}|\beta_{\tau_{I\!I}-1}^k|
+C\sum_{s=\tau_{I\!I}}^t
\Big(1-\frac{\chi_\delta}{8}\Big)^{\max\{t-s-2,0\}}
\big(|\beta_s^s|+|\beta_s^{s-1}|\big)\\
&\leq C\frac{(\tau_{I\!I}+1)\log^{10}n}{\sqrt d}
+\frac{C\log^9n}{\chi_\delta\sqrt d}
\sum_{s=\tau_{I\!I}}^t
\Big(1-\frac{\chi_\delta}{8}\Big)^{\max\{t-s-2,0\}}(s+1)\\
&\leq\frac{C}{\chi_\delta^2}\frac{(t+1)\log^{10}n}{\sqrt d}.
\end{align*}
Together with \eqref{eq:phase3_alpha_row_direct}, this proves
\begin{align}\label{eq:phase3_rows_direct}
\sum_{k=0}^{t-1}|\alpha_{t-1}^k|\vee\sum_{k=0}^t|\beta_t^k|
\leq\frac{C}{\chi_\delta^2}\frac{(t+1)\log^{10}n}{\sqrt d}.
\end{align}

\noindent\textbf{Proof of \eqref{eq:uv_closeness_phase_3}.}
By Lemma \ref{lem:mutual_close},
\begin{align*}
\|\Delta_{V,t+1}\|
\leq\|\Delta'_{V,t+1}\|+\bar\gamma_t\sum_{k=0}^t|\beta_t^k|
\leq\frac{C}{\chi_\delta^2}(t+1)\log^{10}n.
\end{align*}
Also, the definition of $\hat\v^{t+1}$ and the identities $f_k=\operatorname{id}$ and $\gamma_{k,-1}=\rho_k$ give
\begin{align*}
\hat\v^{t+1}-\tilde\v^{t+1}
=\sum_{k=0}^t\beta_t^k\big(\tilde\v^k-\rho_k\btheta^\star\big).
\end{align*}
For $1\leq k\leq t$, \eqref{eq:claim_near_isometry} and $\pnorm{\lambda_{k-1,[0:k-1]}}{}=1$ give
\begin{align*}
\pnorm{\tilde\v^k-\rho_k\btheta^\star}{}
&=\Bigpnorm{\sum_{\ell=0}^{k-1}\lambda_{k-1,\ell}\s^\ell}{}
\leq C\sqrt d\,\pnorm{\lambda_{k-1,[0:k-1]}}{}
=C\sqrt d.
\end{align*}
For $k=0$, $\tilde\v^0-\rho_0\btheta^\star=P_{\btheta^\star}^\perp\v^0$, so $\|\tilde\v^0-\rho_0\btheta^\star\|\leq\|\v^0\|\leq C\sqrt d$. Therefore,
\begin{align*}
\pnorm{\hat\v^{t+1}-\tilde\v^{t+1}}{}
&\leq\sum_{k=0}^t|\beta_t^k|\pnorm{\tilde\v^k-\rho_k\btheta^\star}{}\leq C\sqrt d\sum_{k=0}^t|\beta_t^k|
\leq\frac{C}{\chi_\delta^2}(t+1)\log^{10}n.
\end{align*}

\noindent\textbf{Proof of \eqref{eq:lambda_main_norm_phase_3} and \eqref{eq:phase_3_assump_V_1}.}
Claim 1 already gives $\bar\lambda_t\leq C\sqrt d$. Since $t<\varrho$, \eqref{eq:phase3_signal_direct} gives $\rho_{t+1}^2\leq F_\delta(\rho_t^2)+\frac{C(t+1)\log^{12}n}{\chi_\delta^2\sqrt n}\leq C$. The bound on $\pnorm{\Delta_{V,t+1}}{}$ in \eqref{eq:uv_closeness_phase_3} and \eqref{eq:gamma_norm} then give $\bar\gamma_{t+1}\leq C\sqrt d$. Finally, $f_{t+1}'=1$ and $f_{t+1}''=0$. This proves Claim 2 and verifies $\cA_t^V$.
\end{proof}

In the remainder of this subsection, we give the proofs of Lemmas \ref{lem:phase_3_assump_conseq} and \ref{lem:phase_3_simp}.

\begin{proof}[Proof of Lemma \ref{lem:phase_3_assump_conseq}]
By \eqref{eq:phase_3_assump_U_2}, we have
\begin{align*}
\pnorm{\Delta_{U,t}}{}\vee\pnorm{\hat\u^t-\tilde\u^t}{}
\leq\frac{C}{\chi_\delta^2}(t+1)\log^{10}n.
\end{align*}
For $t>\tau_{I\!I}$, using \eqref{eq:uv_closeness_phase_3} at time $t-1$, we also have
\begin{align*}
\pnorm{\Delta_{V,t}}{}\vee\pnorm{\hat\v^t-\tilde\v^t}{}
\leq\frac{C}{\chi_\delta^2}t\log^{10}n
\leq\frac{C}{\chi_\delta^2}(t+1)\log^{10}n.
\end{align*}
At $t=\tau_{I\!I}$, the proof of \eqref{eq:uv_closeness_phase_2} at time $t-1$ gives
\begin{align*}
\pnorm{\Delta_{V,t}}{}\vee\pnorm{\hat\v^t-\tilde\v^t}{}
&\leq C_\delta t\log^9n
\leq\frac{C}{\chi_\delta^2}(t+1)\log^{10}n,
\end{align*}
for sufficiently large $n$. These bounds verify \eqref{eq:phase3_local_approximation}. Applying Lemma~\ref{lem:phase_3_weak_unif_conc} therefore gives \eqref{eq:phase3_mu_direct}--\eqref{eq:phase3_HF_direct}, as desired.
\end{proof}

\begin{proof}[Proof of Lemma \ref{lem:phase_3_simp}]
We will prove the claim for $\cH_t^U$ in (\ref{eq:induction_Delta_U})-(\ref{eq:induction_alpha_middle}), and briefly verify the claim for $\cH_t^V$. We first note that, under the refined assumptions up to $\cA_{I\!I\!I,t-1}^V$, the following claims hold.
\begin{gather}
\pnorm{f_t'}{\infty}=1,\qquad \pnorm{f_t''}{\infty}=0,\label{eq:temp_phase3_1}\\
\sqrt{F_t\Big(\rho_t,\pnorm{\lambda_{t-1,[0:t-1]}}{};\btheta^\star\Big)}=1,\label{eq:temp_phase3_2}\\
\sqrt{H_{t-1}\Big(\rho_{t-1},\pnorm{\gamma_{t-1,[0:t-1]}}{};\bbeta^\star\Big)}\leq1-\frac{\chi_\delta}{4},\qquad t>\tau_{I\!I},\label{eq:temp_phase3_5}\\
\sum_{k=0}^{t-1}|\beta_{t-1}^k|\leq\frac{C}{\chi_\delta^2}\frac{(t+1)\log^{10}n}{\sqrt d},\label{eq:temp_phase3_3}\\
\bar\gamma_t\vee\bar\lambda_{t-1}\leq C\sqrt d.\label{eq:temp_phase3_4}
\end{gather}
Here (i): Equations (\ref{eq:temp_phase3_1})-(\ref{eq:temp_phase3_2}) follow from $f_t=\operatorname{id}$, and (\ref{eq:temp_phase3_5}) follows from Lemma \ref{lem:phase_3_assump_conseq} when $t>\tau_{I\!I}$; (ii) Claim (\ref{eq:temp_phase3_3}) holds by (\ref{eq:phase_3_assump_V_5}) of $\cA_{I\!I\!I,t-1}^V$; (iii) Claim (\ref{eq:temp_phase3_4}) holds by Lemma \ref{lem:uv_norm} and the bound of $\pnorm{\Delta_{V,t}}{}$ via (\ref{eq:uv_closeness_phase_3}) of $\cA_{I\!I\!I,t-1}^V$.

By (\ref{eq:induction_Delta_U}), we have
\begin{align*}
\pnorm{\Delta'_{U,t}}{}
&\leq\Bigg(\sqrt{F_t\Big(\rho_t,\pnorm{\lambda_{t-1,[0:t-1]}}{};\btheta^\star\Big)}+\cE_t^f\Bigg)
\Bigg(\pnorm{\Delta'_{V,t}}{}+\sum_{k=0}^{t-1}|\beta_{t-1}^k|\pnorm{f_k'}{\infty}\pnorm{\Delta_{V,k}}{}\Bigg)\\
&\quad+C\big(\bar\gamma_t+\bar\lambda_{t-1}\big)\big(\pnorm{f_{[0:t]}'}{\infty}\vee1\big)^2\big(\pnorm{\rho_{[0:t]}}{\infty}\vee1\big)t\sqrt{\frac{\log n}{d}}.
\end{align*}
We bound the terms on the right in turn. By the definition (\ref{eq:E_multiplier}) and $f_t''=0$, since $t<T_{I\!I\!I}(n,\delta)$ and $n$ is large enough,
\begin{align*}
\cE_t^f
\leq C\sqrt{t+1}\Big(\frac{\log n}{d}\Big)^{1/4}
\leq\frac{\chi_\delta}{32}.
\end{align*}
For the sum over $k$, the bound (\ref{eq:temp_phase3_3}) on $\sum_k|\beta_{t-1}^k|$ and the bounds on $\pnorm{\Delta_{V,k}}{}$ for $k<t$ give, for $n$ large enough,
\begin{align*}
\sum_{k=0}^{t-1}|\beta_{t-1}^k|\pnorm{\Delta_{V,k}}{} \leq \frac{C}{\chi_\delta^4}\frac{(t+1)^2\log^{20}n}{\sqrt d} \leq C(t+1).
\end{align*}
Together with (\ref{eq:temp_phase3_2}) and (\ref{eq:temp_phase3_4}), this gives
\begin{align*}
\pnorm{\Delta'_{U,t}}{} \leq \Big(1+\frac{\chi_\delta}{32}\Big)\pnorm{\Delta'_{V,t}}{} + C(t+1)\log^{5/2}n,
\end{align*}
which is \eqref{eq:phase_3_simp_1}.

For $k\leq t-3$, we have by (\ref{eq:induction_alpha_middle}) that
\begin{align*}
|\alpha_{t-1}^k|&\leq|\alpha_{t-2}^{k+1}|\Bigg[\sqrt{F_{k+1}\Big(\rho_{k+1},\pnorm{\lambda_{k,[0:k]}}{};\btheta^\star\Big)F_t\Big(\rho_t,\pnorm{\lambda_{t-1,[0:t-1]}}{};\btheta^\star\Big)}+\cE_t^f\Bigg]\\
&\quad\times\Bigg[\sqrt{H_{k+1}\Big(\rho_{k+1},\pnorm{\gamma_{k+1,[0:k+1]}}{};\bbeta^\star\Big)H_{t-1}\Big(\rho_{t-1},\pnorm{\gamma_{t-1,[0:t-1]}}{};\bbeta^\star\Big)}+\cE_{t-1}^h\Bigg],
\end{align*}
where
\begin{align*}
\cE_{t-1}^h
&\leq C\Bigg[\log^2n\sqrt t\Big(\frac{\log n}{d}\Big)^{1/4}
+\log^{5/2}n\frac{\sqrt t+\sqrt{\log n}}{\sqrt d}
\big(\pnorm{\Delta_{U,t-1}}{}+\pnorm{\hat\u^{t-1}-\tilde\u^{t-1}}{}\big)\Bigg]\\
&\leq C\Bigg[\frac{\sqrt{t+1}\log^3n}{n^{1/4}}
+\frac{(t+1)^{3/2}\log^{13}n}{\chi_\delta^2\sqrt n}\Bigg]
\leq\frac{\chi_\delta}{32}.
\end{align*}
Here we substitute the bound $t\leq T_{I\!I\!I}(n,\delta)$.

For $\tau_{I\!I}\leq k+1\leq t-2$, Lemma \ref{lem:phase_3_assump_conseq} gives $H_{k+1},H_{t-1}\leq1-\chi_\delta/2$. Applying further (\ref{eq:temp_phase3_2}) and (\ref{eq:temp_phase3_5}), we conclude that
\begin{align*}
\big(\mathfrak F_{k+1,t}+\cE_t^f\big)\big(\mathfrak H_{k+1,t-1}+\cE_{t-1}^h\big) \leq\Big(1+\frac{\chi_\delta}{32}\Big)\Big(1-\frac{\chi_\delta}{4}+\frac{\chi_\delta}{32}\Big)\leq 1 -\frac{\chi_\delta}{8},
\end{align*}
for sufficiently large $n$, proving (\ref{eq:phase3_alpha_factor}). 

We now assume $\cA_{I\!I\!I,t}^U$ and $\cH_t^V$. Lemma \ref{lem:phase_3_assump_conseq} gives
\begin{align*}
\sqrt{H_t\Big(\rho_t,\pnorm{\gamma_{t,[0:t]}}{};\bbeta^\star\Big)}
\leq1-\frac{\chi_\delta}{4}.
\end{align*}
Using the derivative bounds \eqref{eq:phase_3_assump_U_1} and the approximation bounds \eqref{eq:phase_3_assump_U_2} in \eqref{eq:E_multiplier}, the same calculation as above gives $\cE_t^h\leq\chi_\delta/32$. Moreover,
\begin{align*}
\sum_{k=0}^{t-1}|\alpha_{t-1}^k|\pnorm{h_k'}{\infty}\pnorm{\Delta_{U,k}}{}
&\leq\frac{C}{\chi_\delta^4}\frac{(t+1)^2\log^{21}n}{\sqrt d}
\leq C(t+1)\log n.
\end{align*}
Using $\bar\lambda_t\vee\bar\gamma_t\leq C\sqrt d$, $\pnorm{h_{[0:t]}'}{\infty}\leq C\log n$, and $\gamma_{k,-1}=\rho_k$, the $V$ inequality \eqref{eq:induction_Delta_V} becomes
\begin{align*}
\pnorm{\Delta'_{V,t+1}}{}
&\leq\Big(1-\frac{\chi_\delta}{4}+\frac{\chi_\delta}{32}\Big)
\Bigg(\pnorm{\Delta'_{U,t}}{}+\sum_{k=0}^{t-1}|\alpha_{t-1}^k|\pnorm{h_k'}{\infty}\pnorm{\Delta_{U,k}}{}\Bigg)
+C(t+1)\log^{5/2}n\\
&\leq\Big(1-\frac{3\chi_\delta}{16}\Big)\pnorm{\Delta'_{U,t}}{}
+C(t+1)\log^{5/2}n.
\end{align*}
This proves \eqref{eq:phase_3_simp_V}. Finally, for $\tau_{I\!I}\leq k\leq t-2$, the same bounds give
\begin{align*}
\big(\mathfrak H_{t,k}+\cE_t^h\big)\big(\mathfrak F_{t,k+1}+\cE_t^f\big)
&\leq\Big(1-\frac{\chi_\delta}{4}+\frac{\chi_\delta}{32}\Big)\Big(1+\frac{\chi_\delta}{32}\Big)
\leq1-\frac{\chi_\delta}{8},
\end{align*}
proving \eqref{eq:phase3_beta_factor}.
\end{proof}
\section{Proofs for strong recovery}\label{sec:proof_strong_recovery}
Throughout this section, fix $\delta>\delta_{\str}$ and $M \geq 1$.

\subsection{Proof of Proposition \ref{prop:phase_3_strong_main}}
We first introduce an additional stopping time to enforce a constant order lower bound on the empirical signal strength $|\rho_t|$:
\begin{align}\label{eq:strong_first_order_lower_exit}
\varrho\coloneqq\inf\Big\{t\geq\tau_{I\!I}:|\rho_t|<\frac{c_{\weak}}{2}\Big\},
\end{align}
where $c_{\weak}$ is defined at the beginning of Section \ref{subsec:analysis_phase2}. Let
\begin{align}\label{eq:strong_gap_parameters}
a_\delta\coloneqq\sqrt{\frac{1+\delta/\delta_{\str}}{2}}>1,
\end{align}
and set
\begin{align}\label{eq:strong_first_order_gap_count}
N_{\delta,M}\coloneqq\max\Bigg\{0,\bigg\lceil\frac{\log(M/c_{\weak})}{\log a_\delta}\bigg\rceil\Bigg\}.
\end{align}
Thus $N_{\delta,M}$ updates with growth factor $a_\delta$ suffice to reach $M$ from $c_{\weak}$. We start with two technical results. As usual, the first controls the AMP approximation. The proof is given in Appendix \ref{subsubsec:phase_3_strong_approx}.

\begin{proposition}\label{prop:phase_3_strong_approx}
With probability $1-O(n^{-10})$, uniformly for $\tau_{I\!I}\leq t\leq\tau_{I\!I\!I}\wedge\varrho\wedge(\tau_{I\!I}+N_{\delta,M})$,
\begin{align}\label{eq:strong_first_order_residual_bound}
\pnorm{\Delta_{U,t}}{}\vee\pnorm{\Delta_{V,t}}{}\leq C_{\delta,M}\log^{C_{\str}}n,
\end{align}
where $C_{\str}> 0$ is a sufficiently large universal exponent.
\end{proposition}

The next lemma derives the main recursion for Phase III of the strong recovery regime.

\begin{lemma}\label{lem:phase_3_SE_strong}
On the same event, for every $\tau_{I\!I}\leq t<\tau_{I\!I\!I}\wedge\varrho\wedge(\tau_{I\!I}+N_{\delta,M})$,
\begin{align}\label{eq:strong_first_order_state_bound}
\big|\rho_{t+1}^2-F_\delta(\rho_t^2)\big|\leq C_{\delta,M}\frac{\log^{C_{\str}}n}{n^{1/4}}.
\end{align}
\end{lemma}
\begin{proof}
The proof is similar to the proof of (\ref{eq:phase_3_recursion}) in Proposition \ref{prop:phase_3_weak_main}, 
with Proposition \ref{prop:phase_3_strong_approx} replacing Proposition \ref{prop:phase_3_approx_weak}.

Let us derive the analogues of the two claims in Lemma \ref{lem:phase_3_weak_unif_conc}. Suppose that $c_{\weak}/2\leq|\rho_t|<M$ and $\pnorm{\Delta_{U,t}}{}\vee\pnorm{\Delta_{V,t}}{}\leq n^{1/4}$. Write $\sigma_t=\pnorm{\gamma_{t,[0:t]}}{}$ and $N_t=d^{-1/2}\pnorm{h_t^\star(\u^t,\y)}{}$. By \eqref{eq:gamma_norm_loo} and \eqref{eq:u_norm}, we have
\begin{align}\label{eq:strong_first_order_mu_sigma}
\big|\hat\mu_t-|\rho_t|\big|+|\sigma_t-1|\leq C_{\delta,M}\frac{\big(1+\pnorm{\Delta_{U,t}}{}+\pnorm{\Delta_{V,t}}{}\big)\log n}{\sqrt n}.
\end{align}

The uniform concentration argument in Lemma \ref{lem:phase_3_weak_unif_conc} gives the bound
\begin{align}\label{eq:strong_first_order_denominator}
\big|N_t^2-F_\delta(\rho_t^2)\big|\leq C_{\delta,M}\frac{\big(1+\pnorm{\Delta_{U,t}}{}+\pnorm{\Delta_{V,t}}{}\big)\log^3n}{\sqrt n}.
\end{align}
Here we replace $\u^t$ by $\tilde\u^t$, apply uniform concentration, and use \eqref{eq:strong_first_order_mu_sigma} to replace $(\hat\mu_t,\sigma_t)$ by $(|\rho_t|,1)$ in the population average. Similarly, tracking the proof of Lemma \ref{lem:phase_3_weak_unif_conc}, we have
\begin{align}\label{eq:strong_first_order_numerator}
&\frac{1}{d}\iprod{\g^{-1}}{h_t^\star(\u^t,\y)}-\delta\rho_t\av{\partial_u h_t^\star(\u^t,\y)}\notag\\
&\qquad=\operatorname{sign}(\rho_t)F_\delta(\rho_t^2)+O_{\delta,M}\Bigg(\frac{\big(1+\pnorm{\Delta_{U,t}}{}+\pnorm{\Delta_{V,t}}{}\big)\log^3n}{\sqrt n}\Bigg).
\end{align}
Since $F_\delta(\rho_t^2)\geq\delta c_{\weak}^2/(4\delta_{\str})$, the errors are smaller than $F_\delta(\rho_t^2)/4$ for sufficiently large $n$. In particular,
\begin{align}\label{eq:strong_first_order_denominator_lower}
N_t^2\geq\frac{1}{2}F_\delta(\rho_t^2)>0.
\end{align}
Plugging (\ref{eq:strong_first_order_denominator}) and (\ref{eq:strong_first_order_numerator}) into the definition
\begin{align*}
\rho_{t+1}=\frac{\sqrt d}{\pnorm{h_t^\star(\u^t,\y)}{}}\Bigg[\frac{1}{d}\iprod{\g^{-1}}{h_t^\star(\u^t,\y)}-\delta\rho_t\av{\partial_u h_t^\star(\u^t,\y)}\Bigg]
\end{align*}
gives $\operatorname{sign}(\rho_{t+1})=\operatorname{sign}(\rho_t)$ and
\begin{align}\label{eq:strong_first_order_noisy_SE}
\big|\rho_{t+1}^2-F_\delta(\rho_t^2)\big|\leq C_{\delta,M}\frac{\big(1+\pnorm{\Delta_{U,t}}{}+\pnorm{\Delta_{V,t}}{}\big)\log^3n}{\sqrt n}.
\end{align}
Applying Proposition \ref{prop:phase_3_strong_approx} and using $\log^3n\leq n^{1/4}$ gives the result.
\end{proof}

Combining these results gives the desired exponential growth of Proposition \ref{prop:phase_3_strong_main}.

\begin{proof}[Proof of Proposition \ref{prop:phase_3_strong_main}]
By Lemma \ref{lem:phase_3_SE_strong}, we have
\begin{align*}
\rho_{t+1}^2=F_\delta(\rho_t^2)+O_{\delta,M}\Big(\frac{\log^{C_{\str}}n}{n^{1/4}}\Big).
\end{align*}
By Proposition \ref{prop:phase_3_asymp}, we have $F_\delta(\rho^2)\geq(\delta/\delta_{\str})\rho^2$ for all $\rho$. Choose $n$ large enough so that the error above is at most $(\delta/\delta_{\str}-1)c_{\weak}^2/8$. Since $\rho_t^2\geq c_{\weak}^2/4$, this yields
\begin{align}\label{eq:strong_first_order_growth_calculation}
\rho_{t+1}^2&\geq\frac{\delta}{\delta_{\str}}\rho_t^2-\frac{(\delta/\delta_{\str}-1)c_{\weak}^2}{8}\geq\Big(\frac{1+\delta/\delta_{\str}}{2}\Big)\rho_t^2=a_\delta^2\rho_t^2.
\end{align}
Hence the stopping time $\varrho$ cannot occur before $\tau_{I\!I\!I}$, i.e. $\tau_{I\!I\!I}\wedge\varrho=\tau_{I\!I\!I}$. Moreover, $a_\delta^{N_{\delta,M}}c_{\weak}\geq M$, so
\begin{align}\label{eq:strong_first_order_first_stop_bound}
\tau_{I\!I\!I}-\tau_{I\!I}\leq N_{\delta,M}\leq C_{\delta,M}.
\end{align}

To bound $\pnorm{\Delta_{U,\tau_{I\!I\!I}}}{}$, we also need to bound $|\rho_{\tau_{I\!I\!I}}|$. By the definition of $\tau_{I\!I\!I}$, the preceding iterate satisfies $|\rho_{\tau_{I\!I\!I}-1}|<M$. Moreover, the definition of $F_\delta$ and $m(\mu)\leq1$ give
\begin{align}\label{eq:strong_F_global_upper}
0\leq F_\delta(\mu)\leq\delta(1+\mu),
\qquad \mu\geq0.
\end{align}
Applying \eqref{eq:strong_recovery_recursion} at time $\tau_{I\!I\!I}-1$ and using \eqref{eq:strong_F_global_upper} gives
\begin{align*}
\rho_{\tau_{I\!I\!I}}^2
&\leq F_\delta\big(\rho_{\tau_{I\!I\!I}-1}^2\big)
+C_{\delta,M}\frac{\log^Cn}{n^{1/4}}\leq\delta\big(1+\rho_{\tau_{I\!I\!I}-1}^2\big)+1\leq\delta(1+M^2)+1,
\end{align*}
where we choose $n$ sufficiently large that $C_{\delta,M}\log^Cn/n^{1/4}\leq1$ and use $|\rho_{\tau_{I\!I\!I}-1}|<M$ by the definition of $\tau_{I\!I\!I}.$
Together with the bound $\pnorm{\Delta_{V,\tau_{I\!I\!I}}}{}\leq C_{\delta,M}\log^{C_{\str}}n$ from \eqref{eq:strong_first_order_residual_iteration} at time $\tau_{I\!I\!I}-1$, Lemma~\ref{lem:fo_v2u} gives
\begin{align*}
\pnorm{\Delta_{U,\tau_{I\!I\!I}}}{}
\leq C_{\delta,M}\log^{C_{\str}}n.
\end{align*}

\end{proof}

\subsection{Proof of Proposition \ref{prop:phase_3_strong_approx}}\label{subsubsec:phase_3_strong_approx}
\begin{proof}
Unlike Proposition \ref{prop:phase_3_approx_weak} for the intermediate recovery regime, since we only need to track the approximation error for $N_{\delta,M}=O_{\delta,M}(1)$ iterations, we can use the simplified first-order analysis given in Appendix \ref{sec:first_order_analysis} rather than the second-order induction in Section \ref{subsec:main_ind}.

We inductively check that
\begin{align}\label{eq:strong_first_order_residual_iteration}
\pnorm{\Delta_{U,t}}{}\vee\pnorm{\Delta_{V,t+1}}{}\leq C_{\delta,M}\log^{C_{\str}}n
\end{align}
for $\tau_{I\!I}\leq t<\tau_{I\!I\!I}\wedge\varrho\wedge(\tau_{I\!I}+N_{\delta,M})$. 

The base case starts at $t=\tau_{I\!I}$. Applying \eqref{eq:uv_closeness_phase_2} at time $\tau_{I\!I}-1$ and using $\tau_{I\!I}\leq C_\delta\log n$ from Proposition~\ref{prop:phase_2_exp_growth} gives
\begin{align}
\pnorm{\Delta_{V,\tau_{I\!I}}}{}&\leq C_\delta\log^{12}n.
\label{eq:strong_first_order_V_base}
\end{align}
By the definition \eqref{eq:def_tau_II} and \eqref{eq:phase2_cubic_recursion} at time $\tau_{I\!I}-1$, where $|\rho_{\tau_{I\!I}-1}|<c_{\weak}$, we also have
\begin{align}
c_{\weak}\leq|\rho_{\tau_{I\!I}}|&\leq C_\delta.
\label{eq:strong_first_order_phase_endpoint_upper}
\end{align}
Suppose the bound for $\pnorm{\Delta_{V,t}}{}$ holds. We verify the bounds for $\pnorm{\Delta_{U,t}}{}$ and $\pnorm{\Delta_{V,t+1}}{}$. The bounds \eqref{eq:uv_closeness_phase_1} and \eqref{eq:uv_closeness_phase_2}, the induction bound \eqref{eq:strong_first_order_residual_iteration} for $\tau_{I\!I}\leq s<t$, and the assumed bound on $\pnorm{\Delta_{V,t}}{}$ control $\pnorm{\Delta_{V,s}}{}$ for every $s\leq t$. Since $|\rho_s|<c_{\weak}\leq1$ for $s<\tau_{I\!I}$ by the definition of $\tau_{I\!I}$, and $|\rho_s|<M$ for $\tau_{I\!I}\leq s\leq t$, we also have $\pnorm{\rho_{[0:t]}}{\infty}\leq C_{\delta,M}$. Applying \eqref{eq:gamma_norm} therefore gives $\bar\gamma_t\leq C_{\delta,M}\sqrt d$, while \eqref{eq:global_current_lambda_control} gives $\bar\lambda_{t-1}\leq C_\delta\sqrt d$.

For the first-order analysis, we need the quantities $B^h_{t,1},B^h_{t,2},B^h_{t,3}$ and $B^f_{t,1},B^f_{t,2},B^f_{t,3}$ defined in Appendix~\ref{sec:first_order_analysis}. Using that $f_t$ is the identity, it is easy to check that
\begin{align*}
B^f_{t,1}(\rho_t,\pnorm{\lambda_{t-1,[0:t-1]}}{})=B^f_{t,2}(\rho_t,\pnorm{\lambda_{t-1,[0:t-1]}}{})=1,\qquad B^f_{t,3}(\rho_t,\pnorm{\lambda_{t-1,[0:t-1]}}{})=0,
\end{align*}
hence by Lemma \ref{lem:fo_v2u}, $K^f_t\leq2$ for sufficiently large $n$, and this lemma yields
\begin{align}\label{eq:strong_first_order_next_U_induction}
\pnorm{\Delta_{U,t}}{}\leq2\pnorm{\Delta_{V,t}}{}+C_{\delta,M}\sqrt{t\log n}\leq C_{\delta,M}\log^{C_{\str}}n.
\end{align}
This yields the desired claim (\ref{eq:strong_first_order_residual_iteration}) for $\pnorm{\Delta_{U,t}}{}$.

On the other hand, we claim that $B^h_{t,j}(\rho_t,\sigma_t)=O_{\delta,M}(1)$ for $j=1,2,3$. Let us show the claim for $B^h_{t,1}(\rho_t,\sigma_t)$. For large enough $n$, we have $\pnorm{\Delta_{U,t}}{}\vee\pnorm{\Delta_{V,t}}{}\leq C_{\delta,M}\log^{C_{\str}}n\leq n^{1/4}$. The estimates \eqref{eq:strong_first_order_mu_sigma} and \eqref{eq:strong_first_order_denominator_lower} therefore give $\hat\mu_t\leq M+1$, $1/2\leq\sigma_t\leq2$, and $N_t^{-1}\leq C/(\sqrt\delta\,c_{\weak})$.

Finally, by the derivative computations in Lemma \ref{lem:h_deriv_comp}, we have
\begin{align}\label{eq:strong_first_order_current_derivative_envelope}
|\partial_u^j h_{\hat\mu_t}^\star(u,y)|\leq C(1+M)^{j+2}(1+y^{(j+1)/2}),\qquad 1\leq j\leq3.
\end{align}
Then by definition of $B^h_{t,1}$, we have
\begin{align*}
B^h_{t,1}(\rho_t,\sigma_t)\leq\frac{C\delta(1+M)^6}{N_t^2}\E\big[(1+g^2)^2\big]\leq C_{\delta,M},\qquad g\sim\N(0,1).
\end{align*}
The bounds of $B^h_{t,2},B^h_{t,3}$ follow similarly so we will omit their proofs. In particular,
\begin{align}
\max_{1\leq j\leq3}B^h_{t,j}(\rho_t,\sigma_t)&\leq\frac{C(1+M)^8}{c_{\weak}^2},\label{eq:strong_first_order_population_derivatives}\\
\max_{1\leq j\leq3}\pnorm{h_t^{(j)}}{\infty}\vee\pnorm{\vartheta_t}{}&\leq C_{\delta,M}\log^2n.\label{eq:strong_first_order_derivative_bounds}
\end{align}
With this, $K_t^h$ in Lemma \ref{lem:fo_u2v} satisfies
\begin{align*}
K^h_t(\rho_t,\sigma_t)\leq C_{\delta,M}\Big(1+\frac{\log^4n}{n^{1/4}}\Big)\leq C_{\delta,M},
\end{align*}
using $t\leq\log^2n$ and $\pnorm{\Delta_{U,t}}{}\leq n^{1/4}$. We also have $\bar\lambda_t+\bar\gamma_t\leq C_{\delta,M}\sqrt d$. Using \eqref{eq:phase_1_assump_U_1} and \eqref{eq:phase_2_assump_U_1} for $s<\tau_{I\!I}$, and \eqref{eq:strong_first_order_derivative_bounds} for $\tau_{I\!I}\leq s\leq t$, we obtain
\begin{align*}
\pnorm{h_{[0:t]}'}{\infty}
=\max_{0\leq s\leq t}\pnorm{h_s'}{\infty}
\leq C_{\delta,M}\log^2n.
\end{align*}
Moreover, since $\gamma_{s,-1}=\rho_s$, we have
\begin{align*}
\pnorm{\gamma_{[0:t],-1}}{\infty}
=\pnorm{\rho_{[0:t]}}{\infty}
\leq C_{\delta,M}.
\end{align*}
Substituting these estimates and $\bar\lambda_t+\bar\gamma_t\leq C_{\delta,M}\sqrt d$ into the additive term in \eqref{eq:first_order_V_recursion}, we obtain
\begin{align*}
&C_\delta(\bar\lambda_t+\bar\gamma_t)
\big(\pnorm{h_{[0:t]}'}{\infty}\vee1\big)^2
\big(\pnorm{\gamma_{[0:t],-1}}{\infty}\vee1\big)
\sqrt{\frac{(t+1)\log n}{n}}\\
&\qquad\leq C_{\delta,M}\log^4n\sqrt{(t+1)\log n}
\leq C_{\delta,M}\log^6n,
\end{align*}
where we used $n=\delta d$ and $t\leq\log^2n$. Lemma \ref{lem:fo_u2v} therefore yields
\begin{align}\label{eq:strong_first_order_V_induction}
\pnorm{\Delta_{V,t+1}}{}\leq C_{\delta,M}\pnorm{\Delta_{U,t}}{} + C_{\delta,M}\log^{C_{\str}}n \leq C_{\delta,M}\log^{C_{\str}}n.
\end{align}
Each step enlarges the constant by a factor depending only on $(\delta,M)$, and there are at most $N_{\delta,M}$ steps, so the final constant depends only on $(\delta,M)$. This completes the induction.
\end{proof}
\section{First-order analysis of AMP approximation}
\label{sec:first_order_analysis}
We work under Assumption~\ref{assump:nonlinearity} and on the events $\cA_t^U,\cA_t^V$ of Definition~\ref{assump:inductive_assumption}, as in the second-order analysis. Fix $\delta>0$ and take $n$ sufficiently large depending on $\delta$. The time $t$ is deterministic and satisfies $t\leq\log^2n$.

For $\rho\in\R$ and $\sigma\in\R_+$, let 
\begin{align}\label{def:B_h}
\begin{aligned}
B^h_{t,1}(\rho,\sigma) &:=  \delta\E\Big(\partial_u h_t(\rho g + \sigma\tilde g,y)\Big)^2, \\
B^h_{t,2}(\rho,\sigma) &:= \delta\E\Big(\tilde g\circ \partial_u h_t(\rho g + \sigma\tilde g,y)\Big)^2, \\
B^h_{t,3}(\rho,\sigma) &:= \delta\E\Big(\partial_u^2 h_t(\rho g + \sigma\tilde g,y)\Big)^2,
\end{aligned}
\end{align}
where $g,\tilde g$ are two independent standard Gaussian variables, $y = g^2$, and the expectation is taken only with respect to $g,\tilde g$ (thus $B^h_{t,1}(\rho,\sigma)$--$B^h_{t,3}(\rho,\sigma)$ are random when $h_t$ is random).
\begin{lemma}\label{lem:fo_u2v}
Let 
\begin{align*}
K^h_{t,1}(\gamma_{t,-1}, \pnorm{\gamma_{t,[0:t]}}{}) &= \max\Big\{\sqrt{B^h_{t,1}(\gamma_{t,-1}, \pnorm{\gamma_{t,[0:t]}}{})}, \sqrt{B^h_{t,2}(\gamma_{t,-1}, \pnorm{\gamma_{t,[0:t]}}{})}\Big\}\\
&\quad + O_\delta\Big(\big(\pnorm{h_t'}{\infty}\vee1\big)\Big(\frac{(t+1)\log n}{n}\Big)^{1/4} + \pnorm{\partial_u^2h_t}{\infty}\sqrt{\frac{(t+1)\log n}{n}}\pnorm{\Delta_{U,t}}{}\Big),\\
K^h_{t,2}(\gamma_{t,-1}, \pnorm{\gamma_{t,[0:t]}}{}) & = \pnorm{\gamma_{t,[0:t]}}{}\Bigg[\sqrt{B^h_{t,3}(\gamma_{t,-1}, \pnorm{\gamma_{t,[0:t]}}{})}\\
&\qquad + O_\delta\Big(\big(\pnorm{\partial_u^2 h_t}{\infty}\vee1\big)\Big(\frac{(t+1)\log n}{n}\Big)^{1/4} + \frac{1}{\sqrt{d}}\pnorm{\partial_u^3 h_t}{\infty}\pnorm{\Delta_{U,t}}{}\Big)\Bigg].
\end{align*}
Under Assumption~\ref{assump:nonlinearity} and on $\cA_t^U\cap\cA_t^V$, with probability $1-O(n^{-12})$,
\begin{align}\label{eq:first_order_V_recursion}
\pnorm{\Delta_{V,t+1}}{} \leq K^h_{t}(\gamma_{t,-1}, \pnorm{\gamma_{t,[0:t]}}{})\pnorm{\Delta_{U,t}}{} + O_\delta\Big(\big(\bar\lambda_t + \bar\gamma_t\big)\big(\pnorm{h_{[0:t]}'}{\infty}\vee1\big)^2(\pnorm{\gamma_{[0:t],-1}}{\infty}\vee 1)\sqrt{\frac{(t+1)\log n}{n}}\Big),
\end{align}
where $K^h_t=K^h_{t,1}+K^h_{t,2}$.
\end{lemma}
\begin{proof}
Recall from (\ref{eq:construct_uv}) that
\begin{align*}
\v^{t+1} &= \sum_{j=-1}^t \q^j \frac{1}{\sqrt{d}} (P_{\bR^{[0:j-1]}}^\perp \g^j)^\top \w^t + \sum_{j=0}^t P_{\Q^{[-1:j]}}^\perp \s^j \frac{1}{\sqrt{d}}\r^{j\top}\w^t - \hat c_t \z^t\\
&= \underbrace{\rho_{t+1}\btheta^\star + \sum_{\ell=0}^t \lambda_{t,\ell}\s^\ell}_{\tilde\v^{t+1}} + \sum_{\ell=0}^t \q^\ell\frac{1}{\sqrt{d}}\iprod{\g^\ell}{\w^t} - \hat c_t\bar\z^t  \underbrace{- \sum_{\ell=0}^{t}\q^\ell \frac{1}{\sqrt{d}} (P_{\bR^{[0:\ell-1]}}\g^\ell)^\top \w^t - \sum_{\ell=0}^t P_{\Q^{[-1:\ell]}} \s^\ell \lambda_{t,\ell}}_{\Delta^{\low}_{V,t+1}}, 
\end{align*}
where $\bar\z^t = P_{\q^{-1}}^\perp \z^t = \sum_{\ell=0}^t \gamma_{t,\ell}\sqrt{d}\q^\ell$, hence we have
\begin{align*}
\Delta_{V,t+1} = \v^{t+1} - \tilde\v^{t+1} = \sum_{\ell=0}^t \sqrt{d}\q^\ell\frac{1}{d}\iprod{\g^\ell}{\w^t} - \hat c_t\bar\z^t + \Delta_{V,t+1}^{\low}.
\end{align*}
This entails
\begin{align}\label{eq:fo_delta_V}
\notag\pnorm{\Delta_{V,t+1}}{} &\leq \underbrace{\Bigpnorm{\sum_{\ell=0}^t \sqrt{d}\q^\ell \frac{1}{d}\iprod{\g^\ell}{\w^t - h_t(\tilde\u^t, \y)}}{}}_{(I)} + \underbrace{|\hat c_t-\tilde c_t|\pnorm{\bar\z^t}{}}_{(I\!I)}\\
&\quad + \underbrace{\Bigpnorm{\sum_{\ell=0}^t \sqrt d\q^\ell \Big(\frac{1}{d}\iprod{\g^\ell}{h_t(\tilde\u^t,\y)} - \tilde c_t \gamma_{t,\ell}\Big)}{}}_{(I\!I\!I)} + \underbrace{\pnorm{\Delta_{V,t+1}^{\low}}{}}_{(IV)},
\end{align}
where $\tilde c_t = \delta \av{h_t'(\tilde\u^t,\y)}$. 

To bound $(I)$, note that there exists some unit-norm $\bomega  \in \spa(\q^0,\ldots,\q^t)$ such that
\begin{align*}
(I) &= \iprod{\sum_{\ell=0}^t \sqrt{d}\q^\ell \frac{1}{d}\iprod{\g^\ell}{\w^t - h_t(\tilde\u^t, \y)}}{\bomega} = \frac{1}{\sqrt{d}}\iprod{\sum_{\ell=0}^t \omega_\ell \g^\ell}{\w^t - h_t(\tilde\u^t, \y)},
\end{align*}
where $\omega_\ell = \iprod{\bomega}{\q^\ell},0\leq\ell\leq t$, and $\omega = (\omega_0,\ldots,\omega_t)$ satisfies $\pnorm{\omega}{} = 1$. By the mean-value theorem, there exists some $\bar\u^t$, with each coordinate between the corresponding coordinates of $\tilde\u^t$ and $\u^t$, such that 
\begin{align*}
\w^t - h_t(\tilde\u^t,\y) = h_t(\u^t,\y) -h_t(\tilde\u^t,\y) = h_t'(\bar\u^t, \y) \circ (\u^t - \tilde\u^t),
\end{align*}
hence
\begin{align}\label{eq:bound_I}
(I) = \frac{1}{\sqrt{d}}\iprod{\Big(\sum_{\ell=0}^t \omega_\ell \g^\ell\Big) \circ h_t'(\bar\u^t,\y)}{\u^t-\tilde\u^t} \leq \frac{1}{\sqrt{d}}\Bigpnorm{\Big(\sum_{\ell=0}^t \omega_\ell \g^\ell\Big) \circ  h_t'(\bar\u^t,\y)}{}\pnorm{\Delta_{U,t}}{}.
\end{align}
Using $\pnorm{a\circ b}{} \leq \pnorm{a}{\infty}\pnorm{\b}{}$, with probability $1-O(n^{-12})$, it holds uniformly over all $\beps\in\R^n$ that
\begin{align*}
\biggpnorm{\Big(\sum_{\ell=0}^t \omega_\ell \g^\ell\Big) \circ \big(h_t'(\tilde\u^t + \beps,\y) - h_t'(\tilde\u^t,\y)\big)}{} &\leq \Bigpnorm{\sum_{\ell=0}^t \omega_\ell \g^\ell}{\infty}\pnorm{h_t'(\tilde\u^t + \beps,\y) - h_t'(\tilde\u^t,\y)}{}\\
&\leq\pnorm{h_t''}{\infty}\pnorm{\beps}{}\Bigpnorm{\sum_{\ell=0}^t\omega_\ell \g^\ell}{\infty}\\
&\leq C\pnorm{h_t''}{\infty}\sqrt{(t+1)\log n}\pnorm{\beps}{},
\end{align*}
where we apply the fact that $\pnorm{\sum_{\ell=0}^t\omega_\ell \g^\ell}{\infty} = O( \sqrt{(t+1)\log n})$ with the same probability uniformly over $\omega = (\omega_0,\ldots,\omega_t) \in \sS^t$. Applying this fact with $\beps = \bar\u^t - \tilde\u^t$ which satisfies $\pnorm{\bar\u^t - \tilde\u^t}{} \leq \pnorm{\Delta_{U,t}}{}$, 
\begin{align}\label{eq:h_prime_approx}
\frac{1}{\sqrt{d}}\biggpnorm{\Big(\sum_{\ell=0}^t \omega_\ell \g^\ell\Big) \circ \big(h_t'(\bar\u^t,\y) - h_t'(\tilde\u^t,\y)\big)}{} \leq C\pnorm{h_t''}{\infty}\sqrt{\frac{(t+1)\log n}{d}}\pnorm{\Delta_{U,t}}{}. 
\end{align}
It remains to bound $\frac{1}{\sqrt d}\Bigpnorm{\big(\sum_{\ell=0}^t\omega_\ell\g^\ell\big)\circ h_t'(\tilde\u^t,\y)}{}$. The difficulty is the dependence on the random coefficients $\omega$ which we handle by a uniform concentration argument.

Recall that $\Theta_t$ is the random parameter set for $h_t$ defined in Assumption \ref{assump:nonlinearity}, where for phase retrieval, $\Theta_t=\R_+\times\R$ and $\vartheta_t=(\hat\mu_t,\pi_t)$ by \eqref{eq:phase_retrieval_parameter_tuple}. Let
\begin{align*}
\mathcal I_t
\coloneqq\Big\{(\gamma,\vartheta,\omega,L)\in\R^{t+2}\times\Theta_t\times\sS^t\times[1,n^5]:
\pnorm{\gamma}{}\vee\pnorm{\vartheta}{}\leq n^5\Big\}.
\end{align*}
Here $L$ caps the derivative in the concentration argument. By \eqref{eq:parameter_cover_assumption}, the tuple has $O(t+1)$ coordinates, each polynomially bounded in $n$. Standard Euclidean nets therefore give an $\epsilon$-net $\mathcal N_\epsilon\subset\mathcal I_t$, where $\epsilon=n^{-C_\epsilon}$, such that
\begin{align*}
|\mathcal N_\epsilon|\leq\exp\big\{C_2(t+1)\log n\big\}.
\end{align*}

For each fixed net point, write
\begin{align*}
G_i&=(g_i^{-1},g_i^0,\ldots,g_i^t),\qquad
U_i=\sum_{\ell=-1}^t\gamma_\ell g_i^\ell,\qquad
Z_i=\sum_{\ell=0}^t\omega_\ell g_i^\ell,\\
X_i&=Z_i^2\min\big\{\mathsf h_t'(U_i,y_i;\vartheta)^2,L^2\big\},\qquad
y_i=(g_i^{-1})^2.
\end{align*}
The variables $X_i$ are independent, and $Z_i\sim\N(0,1)$. Since $0\leq X_i\leq L^2Z_i^2$, Gaussian moments give
\begin{align*}
\E|X_i-\E X_i|^p\leq2^pL^{2p}\E|Z_i|^{2p}\leq p!(CL^2)^p,
\qquad p=2,3,\ldots.
\end{align*}
Thus, Bernstein's inequality gives, for every $x\geq1$,
\begin{align*}
\P\Bigg(\bigg|\frac{1}{d}\sum_{i=1}^n(X_i-\E X_i)\bigg|
>C\delta L^2\Bigg[\sqrt{\frac{x}{n}}+\frac{x}{n}\Bigg]\Bigg)
\leq2e^{-x}.
\end{align*}
Take $x=C_1(t+1)\log n$. Since $t\leq\log^2n$, $\frac{x}{n}\leq\sqrt{\frac{x}{n}}$ for sufficiently large $n$. After fixing $C_\epsilon$, choose $C_1$ sufficiently large and take the union bound over $\mathcal N_\epsilon$. With probability $1-O(n^{-14})$, we obtain
\begin{align*}
\max_{(\gamma,\vartheta,\omega,L)\in\mathcal N_\epsilon}
\frac{1}{L^2}\bigg|\frac{1}{d}\sum_{i=1}^n(X_i-\E X_i)\bigg|
\leq C\delta\sqrt{\frac{(t+1)\log n}{n}}.
\end{align*}

Lastly, we extend to the entire parameter set. For $(\gamma,\vartheta,\omega,L)\in\mathcal I_t$, choose a net point $(\bar\gamma,\bar\vartheta,\bar\omega,\bar L)$ satisfying
\begin{align*}
\pnorm{\gamma-\bar\gamma}{}+\pnorm{\vartheta-\bar\vartheta}{}+\pnorm{\omega-\bar\omega}{}+|L-\bar L|\leq\epsilon.
\end{align*}
Let $\bar U_i$, $\bar Z_i$, and $\bar X_i$ denote the corresponding quantities at the net point. Then
\begin{align*}
|U_i-\bar U_i|\vee|Z_i-\bar Z_i|\leq\epsilon\pnorm{G_i}{},\qquad
|U_i|\vee|\bar U_i|\leq n^5\pnorm{G_i}{},\qquad
|Z_i|\vee|\bar Z_i|\leq\pnorm{G_i}{}.
\end{align*}
By these bounds, the mean value theorem, the polynomial derivative bound, and \eqref{eq:parameter_lipschitz_assumption}, we get
\begin{align*}
\big|\mathsf h_t'(U_i,y_i;\vartheta)-\mathsf h_t'(\bar U_i,y_i;\bar\vartheta)\big|
&\quad\leq n^C\big(1+|U_i|+|\bar U_i|+y_i\big)^C|U_i-\bar U_i|
+n^C\big(1+|\bar U_i|+y_i\big)^C\pnorm{\vartheta-\bar\vartheta}{}\\
&\quad\leq n^C\epsilon\big(1+\pnorm{G_i}{}\big)^C.
\end{align*}
Using
\begin{align*}
\big|\min\{a^2,L^2\}-\min\{b^2,\bar L^2\}\big|
\leq2L|a-b|+2\max\{L,\bar L\}|L-\bar L|,
\end{align*}
we obtain
\begin{align*}
|X_i-\bar X_i|
&\leq L^2(|Z_i|+|\bar Z_i|)|Z_i-\bar Z_i|+2L\bar Z_i^2\big|\mathsf h_t'(U_i,y_i;\vartheta)-\mathsf h_t'(\bar U_i,y_i;\bar\vartheta)\big|
+2\max\{L,\bar L\}\bar Z_i^2|L-\bar L|\\
&\leq n^C\epsilon\big(1+\pnorm{G_i}{}\big)^C.
\end{align*}
The Gaussian coordinate bound $\max_{i,\ell}|g_i^\ell|\leq C\sqrt{\log n}$ holds with probability $1-O(n^{-20})$. Together with $\E(1+\pnorm{G_i}{})^C\leq C(t+2)^{C/2}$, this gives
\begin{align*}
\frac{1}{d}\sum_{i=1}^n\Big(|X_i-\bar X_i|+|\E X_i-\E\bar X_i|\Big)
\leq\delta n^C\epsilon\leq\delta n^{-20}
\end{align*}
for sufficiently large $C_\epsilon$. Moreover, $|L-\bar L|\leq\epsilon$ and $L\geq1$, so $\bar L\leq2L$. The bound on the net therefore implies, uniformly over $\mathcal I_t$,
\begin{align*}
\bigg|\frac{1}{d}\sum_{i=1}^n(X_i-\E X_i)\bigg|
&\leq\bigg|\frac{1}{d}\sum_{i=1}^n(\bar X_i-\E\bar X_i)\bigg|+\delta n^{-20}\\
&\leq C\delta\bar L^2\sqrt{\frac{(t+1)\log n}{n}}+\delta n^{-20}\\
&\leq C\delta L^2\sqrt{\frac{(t+1)\log n}{n}}.
\end{align*}

We can now substitute the AMP coefficients and parameters and take $L=\pnorm{h_t'}{\infty}\vee1$, which is at most $n^5$ by \eqref{eq:cond_derivative_linfty_h} for sufficiently large $n$. For this choice, we have $\E X_i\leq\E\big[Z_i^2\mathsf h_t'(U_i,y_i;\vartheta)^2\big]$. Thus,
\begin{align*}
\frac{1}{d}\sum_{i=1}^n Z_i^2\mathsf h_t'(U_i,y_i;\vartheta)^2
=\frac{1}{d}\sum_{i=1}^n X_i
&\leq\frac{1}{d}\sum_{i=1}^n\E X_i
+C\delta L^2\sqrt{\frac{(t+1)\log n}{n}}\\
&\leq\frac{1}{d}\sum_{i=1}^n
\E\big[Z_i^2\mathsf h_t'(U_i,y_i;\vartheta)^2\big]
+C\delta L^2\sqrt{\frac{(t+1)\log n}{n}}.
\end{align*}
Returning to the original notation, we have proved, with probability $1-O(n^{-12})$,
\begin{align}\label{eq:partial_u_concentration}
&\frac{1}{d}\Bigpnorm{\Big(\sum_{\ell=0}^t\omega_\ell\g^\ell\Big)\circ h_t'(\tilde\u^t,\y)}{}^2\leq\frac{1}{d}\E_{\{\g^\ell\}_{\ell=-1}^t}\Bigpnorm{\Big(\sum_{\ell=0}^t\omega_\ell\g^\ell\Big)\circ h_t'(\tilde\u^t,\y)}{}^2
+C\delta\big(\pnorm{h_t'}{\infty}\vee1\big)^2\sqrt{\frac{(t+1)\log n}{n}}.
\end{align}
Here
\begin{align*}
\frac{1}{d}\E_{\{\g^\ell\}_{\ell=-1}^t} \Bigpnorm{\Big(\sum_{\ell=0}^t \omega_\ell \g^\ell\Big)\circ h_t'(\tilde \u^t, \y)}{}^2 &= \frac{1}{d} \E_{\{\g^\ell\}_{\ell=-1}^t} \Bigpnorm{\Big(\sum_{\ell=0}^t \omega_\ell \g^\ell\Big) \circ  h_t'\Big(\gamma_{t,-1}\g^{-1} + \sum_{\ell=0}^{t}\gamma_{t,\ell}\g^\ell, (\g^{-1})^2\Big)}{}^2. 
\end{align*}
Let $\bZ_1 = \sum_{\ell=0}^t \omega_\ell \g^\ell$ and $\bZ_2 = \sum_{\ell=0}^t \gamma_{t,\ell}\g^\ell$. Then $(\bZ_1,\bZ_2)$ has i.i.d. rows distributed as $(Z_1,Z_2)$, which are jointly centered Gaussian with variances $\sigma_1^2 := \Var Z_1 = 1$, $\sigma_2^2 = \Var Z_2 = \pnorm{\gamma_{t,[0:t]}}{}^2$, and, when $\sigma_2>0$, correlation $r = \sum_{\ell=0}^t \omega_\ell \gamma_{t,\ell}/\sigma_2 = \iprod{\bomega}{\bar\z^t}/\pnorm{\bar\z^t}{}$. Recall the elementary fact that for $(\bZ_1,\bZ_2)$ with such covariance structure and separable function $F:\R^n \rightarrow \R^n$,
\begin{align*}
\E[\pnorm{\bZ_1 \circ F(\bZ_2)}{}^2] = (1-r^2)\E[\pnorm{F(\bZ_2)}{}^2] + \frac{r^2}{\sigma_2^2}\E[\pnorm{\bZ_2 \circ F(\bZ_2)}{}^2].
\end{align*}
Applying this fact with $F(\bZ_2) =  h_t'(\gamma_{t,-1}\g^{-1} + \bZ_2,\y)$, we have (by first conditioning on $\g^{-1}$)
\begin{align*}
&\frac{1}{d} \E_{\{\g^\ell\}_{\ell=-1}^t} \Bigpnorm{\Big(\sum_{\ell=0}^t \omega_\ell \g^\ell\Big) \circ  h_t'\Big(\gamma_{t,-1}\g^{-1} + \sum_{\ell=0}^{t}\gamma_{t,\ell}\g^\ell, (\g^{-1})^2\Big)}{}^2\\
&= \frac{1-r^2}{d}\E\Bigpnorm{ h_t'\big(\gamma_{t,-1}\g^{-1} + \pnorm{\gamma_{t,[0:t]}}{}\tilde \g, \y)}{}^2 + \frac{r^2}{d} \E\Bigpnorm{\tilde \g\circ h_t'\big(\gamma_{t,-1}\g^{-1} + \pnorm{\gamma_{t,[0:t]}}{}\tilde \g, \y\big)}{}^2\\
&\leq \max\Big\{B^h_{t,1}(\gamma_{t,-1}, \pnorm{\gamma_{t,[0:t]}}{}), B^h_{t,2}(\gamma_{t,-1}, \pnorm{\gamma_{t,[0:t]}}{})\Big\}.
\end{align*}
Plugging this into (\ref{eq:partial_u_concentration}), we have
\begin{align}\label{eq:h_prime_main}
\frac{1}{\sqrt d}\Bigpnorm{\sum_{\ell=0}^t \omega_\ell \g^\ell h_t'(\tilde \u^t, \y)}{} &\leq \max\Big\{\sqrt{B^h_{t,1}(\gamma_{t,-1}, \pnorm{\gamma_{t,[0:t]}}{})}, \sqrt{B^h_{t,2}(\gamma_{t,-1}, \pnorm{\gamma_{t,[0:t]}}{})}\Big\}\notag\\
&\quad + O_\delta\Big(\big(\pnorm{h_t'}{\infty}\vee1\big)\Big(\frac{(t+1)\log n}{n}\Big)^{1/4}\Big).
\end{align}
Combining (\ref{eq:h_prime_approx}) with (\ref{eq:h_prime_main}), we have
\begin{align*}
\frac{1}{\sqrt{d}}\Bigpnorm{\Big(\sum_{\ell=0}^t \omega_\ell \g^\ell\Big) \circ h_t'(\bar\u^t,\y)}{} &\leq K^h_{t,1}(\gamma_{t,-1}, \pnorm{\gamma_{t,[0:t]}}{}).
\end{align*}
Plugging this into (\ref{eq:bound_I}), we conclude that
\begin{align*}
(I) \leq K^h_{t,1}(\gamma_{t,-1}, \pnorm{\gamma_{t,[0:t]}}{})\pnorm{\Delta_{U,t}}{}.
\end{align*}

To bound $(I\!I)$, recall that $\hat c_t = \delta \av{h_t'(\u^t,\y)}$, $\tilde c_t = \delta \av{h_t'(\tilde\u^t,\y)}$, and $\pnorm{\bar\z^t}{} = \sqrt{d}\pnorm{\gamma_{t,[0:t]}}{}$, hence
\begin{align*}
(I\!I) &\leq \delta \Big|\av{h_t'(\u^t,\y)} - \av{h_t'(\tilde\u^t,\y)}\Big|\cdot \sqrt{d}\pnorm{\gamma_{t,[0:t]}}{}=\delta\sqrt{d}\pnorm{\gamma_{t,[0:t]}}{} \cdot \Big|\av{h_t''(\bar\u^t,\y) \circ (\u^t - \tilde\u^t)}\Big|, 
\end{align*}
where $\bar\u^t$ is another point lying on the segment connecting $\tilde\u^t$ and $\u^t$. We first note that
\begin{align*}
 \Big|\av{\big(\partial_u^2 h_t(\bar\u^t,\y) - \partial_u^2 h_t(\tilde\u^t,\y)\big)\circ (\u^t - \tilde\u^t)}\Big| &\leq \frac{1}{n}\pnorm{\partial_u^2 h_t(\bar\u^t,\y) - \partial_u^2 h_t(\tilde\u^t,\y)}{}\pnorm{\Delta_{U,t}}{}\\
 &\leq \frac{1}{n}\pnorm{\partial_u^3 h_t}{\infty}\pnorm{\Delta_{U,t}}{}^2,
\end{align*}
which further implies that
\begin{align}\label{eq:onsager_G_1}
\delta\sqrt{d}\pnorm{\gamma_{t,[0:t]}}{} \cdot \Big|\av{\big(\partial_u^2 h_t(\bar\u^t,\y) - \partial_u^2 h_t(\tilde\u^t,\y)\big) \circ (\u^t - \tilde\u^t)}\Big| \leq \frac{1}{\sqrt{d}}\pnorm{\gamma_{t,[0:t]}}{}\pnorm{\partial_u^3 h_t}{\infty}\pnorm{\Delta_{U,t}}{}^2.
\end{align}
On the other hand, we have
\begin{align*}
\Big|\av{\partial_u^2 h_t(\tilde\u^t, \y) \circ (\u^t - \tilde\u^t)}\Big| \leq \frac{1}{n}\pnorm{\partial_u^2 h_t(\tilde\u^t,\y)}{}\pnorm{\Delta_{U,t}}{},
\end{align*}
which further implies that
\begin{align*}
\delta\sqrt{d}\pnorm{\gamma_{t,[0:t]}}{} \cdot \Big|\av{ \partial_u^2 h_t(\tilde\u^t,\y) \circ (\u^t - \tilde\u^t)}\Big| \leq \frac{1}{\sqrt{d}}\pnorm{\gamma_{t,[0:t]}}{}\pnorm{\partial_u^2 h_t(\tilde\u^t,\y)}{}\pnorm{\Delta_{U,t}}{}. 
\end{align*}
Let us further bound $\frac{1}{\sqrt d}\pnorm{\partial_u^2h_t(\tilde\u^t,\y)}{}$. By a similar concentration argument as above, we get
\begin{align}\label{eq:partial_u2_concentration}
\frac{1}{d}\Bigpnorm{\partial_u^2 h_t(\tilde\u^t,\y)}{}^2 &\leq \underbrace{\E_{\{\g^\ell\}_{\ell=-1}^{t}}\frac{1}{d}\Bigpnorm{\partial_u^2 h_t(\gamma_{t,-1}\g^{-1} + \sum_{\ell=0}^t\gamma_{t,\ell}\g^\ell,\y)}{}^2}_{B^h_{t,3}(\gamma_{t,-1}, \pnorm{\gamma_{t,[0:t]}}{})} + C\delta\big(\pnorm{\partial_u^2 h_t}{\infty}\vee1\big)^2\sqrt{\frac{(t+1)\log n}{n}},
\end{align}
then
\begin{align}\label{eq:onsager_G_2}
\notag&\delta\sqrt{d}\pnorm{\gamma_{t,[0:t]}}{} \cdot \Big|\av{\partial_u^2 h_t(\tilde\u^t,\y) \circ (\u^t - \tilde\u^t)}\Big|\\
&\quad\leq \pnorm{\gamma_{t,[0:t]}}{}\Big(\sqrt{B^h_{t,3}(\gamma_{t,-1}, \pnorm{\gamma_{t,[0:t]}}{})} + O_\delta\Big(\big(\pnorm{\partial_u^2 h_t}{\infty}\vee1\big)\Big(\frac{(t+1)\log n}{n}\Big)^{1/4}\Big)\Big)\pnorm{\Delta_{U,t}}{}.
\end{align}
Combining (\ref{eq:onsager_G_1}) and (\ref{eq:onsager_G_2}), we have
\begin{align*}
(I\!I) \leq K^h_{t,2}(\gamma_{t,-1},\pnorm{\gamma_{t,[0:t]}}{})\pnorm{\Delta_{U,t}}{}.
\end{align*}

For $(I\!I\!I)$, we use the same steps, first conditioning on $\g^{-1}$ and $\y$. Orthonormality of the $\q^\ell$ gives
\begin{align*}
(I\!I\!I)
=\sup_{\omega\in\sS^t}\Bigg|\frac{1}{\sqrt d}\sum_{i=1}^n
\Big[Z_{\omega,i}h_t(\tilde u_i^t,y_i)
-\langle\omega,\gamma_{t,[0:t]}\rangle h_t'(\tilde u_i^t,y_i)\Big]\Bigg|,
\end{align*}
where $Z_{\omega,i}=\sum_{\ell=0}^t\omega_\ell g_i^\ell$. Write $Z_{\gamma,i}=\sum_{\ell=0}^t\gamma_{t,\ell}g_i^\ell$, $\sigma=\pnorm{\gamma_{t,[0:t]}}{}$, and $b_i=h_t(\gamma_{t,-1}g_i^{-1},y_i)$. Gaussian integration by parts gives
\begin{align*}
\E\big[Z_{\omega,i}h_t(\gamma_{t,-1}g_i^{-1}+Z_{\gamma,i},y_i)\big]
&=\langle\omega,\gamma_{t,[0:t]}\rangle
\E\big[h_t'(\gamma_{t,-1}g_i^{-1}+Z_{\gamma,i},y_i)\big].
\end{align*}
Here $\g^{-1},\y,$ and the parameters are held fixed. Thus each summand is centered. Split it as
\begin{align*}
&b_iZ_{\omega,i}
+Z_{\omega,i}\big[h_t(\gamma_{t,-1}g_i^{-1}+Z_{\gamma,i},y_i)-b_i\big]-\langle\omega,\gamma_{t,[0:t]}\rangle h_t'(\gamma_{t,-1}g_i^{-1}+Z_{\gamma,i},y_i).
\end{align*}
The sum of the first terms is Gaussian with variance $\pnorm{\b}{}^2$. The remaining part is centered and, by Lipschitz continuity, has absolute value at most
$
\pnorm{h_t'}{\infty}\big(|Z_{\omega,i}||Z_{\gamma,i}|+\sigma\big).$
Gaussian moments give a sub-exponential scale of at most $C\pnorm{h_t'}{\infty}\sigma$. The Gaussian tail bound for the first sum and Bernstein's inequality for the remainder therefore give
\begin{align*}
&\Bigg|\frac{1}{\sqrt d}\sum_{i=1}^n\Big[Z_{\omega,i}h_t(\tilde u_i^t,y_i)-\langle\omega,\gamma_{t,[0:t]}\rangle h_t'(\tilde u_i^t,y_i)\Big]\Bigg|\leq C\pnorm{\b}{}\sqrt{\frac{x}{d}}+C\pnorm{h_t'}{\infty}\sigma\Bigg[\sqrt{\frac{nx}{d}}+\frac{x}{\sqrt d}\Bigg]
\end{align*}
with conditional probability at least $1-Ce^{-x}$.

Set $x=C_1(t+1)\log n$, where $C_1$ is sufficiently large. Conditional on $\g^{-1},\y$, a similar covering argument, including upper bounds for $\pnorm{\b}{}\vee1$ and $\pnorm{h_t'}{\infty}\vee1$ as additional coordinates of the same net, gives
\begin{align*}
(I\!I\!I)
&\leq C\big(\pnorm{\b}{}\vee1\big)\sqrt{\frac{x}{d}}+C\big(\pnorm{h_t'}{\infty}\vee1\big)\sigma
\Bigg[\sqrt{\frac{nx}{d}}+\frac{x}{\sqrt d}\Bigg]+n^{-20}.
\end{align*}

Finally, by the definitions of $\bar\lambda_t,\bar\gamma_t$,
\begin{align*}
\pnorm{\b}{}
&\leq\bar\lambda_t+C\sqrt\delta\,\pnorm{h_t'}{\infty}\bar\gamma_t,
\qquad \sqrt d\,\sigma\leq\bar\gamma_t.
\end{align*}
Consequently,
\begin{align*}
(I\!I\!I)\leq C_\delta\big(\bar\lambda_t+\bar\gamma_t\big)
\big(\pnorm{h_{[0:t]}'}{\infty}\vee1\big)^2
\big(\pnorm{\gamma_{[0:t],-1}}{\infty}\vee1\big)
\sqrt{\frac{(t+1)\log n}{n}}.
\end{align*}

Lastly, we note that by Lemma \ref{lem:bound_lowd_projection}, 
\begin{align*}
(IV) = \pnorm{\Delta_{V,t+1}^{\low}}{} \leq C\pnorm{\lambda_{t,[0:t]}}{}\sqrt{(t+1)\log n} \leq C\sqrt{\frac{(t+1)\log n}{d}}\bar\lambda_t.
\end{align*}
Plugging the estimates of $(I)$-$(IV)$ into (\ref{eq:fo_delta_V}) completes the proof.
\end{proof}

We have a similar bound for the $f$ nonlinearity. For $\rho\in\R$ and $\sigma\in\R_+$, let 
\begin{align}\label{def:B_f}
\begin{aligned}
B^f_{t,1}(\rho,\sigma) &:=  \frac{1}{d}\E\Big[\Bigpnorm{f_t'(\rho\btheta^\star + \sigma \bZ)}{}^2\Big], \\
B^f_{t,2}(\rho,\sigma) &:=  \frac{1}{d}\E\Big[\Bigpnorm{\bZ \circ f_t'(\rho\btheta^\star + \sigma \bZ)}{}^2\Big], \\
B^f_{t,3}(\rho,\sigma) &:=  \frac{1}{d} \E\Big[\Bigpnorm{f_t''(\rho\btheta^\star + \sigma \bZ)}{}^2\Big],
\end{aligned}
\end{align}
where expectation is taken only over $\bZ \sim \N(0,\I_d)$ (thus $B^f_{t,1}(\rho,\sigma)$-$B^f_{t,3}(\rho,\sigma)$ are random when $f_t$ is random).

\begin{lemma}\label{lem:fo_v2u}
Let $t\geq1$ and let
\begin{align*}
K^f_{t,1}(\rho_t, \pnorm{\lambda_{t-1,[0:t-1]}}{}) &= \max\Big\{\sqrt{B^f_{t,1}(\rho_t, \pnorm{\lambda_{t-1,[0:t-1]}}{})}, \sqrt{B^f_{t,2}(\rho_t, \pnorm{\lambda_{t-1,[0:t-1]}}{})}\Big\}\\
&\quad + O_\delta\Big(\big(\pnorm{f_t'}{\infty}\vee1\big)\Big(\frac{t\log n}{n}\Big)^{1/4} + \pnorm{f_t''}{\infty}\sqrt{\frac{t\log n}{n}}\pnorm{\Delta_{V,t}}{}\Big),\\
K^f_{t,2}(\rho_t, \pnorm{\lambda_{t-1,[0:t-1]}}{}) & = \pnorm{\lambda_{t-1,[0:t-1]}}{}\Bigg[\sqrt{B^f_{t,3}(\rho_t, \pnorm{\lambda_{t-1,[0:t-1]}}{})}\\
&\qquad + O_\delta\Big(\big(\pnorm{f_t''}{\infty}\vee1\big)\Big(\frac{t\log n}{n}\Big)^{1/4} + \frac{1}{\sqrt{d}}\pnorm{f_t'''}{\infty}\pnorm{\Delta_{V,t}}{}\Big)\Bigg].
\end{align*}
Under Assumption~\ref{assump:nonlinearity} and on $\cA_t^U\cap\cA_t^V$, with probability $1-O(n^{-12})$,
\begin{align}\label{eq:first_order_U_recursion}
\pnorm{\Delta_{U,t}}{} \leq K^f_{t}(\rho_t, \pnorm{\lambda_{t-1,[0:t-1]}}{})\pnorm{\Delta_{V,t}}{} + O_\delta\Big(\big(\bar\gamma_t + \bar\lambda_{t-1}\big)\big(\pnorm{f_{[0:t]}'}{\infty}\vee1\big)^2(\pnorm{\rho_{[0:t]}}{\infty}\vee 1)\sqrt{\frac{t\log n}{n}}\Big),
\end{align}
where $K^f_{t} = K^f_{t,1} + K^f_{t,2}$.
\end{lemma}
\begin{proof}
    The proof is analogous to that of Lemma \ref{lem:fo_u2v} so we will omit it.
\end{proof}

\section{Additional Background}
\label{sec:background}
\subsection{Review of AMP and Bayes-optimal nonlinearity}\label{sec:bayes_amp}
Let
\begin{align}\label{eq:glm_model}
\y = \varphi(\X\btheta^\star, \beps),
\end{align}
where $\X \in \R^{n\times d}$ has i.i.d. $\N(0,1/d)$ entries, $\btheta^\star \in \R^d$ is deterministic, $\beps \in \R^n$ is some external noise independent of $\X$, and $\varphi:\R\rightarrow\R$ is some nonlinearity applied entrywise. We assume that almost surely,
\begin{align*}
\frac{1}{d}\sum_{j=1}^d \delta_{\theta^\star_j} \overset{W_2}{\to} \sP(\theta^\star), \quad \frac{1}{n}\sum_{i=1}^n \delta_{\eps_i} \overset{W_2}{\to} \sP(\eps), 
\end{align*}
for some independent scalar variables $\theta^\star$ and $\eps$. We consider a general AMP algorithm of the form (\ref{eq:main_amp}):
\begin{equation*}
\begin{aligned}
\v^{t+1} &= \X^\top h_t(\u^t,\y) - \hat c_t f_t(\v^t),\\
\u^{t+1} &= \X f_{t+1}(\v^{t+1}) - \hat b_{t+1}h_t(\u^t,\y),
\end{aligned}
\end{equation*}
the first iterate being $\u^0 = \X f_0(\v^0)$ based on initialization $\v^0$ independent of $\X$. The Onsager coefficients are given by (\ref{eq:general_onsager}):
\begin{align*}
\hat c_t = \frac{1}{d}\sum_{i=1}^n h'_t(u^t_i, y_i), \quad \hat b_{t+1} = \frac{1}{d}\sum_{j=1}^d f'_{t+1}(v^{t+1}_j). 
\end{align*}
The distributional characterization is given by
\begin{align}\label{eq:dist_charaterization}
\frac{1}{d}\sum_{j=1}^d \delta_{v^{t+1}_j} \overset{W_2}{\to} \sP(V_{t+1}), \quad \frac{1}{n}\sum_{i=1}^n \delta_{u^t_i}\overset{W_2}{\to} \sP(U_t),
\end{align}
where
\begin{align*}
V_{t+1} = \mu_{V,t+1}\theta^\star + \sigma_{V,t+1}W_{V,t+1},\quad U_{t} = \mu_{U,t}\eta^\star + \sigma_{U,t}W_{U,t},
\end{align*}
and the state evolution is given by (\ref{eq:gamp_se}):
\begin{equation*}
\begin{aligned}
\mu_{U,t} &= \E[\theta^\star f_t(V_t)],\\
\sigma_{U,t}^2 &= \E[f_t(V_t)^2] - \mu_{U,t}^2,\\
\mu_{V,t+1} &= \delta\Big(\E[\eta^\star h_t(U_t,Y)] - \E[h'_t(U_t,Y)]\E[\theta^\star f_t(V_t)]\Big),\\
\sigma_{V,t+1}^2 &= \delta\E[h_t(U_t,Y)^2].
\end{aligned}
\end{equation*}
Here the joint law of $(V_{t+1},U_t,Y)$ is specified with mutually independent $\theta^\star, \eta^\star, W_{V,t+1},W_{U,t}$, where $\theta^\star \sim \sP(\theta^\star)$ and $\eta^\star, W_{V,t+1},W_{U,t}$ are standard Gaussian variables. The state evolution can be computed in the order of
\begin{align*}
(\mu_{U,0}, \sigma_{U,0}) \Rightarrow (\mu_{V,1}, \sigma_{V,1}) \Rightarrow (\mu_{U,1},\sigma_{U,1}) \Rightarrow (\mu_{V,2}, \sigma_{V,2}) \Rightarrow \ldots,
\end{align*}
initialized from $(\mu_{U,0}, \sigma_{U,0}^2) = \lim_{n,d\rightarrow\infty}(\frac{1}{d}\iprod{\btheta^\star}{f_0(\v^0)}, \frac{1}{d}\pnorm{P_{\btheta^\star}^\perp f_0(\v^0)}{}^2)$.

Next we review the Bayes-optimal choice of nonlinearities $\{f_t,h_t\}$ in (\ref{eq:main_amp}) for the single-index model (\ref{eq:glm_model}). The intuition is that after producing $\v^{t+1}$, we will use $f_{t+1}(\v^{t+1})$ as the estimate for $\btheta^\star$, so we choose the nonlinearities $f_t,h_t$ to minimize the MSE at each stage. The specific choice is by induction. For the base case $f_0$, to minimize the MSE of $f_0(\v^0)$, we choose
\begin{align*}
f_0(x) = \E[\theta^\star | V_0 = x].
\end{align*}
Suppose now we have chosen $h_0,\ldots,h_{t-1}$ and $f_0,\ldots, f_t$ which specify (\ref{eq:main_amp}) up to $\v^t, \u^t$, we will now specify $h_t$ which determines $\v^{t+1}$, and then $f_{t+1}$ which determines $\u^{t+1}$.

We first note that once $h_t$ is chosen, the choice of $f_{t+1}$ (to determine $\u^{t+1}$) is straightforward. Since the empirical distribution of $\v^{t+1}$ converges to $\mu_{V,t+1}\theta^\star + \sigma_{V,t+1}W_{V,t+1}$, to minimize the MSE of $f_{t+1}(\v^{t+1})$, we take $f_{t+1}(\cdot)$ to be the optimal choice of posterior mean:
\begin{align*}
f_{t+1}^\star(x) = \E\Big[\theta^\star| \mu_{V,t+1}\theta^\star + \sigma_{V,t+1}W_{V,t+1} = x\Big]. 
\end{align*}
In fact, this choice also corresponds to the one maximizing the SNR $\frac{\mu_{U,t+1}^2}{\sigma_{U,t+1}^2}$. It is clear that the MSE with the above choice depends on the SNR $\frac{\mu_{V,t+1}^2}{\sigma_{V,t+1}^2}$, so $h_t$ is chosen to maximize this SNR. By definition, 
\begin{align*}
\mu_{V,t+1} = \delta\Big(\E[\eta^\star h_t(U_t,Y)] - \E[h_t'(U_t,Y)]\E[\theta^\star f_t(V_t)]\Big) = \delta\Big(\E[\eta^\star h_t(U_t,Y)] - \mu_{U,t} \E[h_t'(U_t,Y)]\Big).
\end{align*}
Using $U_t = \mu_{U,t}\eta^\star + \sigma_{U,t} W_{U,t}$, applying integration-by-parts on $\eta^\star$ further yields
\begin{align*}
\mu_{V,t+1} = \delta\E[\partial_{\eta^\star} h_t(U_t,\varphi(\eta^\star,\eps))],
\end{align*}
where the derivative $\partial_{\eta^\star}$ is with respect to $\eta^\star$ in the second argument of $h_t$ only. Since $\eta^\star|U_t$ is also Gaussian, applying  integration-by-parts in reverse over the conditional distribution $\eta^\star|U_t$ yields
\begin{align*}
\mu_{V,t+1} &= \delta\E\Big[\E\Big[h_t(U_t,\varphi(\eta^\star,\eps))\frac{\eta^\star - \E[\eta^\star|U_t]}{\Var(\eta^\star|U_t)}|U_t\Big]\Big]\\
&= \delta\E\Big[\E\Big[\E\Big[h_t(U_t,\varphi(\eta^\star,\eps))\frac{\eta^\star - \E[\eta^\star|U_t]}{\Var(\eta^\star|U_t)}|U_t,\varphi(\eta^\star,\eps)\Big]|U_t\Big]\Big]\\
&= \delta\E\Big[\E\Big[h_t(U_t,Y)\frac{\E[\eta^\star|U_t,Y] - \E[\eta^\star|U_t]}{\Var(\eta^\star|U_t)}|U_t\Big]\Big]\\
&= \delta\E\Big[h_t(U_t,Y)\frac{\E[\eta^\star|U_t,Y] - \E[\eta^\star|U_t]}{\Var(\eta^\star|U_t)}\Big].
\end{align*}
On the other hand, $\sigma^2_{V,t+1} = \delta\E[h_t^2(U_t,Y)]$, so by Cauchy-Schwarz, the optimal choice of $h_t$ to maximize $\mu_{V,t+1}^2/\sigma_{V,t+1}^2$ is
\begin{align*}
h_t^\star(u,y) = \frac{\E[\eta^\star|U_t = u,Y = y] - \E[\eta^\star|U_t = u]}{\Var(\eta^\star|U_t = u)} = \frac{(\mu_{U,t}^2 + \sigma_{U,t}^2)\E[\eta^\star|U_t = u,Y = y] - \mu_{U,t}u}{\sigma_{U,t}^2},
\end{align*}
where the last step follows from 
\begin{align*}
\E[\eta^\star|U_t] = \frac{\mu_{U,t}}{\mu_{U,t}^2 + \sigma_{U,t}^2}U_t, \quad \Var[G|U_t] = \frac{\sigma_{U,t}^2}{\mu_{U,t}^2 + \sigma_{U,t}^2}.
\end{align*}

\textbf{Un-normalized Bayes-optimal AMP for noiseless phase retrieval}
Consider the noiseless phase retrieval model $\y = (\X\btheta^\star)^2$, with the AMP nonlinearities in (\ref{eq:main_amp}) chosen as 
\begin{align}\label{eq:bayes_opt_choice}
f_t(x) = x, \quad h_t(u,y) = h_t^\star(u,y), \quad t\geq 0.
\end{align}
Then the AMP iterates (\ref{eq:main_amp}) become
\begin{align*}
\v^{t+1} = \X^\top h_t^\star(\u^t,\y) - \hat{c}_t\v^t, \qquad \u^{t+1} = \X \v^{t+1} - \hat{b}_{t+1} h_t^\star(\u^t,\y),
\end{align*}
where the Onsager coefficients are
\begin{align}\label{eq:unscaled_onsager}
\hat{b}_{t+1} := \frac{1}{d}\sum_{j=1}^d f_{t+1}'(v^{t+1}_j) = 1, \quad \hat c_t := \frac{1}{d}\sum_{i=1}^n h_t'(u_i;y_i) = \frac{1}{d}\sum_{i=1}^n \partial_u h_t^{\star}(u_i,y_i).
\end{align}
Let us further simplify the form of $h_t^\star$ and the associated state evolution in (\ref{eq:gamp_se}). First note that by direct calculation, 
\begin{align}\label{eq:mean_noiseless}
\E[\eta^\star|U_t = u,Y = y] = \frac{\sqrt{y}\phi(\frac{u - \mu_{U,t}\sqrt{y}}{\sigma_{U,t}})\phi(\sqrt{y}) - \sqrt{y}\phi(\frac{u + \mu_{U,t}\sqrt{y}}{\sigma_{U,t}})\phi(-\sqrt{y})}{\phi(\frac{u - \mu_{U,t}\sqrt{y}}{\sigma_{U,t}})\phi(\sqrt{y}) + \phi(\frac{u + \mu_{U,t}\sqrt{y}}{\sigma_{U,t}})\phi(-\sqrt{y})} = \sqrt{y}\tanh\Big(\frac{\mu_{U,t}u\sqrt{y}}{\sigma_{U,t}^2}\Big),
\end{align}
and with $\sech(u) = 1/\cosh(u)$,
\begin{align}\label{eq:var_noiseless}
\Var[\eta^\star|U_t = u, Y = y] &= \E[\eta^{\star 2}|U_t = u, Y = y] - \big(\E[\eta^\star|U_t = u, Y = y]\big)^2 = y \sech^2\Big(\frac{\mu_{U,t}u\sqrt{y}}{\sigma_{U,t}^2}\Big).
\end{align}
Moreover, we have
\begin{align}\label{eq:post_mean_derivative}
\partial_u \E[\eta^\star|U_t = u, Y = y] = \frac{\mu_{U,t}}{\sigma_{U,t}^2}\Var[\eta^\star|U_t = u, Y = y] = y\frac{\mu_{U,t}}{\sigma_{U,t}^2}\sech^2\Big(\frac{\mu_{U,t}u\sqrt{y}}{\sigma_{U,t}^2}\Big),
\end{align}
so the Onsager coefficient $\hat c_t$ in (\ref{eq:unscaled_onsager}) can be computed as 
\begin{align}\label{eq:opt_h_derivative}
\notag\partial_u h_t^\star(u,y) &= \frac{(\mu_{U,t}^2 + \sigma_{U,t}^2)\partial_u \E[\eta^\star|U_t = u, Y = y] - \mu_{U,t}}{\sigma_{U,t}^2}\\
&= \frac{y\mu_{U,t}(\mu_{U,t}^2 + \sigma_{U,t}^2)\sech^2\Big(\frac{\mu_{U,t}u\sqrt{y}}{\sigma_{U,t}^2}\Big) - \sigma_{U,t}^2\mu_{U,t}}{\sigma_{U,t}^4}.
\end{align}

Next, note that we can simplify the general SE recursion (\ref{eq:gamp_se}) as
\begin{align}\label{eq:mu_sigma_equal}
\mu_{U,t} = \mu_{V,t} =: \mu_t,\quad \sigma_{U,t}^2 = \sigma_{V,t}^2 =: \sigma_t^2.
\end{align}
The evolution of $\{(\mu_t,\sigma_t)\}_{t\geq 0}$ can be further simplified as follows. For $\mu_{t+1}$, we use the identity
\begin{equation}\label{eq:h_ast_id}
\begin{aligned}
\E[\eta^\star h_t^\star(U_t,Y)] &= \frac{\mu_{t}^2 + \sigma_{t}^2}{\sigma_{t}^2}\E\Big[(\E[\eta^\star|U_t, Y])^2\Big] - \frac{\mu_{t}^2}{\sigma_{t}^2},\\
\E[\partial_u h_t^\star(U_t,Y)] &= \frac{\mu_{t}^2 + \sigma_{t}^2}{\sigma_{t}^2}\frac{\mu_{t}}{\sigma_{t}^2}\E\Big[\Var(\eta^\star|U_t,Y)\Big] - \frac{\mu_{t}}{\sigma_{t}^2},\\
\E\Big[\Var(\eta^\star|U_t,Y)\Big] &= 1 - \E\Big[(\E[\eta^\star|U_t,Y])^2\Big],
\end{aligned}
\end{equation}
to conclude that
\begin{align*}
\mu_{t+1} &= \delta\frac{\mu_{t}^2+\sigma_{t}^2}{\sigma_{t}^2}\Big[\E(\E[G|U_t,Y])^2 - \frac{\mu_{t}^2}{\sigma_{t}^2} \E[\Var(G|U_t,Y)]\Big]\\
&=\delta\frac{\mu_{t}^2+\sigma_{t}^2}{\sigma_{t}^2}\Big[\frac{\mu_{t}^2 + \sigma_{t}^2}{\sigma_{t}^2}\E(\E[G|U_t,Y])^2 - \frac{\mu_{t}^2}{\sigma_{t}^2}\Big].
\end{align*}
Lastly, using $\E[U_t \E[\eta^\star|U_t,Y]] = \mu_{t}$, we have
\begin{align*}
\sigma_{t+1}^2 = \delta\frac{\mu_{t}^2+\sigma_{t}^2}{\sigma_{t}^4}\Big((\mu_{t}^2 + \sigma_{t}^2) \E\Big[(\E[\eta^\star|U_t,Y])^2\Big] - \mu_{t}^2\Big).
\end{align*}
In summary, we have
\begin{align}\label{eq:mu_sigma_V}
\mu_{t+1} = \sigma_{t+1}^2 = \delta\frac{\mu_{t}^2+\sigma_{t}^2}{\sigma_{t}^2}\Big[\frac{\mu_{t}^2 + \sigma_{t}^2}{\sigma_{t}^2}\E(\E[\eta^\star|U_t,Y])^2 - \frac{\mu_{t}^2}{\sigma_{t}^2}\Big],
\end{align}
where the mean formula for $\E[\eta^\star|U_t,Y]$ is given by (\ref{eq:mean_noiseless}). 

With the above mean variance relation $\mu_{t+1} = \sigma_{t+1}^2$, let us simplify the state evolution in terms of $\mu_t$ only. Since $U_t = \mu_{t}\eta^\star + \sqrt{\mu_{t}}W_{U,t}$ and $Y = (\eta^\star)^2$, we have by (\ref{eq:mean_noiseless}) that for any $t\geq 1$,
\begin{align}
\label{eq:post_mean_calc}
\E[\eta^\star|U_t,Y] = |\eta^\star|\tanh\Big(\frac{\mu_{t}(\mu_{t}\eta^\star + \sigma_{t}W_{U,t})|\eta^\star|}{\sigma_{t}^2}\Big) = |\eta^\star|\tanh\Big(\mu_t \eta^\star|\eta^\star| + \sqrt{\mu_t}W_{U,t}|\eta^\star|\Big).
\end{align}
As a shorthand, let us use $m(\mu)$ to denote the second moment of the posterior mean, i.e., 
\begin{align}\label{eq:cond_mean_second_moment}
\notag m(\mu) &: = \E(\E[\eta^\star|U_t = \mu\eta^\star + \sqrt{\mu}W, Y = (\eta^\star)^2])^2\\
&= \E_{G,W\sim \N(0,1)} \Big[G^2\tanh^2(\mu G|G| + \sqrt{\mu}W|G|)\Big].
\end{align}
Using this to simplify (\ref{eq:mu_sigma_V}), we have
\begin{align}\label{eq:population_recursion}
\mu_{t+1} = F_\delta(\mu_t), \quad \text{ where } \quad F_\delta(\mu) = \delta(1+\mu)\Big((1+\mu)m(\mu) - \mu\Big).
\end{align}

\textbf{Normalized Bayes-optimal AMP for noiseless phase retrieval} The normalized choice replaces (\ref{eq:bayes_opt_choice}) by the one in (\ref{eq:rescaled_nonlinearity}):
\begin{align*}
f_t(x) = x, \quad h_t(u,y) = \frac{h_t^\star(u,y)}{\pnorm{h_t^\star(\u^t,\y)}{}/\sqrt{d}}, \quad t\geq 0.
\end{align*}
Following the same argument that leads to (\ref{eq:mu_sigma_equal}), we still have $\mu_{V,t} = \mu_{U,t} =: \mu_t$ and $\sigma_{V,t}^2 = \sigma_{U,t}^2 =: \sigma_{t}^2$, but crucially $\sigma_t^2 = 1$. For the recursion of $\mu_t$, we have by the general SE (\ref{eq:gamp_se}) that
\begin{align*}
\mu_{t+1} = \frac{\delta}{\lim_n \pnorm{h_t^\star(\u^t,\y)}{}/\sqrt{d}}\Big(\E[\eta^\star h_t^\star(U_t,Y)] - \E[\partial_uh_t^\star(U_t,Y)]\mu_t\Big).
\end{align*}
Using the identities of (\ref{eq:h_ast_id}) with $\sigma_t^2\equiv 1$, we have by direct calculation that
\begin{align*}
\E[\eta^\star h_t^\star(U_t,Y)] &= (\mu_t^2 + 1)\E (\E[\eta^\star|U_t,Y])^2 - \mu_t^2,\\
\E[\partial_u h_t^\star(U_t,Y)] &= (\mu_t^2+1)\mu_t(1-\E(\E[\eta^\star|U_t,Y])^2) - \mu_t, \\
\frac{1}{d}\pnorm{h_t^\star(\u^t,\y)}{}^2 &\rightarrow \delta\Big((\mu_t^2+1)^2 \E (\E[\eta^\star|U_t,Y])^2 - \mu_t^2(\mu_t^2+1)\Big),
\end{align*}
and
\begin{align*}
\E (\E[G|U_t = \mu_t G+W, Y = G^2])^2 = \E [G^2\tanh^2(\mu_t^2|G|G + \mu_t |G|W)] = m(\mu_t^2).
\end{align*}
This leads to the recursion
\begin{align*}
\mu_{t+1} = \sqrt{\delta(\mu_t^2+1)} \sqrt{(\mu_t^2+1)m(\mu_t^2) - \mu_t^2},
\end{align*}
or equivalently,
\begin{align}
\label{eq:SE_norm}
\mu_{t+1}^2 = F_\delta(\mu_t^2).
\end{align}

\subsection{Connection to Bolthausen's condition}\label{sec:bolthausen}

In this section, we review Bolthausen's condition \cite{bolthausen2014iterative} in the asymmetric AMP setting. For simplicity, we discuss the standard AMP dynamics with time-independent nonlinearities $f,h$, that is, 
\begin{equation}\label{eq:artificial_amp}
\begin{aligned}
\v^{t+1} &= \X^\top h(\u^t) - \hat c_t f(\v^t),\\
\u^{t+1} &= \X f(\v^{t+1}) - \hat b_{t+1}h(\u^t),
\end{aligned}
\end{equation}
where the initialization is $\v^0$ and the first iterate is $\u^0 = \X f(\v^0)$. As in (\ref{eq:main_amp}), $\X$ has i.i.d. $\N(0,1/d)$ entries, $n/d = \delta$, and $\{c_t\}_{t\geq 0}$ and $\{b_{t+1}\}_{t\geq 0}$ are deterministic Onsager coefficients given by
\begin{align*}
\hat c_t = \frac{1}{d}\sum_{i=1}^n h'(u^t_i), \quad b_{t+1} = \frac{1}{d}\sum_{j=1}^d f'(v^{t+1}_j).
\end{align*}
Suppose that $\frac{1}{d}\sum_{j=1}^d \delta_{v^0_j} \overset{W_2}{\to} \sP(V_0)$ for some scalar variable $V_0$ and $f,h$ are twice continuously differentiable with uniformly bounded first and second derivatives, then for any $t = O(1)$ the empirical distribution of (\ref{eq:artificial_amp}) satisfies that
\begin{align*}
\frac{1}{n}\sum_{i=1}^n \delta_{(u^0_i,\ldots,u^t_i)} \overset{W_2}{\to} \sP(U_0,\ldots,U_t), \quad \frac{1}{d}\sum_{j=1}^d \delta_{v^0_j,v^1_j,\ldots, v^{t+1}_j} \overset{W_2}{\to} \sP(V_0,V_1,\ldots,V_{t+1}),
\end{align*}
where $(U_0,\ldots,U_t)$ and $(V_1,\ldots,V_{t+1})$ are two $(t+1)$-dimensional mean zero Gaussian vectors independent of $V_0$, and the covariance is recursively defined by
\begin{align}\label{eq:state_evolution_app}
\E[U_tU_s] = \E[f(V_t)f(V_s)], \quad \E[V_{t+1}V_{s+1}] = \delta\E[h(U_t)h(U_s)],
\end{align}
with initialization $\E[U_0^2] = \lim_{n,d\rightarrow\infty} \frac{1}{d}\pnorm{f(\v^0)}{}^2 =  \E[f(V_0)^2]$. We define the marginals of (\ref{eq:state_evolution_app}) by
\begin{align*}
\sigma_{U,t}^2 &= \E[f^2(V_t)] = \E_{G\sim \N(0,1)}[f^2(\sigma_{V,t}G)],\\
\sigma_{V,t+1}^2 &= \delta\E[h^2(U_t)] = \delta\E_{G\sim \N(0,1)}[h^2(\sigma_{U,t}G)].
\end{align*}

Suppose that the above marginal system
\begin{align}\label{eq:var_fixed_point}
\sigma_U^2 = \E_{G\sim \N(0,1)}[f^2(\sigma_V G)], \quad \sigma_V^2 = \delta\E_{G\sim N(0,1)} [h^2(\sigma_UG)]
\end{align}
admits a fixed point $(\sigma_{U,\infty}, \sigma_{V,\infty})$ in $\R_+^2$, and we initialize the AMP dynamic (\ref{eq:artificial_amp}) at this fixed point, i.e., 
\begin{align}\label{eq:stable_init}
\z^0 := f(\v^0) = \sigma_{U,\infty}\bm{1}_d.
\end{align}
We will show that under Bolthausen's condition 
\begin{align}\label{eq:bolthausen_app}
\E\Big[\big(f'(\sigma_{V,\infty}G)\big)^2\Big] \cdot \delta\E\Big[\big(h'(\sigma_{U,\infty}G)\big)^2\Big] < 1,
\end{align} 
where $G\sim \N(0,1)$, the AMP dynamic (\ref{eq:artificial_amp}) converges on an $O(1)$ horizon, that is, almost surely as $n,d\rightarrow\infty$,
\begin{align}\label{eq:amp_converge_app}
\lim_{t\rightarrow\infty}\lim_{n,d\rightarrow\infty}\frac{1}{d}\pnorm{f(\v^{t}) - f(\v^{t+1})}{}^2 = 0, \quad \lim_{t\rightarrow\infty}\lim_{n,d\rightarrow\infty}\frac{1}{d}\pnorm{h(\u^{t}) - h(\u^{t+1})}{}^2 = 0.
\end{align}
For simplicity of presentation, we assume the nonlinearities $f,h$ are sufficiently smooth and the state evolution (\ref{eq:state_evolution_app}) holds over any $O(1)$ horizon. 

We first show that under the initialization (\ref{eq:stable_init}), the joint empirical distribution of the AMP dynamic (\ref{eq:artificial_amp}) satisfies a few special properties. In the following, define for any $t\geq 0$
\begin{align}\label{def:z_and_w}
\z^{t} := f(\v^{t}), \quad \w^t := h(\u^t),
\end{align}
and as in (\ref{def:gamp_q_and_r}), let $\{\q^\ell\}_{\ell=0}^t$ and $\{\r^\ell\}_{\ell=0}^t$ be the orthonormal basis for $\{\z^\ell\}_{\ell=0}^{t}$ and $\{\w^\ell\}_{\ell=0}^t$ respectively, i.e.,
\begin{equation}\label{def:q_and_r}
\begin{gathered}
\q^0 := \frac{\z^0}{\pnorm{\z^0}{}}, \quad \q^\ell := \frac{P_{\Q^{\ell-1}}^\perp\z^\ell}{\pnorm{P_{\Q^{\ell-1}}^\perp\z^\ell}{}}, \quad 1\leq \ell\leq t,\\
\r^0 := \frac{\w^0}{\pnorm{\w^0}{}}, \quad \r^\ell := \frac{P^\perp_{\bR^{\ell-1}}\w^\ell}{\pnorm{P^\perp_{\bR^{\ell-1}}\w^\ell}{}}, \quad 1\leq \ell\leq t,
\end{gathered}
\end{equation}
where $\Q^{\ell-1} = [\q^0,\ldots,\q^{\ell-1}]\in\R^{d\times\ell}$, and $\bR^{\ell-1} = [\r^0,\ldots, \r^{\ell-1}] \in \R^{n\times \ell}$. Similar to (\ref{eq:inner_product}), let
\begin{align*}
\hat\gamma_{t,\ell} := \frac{1}{\sqrt{d}}\iprod{\q^\ell}{\z^t}, \quad \hat\lambda_{t,\ell} := \frac{1}{\sqrt{d}}\iprod{\r^\ell}{\w^t}, \quad 0\leq \ell \leq t. 
\end{align*}

\begin{lemma}\label{lem:stable_se}
\begin{enumerate}
\item[(1)] For all $t\geq 0$, it holds that
\begin{equation*}
\begin{gathered}
    \sigma_{U,t}^2 = \sigma_{U,\infty}^2, \quad \sigma_{V,t}^2 = \sigma_{V,\infty}^2, \\
    c_t = c_\infty := \delta\E_{Z\sim \N(0,1)}[h'(\sigma_{U,\infty}Z)], \quad b_{t+1} = b_\infty := \E_{Z\sim \N(0,1)}[f'(\sigma_{V,\infty}Z)].
\end{gathered}
\end{equation*}
\item[(2)] There exist real sequences $\{\gamma_\ell\}_{\ell\geq 0}$, $\{\lambda_\ell\}_{\ell\geq 0}$, $\{e_\ell\}_{\ell\geq 0}$, $\{d_\ell\}_{\ell\geq 0}$ such that almost surely,
\begin{align*}
\gamma_{t,\ell} &:= \lim_{n,d\rightarrow\infty}\hat\gamma_{t,\ell}  = \gamma_\ell, 
\quad \ell < t, \quad \gamma_{t,t} := \lim_{n,d\rightarrow\infty}\hat\gamma_{t,t} = \sqrt{\sigma_{U,\infty}^2 - \sum_{\ell=0}^{t-1}\gamma_\ell^2},\\
\lambda_{t,\ell} &:= \lim_{n,d\rightarrow\infty} \hat\lambda_{t,\ell} = \lambda_\ell, \quad \ell < t, \quad\lambda_{t,t} := \lim_{n,d\rightarrow\infty}\hat\lambda_{t,t} 
 = \sqrt{\sigma_{V,\infty}^2 - \sum_{\ell=0}^{t-1} \lambda_\ell^2},
\end{align*}
and
\begin{align*}
\lim_{n,d\rightarrow\infty} \frac{1}{d}\z^{\ell\top}\z^t &= e_\ell, \quad \ell <t, \quad \lim_{n,d\rightarrow\infty} \frac{1}{d}\pnorm{\z^t}{}^2 = \sigma_{U,\infty}^2,\\
\lim_{n,d\rightarrow\infty} \frac{1}{d}\w^{\ell\top}\w^t &= d_\ell, \quad \ell <t, \quad \lim_{n,d\rightarrow\infty} \frac{1}{d}\pnorm{\w^t}{}^2 = \sigma_{V,\infty}^2.
\end{align*}
\end{enumerate}
\end{lemma}
\begin{proof}
(1) holds since $\sigma_{U,0}^2 = \sigma_{U,\infty}^2$ by definition, and the rest follows by the fixed point definition. For (2), the claims on $\hat\gamma_{t,t},\hat\lambda_{t,t}, \frac{1}{d}\pnorm{\z^t}{}^2,\frac{1}{d}\pnorm{\w^t}{}^2$ follow by (1), and the rest follows from a simple induction via the state evolution (\ref{eq:state_evolution_app}), with the base case verified by
\begin{align*}
\gamma_{t,0} = \lim_{n,d\rightarrow\infty} \frac{1}{\sqrt d}\iprod{\q^0}{\z^t} = \lim_{n,d\rightarrow\infty}\frac{\frac{1}{d}\iprod{\z^0}{\z^t}}{\frac{1}{\sqrt{d}}\pnorm{\z^0}{}} = \frac{\sigma_{U,\infty}\E[f(\sigma_{V,\infty}G)]}{\sigma_{U,\infty}} = \E[f(\sigma_{V,\infty}G)], 
\end{align*}
which is independent of $t$.
\end{proof}

The above result states that the state evolution (\ref{eq:state_evolution_app}) is determined by the four scalar sequences $\{\gamma_\ell\}_{\ell\geq 0}$, $\{\lambda_\ell\}_{\ell\geq 0}$, $\{e_\ell\}_{\ell\geq 0}$, $\{d_\ell\}_{\ell\geq 0}$. We now specify their values.

\begin{lemma}
Define the initialization
\begin{align}\label{eq:e0_def}
e_0 = \sigma_{U,\infty}\E[f(\sigma_{V,\infty}Z)].
\end{align}
Then $\{e_\ell\}_{\ell\geq 0}$ and $\{d_\ell\}_{\ell\geq 0}$ are recursively given by
\begin{align}
d_\ell &= \delta\E\Big[h\Big(\sqrt{e_\ell}G_1 + \sqrt{\sigma_{U,\infty}^2 - e_\ell}G_2\Big)h\Big(\sqrt{e_\ell}G_1 + \sqrt{\sigma_{U,\infty}^2 - e_\ell}G_3\Big)\Big], \label{eq:d_recursion}\\
e_{\ell+1} &= \E\Big[f\Big(\sqrt{d_\ell}G_1 + \sqrt{\sigma_{V,\infty}^2 - d_\ell}G_2\Big)f\Big(\sqrt{d_\ell}G_1 + \sqrt{\sigma_{V,\infty}^2 - d_\ell}G_3\Big)\Big],\label{eq:e_recursion}
\end{align}
where $G_1,G_2,G_3$ are i.i.d. $\N(0,1)$ variables. Moreover, let
\begin{align*}
\gamma_0 = \E[f(\sigma_{V,\infty}G)], \quad \lambda_0 = \frac{d_0}{\sigma_{V,\infty}},
\end{align*}
and the rest of $\{\gamma_\ell\}_{\ell\geq 0}$ and $\{\lambda_\ell\}_{\ell\geq 0}$ can be computed from $\{d_\ell\}_{\ell\geq 0}$ and $\{e_\ell\}_{\ell\geq 0}$ via
\begin{align}\label{eq:gamma_lambda_proof}
\gamma_{\ell+1} = \frac{e_{\ell+1} - \sum_{s=0}^\ell \gamma_s^2}{\sqrt{\sigma_{U,\infty}^2 - \sum_{s=0}^\ell \gamma_s^2}}, \quad \lambda_\ell = \frac{d_\ell - \sum_{s=0}^{\ell-1} \lambda_s^2}{\sqrt{\sigma_{V,\infty}^2 - \sum_{s=0}^{\ell-1}\lambda_s^2}}.
\end{align}
\end{lemma}
\begin{proof}
We first prove (\ref{eq:gamma_lambda_proof}). Using Lemma \ref{lem:stable_se}, $\gamma_{\ell+1}$ can be computed as
\begin{align}\label{eq:recursion_gamma}
\notag\gamma_{\ell+1} &= \lim_{n,d\rightarrow\infty} \frac{1}{\sqrt{d}}\q^{\ell+1\top}\z^{t+1} = \lim_{n,d\rightarrow\infty} \frac{1}{\sqrt{d}} \frac{(P_{\Q^{\ell}}^\perp\z^{\ell+1})^\top \z^{t+1}}{\pnorm{P_{\Q^\ell}^\perp \z^{\ell+1}}{}}\\
\notag&= \lim_{n,d\rightarrow\infty} \frac{\z^{\ell+1\top}\z^{t+1}/d - \sum_{s=0}^\ell \q^{s\top}\z^{t+1}/\sqrt{d} \cdot \q^{s\top}\z^{\ell+1}/\sqrt{d}}{\pnorm{P_{\Q^\ell}^\perp \z^{\ell+1}}{}/\sqrt{d}}\\
&= \frac{e_{\ell+1} - \sum_{s=0}^\ell \gamma_s^2}{\sqrt{\sigma_{U,\infty}^2 - \sum_{s=0}^\ell \gamma_s^2}},
\end{align}
as desired. Similarly, 
\begin{align}\label{eq:recursion_lambda}
\lambda_\ell &= \lim_{n,d\rightarrow\infty} \frac{1}{\sqrt{d}}\r^{\ell\top}\w^t = \lim_{n,d\rightarrow\infty} \frac{(P_{\bR^{\ell-1}}^\perp\w^\ell)^\top \w^t/d}{\pnorm{P_{\bR^{\ell-1}}^\perp\w^\ell}{}/\sqrt{d}} = \frac{d_\ell - \sum_{s=0}^{\ell-1} \lambda_s^2}{\sqrt{\sigma_{V,\infty}^2 - \sum_{s=0}^{\ell-1}\lambda_s^2}}.
\end{align}
Next we show (\ref{eq:d_recursion}) and (\ref{eq:e_recursion}). To this end, we define as in (\ref{eq:tilde_u_intro}) the first-order AMP approximation
\begin{align*}
\tilde\v^{t+1} &:= \sum_{\ell=0}^t \lambda_{t,\ell}\s^\ell = \sum_{\ell=0}^{t-1} \lambda_\ell \s^\ell + \sqrt{\sigma_{V,\infty}^2 - \sum_{\ell=0}^{t-1}\lambda_\ell^2}\s^t,\\
\tilde\u^t &:= \sum_{\ell=0}^t \gamma_{t,\ell}\g^\ell = \sum_{\ell=0}^{t-1} \gamma_\ell\g^\ell + \sqrt{\sigma_{U,\infty}^2 - \sum_{\ell=0}^{t-1}\gamma_\ell^2}\g^t,
\end{align*}
where $\{\g^\ell\}_{\ell=0}^t$ (resp. $\{\s^\ell\}_{\ell=0}^t$) are i.i.d. $\N(0,\I_n)$ (resp. $\N(0,\I_d)$) vectors, and we replace $\{\hat\lambda_{t,\ell}\}_{\ell=0}^t$ and $\{\hat\gamma_{t,\ell}\}_{\ell=0}^t$ by their asymptotic limits in Lemma \ref{lem:stable_se}. For $O(1)$ iterations, a standard asymptotic state evolution analysis yields that, almost surely, 
\begin{align*}
\lim_{n,d\rightarrow\infty}\frac{1}{\sqrt{d}}\pnorm{\v^{t+1}-\tilde\v^{t+1}}{} = \lim_{n,d\rightarrow\infty}\frac{1}{\sqrt{d}}\pnorm{\u^{t}-\tilde\u^{t}}{} = 0,
\end{align*}
and hence with $\tilde \g$ denoting a fresh $\N(0,\I_n)$ vector, it holds for any $\ell < t$ that 
\begin{align*}
\notag d_\ell &= \lim_{n,d\rightarrow\infty} \frac{1}{d}\iprod{\w^\ell}{\w^t} = \lim_{n,d\rightarrow\infty} \frac{1}{d}\iprod{h(\tilde\u^\ell)}{h(\tilde \u^t)}\\
\notag &= \lim_{n,d\rightarrow\infty} \frac{1}{d}\iprod{ h\Big(\sum_{j=0}^{\ell-1}\gamma_j\g^j + \sqrt{\sigma_{U,\infty}^2 - \sum_{j=0}^{\ell-1}\gamma_j^2}\g^\ell\Big)}{h\Big(\sum_{j=0}^{\ell-1}\gamma_j\g^j + \gamma_{\ell}\g^{\ell} + \sqrt{\sigma_{U,\infty}^2 - \sum_{j=0}^{\ell}\gamma_j^2}\tilde{\g}\Big)}\\
&= \delta\E\Big[h\Big(\sum_{s=0}^{\ell-1}\gamma_sg^s + \sqrt{\sigma_{U,\infty}^2 - \sum_{s=0}^{\ell-1}\gamma_s^2}g^\ell\Big)h\Big(\sum_{s=0}^{\ell}\gamma_sg^s + \sqrt{\sigma_{U,\infty}^2 - \sum_{s=0}^\ell \gamma_s^2}\tilde{g}\Big)\Big].
\end{align*}
Since
\begin{align*}
\Cov\Big(\sum_{s=0}^{\ell-1}\gamma_sg^s + \sqrt{\sigma_{U,\infty}^2 - \sum_{s=0}^{\ell-1}\gamma_s^2}g^\ell, \sum_{s=0}^{\ell}\gamma_sg^s + \sqrt{\sigma_{U,\infty}^2 - \sum_{s=0}^\ell \gamma_s^2}\tilde{g}\Big) = \sum_{s=0}^{\ell-1}\gamma_s^2 + \gamma_\ell\sqrt{\sigma_{U,\infty}^2 - \sum_{s=0}^{\ell-1}\gamma_s^2} = e_\ell, 
\end{align*}
where the last step follows from (\ref{eq:recursion_gamma}), we have
\begin{align*}
d_\ell = \delta\E\Big[h\Big(\sqrt{e_\ell}G_1 + \sqrt{\sigma_{U,\infty}^2 - e_\ell}G_2\Big)h\Big(\sqrt{e_\ell}G_1 + \sqrt{\sigma_{U,\infty}^2 - e_\ell}G_3\Big)\Big],
\end{align*}
as desired in (\ref{eq:d_recursion}). Similarly, with $\tilde\s$ denoting a fresh $\N(0,\I_d)$ vector, we have for any $\ell < t$ that
\begin{align*}
\notag e_{\ell+1} &= \lim_{n,d\rightarrow\infty} \frac{1}{d}\iprod{\z^{\ell+1}}{\z^{t+1}} = \lim_{n,d\rightarrow\infty} \frac{1}{d}\iprod{f(\v^{\ell+1})}{f(\v^{t+1})}\\
\notag&=\lim_{n,d\rightarrow\infty} \frac{1}{d}\iprod{ f\Big(\sum_{j=0}^{\ell-1}\lambda_j \s^j + \sqrt{\sigma_{V,\infty}^2 - \sum_{j=0}^{\ell-1}\lambda_j^2}\s^\ell\Big)}{f\Big(\sum_{j=0}^{\ell-1}\lambda_j \s^j + \lambda_\ell\s^\ell + \sqrt{\sigma_{V,\infty}^2 - \sum_{j=0}^{\ell}\lambda_j^2}\tilde{\s}\Big)}\\
&= \E\Big[f\Big(\sum_{j=0}^{\ell-1}\lambda_j s^j + \sqrt{\sigma_{V,\infty}^2 - \sum_{j=0}^{\ell-1}\lambda_j^2}s^\ell\Big)f\Big(\sum_{j=0}^\ell \lambda_j s^j + \sqrt{\sigma_{V,\infty}^2 - \sum_{j=0}^\ell \lambda_j^2}\tilde s\Big)\Big].
\end{align*}
By (\ref{eq:recursion_lambda}), we have
\begin{align*}
\Cov\Big(\sum_{j=0}^{\ell-1}\lambda_j s^j + \sqrt{\sigma_{V,\infty}^2 - \sum_{j=0}^{\ell-1}\lambda_j^2}s^\ell, \sum_{j=0}^\ell \lambda_js^j  
 + \sqrt{\sigma_{V,\infty}^2 - \sum_{j=0}^\ell \lambda_j^2}\tilde s\Big) = \sum_{j=0}^{\ell-1} \lambda_j^2 + \lambda_\ell\sqrt{\sigma_{V,\infty}^2 - \sum_{j=0}^{\ell-1}\lambda_j^2} = d_\ell, 
\end{align*}
hence as desired in (\ref{eq:e_recursion}),
\begin{align*}
e_{\ell+1} = \E\Big[f\Big(\sqrt{d_\ell}G_1 + \sqrt{\sigma_{V,\infty}^2 - d_\ell}G_2\Big)f\Big(\sqrt{d_\ell}G_1 + \sqrt{\sigma_{V,\infty}^2 - d_\ell}G_3\Big)\Big].
\end{align*}
\end{proof}

With $G_1,G_2,G_3$ denoting i.i.d. $\N(0,1)$ variables, let
\begin{equation}\label{eq:def_Tf}
\begin{aligned}
T_f(x) &:= \E\Big[f(\sqrt{x}G_1 + \sqrt{\sigma_{V,\infty}^2 - x}G_2)f(\sqrt{x}G_1 + \sqrt{\sigma_{V,\infty}^2 - x}G_3)\Big], \quad x \in [0,\sigma_{V,\infty}^2],\\
T_h(x) &:= \delta\E\Big[h(\sqrt{x}G_1 + \sqrt{\sigma_{U,\infty}^2 - x}G_2)h(\sqrt{x}G_1 + \sqrt{\sigma_{U,\infty}^2 - x}G_3)\Big], \quad x\in [0,\sigma_{U,\infty}^2].
\end{aligned}
\end{equation}
Then the recursions (\ref{eq:d_recursion}) and (\ref{eq:e_recursion}) can be written as
\begin{align*}
e_{\ell+1} = T_f(d_\ell), \quad d_\ell = T_h(e_\ell),
\end{align*}
with initialization $e_0$ given by (\ref{eq:e0_def}). In view of (\ref{eq:var_fixed_point}), the above recursion admits a fixed point $(d_\infty,e_\infty)$ given by
\begin{align*}
d_\infty = \sigma_{V,\infty}^2, \quad e_\infty = \sigma_{U,\infty}^2.
\end{align*}
The following result identifies the Bolthausen condition (\ref{eq:bolthausen_app}) as a sufficient condition for the convergence to this fixed point, which, with the notation (\ref{eq:def_Tf}), is equivalent to
\begin{align*}
T_f'(\sigma_{V,\infty}^2)\cdot T_h'(\sigma_{U,\infty}^2) < 1.
\end{align*}
\begin{lemma}\label{lem:bolthause_convexity}
\begin{itemize}
\item[(1)] $T_f$, $T_h$, $T_1:= T_f \circ T_h$, and $T_2 := T_h \circ T_f$ are non-decreasing and convex.
\item[(2)] Suppose that $\Var[f(\sigma_{V,\infty}G)] > 0$ so that $e_0 = \sigma_{U,\infty}\E[f(\sigma_{V,\infty}G)] < \sigma_{U,\infty} \sqrt{\E[f^2(\sigma_{V,\infty}G)]}= \sigma_{U,\infty}^2$. For all $t\geq 0$, it holds that
\begin{align*}
\sum_{s=0}^{t-1} \lambda_s^2 < d_t < \sigma_{V,\infty}^2, \quad \sum_{s=0}^t \gamma_s^2 < e_{t+1} < \sigma_{U,\infty}^2.
\end{align*}
\item[(3)] Suppose $T_f'(\sigma_{V,\infty}^2)T_h'(\sigma_{U,\infty}^2) < 1$, then the following limits exist and hold:
\begin{align*}
\sum_{s=0}^\infty \lambda_s^2 = \lim_{t\rightarrow\infty} d_t = \sigma_{V,\infty}^2, \quad \sum_{s=0}^\infty \gamma_s^2 = \lim_{t\rightarrow\infty} e_t = \sigma_{U,\infty}^2.
\end{align*}
Moreover, for any $t\geq 0$,
\begin{align*}
\sigma_{V,\infty}^2 - d_{t+1} &\leq T_f'(\sigma_{V,\infty}^2)T_h'(\sigma_{U,\infty}^2)\big(\sigma_{V,\infty}^2 - d_t\big),\\
\sigma_{U,\infty}^2 - e_{t+1} &\leq T_f'(\sigma_{V,\infty}^2)T_h'(\sigma_{U,\infty}^2)\big(\sigma_{U,\infty}^2 - e_t\big). 
\end{align*}
\end{itemize}
\end{lemma}
\begin{proof}
We have
\begin{align}\label{eq:Th_derivative}
\notag T_h'(x) &= 2\delta\E\Big[h'(\sqrt{x}G_1 + \sqrt{\sigma_{U,\infty}^2 - x}G_2)\Big(\frac{1}{2\sqrt{x}}G_1 - \frac{1}{2\sqrt{\sigma_{U,\infty}^2 - x}}G_2\Big)h(\sqrt{x}G_1 - \sqrt{\sigma_{U,\infty}^2 - x}G_3)\Big]\\
\notag &= \delta\E\Big[h'(\sqrt{x} G_1 + \sqrt{\sigma_{U,\infty}^2 - x}G_2)h'(\sqrt{x}G_1 + \sqrt{\sigma_{U,\infty}^2 - x}G_3)\Big]\\
&= \delta\E \Big(\E\Big[h'(\sqrt{x}G + \sqrt{\sigma_{U,\infty}^2 - x}\tilde{G})|G\Big]\Big)^2 \geq 0,
\end{align}
where the first step follows from symmetry, and the second follows from Gaussian integration by parts. A similar calculation yields that
\begin{align*}
T_h''(x) &= \delta\E\Big[h''(\sqrt{x} G_1 + \sqrt{\sigma_{U,\infty}^2 - x}G_2)h''(\sqrt{x}G_1 + \sqrt{\sigma_{U,\infty}^2 - x}G_3)\Big]\\
&= \delta\E \Big(\E\Big[h''(\sqrt{x}G + \sqrt{\sigma_{U,\infty}^2 - x}\tilde G)|G\Big]\Big)^2 \geq 0.
\end{align*}
A similar argument for $T_f(\cdot)$ yields that $T_h$ and $T_f$ are both non-decreasing and convex. Elementary considerations show that $T_1$ and $T_2$ are also non-decreasing and convex.

(2) Clearly it holds that $d_t < \sigma_{V,\infty}^2$ and $e_{t+1} < \sigma_{U,\infty}^2$ for any $t\geq 0$. We will prove via induction that $\sum_{s=0}^{t-1} \lambda_s^2 < d_t$, and a similar argument will yield $\sum_{s=0}^t \gamma_s^2 < e_{t+1}$. The base case $t= 0$ clearly holds. Assume the claim holds for some $t$, then
\begin{align*}
\lambda_t = \frac{d_t - \sum_{s=0}^{t-1}\lambda_s^2}{\sqrt{\sigma_{V,\infty}^2 - \sum_{s=0}^{t-1} \lambda_s^2}} < \sqrt{d_t - \sum_{s=0}^{t-1}\lambda_s^2},
\end{align*}
which implies $\sum_{s=0}^t \lambda_s^2 < d_t$ and hence the claim since $d_t < d_{t+1}$.

(3) Since $\{d_t\}_{t\geq 0}$ is non-decreasing and bounded by $\sigma_{V,\infty}^2$, it must admit a limit, and under the condition $T_f'(\sigma_{V,\infty}^2)T_h'(\sigma_{U,\infty}^2) < 1$ this limit must be $\sigma_{V,\infty}^2$. Similarly, the limit $\sum_{s=0}^\infty \lambda_s^2$ exists. If $\sum_{s=0}^\infty \lambda_s^2 < \sigma_{V,\infty}^2$, then $\lim_{t\rightarrow\infty} \lambda_t = \sqrt{\sigma_{V,\infty}^2 - \sum_{s=0}^\infty \lambda_s^2} > 0$, which contradicts the boundedness of $\sum_{s=0}^\infty \lambda_s^2$. Lastly, since $d_{t+1} = T_2(d_t)$, we have by convexity of $T_2$ that
\begin{align*}
\sigma_{V,\infty}^2 - d_{t+1} = T_2(\sigma_{V,\infty}^2) - T_2(d_t) \leq T_2'(\sigma_{V,\infty}^2)(\sigma_{V,\infty}^2 - d_t) = T_f'(\sigma_{V,\infty}^2)T_h'(\sigma_{U,\infty}^2)(\sigma_{V,\infty}^2 - d_t),
\end{align*}
as desired.
\end{proof}

Applying Lemma \ref{lem:bolthause_convexity}, the desired AMP convergence (\ref{eq:amp_converge_app}) readily holds since almost surely, 
\begin{align*}
\lim_{t\rightarrow\infty}\lim_{n,d\rightarrow\infty}\frac{1}{d}\pnorm{\z^{t} - \z^{t+1}}{}^2 = \lim_t  2(\sigma_{U,\infty}^2 - e_t) = 0, \quad \lim_{t\rightarrow\infty}\lim_{n,d\rightarrow\infty}\frac{1}{d}\pnorm{\w^{t} - \w^{t+1}}{}^2 = \lim_t  2(\sigma_{V,\infty}^2 - d_t) = 0.
\end{align*}

\end{document}